\documentclass[11pt]{article}

\usepackage[english]{babel}

\usepackage[letterpaper,top=2cm,bottom=2cm,left=3cm,right=3cm,marginparwidth=2cm]{geometry}

\usepackage{bm} 
\usepackage{amsmath}
\usepackage{graphicx}
\usepackage[colorlinks=true, allcolors=black]{hyperref}
\newcommand{\w}{{\mathrm{w}}}
\usepackage{amsthm}
\usepackage{amsmath}
\usepackage{amssymb}
\usepackage{mathrsfs}
\usepackage{graphicx}
\usepackage{caption}
\usepackage{subcaption}
 \usepackage[colorlinks=true]{hyperref}
\usepackage{ucs}
\usepackage{yfonts}
\usepackage{bbm}
\usepackage{geometry}
\usepackage{graphicx}
\usepackage{enumerate}
\usepackage{xifthen}
\usepackage{upgreek} 
\usepackage{stmaryrd} 
\newcommand{\esp}[2]{%
\ifthenelse{\equal{#2}{}}{{E}\left[ #1 \right] }{{E}^{#2}\left[ #1 \right] }%
}

\newcommand{\espk}[2]{%
\ifthenelse{\equal{#2}{}}{\mathbb{E}^{\dag}\left[ #1 \right] }{\mathbb{E}^{\dag}_{#2}\left[ #1 \right] }%
}

\usepackage{mathrsfs}
\usepackage{bm} 
\usepackage{mathabx} 
\usepackage{geometry}
\newtheorem{theo}{Theorem}
\newtheorem{lem}[theo]{Lemma}
\newtheorem{prop}[theo]{Proposition}
\newtheorem{def1}[theo]{Definition}
\newtheorem{cor}[theo]{Corollary}

\newtheoremstyle{named}{}{}{\itshape}{}{\bfseries}{.}{.5em}{\thmnote{#3}}
\theoremstyle{named}

\usepackage{tcolorbox}

\numberwithin{equation}{section}
\numberwithin{theo}{section}

\newcommand{\dd}{{\mathrm{d}}}

\renewcommand{\epsilon}{\varepsilon}

\usepackage{array} 
\newcounter{hypocounter}

\usepackage{pict2e,picture,graphicx}
\usepackage{mathtools}
\usepackage{todonotes}
\makeatletter
\DeclareRobustCommand{\Arrow}[1][]{%
\check@mathfonts
\if\relax\detokenize{#1}\relax
\settowidth{\dimen@}{$\m@th\rightarrow$}%
\else
\setlength{\dimen@}{#1}%
\fi
\sbox\z@{\usefont{U}{lasy}{m}{n}\symbol{41}}%
\begin{picture}(\dimen@,\ht\z@)
\roundcap
\put(\dimexpr\dimen@-.7\wd\z@,0){\usebox\z@}
\put(0,\fontdimen22\textfont2){\line(1,0){\dimen@}}
\end{picture}%
}
\makeatother

\makeatletter
\newcommand*\bigcdot{\mathpalette\bigcdot@{.5}}
\newcommand*\bigcdot@[2]{\mathbin{\vcenter{\hbox{\scalebox{#2}{$\m@th#1\bullet$}}}}}
\makeatother

\DeclareMathSymbol{\shortminus}{\mathbin}{AMSa}{"39}

\newcommand{\trev}{\sigma_h \shortminus \hspace{0,2mm}  \bigcdot}

\usepackage{titling} 

\thanksmarkseries{fnsymbol} 

\makeatletter
\renewcommand*{\@fnsymbol}[1]{\ensuremath{\ifcase#1\or *\or \dagger\or \ddagger\or
   \mathsection\or \mathparagraph\or \|\or **\or \dagger\dagger
   \or \ddagger\ddagger \else\@ctrerr\fi}}
\makeatother
\title{Excursion theory for Markov processes indexed by Lévy trees}
\author{ Armand Riera\thanks{Sorbonne Universit\'e, LPSM, France.\hfill  \texttt{riera@lpsm.paris}}  ~ and ~ Alejandro Rosales-Ortiz\thanks{Universit\"at Z\"urich, Institute of Mathematics, Switzerland.\hfill  \texttt{alejandro.rosalesortiz@math.uzh.ch}}.}

\date{}

\newcommand{\lambdaC}{\resizebox{3 mm}{2.5mm}{$\curlywedge$}}

\begin{document}
\maketitle
\vspace{-7mm}
\begin{abstract}
We develop an excursion theory that describes the evolution of a Markov process indexed by a Lévy tree away from a regular and instantaneous point $x$ of the state space. The theory builds upon a notion of local time at $x$ that was recently introduced in  \cite{2024structure}.   Despite the radically different setting, our results exhibit striking similarities to the classical excursion theory for $\mathbb{R}_+$-indexed Markov processes. We then show that the genealogy of the excursions  can be encoded in a Lévy tree  called the tree coded by the local time. In particular,  we recover by different methods the excursion theory of Abraham and Le Gall \cite{ALG15}, which was developed for Brownian motion indexed by the Brownian tree. 
\end{abstract}

\tableofcontents

\section{Introduction}

The purpose of this work is to develop an excursion theory for Markov processes  indexed by  Lévy trees, away from a regular and instantaneous point $x$ of the state space.   This theory builds upon  a notion of local time process introduced recently in the companion paper \cite{2024structure}. Informally,   a Markov process indexed by a rooted $\mathbb{R}$-tree can be understood as follows:  the Markov process starts at the root of the tree at a point in the state space, travels along the branches of the tree away from the root, and at every branching point it splits into independent copies that continue to evolve with the same dynamics.  In this work, the tree itself is random, and to ensure that the tree-indexed process possesses basic Markovian properties  the family of random trees that we consider are  Lévy trees. This class of random rooted $\mathbb{R}$-trees plays a central role in modern probability theory. For instance, they appear as scaling limits of Galton-Watson trees \cite{DLG02} and, as their discrete analogs, they  are characterised by a branching property  \cite{WeillLevyTrees}.  One notable member of this family is the celebrated Brownian tree\footnote{Also known as Aldous' Continuum Random Tree \cite{AdousCRT} with free total mass.} which has played a key role in the study of different models of random geometry.  We refer to   \cite{AdousCRT, AldousTreeBasedModels, Btree} and references within for background and related works. The concept of a Markov process indexed by a Lévy tree has appeared over the years in various settings, making it  also   an important  canonical probabilistic object. They are closely related to  the theory of super-processes \cite{DLG02}, and   Brownian motion indexed by the Brownian tree has been essential  for the study of the so-called Brownian geometry -- a family of random surfaces arising as scaling limits of random  planar maps \cite{baur2016classification, bettinelli2022compact, CLGmodif, LG11,Mie11}. We refer to the discussion at the end of this introduction for a more detailed account. 
\par 
In this work, we show that it is possible to build a theory that describes the evolution of the tree-indexed process  between visits to the point $x$ along the branches of the Lévy tree. We shall see that, despite the radically different setting, this theory shares striking similarities with  classical  excursion theory for $\mathbb{R}_+$-indexed Markov processes.  To highlight these similarities,  let us start by briefly recalling some  elements of the classical theory. 
\par 
 Consider a strong Markov process $(\xi_t)_{t  \geq 0}$ taking values in some Polish space $E$, and suppose there exists a point $x$ of the state space which is  instantaneous and regular  for the process. For the purposes of this work, we further assume that the motion is  continuous,  the point $x$ is recurrent and   $\int_0^\infty \mathrm{d}s~\mathbbm{1}_{\xi_s=x}=0$ a.s. The path of 
 the Markov process can be decomposed in excursions away from $x$, where if we write $(a_i,b_i)_{i \in \mathcal{I}}$ for the connected components of $\mathbb{R}_+ \setminus \{ t \geq 0 : \xi_t = x \}$, the excursion associated with the excursion interval $(a_i,b_i)$ is the piece of path given  by $\xi^{i,*} := (\xi_{(a_i + t)\wedge b_i } : t \geq 0)$. The study of the family of excursions relies on  a remarkable continuous additive functional $(\mathcal{L}_t)_{t \geq 0}$ of the Markov process named the local time at $x$. This process is unique up to a multiplicative constant that we fix arbitrarily. Due to the fact that the support of the Stieltjes measure $\dd \mathcal{L}$  is  precisely the set $\{t \geq 0 : \xi_t = x \}$, the local time is well suited to index the family of excursions,  and the order induced by $(\mathcal{L}_t)_{t \geq 0}$  is consistent with the order induced by time. The main result of the theory   states that the point measure 
\begin{equation*}
   \sum_{i \in \mathcal{I}} \delta_{(\mathcal{L}_{a_i}, \xi^{i,*})}
\end{equation*}
is a Poisson point measure with intensity $\mathbbm{1}_{\mathbb{R}_+}(t)\dd t \otimes \mathcal{N}$, where $\mathcal{N}$ is an infinite measure called the excursion measure.  Informally,  one can think of $\mathcal{N}$ as describing the law of a typical excursion. For historical reasons, the point measure in the last display is often called the excursion process of $(\xi_t)_{t \geq 0}$. In addition to the collection of excursions, there is another closely related family of trajectories that plays an important  role in the theory. Namely,  it is natural to consider the family of trajectories obtained by shifting $(\xi, \mathcal{L})$ at the starting point of an excursion, i.e. $(\xi^i,\mathcal{L}^i):=((\xi_{a_i+t},\mathcal{L}_{a_i+t}):~t\geq 0)$, $i\in \mathcal{I}$.  Although this family of trajectories cannot possess a Poissonian structure due to the lack of independence, the theory of exit systems introduced by Maisonneuve \cite{Maisonneuve} gives a framework for their study. In particular, \cite{Maisonneuve} provides a collection of measures $\mathcal{N}_{x,r}$, $r \geq 0$, which, roughly speaking, encode the law of  a typical element of this family. Here, $r \geq 0$ stands for the value of the local time at the starting point of the trajectory. 
 \par 
Let us now start by giving an informal overview of the results of this work. Precise versions, along with related definitions and notations, will be discussed  afterwards.  Lévy trees are constructed from excursions above the running infimum of a subclass of spectrally positive Lévy processes, and their law is  determined by the corresponding Laplace exponent $\psi$. Any Lévy tree is equipped with a clockwise exploration that informally, starts at the root of the tree and travels along the branches from ``left to right", following the contour of the tree. From the corresponding  $\psi$-Lévy excursion, one can construct a functional  $H= (H_t)_{t \geq 0}$ called the height function, which encodes at every $t\geq 0$  the distance to the root during the clockwise exploration. The tree and its clockwise exploration can be fully recovered from $H$, which is why we denote the corresponding tree by  $\mathcal{T}_H$. These constructions hold under the excursion measure above the running infimum of the $\psi$-Lévy process, as well as under the underlying probability measure. In the former case, the tree $\mathcal{T}_H$ is compact and  referred to as a $\psi$-Lévy tree.  In the latter, the resulting tree is no longer compact, but it can be thought as a concatenation of $\psi$-Lévy trees at their root - one for each excursion of the Lévy process above  its ruining infimum.  For this reason, $\mathcal{T}_H$  under the underlying probability measure is referred to as a forest of $\psi$-Lévy trees. For each point  $a \in \mathcal{T}_H$, we write $(\xi_a, \mathcal{L}_a)$ for the value at the point $a$ of the Markov process paired with its local time at $x$.  The tree-indexed process $(\xi_a, \mathcal{L}_a)_{a \in \mathcal{T}_H}$  should be interpreted as the Markov process   $(\xi_t, \mathcal{L}_t)_{t \in \mathbb{R}_+}$ indexed by the Lévy tree $\mathcal{T}_H$, and we start it from the root of $\mathcal{T}_H$ at $(x,0)$.   We often think of  $(\xi_a, \mathcal{L}_a)_{a \in \mathcal{T}_H}$  as a collection of labels for $\mathcal{T}_H$. The role played  by the set  $\{t \geq 0: \xi_t = x \}$ in classical excursion theory is now taken over by
 \begin{equation}\label{intro:puntosx}
     \mathscr{Z} := \{ a \in \mathcal{T}_H : \xi_a = x  \},
 \end{equation}
 which is a  random subset of $\mathcal{T}_H$. The structure of this set was studied in detail in the companion paper \cite{2024structure},  and in this work we shall  investigate  the  family of excursions away from $x$.  
 \par 
The excursions consist of the restrictions of  $(\xi_a)_{a \in \mathcal{T}_H}$ to the closures of connected components  of $\mathcal{T}_H \setminus  \mathscr{Z}$, which we refer to as excursion components.  The closest point of an excursion component  to the root of $\mathcal{T}_H$ is called  a debut point, and we denote the collection of debut points  by $\mathcal{D}$. Excursion components and  debut points  are in one-to-one  correspondence.  For every $u \in \mathcal{D}$, we write $\mathcal{C}_u$ for the  associated excursion component and we set $\xi^{u,*} = (\xi_a :  a \in \mathcal{C}_u)$  for the corresponding excursion.  The interiors of  excursion components play a  role analogous  to the one of excursion intervals in the time-indexed setting, but the former are of a more intricate nature due to the fact that each $\mathcal{C}_u$ is itself a random tree. In particular, each excursion is a tree-indexed process in its own right. The family of excursions $(\xi^{u,*})_{u \in \mathcal{D}}$  can be indexed by means of  a remarkable additive functional $(A_t)_{t \geq 0}$, recently introduced in \cite{2024structure},  which  plays a crucial role in this work. Informally, at each time $t \geq 0$, the variable $A_t$ measures  the amount of time spent by the clockwise exploration of the tree at points with label $x$  up to time $t$. Hence, the process $(A_t)_{t \geq 0}$ plays the role of the local time at $x$ for the tree-indexed process when the tree is explored in clockwise order. We mark each excursion $\xi^{u,*}$, $u \in \mathcal{D}$,  with the value of $(A_t)_{t \geq 0}$ at the  time of first visit to $u$ by the clockwise exploration.  When working under the underlying probability measure, we prove that the family of excursions indexed by  their corresponding  mark  is  of Poissonian nature [Theorem \ref{theorem:excursionPPP}],    and  their law can be described by an (infinite) measure $\mathbf{N}_x^{*}$ playing the role of $\mathcal{N}$ in the tree indexed setting. For this reason, we named $\mathbf{N}_x^*$ the excursion measure away from $x$ of the tree-indexed process $(\xi_a)_{a \in \mathcal{T}_H}$. Our result shows that despite the complex structure of the underlying tree, the family of excursions ordered with respect to the clockwise exploration, are roughly speaking i.i.d. and distributed according to $\mathbf{N}_x^{*}$. The proof of this theorem relies on a precise description of not only the excursions, but as well of their whole descendant line in the tree $\mathcal{T}_H$. In this direction, we introduce a closely related family of tree-indexed processes $(\xi^{u}, \mathcal{L}^u)_{u \in \mathcal{D}}$, where each $(\xi^{u},\mathcal{L}^u)$ consists in the restriction of $(\xi_a, \mathcal{L}_a)_ {a \in \mathcal{T}_H}$ to the subtree stemming from $u$, i.e. $\{ a \in \mathcal{T}_H : u \preceq a \}$.  Here   $\preceq$ stands for  the genealogical order in the tree, in the sense that we write $u \preceq a$ if $u$ belongs to the geodesic path connecting $a$ to the root. In particular, the excursion $\xi^{u,*}$ is just the restriction of $\xi^u$ to the excursion component $\mathcal{C}_u$ and we call the family  $(\xi^{u}, \mathcal{L}^u)_{u \in \mathcal{D}}$  the subtrajectories stemming from $u \in \mathcal{D}$. We introduce a collection of measures $\mathbf{N}_{x,r}$, $r \geq 0$,   and we show that these  describe the law of a typical element of the family  $(\xi^{u}, \mathcal{L}^u)_{u \in \mathcal{D}}$,  where $r \geq 0$ stands for  the value of the local time $\mathcal{L}^u$ at the debut $u$. The relationship is established by means of an \textit{exit formula}  [Theorem \ref{theorem:exit}] which shares  strong resemblance with a classic result  of   Maisonneuve  \cite[Theorem 4.1]{Maisonneuve},  where the role played by $(\mathcal{N}_{x,r}, r \geq 0)$ in the classic theory is now played by $(\mathbf{N}_{x,r}: r \geq 0)$. Notably, the family $(\mathbf{N}_{x,r}$:  $r \geq 0)$ admits a very simple representation:  the measure  $\mathbf{N}_{x,r}$, for $r \geq 0$, can be thought of as the law of the Markov process distributed according to $\mathcal{N}_{x,r}$ indexed by a $\psi$-Lévy tree. The excursion measure $\mathbf{N}_x^*$ is simply given by the pushforward of  $\mathbf{N}_{x,0}$ through a pruning operation on the branches of the tree at the first return  to $x$ by the labels.  In contrast with the time-indexed setting where the boundary of an excursion interval consists of only two points, the boundary of an excursion component\footnote{The set of points in the component with label $x$}  is in general a rather complex fractal subset of $\mathcal{T}_H$. We associate to each excursion component $\mathcal{C}_u$ a non-negative variable $\ell_u$  that can be interpreted as its boundary size.  It is a  non-trivial matter  to prove that one can make sense, in a canonical way, of the notion of boundary size $\ell_u$ of every such excursion component and Section \ref{section:L} is entirely devoted to this task. Our approach heavily relies on the theory of exit local times, we refer to \cite{DLG02, 2024structure} for background. In the last part of the work, we argue that the genealogy of the excursions and their respective  boundary sizes $(\ell_u)_{u \in \mathcal{D}}$ can be encoded in another Lévy tree  $\mathcal{T}_{\widetilde{H}}$,  with explicit height process $\widetilde{H}$,  which has already been studied in the work \cite{2024structure}. In short, $\mathcal{T}_{\widetilde{H}}$ can be obtained from $\mathcal{T}_H$ by identifying each excursion component $\mathcal{C}_u$ in a single point, say $[u]$. This tree encodes by construction the genealogy induced by $\mathcal{T}_H$ on the excursions, but in fact the relationship is  more precise. Our last main result yields that every such point $[u]$, issued from identifying an excursion component with non-null boundary size, is a branching point of infinite  multiplicity\footnote{This is, $\mathcal{T}_{\widetilde{H}} \setminus [u]$ has infinite connected components.}   of the tree ${\mathcal{T}_{\widetilde{H}}}$, and that its ``mass'' in the sense of \cite[Theorem 4.7]{FractalAspectsofLevyTrees} is precisely $\ell_u$.   As a consequence we deduce that,  conditionally on $\mathcal{T}_{\widetilde{H}}$, the family of excursions  with positive boundary size are independent, and the respective conditional distribution of every such excursion $\xi^{u,*}$ is given by $\mathbf{N}^{*,\ell_u}_x$ where the measure $\mathbf{N}^{*,\ell}_x$, for $\ell > 0$, should read as the law of an excursion conditioned to have boundary size equal to $\ell$. 
\par 
The preceding discussion is informal and for instance we did not give a rigorous definition for the tree indexed Markov process. Defining an appropriate framework for  random processes indexed by random sets is a challenge on its own right,  and to circumvent this difficulty we rely in the formalism of Lévy snakes developed in \cite{DLG02,2024structure}. We stress that this is certainly not a disadvantage: 
the Lévy snake provides a temporal exploration of the resulting labelled tree which is essential for our purposes. An  overview of the theory can be found in Section \ref{section:framework},  here we limit ourselves to recalling  some of its main elements. The tree-indexed process $(\xi_a, \mathcal{L}_a)_{a \in \mathcal{T}_H}$ is made of  two layers of randomness: on the one hand, the underlying Lévy tree $\mathcal{T}_H$ and on the other hand, the corresponding labels $(\xi_a, \mathcal{L}_a)$, for $a \in \mathcal{T}_H$.   We begin by briefly reviewing the theory of Lévy trees, and we start  with the construction of a tree from a continuous, non-negative function.  This procedure is standard, we refer to  \cite{Duq.trees.notes} for background on coding of trees by continuous functions. For any continuous, non-negative function $h: \mathbb{R}_+ \to \mathbb{R}_+$, we introduce the pseudo-distance 
\begin{equation*}
d_{h}(s,t):=h(s)+h(t) - 2\cdot \min\limits_{[s\wedge t ,  s\vee t]} h, \quad \text{ for } s,t \geq 0,
\end{equation*}
and with the convention $\inf\emptyset =\infty$ we write $\sigma_h=\inf\{s\geq 0:~h_t=0 \text{ for every } t \geq s\}$ for the lifetime of $h$.   With the convention $[0,\infty]=[0,\infty)$, the pseudo-distance $d_h$ induces an equivalence relation on $[0,\sigma_h]$  by identifying $s \sim_h t$ when  $d_h(s,t) = 0$. We set $\mathcal{T}_h := [0,\sigma_h] / \sim_h$ for the resulting quotient space and we write  $p_h = (p_h(t):t \in [0,\sigma_h] )$ for the canonical projection from $[0,\sigma_h]$  onto $\mathcal{T}_h$.  The metric $d_h$ is well defined in the quotient space $\mathcal{T}_h$, and we call $p_h$  the clockwise exploration of $\mathcal{T}_h$. The pointed metric space $(\mathcal{T}_h, d_h, p_h(0))$ is the real tree coded by the function $h$ in the sense of \cite{Duq.trees.notes}. For the sake of simplicity we   often just write  $\mathcal{T}_h$ for the triplet $(\mathcal{T}_h, d_h, p_h(0))$. The multiplicity of a point $a \in \mathcal{T}_h$ is the number of connected components of $\mathcal{T}_{h}\setminus \{ a \}$. The points of multiplicity $1$ are called leaves, the ones of multiplicity  at least $2$ are named interior points, and any point of multiplicity strictly greater than $2$ is referred to as a branching point. We now randomise the coding function $h$, and in this direction we write $P$ for the law of a $\psi$-Lévy process $X$ and  $N$ for the associated   excursion measure of $X$ over its running infinimum. Under both measures,  we still write $H$ for the  height  function of $X$. In particular, the random variable  $\sigma_H$  under   $P$ is infinite, and it is finite  under $N$. The pointed metric space $\mathcal{T}_H$ under $N$ is a compact real tree that we refer to  as the $\psi$-Lévy tree coded by $H$.    This  construction  still makes sense under $P$ and, as we already mentioned, the resulting tree $\mathcal{T}_H$ is non-compact and referred to as a forest of $\psi$-Lévy trees.  When the underlying Lévy process is a Brownian motion, $\mathcal{T}_H$ is the Brownian tree.  In this case, all branching points are of multiplicity $3$, due to the fact that  local minima of Brownian motion are distinct. In contrast, when the Lévy process has no Brownian component, every branching point is of infinite multiplicity. The height process is in general not a Markov process, and this makes its direct  study a rather difficult task. To circumvent this difficulty, another closely related process was introduced in \cite{LGLJ98} that has been key in the study of Lévy trees. The so-called exploration process $\rho = (\rho_t: t \geq 0)$ is a rcll (right-continuous with left limits)  strong Markov process with values in $\mathcal{M}_f(\mathbb{R}_+)$, the set of finite measures in $\mathbb{R}_+$ endowed with the weak topology. It is  defined in terms of the $\psi$-Lévy process/excursion, and at time $t = 0$ it starts at the null measure $0$.  Informally, at each fixed $t \in  [0,\sigma_H]$, the variable $\rho_t$ encodes the heights of the subtrees attached to the right of the geodesic path connecting the root of the tree to the point $p_H(t) \in \mathcal{T}_H$. If for any element $\mu \in \mathcal{M}_f(\mathbb{R}_+)$ we write $H(\mu)$ for  the supremum of its topological support,  the exploration process and the height process are related by the identity $H_t = H(\rho_t)$, for $t \geq 0$. Due to its Markovian  nature, the study of $H$ and the underlying Lévy tree $\mathcal{T}_H$  often relies on the exploration process $(\rho_t : t \geq 0)$. The second layer of randomness of the tree-indexed Markov process is given by the labels on top of the Lévy tree, and we now explain how these can be defined by making use of the formalism of snakes. To simplify notation, set $\overline{E} := E \times \mathbb{R}_+$ and   write $\mathcal{W}_{\overline{E}}$
for the collection of finite $\overline{E}$-valued paths. Hence, any element $\overline{\w} \in \mathcal{W}_{\overline{E}}$ is a map $\overline{\w}: \overline{E} \mapsto [0,\zeta_{\overline{\w}}]$ where $\zeta_{\overline{\w}}$ is a finite non-negative  number called its lifetime. We equip the space $\mathcal{W}_{\overline{E}}$   with a metric, whose definition is recalled in  \eqref{d:W:E} and under which $\mathcal{W}_{\overline{E}}$ is a complete metric space.   Now, conditionally on $H$,  we define a $\mathcal{W}_{\overline{E}}$-valued process $\overline{W} = (W_t,\Lambda_t)_{t \geq 0}$ satisfying that, for every fixed $t \geq 0$, the pair  $(W_t, \Lambda_t)$ is a random element of  $\mathcal{W}_{\overline{E}}$ with lifetime $\zeta_{\overline{W}_t} = H_t$, and is distributed as  the Markov process $(\xi_h, \mathcal{L}_h)_{h \geq 0}$ started at $(x,0)$  restricted to $[0,H_t]$. Further, its dynamics as $t \geq 0$ varies can be understood  as follows: informally, when $H$ decreases, the path is erased from its tip and, when $H$ increases, the path is extended by adding independent  “little pieces” of trajectories of $(\xi_h,\mathcal{L}_h)_{h \geq 0}$ at the tip. The  key property  for  defining a $\mathcal{T}_H$-indexed process in terms of the pair $(W , \Lambda)$ is that    $(W , \Lambda)$ is compatible with the equivalence relation $\sim_H$. Namely, for every $s , t  \geq 0$, we have
\begin{equation}\label{equation:intro*}
    \big(W_s(h), \Lambda_s(h)\big) = \big(W_t(h), \Lambda_t(h)\big), \quad \text{ for every } 0 \leq h \leq \min_{[s \wedge t, s \vee t]}H. 
\end{equation}
 In particular, if for $t \geq 0$ we write $(\widehat{W}_t, \widehat{\Lambda}_t):= (W_t(H_t), \Lambda_t(H_t))$ for the tip of the path $(W_t, \Lambda_t)$,  we can assign a label $(\xi_a, \mathcal{L}_a)$  on any $a \in \mathcal{T}_H$ by setting $(\xi_a, \mathcal{L}_a) := (\widehat{W}_t, \widehat{\Lambda}_t)$ where $t$ is any element of  $p_H^{-1}(\{a\})$. It readily follows from our construction that for any fixed  
 $t \in  [0,\sigma_H]$, the path $(W_t, \Lambda_t)$ codes the labels on the geodesic connecting the root 
 of $\mathcal{T}_H$ to $p_H(t)$.   Under both $P$ and $N$,  the process  $(\rho , W, \Lambda)$ takes values in  $\mathbb{D}(\mathbb{R}_+, \mathcal{M}_f(\mathbb{R}_+) \times \mathcal{W}_{\overline{E}})$, the space of $\mathcal{M}_f(\mathbb{R}_+) \times \mathcal{W}_{\overline{E}}$ valued rcll functions, and is referred to as the $\psi$-Lévy snake with spatial motion $(\xi_t, \mathcal{L}_t)_{t \geq 0}$. When $\rho$ is considered under $N$ (resp.  ${P}$) we denote the law of this process by $\mathbb{N}_{x,0}$ (resp. $\mathbb{P}$). Since both measures are used throughout this work, it is convenient to assume that $(\rho, W, \Lambda)$ is the canonical process in  $\mathbb{D}(\mathbb{R}_+, \mathcal{M}_f(\mathbb{R}_+) \times \mathcal{W}_{\overline{E}})$. 
\par 
We are now in position to give a more precise version of our results using the formalism of Lévy snakes.  In what follows we argue under $\mathbb{P}$,  analogous versions of our results under $\mathbb{N}_{x,0}$ can be found in the manuscript.  We start by defining, for each debut $u \in \mathcal{D}$, a piece of path of  $(\rho , \overline{W})$ which  codes the tree-indexed process $(\xi^u, \mathcal{L}^u)$ in a sense that will be made precise --  recall that $(\xi^u, \mathcal{L}^u)$ is the restriction of $(\xi_a, \mathcal{L}_a)_{a \in \mathcal{T}_H}$ to the subtree  stemming from $u$. In this direction, for every $u \in \mathcal{D}$, we set $\mathfrak{g}(u) := \inf\{ t \geq 0 : p_H(t) = u \}$ and $\mathfrak{d}(u) := \sup\{ t \geq 0 : p_H(t) = u \}$ respectively for the first and last time 
  when the exploration $p_H$ visits the component $\mathcal{C}_u$.  Remark that the subtree stemming from $u$ is precisely the image of $[\mathfrak{g}(u), \mathfrak{d}(u)]$ under the projection $p_H$. We define a ``shift'' operation $\uptheta_z$ for $z\geq 0$ on elements of $\mathcal{M}_f(\mathbb{R}_+)$ and $\mathcal{W}_{\overline{E}}$ as follows. We   let  $\uptheta_z(\mu)$, and  $\uptheta_z(\w)$ if $z \leq \zeta_{\w}$, be the element of $\mathcal{M}_f(\mathbb{R}_+)$ and $\mathcal{W}_{\overline{E}}$  defined respectively by: 
    \begin{equation} \label{operacion:translacion:intro}
    \langle \uptheta_{z} (\mu), f \rangle   := \int  \mu (\dd r)  f(r - z ) \mathbbm{1}_{\{r > z\}},     
    \quad \text{ and } \quad 
    \uptheta_{z} (\overline{\w}) := ( \overline{\w}(z+r),  0 \leq r \leq \zeta_\w-z). 
\end{equation}
 Recall as well the notation $H(\mu)$ for the supremum of the topological support of $\mu$.  With these notations in hand, we define the subtrajectory associated with  the interval $[\mathfrak{g}(u), \mathfrak{d}(u)]$ to be the $\mathcal{M}_f(\mathbb{R}_+) \times \mathcal{W}_{\overline{E}}$-valued process defined by the relation
\begin{equation*}
    (\rho_t^u, \overline{W}_t^u) := 
        \big( \uptheta_{H_{\mathfrak{g}(u)}}(\rho_{(\mathfrak{g}(u)+t) \wedge \mathfrak{d}(u)}), \uptheta_{H_{\mathfrak{g}(u)}}(\overline{W}_{(\mathfrak{g}(u)+t) \wedge \mathfrak{d}(u)} )\big), \quad \text{ for } t \geq 0. 
\end{equation*}
For each $u \in \mathcal{D}$, we refer to  $(\rho^u, \overline{W}^u)$ as the  subtrajectory stemming from the debut $u$. The process $H(\rho^u)$ is continuous and  we write $\mathcal{T}_{H(\rho^u)}$ for the tree coded by $H(\rho^u)$. Further, the pair 
$\overline{W}^u := ({W}^u_t, \Lambda^u_t)_{t \geq 0}$  is compatible with the equivalence relation $\sim_{H(\rho^u)}$, in the sense  that  
\begin{equation*}
    (W^u_s(h), \Lambda^u_s(h)) = (W^u_t(h), \Lambda^u_t(h)), \quad \text{ for every } 0 \leq h \leq \min_{[s \wedge t, s \vee t]}H(\rho^u),  
\end{equation*}
for all $s, t \geq 0$. The corresponding $\mathcal{T}_{H(\rho^u)}$-indexed process -- obtained by an analogous procedure as before  -- and  $(\xi^u, \mathcal{L}^u)$ are related through a label-preserving isometry between $\mathcal{T}_{H(\rho^u)}$ and  $\{ a \in \mathcal{T}_H : u \preceq a \}$. In Section~\ref{section:excursionmeasure},  we shall argue that one can define a family of infinite measures $(\mathbf{N}_{x,r}: r \geq 0)$ on $\mathbb{D}(\mathbb{R}_+, \mathcal{M}_f(\mathbb{R}_+) \times \mathcal{W}_{\overline{E}})$ that code the law of a typical subtrajectory stemming from $u$. Informally, for each fixed $r \geq 0$, one can think of $\mathbf{N}_{x,r}$ as the 
 law of  the $\psi$-Lévy snake with spatial motion $\mathcal{N}_{x,r}$ when $\rho$ is considered under the excursion measure  $N$.  In particular, the tree coded by $H$ is still a Lévy tree.\footnote{ Since under $\mathbf{N}_{x,r}$ the conditional distribution of $W$ given $H$ is an infinite measures, this statement should be handled with care.} This description is informal since $\mathcal{N}_{x,r}$ is an infinite measure, but it can be made precise by introducing an appropriate entrance law for the Lévy snake. Section \ref{section:excursionmeasure} is devoted to defining the measures $(\mathbf{N}_{x,r}: r \geq 0)$ and to establishing some  of its   basic properties. It is important to stress that, due to their simple nature,  results for the Lévy snake can be often carried-over under $(\mathbf{N}_{x,r}: r \geq 0)$ with minor modifications. This makes their study rather straightforward, in virtue of the now comprehensive  literature concerning Lévy snakes. In Section  \ref{subsection:exitformula} we relate the family of measures $(\mathbf{N}_{x,r}: r \geq 0)$ with  $\mathbb{P}$. In this direction, for every   $u \in \mathcal{D}$, consider the functional  
\begin{equation}\label{equation:trimming:intro}
    \int_0^t \dd s \, \mathbbm{1}_{  s \,  \in \,  (\mathfrak{g}(u),\mathfrak{d}(u) )  }, \quad t \geq 0, 
\end{equation}
and write $(\mathrm{T}_t^{u}: t \geq 0)$ for its right-inverse. If we set $\texttt{T}_u (\rho,\overline{W}) := (\rho_{\mathrm{T}_t^{u}}, W_{\mathrm{T}_t^{u}})_{t \geq 0}$, it is plain that this  time-changed process is obtained from gluing the two pieces of path $( (\rho_t, \overline{W}_t) : 0 \leq t < \mathfrak{g}(u))$ and $( (\rho_t, \overline{W}_t) :  \mathfrak{d}(u) \leq t)$,  so that  the subtrajectory $(\rho^u, \overline{W}^u)$  stemming from the debut point  $u$   has been ``trimmed" from   $(\rho, \overline{W})$. In terms of labelled trees, since the subtrajectory $(\rho^u, \overline{W}^u)$ codes the labels of the subtree stemming from $u$, one can think of   $\texttt{T}_u (\rho,\overline{W})$ as coding  the labels on the closure of the complementary set $\mathcal{T}_H \setminus \{v \in \mathcal{T}_H : u \preceq v  \}$.  The following theorem is the main result of Section  \ref{subsection:exitformula}  and  
 relates the family of measures $(\mathbf{N}_{x,r}: r \geq 0)$ with $\mathbb{P}$
 through the so-called exit formula.
\begin{enumerate}
    \item[\mbox{}] \textbf{Theorem}. For every  non-negative measurable functions $F,G$ on $\mathbb{D}(\mathbb{R}_+, \mathcal{M}_f(\mathbb{R}_+)\times \mathcal{W}_{\overline{E}})\times \mathbb{R}$ \\ and 
 $\mathbb{D}(\mathbb{R}_+, \mathcal{M}_f(\mathbb{R}_+)\times \mathcal{W}_{\overline{E}})$, we have 
    \begin{equation}\label{INTRO:exit}
   \mathbb{E}\Big(  \sum_{u \in \mathcal{D}} F\big( \,  {\texttt{T}}_u (\rho,\overline{W}), \mathfrak{g}(u) \big) \cdot G(\rho^u , \overline{W}^u)  \Big)  = 
    \mathbb{E}\Big( \int_0^\infty \mathrm{d} A_s  ~F(\rho,\overline{W},s ) \cdot \mathbf{N}_{x, \widehat{\Lambda}_s}\big(G(\rho,\overline{W})\big)\Big).
\end{equation}
\end{enumerate}
This result shares   strong   similarities with the exit formula from Maisonneuve  \cite[Theorem 4.1]{Maisonneuve},  and our nomenclature is borrowed from there. It allows for instance  to transfer properties holding for $(\mathbf{N}_{x,r}: r \geq 0)$ to the family of subtrajectories  stemming from $u \in \mathcal{D}$ under  $\mathbb{P}$ and vice-versa, and  both directions will be of use. It will be crucial in the study of  the boundary measure of excursion components, the excursions away from $x$, and their descendent line in the tree. 
\par 
In Section \ref{section:Poisson} we turn our attention to the study of the family of excursions away from $x$.  In the same vein  as before, for each $u \in \mathcal{D}$, we associate with each excursion  $\xi^{u,*}$  a  process constructed in terms of deterministic operations on $(\rho^u , W^u)$ which codes both the  excursion component $\mathcal{C}_u$ and the respective labels when the component is explored from left to right. This can now be easily done by removing from each $(\rho^u, W^u)$ every path $(\rho^u_s, W^u_s)$, for which $W^u_s$ has returned to $x$ before the end of its lifetime. Formally,
for any finite $E$-valued path $\w$, write $\uptau(\w) := \inf\{ h > 0 : \w(h) = x \}$ for its first return time to $x$. For every element $(\upvarrho , \omega, \lambdaC) \in \mathbb{D}(\mathbb{R}_+, \mathcal{M}_f(\mathbb{R}_+) \times \mathcal{W}_{\overline{E}}) $, we consider the functional $\int_0^t \mathbbm{1}_{\{ H(\upvarrho_s) \leq \uptau(\omega_s)  \}}\dd s$, for $t\geq 0$, and write $\Gamma(\upvarrho, \omega)$ for its right inverse. If we set $\text{tr}(\upvarrho, \omega) := (\upvarrho_{\Gamma(\upvarrho , \omega)}, \omega_{\Gamma(\upvarrho, \omega)})$, for every  $u \in \mathcal{D}$ the process $\mathrm{tr}(\rho^u, W^u)$     is informally obtained from $(\rho^u, W^u)$ by removing  every pair $(\rho^u_s, W^u_s)$ for $s \geq 0$ for which $\uptau(W^u_s) < H(\rho^u_s)$. Now,  to simplify notation, 
 for every $u \in \mathcal{D}$,  we write   
\begin{equation*}
  (\rho^{u,*} , W^{u,*}):=   \mathrm{tr}\big( \rho^u, W^{u}   \big).  
\end{equation*}
For every $u \in \mathcal{D}$, the process $H(\rho^{u,*})$ is continuous,   $W^{u,*}$ is well defined in $\mathcal{T}_{H(\rho^{u,*})}$, and  the corresponding tree-indexed process is related to  $\xi^{u,*}$  by a label-preserving isometry between $\mathcal{T}_{H(\rho^{u,*})}$ and  $\mathcal{C}_u$. For this reason,  with a slight abuse of notation, we will refer to the family  $(\rho^{u,*} , W^{u,*})_{u \in \mathcal{D}}$   as the excursions away from $x$ of $(\xi_a)_{a \in \mathcal{T}_H}$.   Recalling that $\mathbf{N}_{x,0}$ codes the law of a typical subtrajectory stemming from $u$,  it is  natural to define a candidate measure in $\mathbb{D}(\mathbb{R}_+, \mathcal{M}_f(\mathbb{R}_+) \times \mathcal{W}_E)$ for the excursion measure,  by setting 
\begin{equation*}
    \mathbf{N}_{x}^*(\dd \upvarrho , \dd \omega) := \mathbf{N}_{x,0} \big( \mathrm{tr}(\rho , W) \in (\dd \upvarrho, \dd \omega) \big). 
\end{equation*}
In other words, $\mathbf{N}_{x}^*$ is the law of $\mathrm{tr}(\rho , W)$ under  $\mathbf{N}_{x,0}$.     Note that the excursion measure does not  keep track of the local time $(\mathcal{L}_a)_{a \in \mathcal{T}_H}$.  The reason being, that the local time  is constant on any excursion component $\mathcal{C}_u$ [Lemma \ref{lemma:contanteExcursion}] and therefore  there is no need to keep track of it in the excursion measure $\mathbf{N}_x^*$. It is important to note that in contrast with the family $(\mathbf{N}_{x,r}: r \geq 0)$, the tree coded by  $H(\rho)$ under $\mathbf{N}_{x}^*$ is \textit{not} a Lévy tree. For this reason, the study of the measure $\mathbf{N}_{x}^*$ is systematically performed through the measures $(\mathbf{N}_{x,r}: r \geq 0)$ which are of much simpler nature. The second main result of this work is a version of Itô's theorem in the tree-indexed setting. In particular it justifies our nomenclature for the measure $\mathbf{N}_x^*$. 
\begin{enumerate}
    \item[\mbox{}] \textbf{Theorem}. Under $\mathbb{P}$, the point measure 
    \begin{equation*}
  \mathcal{E} =  \sum_{u \in \mathcal{D}} \delta_{(A_{\mathfrak{g}(u)} , \rho^{u,*},  W^{u,*} )}
\end{equation*}
    is a Poisson point measure in $\mathbb{R}_+ \times \mathbb{D}(\mathbb{R}_+, \mathcal{M}_f(\mathbb{R}_+) \times \mathcal{W}_E)$ with intensity   $\mathbbm{1}_{\mathbb{R}_+}(t)\dd t \otimes \mathbf{N}_x^*$.  
\end{enumerate}
We refer to $\mathcal{E}$ as the excursion process of $(\xi_a)_{a \in \mathcal{T}_H}$.  The proof of this result relies mainly in the special Markov property of the Lévy snake established in \cite[Theorem 3.7]{2024structure} paired with  the exit formula \eqref{INTRO:exit}. We refer to Section  \ref{section:Poisson} for details.   Let us mention that in contrast to the other main results presented in this introduction, this theorem does not admit a straight adaptation  under    $\mathbb{N}_{x,0}$.  It is however possible to characterise  the law of $\mathcal{E}$ under $\mathbb{N}_{x,0}$ by making use of the last main result of this work, that we shall now present. 
\par    
In the last part of this manuscript, we relate the tree-like structure of the set \eqref{intro:puntosx} with the collection of excursions. This additional structure of  the set of points with label $x$ is not present   in the time-indexed setting and therefore this last part somewhat deviates from the classic excursion theory for  $\mathbb{R}_+$-indexed Markov processes. In this direction, recall that for each $u \in \mathcal{D}$, we write $\ell_u$ for the boundary size of the component $\mathcal{C}_u$.   We shall argue that the genealogy of the excursions  as well as their respective boundary sizes $(\ell_u)_{u \in \mathcal{D}}$ can be   encoded in a random real tree. Our study will shed  light on the law of the excursion process  $\mathcal{E}$ under $\mathbb{N}_{x,0}$,  thereby complementing the statement of Theorem \ref{theorem:excursionPPP}. The real tree suited for this purpose has already made its appearance in the previous work \cite{2024structure}.  More precisely, if we set 
\begin{equation*}
    \widetilde{H}_t  := 
    \widehat{\Lambda}_{A^{-1}_t}, \quad  t \geq 0,
\end{equation*}
then  \cite[Theorem 5.1]{2024structure} states that $\widetilde{H}$ under $\mathbb{P}$ has the law of the height process of a forest of $\widetilde{\psi}$-Lévy trees. The Laplace exponent $\widetilde{\psi}$ is explicit, see \cite[Proposition 4.7]{2024structure} and it can be shown that it does not have Brownian component \cite[Corollary 4.1]{2024structure}. In particular, $\widetilde{\psi}$ can be uniquely decomposed in the form:
$$
\widetilde{\psi}(\lambda)=\widetilde{\alpha}\lambda+\int_{(0,\infty)}\widetilde{\pi}(\dd \ell)\big(\exp(-\lambda \ell)-1+\lambda \ell\big),
$$ 
for some constant $\widetilde{\alpha}\in \mathbb{R}_+$ and a Lévy measure $\widetilde{\pi}$ on $(0,\infty)$  satisfying $\int \widetilde{\pi}(\dd \ell) ~(\ell\wedge \ell^2)<\infty$  by \cite[Section 1.1]{DLG02}.  Since  the constancy intervals of $(A_t)_{t \geq 0}$ and   $(\widehat{\Lambda}_t)_{t \geq 0}$ coincide \cite[Theorem 4.20]{2024structure} and  $(\widehat{\Lambda}_t)_{t \geq 0}$ is constant on each $p_H^{-1}(\mathcal{C}_u)$ for $u \in \mathcal{D}$,  the tree coded by $\widetilde{H}$  does fulfil the description given above, where we    informally\footnote{This description can be made precise by making use of the notion of subordination of trees introduced in \cite{Subor} we refer to Section \ref{section:joinlaw} and \cite{2024structure} for details.} obtained   $\mathcal{T}_{\widetilde{H}}$ from $\mathcal{T}_H$ by identifying each excursion component $\mathcal{C}_u$ in a single point $[u] \in \mathcal{T}_{\widetilde{H}}$.   The tree $\mathcal{T}_{\widetilde{H}}$ is then a forest of $\widetilde{\psi}$-Lévy trees that we refer to as the tree coded by the local time. The Lévy measure $\widetilde{\pi}$  is closely related to our notion of boundary size in the following sense.  We show that the boundary size  $\ell_u$ of any  excursion component  $\mathcal{C}_u$  is a measurable function of the corresponding excursion   $(\rho^{u,*}, W^{u,*})$. In particular, this notion of  boundary size is also well defined under $\mathbf{N}_{x}^*$ and we establish in Proposition \ref{corollary:exponente}  that its ``law'' when restricted to $(0,\infty)$ is precisely the Lévy measure $\widetilde{\pi}$. We are now in position to present a precise version of the  relation between the branching points of the tree $\mathcal{T}_{\widetilde{H}}$ and the collection of boundary sizes   $(\ell_u)_{u \in \mathcal{D}}$. In this direction, under $\mathbb{P}$,  we  partition the collection of excursion debuts in the two  following  families $\mathcal{D}_+ :=\{u\in \mathcal{D} :~ \ell_u >0\}$ and  $\mathcal{D}_0 :=\{u\in \mathcal{D}  : ~ \ell_u =0\}$. If we set 
\begin{equation}\label{equation:jumpLevy:intro}
       \widetilde{\mathcal{E}} :=  \sum_{u \in \mathcal{D}_+} \delta_{(A_{g(u)} , \,  \ell_u )}, 
\end{equation}
standard properties of Poisson measures  and the previous discussion entail that, under $\mathbb{P}$, the point measure $\widetilde{\mathcal{E}}$ is a Poisson point measure with intensity $\mathbbm{1}_{\mathbb{R}_+}(t) \dd t \otimes  \widetilde{\pi}(\dd \ell)$. We write $\widetilde{\mathcal{E}}^{(c)}$ for the compensated measure $\widetilde{\mathcal{E}}^{(c)}( \dd  t,  \dd \ell):=  \mathcal{E}(\dd  t,  \dd \ell)- \dd  t \widetilde{\pi}(\dd \ell)$. 
\begin{enumerate}
    \item[\mbox{}] \textbf{Theorem}. Under $\mathbb{P}$, the process $(\widetilde{X}_t)_{t \geq 0}$ defined as 
\begin{equation} 
\widetilde{X}_t  \coloneqq \widetilde{\alpha} t + \int_{[0,t]\times \mathbb{R}_+} \widetilde{\mathcal{E}}^{(c)}( \dd  s,  \dd \ell)~ \ell,\quad t\geq 0,  
\end{equation}
is a $\widetilde{\psi}$-Lévy process, and $\widetilde{H}$ is its height process.  
\end{enumerate}
 If we write $\mathcal{D}_{\widetilde{X}}$ for the set of jump-times of $\widetilde{X}$, the map $u \to A_{\mathfrak{g}(u)}$ is a bijection between the sets $\mathcal{D}_+$,  $\mathcal{D}_{\widetilde{X}}$ and if for $t \in \mathcal{D}_{\widetilde{X}}$ it holds that  $t = A_{\mathfrak{g}(u)}$ it is plain from the definition of $\widetilde{X}$ that  $\Delta \widetilde{X}_{t}=\ell_u$. General results on Lévy trees from \cite{FractalAspectsofLevyTrees} then yield that the set of branching points of $\mathcal{T}_{\widetilde{H}}$ are precisely the points $[u] \in \mathcal{T}_{\widetilde{H}}$ for $u \in \mathcal{D}_+$, and the fractal mass of every such point $[u]$ in the sense of \cite[Theorem 4.7]{FractalAspectsofLevyTrees} is given by the corresponding boundary size $\ell_u$. One important application of this theorem is that it allows to derive the   conditional distribution of the family of excursions conditionally on $\widetilde{H}$.  Namely, since $\widetilde{X}$ and $\widetilde{H}$ are measurable functions of each-other  (see the discussion at the end of Section \ref{subsection:explorationprocess}),   we deduce  that under $\mathbb{P}$,  the collection $(\rho^{u,*}, W^{u,*})_{u \in \mathcal{D}_0}$ is independent from $\widetilde{H}$ and $(\rho^{u,*}, W^{u,*})_{u \in \mathcal{D}_+}$. Furthermore, conditionally on $\widetilde{H}$  the family of excursions $(\rho^{u,*}, W^{u,*})_{u \in \mathcal{D}_+}$ are independent, and for each $u \in \mathcal{D}_+$,   the  conditional law of $(\rho^{u,*}, W^{u,*})$ is given by $\mathbf{N}_x^{*, \Delta \widetilde{X}_t }$, where $t = A_{\mathfrak{g}(u)}$. Here it is important to consider an enumeration of the atoms measurable with respect to $\widetilde{H}$, we refer to Section \ref{section:joinlaw} for details.  This last result also  bears  similarities  with the time-indexed setting, where excursions are in bijection with jumps  of the inverse local time process;   conditionally on the inverse local time, the excursions are independent, and the distribution of an excursion corresponding to a jump of size $\Delta_i$ is given by the excursion measure conditioned to have duration $\Delta_i$. 
\\
\\
We conclude this introduction with a discussion concerning the relation of our theory with the closely related work of Abraham and Le Gall  \cite{ALG15},  as well as with some connections to  related theories and open problems.  
\\
\\
\textit{Connection with Abraham-Le Gall excursion theory} \cite{ALG15}. An excursion theory for Brownian motion indexed by the Brownian tree was developed by Abraham and Le Gall in  \cite{ALG15} using radically different methods. Specifically, \cite{ALG15} considers the case when $(\xi_t)_{t \geq 0}$ is a 
 linear Brownian motion, the underlying tree is a Brownian tree, and $x=0$.  
The work \cite{ALG15}  relies  on  symmetries and scaling properties of Brownian motion indexed by the Brownian tree, and these arguments cannot be exploited in our general setting. Another key difference with the present work is that \cite{ALG15} does not introduce a notion of local time suitable to index the excursions. In the last section of this work, we verify the consistency of our excursion theory with that in \cite{ALG15}; that is, we show that the excursion measure introduced in \cite{ALG15} and their notion of boundary measure coincide with those in this work in the special case of Brownian motion indexed by the Brownian tree. In the process, we derive the master formula \cite[Theorem 23]{ALG15} as a particular case of the exit formula [Theorem \ref{theorem:exit}]. Finally, Corollary  \ref{corolary:condlaw} is a more precise version of \cite[Theorem 40]{ALG15}.  We stress that although the methods employed in \cite{ALG15} are drastically different from those used here,   the results and ideas employed by Abraham and Le Gall have been fundamental for the development of this manuscript. Furthermore some results in \cite{ALG15}  are exclusive to Brownian motion indexed by the Brownian tree and therefore can not be established in our general setting.
\\
\\
\textit{Related works and applications.}   As in \cite{ALG15}, one of the main motivations of this work is the connection between  Lévy trees with labels and continuous models of 2D random geometry. As mentioned above, Brownian motion indexed by the Brownian tree has served as building block for the family of random metric spaces that are generically referred to as   Brownian geometries \cite{LeGallLectures, LeGallRiera-nonCompactModels}. One notable member of this rich family is  the celebrated Brownian sphere (or Brownian map) \cite{LG11, Mie11}. The theory developed in \cite{ALG15} has been shown to be a powerful tool in the study of Brownian geometries, and specially when establishing spatial Markov properties --  continuous analogues of  peelings of planar maps \cite{peeling}.  In short, when performing a metric exploration in a Brownian surface (by considering, for example, increasing balls or hulls) the space is  divided progressively in  discovered  and  undiscovered  components. Spatial Markov properties characterise the distribution of these components conditional on their boundary lengths. The  present excursion theory should allow  to obtain more precise versions of these results, and should permit for instance to recover  the initial surface through a gluing procedure of the components. As a notable example,  the Brownian sphere can be divided in the Lévy net \cite{Axiomatic, Subor} (coded by the tree of local time) and a collection of  disjoint Brownian disks (which are coded by excursions of Brownian motion indexed by the Brownian tree),  that have been shown to be independent conditionally on their boundary lengths \cite{fragmentations}. Our work should allow for the reconstruction of the Brownian sphere by gluing the Brownian disks back into the Lévy net, in a completely  explicit way.  In recent years, new models of 2D random geometries closely related to  Brownian motion indexed by stable trees have been introduced, see e.g. \cite{archer2024stable,kortchemski2024random}. Our results should shed light in the study of these new models and permit to solve questions that until now where out of reach. It is expected that these geometries can be connected  to different versions of Lévy nets \cite{Axiomatic} and to stable causal maps \cite{causal:maps}. In particular, we hope to obtain  similar reconstruction results for these models, and this is currently a work in progress.  We also expect the present theory to  relate  Brownian motion  indexed by stable trees to large classes of growth-fragmentation processes, and more precisely with the recently introduced family of self-similar Markov trees \cite{BCR:2024}. These conform  a family of decorated  $\mathbb{R}$-trees that code the genealogy of general growth-fragmentation processes,  and are scaling limits of multi-type Galton–Watson trees when the types along branches are in the domain of attraction of a positive self-similar  Markov process.   This connection has already been observed in the special case of Brownian motion indexed by the Brownian tree \cite{fragmentations}  thanks to \cite{ALG15}, see also Section~4.3 in \cite{BCR:2024}, and the stable case seems to be connected to some discret models of parking on trees \cite{BCR:2024,C:C:Parking,chen2021enumerationfullyparkedtrees}. Another  direction that   has not been yet  explored  are  invariance principles towards the excursion measure $\mathbf{N}^*_x$. Invariance principles for Lévy snakes are an important topic \cite{Marzouk}, and we hope that the definition of the excursion measure $\mathbf{N}_x^*$ in terms of  a pruning operation of $\mathbf{N}_{x,0}$  will be useful to extend known convergence results. Finally, it has been recently pointed out to us that the additive $A$ should also appear in scaling limit results  concerning local times of Branching random walks and this direction is currently under study \cite{Bai-Chen-Hu}.
\\ \\
\noindent \textit{The rest of the work is organised as follows.} In Section \ref{section:preliminaries} we recall the main elements of the theory of Lévy trees and Lévy snakes that are needed for this work. In Section  \ref{section:framework} we fix the setup and hypotheses under which we develop the excursion theory, these are maintained throughout the rest of the manuscript.  In Section \ref{section:excursionstrajectories} we introduce multiple functionals and notions   that will be  used extensively -- namely, the truncation operation on snake paths,  the exit local time $(L_t)_{t \geq 0}$, debut points, the notion of excursion away from $x$ and the one of a subtrajectory stemming from a debut point. We then proceed in Section \ref{section:excursionmeasure} to define candidate measures $\mathbf{N}_{x}^*$ and  $\mathbf{N}_{x,r}$, for $r\geq 0$, to study  the excursions away from $x$, and the subtrajectories stemming from a debut point.  Next,  we establish     spinal decompositions in Section \ref{section:MainSpinalDecomp} that will be often needed for technical reasons. The  main aspects of the excursion theory are developed in Section \ref{section:excursiontheory}. In this direction, we start by recalling properties of the local time at $x$ process $(A_t)_{t \geq 0}$. We then prove the exit formulas  [Theorem \ref{theorem:exit}, Corollary \ref{corollary:exitPx}] and show that the excursion process of a Markov process indexed by a forest of Lévy trees is a Poisson point measure  [Theorem \ref{theorem:excursionPPP}] with intensity $\mathbbm{1}_{\mathbb{R}_+}(t)\dd t \otimes \mathbf{N}_x^*$. These constitute the main results of this work.  
 In Section \ref{section:L} we develop a theory, analogous to that of exit local times, under the measures $\mathbf{N}_{x,r}, \mathbf{N}_{x}^*$. For instance, we prove a version of the special Markov property [Proposition \ref{proposition:L*}] under $\mathbf{N}_{x,r}$. Section \ref{section:joinlaw} is devoted to the study of the relation between the so-called tree coded by the local time and the family of excursions. The main result of the section is Theorem \ref{theo:X:H:tilde}. Finally, in Section \ref{sec:ref:consis:} we show that our theory is consistent with the work of Abraham and Le Gall \cite{ALG15}.
\\
\\
 \textbf{Acknowledgments.} We thank Mathieu Merle and  Victor Rivero for  insightful  conversations concerning exit systems. The second author would like to express his gratitude to Jean Bertoin for his support and  encouragement during the elaboration of this work. The research of the second author was supported by the Swiss National Science Foundation (SNSF), grant number 212115.

\section{Preliminaries} \label{section:preliminaries}

This section is divided into two main parts. We begin  Section \ref{section:LablesandTrees}  by introducing a general construction of trees that is well suited for defining (possibly random) labels on them. This construction will allow us to  define the notion of a Markov process indexed by a tree at the end of Section \ref{section:LablesandTrees}. We then proceed in Section \ref{section:MPonLT} to randomise the underlying tree within the class of Lévy trees. This is done by making use of the classical theory of Lévy trees and Lévy snakes \cite{DLG02}, and for ease of reading we  provide a brief overview of both theories.  If $M$ is a Polish space,  we will use the standard notation $\mathbb{C}(\mathbb{R}_+, M)$ and $\mathbb{D}(\mathbb{R}_+, M)$ respectively for the space of $M$-valued continuous and rcll\footnote{Right-continuous with left limits.} paths indexed by $\mathbb{R}_+$.  These spaces are endowed with the topology of local uniform convergence and the Skorokhod topology respectively, along with  the corresponding Borel sigma-fields.

\subsection{Labels on trees coded by excursions and snakes}\label{section:LablesandTrees}

The family of trees that we  work with are coded by  continuous non-negative functions on  $\mathbb{R}_+$, that we broadly refer to as coding functions.   The construction is classic, and we recall its main elements in Section~\ref{sec:trees:1}.    Similarly, the labels on the tree are also defined in terms of a path-valued function in $\mathbb{R}_+$ that is generically referred to as a  \textit{snake trajectory}. This notion is introduced and studied in Section \ref{secsnake}.

\subsubsection{Trees coded by non-negative functions}\label{sec:trees:1}

In this short section, we recall standard notation and basic properties of  $\mathbb{R}$-trees.   For a detailed account, the reader is referred to the lecture notes by Evans \cite{evans} or by Le Gall \cite{LG05}.  An $\mathbb{R}$-tree $(\mathcal{T},d)$ is a uniquely arcwise connected metric space, in which each arc is isometric to a compact interval of $\mathbb{R}$. In this work, we  exclusively consider rooted  $\mathbb{R}$-trees, which further imposes  $\mathcal{T}$ to have a distinguished point $\vartheta \in \mathcal{T}$, called the root.   For every $u,v\in \mathcal{T}$ we write $\llbracket u,v \rrbracket$ for the range of the unique injective path connecting $u$ and $v$, and  $u \curlywedge v$  for the unique element of $\mathcal{T}$ verifying the relation  $\llbracket \vartheta , u \curlywedge v \rrbracket =  \llbracket \vartheta ,  u \rrbracket \cap \llbracket \vartheta , v \rrbracket$. It is   natural to define a partial order $\preceq$ 
 coding  the genealogy of $\mathcal{T}$. Namely, we shall write  $u \preceq v$ if $u\in\llbracket \vartheta , v \rrbracket$, and when the latter holds we say that $u$ is an ancestor of $v$. The element $u \curlywedge v$ is called the  common ancestor of $u$ and $v$. The multiplicity of $u$ is defined as the (possibly infinite) number of connected components of $\mathcal{T} \setminus \{ u \}$. For every $i \in \mathbb{N}^* \cup { \infty }$, we write $\text{Mult}_i(\mathcal{T})$ for the family of points in $\mathcal{T}$ with multiplicity $i$. These definitions allow us to extend classic notions from discrete trees to the framework of $\mathbb{R}$-trees. For instance, we shall make use of the following standard nomenclature:
 \begin{itemize}
\item The elements of $\text{Mult}_1(\mathcal{T})$ are  the \textit{leaves} of $\mathcal{T}$;
\item The elements with multiplicity at least $2$  conform the \textit{skeleton} of $\mathcal{T}$;
\item  The elements with multiplicity at least $3$  are referred to as  the \textit{branching points} of $\mathcal{T}$.
 \end{itemize}
 Let us now present a canonical way to construct  $\mathbb{R}$-trees in terms of continuous non-negative functions.  This method is standard and we refer to \cite{Duq.trees.notes}  for a thorough study on coding of $\mathbb{R}$-trees.
 Fix a continuous non-negative function $h: \mathbb{R}_+ \mapsto \mathbb{R}_+$ and  set  
 \begin{equation}\label{def:sigma:h}
     \sigma_h:=\inf \big\{t\geq 0:~h(s)=0 \, \text{ for every }s \geq t\big\} 
 \end{equation}
  for its lifetime, where as  usual  $\inf \emptyset = \infty$ (this convention will be maintained  throughout this work).   For every $s, t \in \mathbb{R}_+$ we set 
\begin{equation}\label{def:m:h}
    \displaystyle m_{h}(s,t):= \inf_{s\wedge t\leq r \leq s\vee t}h(r),  
\end{equation}
and, with the convention $[a,\infty]:=[a,\infty)$ for  $a\in \mathbb{R}$, we  define a pseudometric $d_h$ on $[0,\sigma_h]$ by  
\begin{equation}\label{def:d:h}
    d_h(s,t) := h(s) + h(t) - 2 \cdot  m_{h}(s ,  t),   \quad \text{for  }s,t \in [0,\sigma_h]. 
\end{equation}
We  let  $\sim_h$ be the equivalence relation on $[0,\sigma_h]$ identifying  every pair of points $s,t \in [0,\sigma_h]$ satisfying that $d_h(s,t) = 0$, and for every such pair we write $s\sim_h t$.    We set $\mathcal{T}_h := [0,\sigma_h] / \sim_h$  for the corresponding quotient space  and   we write   $p_h:[0,\sigma_h] \mapsto \mathcal{T}_h$ for  the associated canonical projection.  The metric space $(\mathcal{T}_h, d_h)$  that we  root at $p_h(0)$  is an $\mathbb{R}$-tree which, with a slight abuse of notation, we  simply denote by $\mathcal{T}_h$ and the function $h$ is referred to as the coding function of $\mathcal{T}_h$.   When the duration $\sigma_h$ is finite,  the resulting tree $\mathcal{T}_h$ is a compact metric space.  Since we have $s\sim_h t$ if and only if   $m_h(s , t) = h(s ) = h (t)$, the function $h$ is well defined in the quotient space $\mathcal{T}_h$, and for each $u \in \mathcal{T}_h$ we write $h(u)$ for $h(t)$, where $t$ is any arbitrary element of $p_h^{-1}(\{u\})$.  
\par  Informally, the continuous mapping $( p_h(t) : t \in [0,\sigma_h] )$ explores the tree $\mathcal{T}_h$ from left to right following its contour, and for this reason,  we  call it the clockwise exploration of $\mathcal{T}_h$. For every $u \in \mathcal{T}_h$ and with the convention $\sup \emptyset  = \infty$, we  write respectively
\begin{equation} \label{definition:gauchedroite}
    \mathfrak{g}(u) := \inf \big\{t \geq 0 : p_h(t) = u \big\}  \quad \text{ and } \quad  \quad \mathfrak{d}(u):= \sup \big\{t >g(u) : p_h(t) = u \big\} \footnote{The notation $\mathfrak{g}$ and $\mathfrak{d}$ stems from ``gauche" and ``droite", which in french means left and right.}
\end{equation}
for  the first time and the last time  at which  the exploration $p_h$ visits the point $u$.  
We stress that,  in contrast with $\mathfrak{g}(u)$, the right-end point $\mathfrak{d}(u)$ might be infinite but, by \eqref{def:d:h}, the quantity $\mathfrak{d}(u)$ always coincides with  $\inf\{t> \mathfrak{g}(u):~h(t) < h(u)\}$.    In this work, we  often consider coding functions $h$ satisfying the additional condition  $h(0)=0$. In this scenario, we simply have $d_{h}(0,t)=h(t)$ for every $t\in [0,\sigma_h]$, and consequently the function $h$ codes precisely the distances to the root of the tree.   It is also worth noting that in that case, the genealogy of $\mathcal{T}_h$ is encoded in $h$ in a very simple way. Namely, it is easy to check that for any pair of points  $u,v \in \mathcal{T}_h$, the relation $u \preceq v$ holds if and only if $\mathfrak{g}(u)\leq \mathfrak{g}(v)\leq \mathfrak{d}(u)$.  Moreover, if $u \preceq v$ and letting 
\begin{equation*}
\gamma(r) := \sup \big\{ s \leq \mathfrak{g}(v) : h(s) \leq r \big\}, \quad \text{for } r \in [h(u),h(v)],
\end{equation*}
 definition \eqref{def:d:h} gives $d_h(p_h(\gamma(r)), p_h(\gamma(r'))) =| r-r' |$ for $ r,r' \in [h(u),h(v)]$, which yields that the path  $(p_h(\gamma(r)) : r \in [h(u), h(v)])$ is the geodesic path connecting $u$ and $v$. It then follows from the definition of $\llbracket u,v \rrbracket$ that
\begin{equation}\label{equation:geodesicpath}
\llbracket u, v \rrbracket = p_h \big(\big\{ s \in [\mathfrak{g}(u),\mathfrak{g}(v)] : m_h(s,\mathfrak{g}(v)) = h(s) \big\}\big).
\end{equation}
It is possible to extend the previous discussion to coding functions $h$   that do  not fulfil the condition $h(0) = 0$, but the general version of \eqref{equation:geodesicpath} is more difficult  to work with.  Since this generalisation is not needed in this work we do not discuss it here. 

\subsubsection{Snake trajectories}\label{secsnake} 

From now on we fix a Polish space $(E,\dd_E)$. The goal of this section is to recall the notion of  \textit{snake trajectory}, first defined in \cite{ALG15} and further developed in \cite{2024structure}. Snake trajectories provide an appropriate framework for defining functions indexed by $\mathbb{R}$-trees with values in $E$, and throughout this work we will make extensive use of their properties.  In this direction, we let $\mathcal{W}_E$ be the space of $E$-valued finite paths in $E$. Hence, every element $\w$ of $\mathcal{W}_E$ is a continuous mapping $\w : [0,\zeta_\w] \mapsto E$, where $\zeta_\w$ is a finite non-negative number called the lifetime of $\w$. The endpoint or tip of the path  $\w$ is denoted by   $ \widehat{\w} := \w(\zeta_\w)$. For every $y \in E$, we write  $\mathcal{W}_{E,y} \subseteq \mathcal{W}_E$ for the collection of continuous finite paths starting from $y$. With a slight abuse of notation, we still write  $y$ for the unique element of $\mathcal{W}_{E,y}$ with lifetime $\zeta_{\w}=0$. If for every $\w, \w' \in \mathcal{W}_E$ we set 
\begin{equation}\label{d:W:E}
    d_{\mathcal{W}_{E}}(\mathrm{w},\mathrm{w}^{\prime}):=|\zeta_{\mathrm{w}}-\zeta_{\mathrm{w}^{\prime}}|+\sup \limits_{r\geq 0}d_{E}\big(\mathrm{w}(r\wedge \zeta_{\mathrm{w}}),\mathrm{w}^{\prime}(r\wedge \zeta_{\mathrm{w}^{\prime}})\big),
\end{equation}
 the function $d_{\mathcal{W}_E}: \mathcal{W}_E \times \mathcal{W}_E \to \mathbb{R}$ is a metric on $\mathcal{W}_E$ and $(\mathcal{W}_E, d_{\mathcal{W}_E})$ is a Polish space. In what follows, we   systematically write $\omega = (\omega_s : s \geq 0)$ for a $\mathcal{W}_E$-valued function indexed by $\mathbb{R}_+$.  In particular, for every fixed $s\geq 0$, $\omega_s$ is an element of $\mathcal{W}_E$ with lifetime  $\zeta_{\omega_s}$, and the value at its tip is just  $\widehat{\omega}_s$.  We shall refer to 
\begin{equation*}
\zeta(\omega) := (\zeta_{\omega_s} : s \geq 0)
\end{equation*}
 as the  lifetime process of $\omega$, and to simplify notation we write $\zeta_s(\omega)$ for $\zeta_{\omega_s}$. Remark that the definition of the metric $d_{\mathcal{W}_E}$ in \eqref{d:W:E} yields that $\zeta(\omega)$ is continuous if $\omega$ is continuous. Finally, we will use  indifferently the notation 
\begin{equation*}
    \sigma(\omega) = \sigma_{\zeta(\omega)} = \inf \big\{t\geq 0:~\zeta_{\omega_s}=0  \text{ for every }s \geq t\big\},
\end{equation*}
for the lifetime of $\zeta(\omega)$. 
\begin{def1}[Snake trajectory] \label{definition:snaketrajectory}
 Fix $\w \in \mathcal{W}_E$. A continuous function  $\omega:\mathbb{R}_+ \to \mathcal{W}_E$   started at  $\omega_0 = \w$
 and  satisfying for any $s,t \geq 0$,  that 
  \begin{equation}\label{weak:Snake:omega}
    \omega_{s}(r) = \omega_{t}(r),  \quad \quad  \text{ for every } \, \,  0 \leq r \leq m_{\zeta(\omega)}(s ,t),
\end{equation}
will be referred to as a snake trajectory started from $\w$.  We denote  the collection of snake trajectories started from  $\w$ by $\mathcal{S}^\circ_{\w}$, and we write  
\begin{equation*}
    \mathcal{S}^\circ := \bigcup_{\w \in \mathcal{W}_E} \mathcal{S}^\circ_\w ,
\end{equation*}
for the set of all snake trajectories.\footnote{Note that \cite{ALG15} defines a snake trajectory under the additional conditions that $\omega$ starts from  some $y \in E$  (instead of a general path $\w \in \mathcal{W}_E$)  and that $\sigma(\omega)$ is  finite.}
\end{def1} 
When the property \eqref{weak:Snake:omega} is satisfied for some $\omega : \mathbb{R}_+ \to  \mathcal{W}_E$ we say that $\omega$ satisfies the \textit{strong snake property}.  Let us now explain why snake trajectories are well suited for coding functions indexed by $\mathbb{R}$-trees. First, as was mentioned above,  the continuity of $\omega \in \mathcal{S}^\circ$ ensures that the lifetime process $\zeta(\omega)$ is continuous,  so we can consider the tree coded by $\zeta(\omega)$, i.e. $\mathcal{T}_{\zeta(\omega)}$. Now, the strong snake property \eqref{weak:Snake:omega}  entails that the continuous function $\widehat{\omega} = (\widehat{\omega}_s : s \geq 0)$ is  compatible with the equivalence relation $\sim_{\zeta(\omega)}$ and therefore, is well defined in the quotient space  $\mathcal{T}_{\zeta(\omega)}$. With a slightly abuse of notation, for every $a\in \mathcal{T}_{\zeta(\omega)}$, we write  $\widehat{\omega}_a$ for $\widehat{\omega}_s$, where $s$ is any element of $p_{\zeta(\omega)}^{-1}(\{a\})$. We  interpret $\widehat{\omega}_a$ as a label for  the point  $a \in \mathcal{T}_{\zeta(\omega)}$ and $(\widehat{\omega}_a:~a\in \mathcal{T}_{\zeta(\omega)})$  as a process indexed by $\mathcal{T}_{\zeta(\omega)}$ with values on $E$. 
\par 
Let us now introduce a family of probability measures on the space of snake trajectories which  will allow us to define the notion of Markov process indexed by a (deterministic) $\mathbb{R}$-tree.
\medskip
\\
\noindent\textbf{Snakes driven by continuous functions.} Fix an $E$-valued strong Markov process with continuous sample paths   and for every $y \in E$, we let  $\Pi_y$ be its law in $\mathbb{C}(\mathbb{R}_+ , E)$  started from $y$. To simplify notation, we set $\Pi := (\Pi_y)_{y \in E}$ and write $\xi = (\xi_t : t \geq 0)$  for the canonical process on $\mathbb{C}(\mathbb{R}_+ , E)$.\footnote{It is implicitly assumed in our definition that the mapping $y \mapsto \Pi_y$ is measurable.} For every $\w \in \mathcal{W}_E$ and  $a,b \in \mathbb{R}_+$  with $ a \leq  \zeta_{\w}$ and $a \leq b$, we let  $R_{a,b}(\w, \dd \w' )$ be the probability measure on $\mathcal{W}_E$ characterised by the following properties:  
 \begin{itemize}
    \item[$\bullet$] $R_{a,b}(\mathrm{w},\dd \mathrm{w}^{\prime})$-a.s., $\mathrm{w}^{\prime}(s)=\mathrm{w}(s)$ for every $s\in[0,a]$.
    \item[$\bullet$] $R_{a,b}(\mathrm{w},\dd \mathrm{w}^{\prime})$-a.s., $\zeta_{\mathrm{w}^{\prime}}=b$.
    \item[$\bullet$] Under $R_{a,b}(\mathrm{w},\dd \mathrm{w}^{\prime})$, $(\mathrm{w}^{\prime}(s+a))_{s\in[0,b-a]}$
is distributed as $(\xi_{s})_{s\in[0,b-a]}$ under $\Pi_{\mathrm{w}(a)}$.
\end{itemize}
In other words, $\w'$ under $R_{a,b}(\w, \dd \w')$ coincides with $\w$ up to time $a$,  and then it is  distributed as the Markov process $(\xi_t : 0 \leq t \leq b-a)$ under $\Pi_{\w(a)}$. \par 
We will now endow $\mathcal{W}_E^{\mathbb{R}_+}$ with a probability measure and in this direction, we write $W = (W_t : t \geq 0)$ for the canonical process on $\mathcal{W}_E^{\mathbb{R}_+}$. We fix an arbitrary $\w \in \mathcal{W}_E$ as well as a continuous non-negative function $h$ on $\mathbb{R}_+$ such that $\zeta_\w = h(0)$,  and   for $s,t \geq 0$ recall from \eqref{def:m:h} the definition of $m_h(s,t)$. We let $Q_\w^h(\dd W)$ be the probability measure on $\mathcal{W}_E^{\mathbb{R}_+}$ characterised by the relation: 
\begin{align}\label{equation:fddsSnake}
     Q^h_\w \Big(  W_{t_0}\in A_0, \dots, W_{t_n}\in A_n\Big) 
    := \mathbbm{1}_{A_0}(\w) \int_{A_1\times \dots\times A_n}  R_{ m_{h}(0,t_1),h(t_1) }(\w,\dd \w_1) \dots R_{ m_{h}(t_{n-1},t_n),h(t_n) }(\w_{n-1},\dd \w_n),
\end{align}
for arbitrary   $0 = t_0 < t_1 < \dots < t_n$ and   $n \geq 1$.  The canonical process $W$ under $Q_\w^h$ is a time-inhomogenous $\mathcal{W}_E$-valued Markov process, called \textit{the snake driven by $h$ with spatial motion $\Pi$  started from $\emph{w}$}. The function $h$ is referred to as the driving function since  for every $t \geq 0$,  $Q_\w^h$-a.s. it holds that $\zeta_{W_t} = h(t)$.  Informally, under $Q^h_{\rm{w}}$,  when $h(t)$ decreases the path $W_t$ is erased from its tip, while when $h(t)$ increases, the path $W_t$ is extended by adding “little pieces” of trajectories of the Markov process $\Pi$ at its tip. The term  snake stems from  the fact that by construction,   for every fixed  $s, t \in \mathbb{R}_+$ and $Q_\w^h$-a.s.   we have 
\begin{equation}\label{weak:Snake}
    W_{s}(r) = W_{t}(r),  \quad \quad  \text{ for every } 0 \leq r \leq m_h(s,t). 
\end{equation}
 We stress that the equality $\zeta_{W_t} = h(t)$ and \eqref{weak:Snake} only hold $Q_\w^h$-a.s. for fixed $s,t \in \mathbb{R}_+$, and in particular we cannot say that $W$ is $Q_\w^h$ -a.s. an element of $\mathcal{S}_\w^\circ$.  This difficulty can be circumvented as soon as $W$ has a continuous modification  under the metric $d_{\mathcal{W}_E}$, since it is plain from the definition \eqref{d:W:E} of  $d_{\mathcal{W}_E}$ that this modification will now verify $Q_\w^h$-a.s.,  
 for every pair $s,t \in \mathbb{R}_+$, both the identity  $\zeta_{W_t} = h(t)$ and \eqref{weak:Snake}. To this end, let us recall from \cite{2024structure}  sufficient conditions on the pair $(h, \Pi)$  ensuring that $W$ under $Q^{h}_{\mathrm{w}}$ has a continuous modification with respect to  $d_{\mathcal{W}_E}$. With the  convention   $[a,\infty]:=[a,\infty)$ for $a  \in  \mathbb{R}$,  consider a family of disjoint intervals $([a_i, b_i], i \in \mathcal{J})$ indexed by an arbitrary subset $\mathcal{J} \subseteq \mathbb{N}$,  with $a_i < b_i$, for $a_i, b_i \in \mathbb{R}_+ \cup \{ \infty \}$. 
A continuous, non-negative function $h: \mathbb{R}_+ \mapsto \mathbb{R}_+$ is said to be locally $\theta$-Hölder continuous in  $([a_i,b_i] : i \in \mathcal{J} )$ for some $\theta \in (0,1]$ if, for every $n > 0$,  there exists a constant $C_n> 0$ such that
\begin{equation*}
    |h(s) - h(t)| \leq C_n|s-t|^\theta, \quad  \text{ for every } s,t \in [a_i \wedge n,b_i \wedge n],\, i \in \mathcal{J}.   
\end{equation*}
Now, we consider the following assumptions on the pair $(h, \Pi)$.  
\begin{enumerate}
    \item[{(i)}] There exists  a constant $C_\Pi> 0$ and two positive numbers $p,q > 0$ such that, for every $y \in E$ and $t\geq 0$, we have:
\begin{equation} \label{regularidad:estimate}
    \Pi_{y}\big(  \sup_{0 \leq s \leq t } \dd_E(\xi_s , y)^p  \big) \leq C_\Pi\cdot t^{q}.
\end{equation}
    \item[{(ii)}]  If we let  $((a_i, b_i):~ i \in \mathcal{J})$ be the excursion intervals  of $h$ above its running infimum, the  function  $h$ is locally $\theta$-Hölder continuous in $([a_i, b_i]: i \in \mathcal{J})$, for some $\theta\in(0,1)$ with $\theta  q > 1$.
\end{enumerate}
Proposition 1 in \cite{2024structure} states that under conditions (i) and (ii) on $( h,\Pi )$,  for every $\w \in \mathcal{W}_E$ with $\zeta_{{\w}} = h(0)$, the process $W$ under $Q_\w^h$ possesses a continuous modification.  Hence, when (i) and (ii) are fulfilled,  the measure $Q^h_{\rm{w}}$ can be defined in the space $\mathbb{D}(\mathbb{R}_+,\mathcal{W}_E)$.\footnote{The reason for considering $Q^h_{\rm{w}}$ as a measure on $\mathbb{D}(\mathbb{R}_+,\mathcal{W}_E)$, rather than in $\mathbb{C}(\mathbb{R}_+,\mathcal{W}_E)$ or even $\mathcal{S}^\circ$, stems simply from the fact that we will have to consider different rcll processes in various spaces, and it will thus be more convenient to see  $Q^h_{\rm{w}}$ as a measure on $\mathbb{D}(\mathbb{R}_+,\mathcal{W}_E)$ to have a comprehensible framework.}  With a slight abuse of notation, we still denote this measure by $Q^h_{\rm{w}}$, and it is plain from a previous remark that for $Q^h_{\rm{w}}(\mathrm{d} \omega)$-almost every $\omega$ we now have $\omega \in \mathcal{S}^\circ$.  In particular,  under conditions (i) and (ii), we can consider the process $(\widehat{\omega}_a)_{a\in   \mathcal{T}_h}$ under  $Q^h_{\rm{w}}(\mathrm{d} \omega)$, and we call it a $\Pi$-Markov process indexed by $\mathcal{T}_h$,  started from $\mathrm{w}$.    
\par 
The informal interpretation given after \eqref{equation:fddsSnake} takes a nicer form when $h(0)=0$. In this scenario, the duration $\zeta_\w$ of the initial condition $\w$ is null so that $\w = y$ for some $y \in E$. Consequently, for each $t \geq 0$ the path $W_t$ under $Q_{y}^h$ is distributed as $\xi$ under $\Pi_{y}$ restricted to $[0,h(t)]$.   Moreover,  the strong snake property and \eqref{equation:geodesicpath} yield that for every $u\in \mathcal{T}_h$  the 
 mappings 
\begin{equation} \label{equation:snakeproperty2}
    \upsilon \mapsto  \widehat{\omega}_v, \quad  \text{and} \quad  v \mapsto \omega_{t_u}(h(v)),\quad \text{ for } v\in \llbracket 0,u \rrbracket 
\end{equation}
are identical, where  $t_u$ is any arbitrary element of $p_{h}^{-1}(\{ u \})$.  Therefore,   the process  $(\widehat{\omega}_a)_{a\in   \mathcal{T}_h}$ behaves along every ancestral line as the Markov process $\xi$ under $\Pi_{y}$. 
\par 
At this point, it is important to stress that we will not work  with random variables taking values in the space of  metric spaces  with  labels. Defining an appropriate framework for the latter is a technical challenge that has been overcome in special situations  \cite{BCR:2024}, and working with snake trajectories circumvents this difficulty since it allows us to consistently work with processes indexed by $\mathbb{R}_+$.  Note that this is by no means a setback, since it will allow us to make use of the classical theory and machinery for $\mathbb{R}_+$-indexed stochastic processes. throughout this work, our analysis is then always carried over through the lens 
 of   $\mathbb{R}_+$-indexed processes, and when convenient, we interpret our results in terms of labelled trees, since these are well  suited for heuristic purposes. Nevertheless, to avoid measurability problems, when doing so we limit ourselves to properties  that can  be described in terms of some $\mathbb{R}_+$-indexed process.

\subsection{Markov processes indexed by Lévy trees}
\label{section:MPonLT}

The objective of the   section is to introduce the notion of a \textit{Markov process  indexed by a $\psi$-Lévy tree}. To this end, we first present $\psi$-Lévy trees through their coding function, the so-called height process of a $\psi$-Lévy process. We will then assign labels to the $\psi$-Lévy tree by considering  the snake driven by the height process of the Lévy tree.

\subsubsection{The height process}\label{subsection:height}

The law of the height process is determined by a function $\psi:\mathbb{R}_+\to \mathbb{R}_+$ known as its  branching mechanism. In this work, we always assume that $\psi$ can be written in the following form 
\begin{equation} \label{equation:psi}
     \psi(\lambda)=\alpha\lambda+\beta \lambda^{2}+\int_{(0,\infty)}\pi(\dd x)\big(\exp(-\lambda x)-1+\lambda x\big), \quad \lambda \geq 0, 
 \end{equation}
for some non-negative $\alpha, \beta \in \mathbb{R}_+$, and where $\pi(\dd x)$ is  a measure on $(0,\infty)$ satisfying the integrability condition  $\int_{(0,\infty)} \pi(\dd x) (x\wedge x^{2})<\infty$.  We  impose also  the following additional condition on the branching mechanism  
\begin{equation}\label{1_infinity_psi} 
 \int_{1}^{\infty}\frac{\dd \lambda}{\psi(\lambda)}<\infty,  
\end{equation}
which in particular will ensure that the corresponding Lévy tree is compact.  It is classic that for every such function $\psi$ one can associate a spectrally positive Lévy process with Laplace exponent $\psi$.  In this direction, we write $X$ for the canonical process on $\mathbb{D}(\mathbb{R}_+, \mathbb{R})$,   and we denote  the law of a  Lévy process with Laplace exponent $\psi$  started from $0$ by $P$.   In other words, $X$ under $P$ is a Lévy process with  $X_0 = 0$ and with distribution characterised by
\[{E}[\exp(-\lambda X_{1})]=\exp\big(\psi(\lambda)\big),\quad \text{ for } \lambda\geq 0. \]
From our conditions on $\psi$, it follows that under $P$, the process $X$ does not have negative jumps, its paths have infinite variation by \eqref{1_infinity_psi}, and either oscillate or drift towards $-\infty$, depending on whether $\alpha = 0$ or $\alpha > 0$, respectively. For a more detailed account on our standing hypothesis on $\psi$ we refer to \cite{ALG15,2024structure}. Let us now turn  our attention to the so-called height process. Our presentation follows  \cite[Chapter 1]{DLG02}  and we refer to \cite[Section VIII-1]{RefSnake} for heuristics stemming from the discrete setting. Let us start by introducing some notation.  For every $0 \leq s\leq t$, we set 
\begin{equation*}
    I_{s,t}:=\inf_{s \leq u \leq t} X_u,
\end{equation*}
for the infimum of $X$ in $[s,t]$, and    when $s = 0$ we  simply write  $I_{t}=I_{0,t}$ for the running infimum of $X$ at time $t$.  Since $X$  drifts towards $-\infty$ or oscillates, we must have $I_{t}\to -\infty$ as $t \to  \infty$. By  \cite[Lemma 1.2.1]{DLG02}, for 
 each fixed $t\geq 0$,  the limit:
\begin{equation} \label{definition:alturaH}
    H_t := \lim\limits_{\varepsilon \to 0}\frac{1}{\varepsilon }\int_{0}^t \dd r~ \mathbbm{1}_{\{ X_{r}<I_{r,t}+\varepsilon \}}
\end{equation}
exists in probability. Roughly speaking, for every fixed $t\geq 0$, the variable $H_t$ measures the size of the set: 
\begin{equation*}
    \{ 0 < r  \leq t :~ X_{r-} \leq  I_{r,t} \}
\end{equation*}
and we refer to  $H = (H_t : t \geq 0)$ as the height process of $X$.
By \cite[Theorem 1.4.3]{DLG02}, condition \eqref{1_infinity_psi} ensures that $H$ possesses a continuous modification, that we consider from now on and that we still denote by $H$. Hence, under $P$, the process $H$ is  continuous, non-negative and it starts at $H_0 = 0$.  
\par
\indent  It will be crucial for our purposes to also define the height process $H$ under the excursion measure of the reflected Lévy process $X-I = (X_t - \inf_{[0,t]}X : t \geq 0)$. Let us be more precise. Under our current hypothesis, the point $0$ is regular and instantaneous for the Markov process $X-I$.  The process $-I$ is a local time at $0$ for $X-I$ and we denote the corresponding excursion measure by $N$, and we refer  to $X$ under $N$ as a $\psi$-Lévy excursion. For $e \in \mathbb{D}(\mathbb{R}_+ , \mathbb{R})$, we still write  $\sigma_e$ for the lifetime of  $e$, i.e. $\sigma_e :=  \inf \big\{t\geq 0:~e(s)=0 \, \text{ for every }s \geq t\big\}$.   Now, denote the excursion intervals of $X-I$ away from $0$ by $(a_i,b_i)_{i \in \mathbb{N}}$ --  recall that these are defined as the connected components of the open set $\{ t \geq 0 : X_t - I_t > 0 \}$ --  and for every $i \in \mathbb{N}$, we set $e_i =( X_{(a_i + t) \wedge b_i} - I_{a_i} : t \geq 0)$ for the corresponding excursion. Remark that since $I_t \rightarrow  - \infty$ as $t \uparrow \infty$, every excursion interval has finite duration $\sigma_{e_i}=b_i-a_i$, and it can be further shown \cite[Lemma 1.2.2 or see \eqref{equation:zeros} below]{DLG02} that   $(a_i,b_i)_{i \in \mathbb{N}}$ are precisely the excursion intervals away from $0$ of  $H$.   Now, the key observation is that the (shifted) restriction $(H_{(a_i + t)\wedge b_i} : t \geq 0)$ of $H$ to an arbitrary excursion interval $[a_i,b_i]$  can be written in terms of a functional that only depends on the corresponding excursion $e_i$, say $H(e_i)$. Informally, this should not come as a surprise,  since the definition of $H$ in \eqref{definition:alturaH}  only depends on the excursion straddling $t$; we refer to the discussion preceding Lemma 1.2.4 in \cite{DLG02} for a more detailed account. To simplify notation, under $N(\dd e)$ we write $H$ for $H(e)$.  We now have all the necessary ingredients to define Lévy trees.
\\
\\
\textbf{Lévy trees.} 
Under $N$, the quotient space $\mathcal{T}_H$ equipped with the metric $d_H$ and rooted at $p_H(0)$ is  the Lévy tree with branching mechanism $\psi$. In particular,   when $\psi(\lambda) = \lambda^2/2$ for $\lambda \geq 0$, the height process $H$ under $N$ is a non-negative Brownian excursion, the corresponding tree $\mathcal{T}_H$  is the so-called Brownian tree, and $\mathcal{T}_H$ under $N(\, \cdot \,  | \sigma_{H} = 1)$ is the celebrated Brownian continuum random tree. If we work instead under $P$, the height process $H$ is still a continuous, non-negative function    on $\mathbb{R}_+$ with infinite lifetime $\sigma_H$. Thus, $\mathcal{T}_H$ under  $P$ is a non-compact $\mathbb{R}$-tree  which can be thought of as a collection of Lévy trees concatenated at their root.  Let us be more precise.  Write $(a_i,b_i)_{i \in \mathbb{N}}$ for the excursion intervals of $H$ away from $0$ and for every $i \in \mathbb{N}$, set  $H^i:= (H_{(a_i + t)\wedge b_i} : t \geq 0)$.  The tree $\mathcal{T}_H$ under $P$  can be interpreted as the concatenation at the root  of the collection of trees $(\mathcal{T}_{H^i} : i \in \mathbb{N})$, according to  the order induced by the local time  $-I$.  For this reason, under $P$ we refer to $\mathcal{T}_H$ as a forest of  $\psi$-Lévy trees.  We stress that under $P$ and $N$ we have $H_0=0$, so $H$ codes  the distances to the root of the tree $\mathcal{T}_H$ when  explored in clockwise order in the sense of Section \ref{sec:trees:1}. 
\\
\\
One of the main difficulties arising in the study of the height process is that it is not Markovian as soon as $\pi \neq 0$. To circumvent this difficulty, we will need to introduce a measure-valued Markov process,  called the exploration process,   which roughly speaking carries the information needed to make  $H$ Markovian.

\subsubsection{The exploration process} \label{subsection:explorationprocess}

We denote the space of finite measures in $\mathbb{R}_+$ equipped with the topology of weak convergence by  $\mathcal{M}_f(\mathbb{R}_+)$, and with a slight abuse of notation we still denote   the identically null measure by  $0$. For each fixed $t \geq 0$, we write $\dd_s I_{s,t}$ for the Lebesgue-Stieltjes measure associated with the non-decreasing, continuous mapping $r \mapsto I_{r,t}$ for $r \in [0,t]$. The exploration process is the $\mathcal{M}_f(\mathbb{R}_+)$-valued process $\rho = (\rho_t : t \geq 0)$ defined,  for every measurable function $f:\mathbb{R}\to \mathbb{R}_+$,  by the relation 
    \begin{equation*}
        \langle \rho_t , f \rangle := \int_{[0,t]} \dd_s I_{s,t}~ f(H_s), \quad t \geq 0. 
    \end{equation*}
    Note that in particular, the total mass $\langle \rho_t , 1  \rangle$ for $t \geq 0$ is simply given by 
 \begin{equation}\label{eq:rho:1:mass}
 \langle \rho_t,1 \rangle=X_t-I_t,\quad t\geq 0.
 \end{equation}   
  Despite its rather technical definition, the exploration process possesses crucial properties which make its study possible. Notably,   $(\rho_t : t \geq 0)$ is an   $\mathcal{M}_{f}(\mathbb{R}_{+})$--valued  rcll  strong Markov process  \cite[Proposition 1.2.3]{DLG02} and  the decomposition of the measure $\rho_t$ on its continuous and atomic parts is given by 
    \begin{equation}\label{rho_atoms}
\rho_{t}(\dd r) = \beta \mathbbm{1}_{[0,H_{t}]}(r)\dd r+\mathop{\sum \limits_{0<s\leq t}}_{X_{s-}<I_{s,t}}(I_{s,t}-X_{s-})\:\delta_{H_{s}}(\dd r),\quad t\geq 0. 
\end{equation}
It was later established in \cite[Theorem 2.1]{Abraham-Delmas} that the exploration process is in fact a Feller process. Let us now briefly address some important  properties of $\rho$ as well as its connection with $H$. To this end,  we use the notation $\mu$ for an arbitrary element of $\mathcal{M}_f(\mathbb{R}_+)$ and we denote the supremum of  its topological support  by $H(\mu)$, i.e. $H(\mu) := \sup ( \text{supp } \mu )$, with the convention $H(0) = 0$. The following properties hold $P$ - a.s.  
\begin{itemize}
    \item[\rm{(i)}]   $\text{supp } \rho_t = [0,H_t]$ for  every $t\geq 0$ at which $\rho_t \neq 0$.
    \item[\rm{(ii)}] The process $t \mapsto \rho_t$ is rcll with respect to the total variation distance. 
    \item[\rm{(iii)}] We have: 
\begin{equation}  \label{equation:zeros}
       \{ t \geq 0 : \rho_t = 0 \} = \{ t \geq 0 : X_t - I_t = 0 \} = \{ t \geq 0 : H_t = 0 \}
\end{equation}
 and the Lebesgue measure of these sets is null.
\end{itemize}
Observe that by points (i) and (iii) it holds that $(H_t: t \geq 0) = (H(\rho_t): t \geq 0)$,  and by (iii) the excursion intervals away from $0$ of $X-I$, $H$ and $\rho$  coincide   (in particular, under $P$ the process $\rho$ starts from $\rho_0 = 0$). Point (ii) was proved in \cite[Proposition 1.2.3]{DLG02} while  points (i) and the first statement in (iii) are a direct consequence of \cite[Lemma 1.2.2]{DLG02} and \eqref{eq:rho:1:mass}. Finally, the second statement in (iii) follows from  Theorems~6.5 and 6.7 in \cite{kyprianou2014fluctuations}. As we already mentioned, the process $\rho$ is a Markov process which starts from $0$ under $P$. In order to define its  law when  started from an arbitrary $\mu \in \mathcal{M}_f(\mathbb{R}_+)$  we need to introduce  two deterministic operations on the elements of $\mathcal{M}_f(\mathbb{R}_+)$. 
\smallskip \\
\textit{Pruning.} For $\mu \in \mathcal{M}_f(\mathbb{R}_+)$ and $a \geq 0$,  we write $\kappa_a\mu$ for the element of $\mathcal{M}_f(\mathbb{R}_+)$ characterised by the relation: 
\begin{equation}\label{definition:pruning}
    \kappa_a\mu ( [0,r] ) = \mu([0,r]) \wedge  ( \langle \mu , 1 \rangle - a),  
\end{equation}
with the convention $\kappa_a\mu = 0$ if $a \geq  \langle \mu , 1 \rangle$.
Roughly speaking, the operation $\mu \mapsto \kappa_a \mu$ prunes the measure $\mu$ from the tip of its topological support so that the remaining mass is $\langle \mu , 1 \rangle - a$. Remark that despite the fact that $H(\mu)$ might be infinite, 
  the measure  $\kappa_a \mu$ has compact support for every $a > 0$. Further, if $\mu$ has compact support and  $\text{supp } \mu = [0,H(\mu)]$, the mapping $\mathbb{R}_+ \ni  a \mapsto H(\kappa_a \mu)$ is continuous.  \smallskip  \\
  \textit{Concatenation}. Consider  two elements $\mu , \nu$ of $\mathcal{M}_f(\mathbb{R}_+)$ and assume $H(\mu) < \infty$. We write  $[\mu,\nu]$ for the element of $\mathcal{M}_f(\mathbb{R}_+)$ defined  by:  
  \begin{equation*}
      \langle  [\mu , \nu] ,f  \rangle := \int_{[0, H(\mu) ]} \mu( \dd r) \,  f(r) + \int_{\mathbb{R}_+} \nu(\dd r)f(H(\mu) + r),   
  \end{equation*}
  where $f:\mathbb{R}\to \mathbb{R}_+$ is any arbitrary measurable function. \smallskip \\
  \indent 
  We can  now  define the law of the exploration process started from an arbitrary measure $\mu \in \mathcal{M}_f(\mathbb{R}_+)$. Under $P$, we write  $\rho^{\mu} = ( \rho_t^{\mu} : t \geq 0 )$ for the $\mathcal{M}_f(\mathbb{R}_+)$-valued process defined at time $t = 0$ as $\rho^\mu_0 := \mu$ and  
  \begin{equation} \label{equation:LeyExploration}
      \rho^{\mu}_t := [\kappa_{-I_t} \mu , \rho_t],\quad \text{for } t>0. 
  \end{equation}
   Remark that the right-hand side is well defined since $\kappa_a \mu$ has compact support for every $a > 0$.  We will use the notation $\mathbf{P}_\mu$ to denote the law  of $\rho^\mu$ in $\mathbb{D}(\mathbb{R}_+, \mathcal{M}_f(\mathbb{R}_+))$. If for $r \geq 0$ we set $T_r := \inf \{ t \geq 0 : -I_{t} = r \}$,  it follows from our definitions that 
   \begin{equation}\label{equation:LevyPmu}
       \langle \rho_t^{\mu}, 1 \rangle 
       = X_t + \langle \mu , 1 \rangle, \text{ for } 0 \leq t \leq T_{\langle \mu , 1 \rangle}, \quad \text{ and } \quad  \langle \rho_t^{\mu}, 1 \rangle  = X_t - I_t,  \text{ for } t \geq T_{\langle \mu , 1 \rangle}. 
   \end{equation}
   Said otherwise,  the process $(\langle \rho_t^\mu , 1 \rangle : 0 \leq t \leq T_{\langle \mu,1\rangle} )$ is distributed as the Lévy process $X$ started from $\langle \mu , 1 \rangle$ taken up to its hitting time to $0$,  and  $(\langle \rho_{t+T_{\langle \mu,1\rangle}}^\mu , 1 \rangle : t\geq 0)$ has the same law as $(\langle \rho_t , 1 \rangle:~t\geq 0)$ under $P$.   In this work, we will exclusively consider  initial conditions $\mu$ in  the following  subset of $\mathcal{M}_f(\mathbb{R}_+)$   
\begin{equation}\label{df:M:O:f}\mathcal{M}^{0}_{f}:=\Big\{\mu\in \mathcal{M}_{f}(\mathbb{R}_{+}):\:H(\mu)<\infty \ \text{ and } \text{supp } \mu = [0,H(\mu)]\Big\}\cup\big\{0\big\}. 
\end{equation}
Let us explain the reason behind this restriction. First and foremost, remark that for $\mu \in \mathcal{M}_f^0$ the process $H(\rho^\mu)$  is continuous, a property that might fail for arbitrary elements $\mu \in \mathcal{M}_f(\mathbb{R}_+)$. Furthermore, by property (i) above  and \eqref{equation:LeyExploration},   the process $\rho^\mu$  for $\mu\in \mathcal{M}^{0}_{f}$ takes values in $\mathcal{M}^{0}_{f}$. This fact will be used frequently in the sequel.  
 \par  
 Informally, the exploration process  encodes at each time $t \geq 0$, the height of the subtrees in $\mathcal{T}_H$ attached ``to the right'' of the geodesic path $\llbracket 0,p_H(t) \rrbracket$. We now introduce another measure-valued process closely related to the exploration process that informally codes at each time $t \geq 0$, the height of the subtrees attached ``to the left'' of the same geodesic path. This process will  arise in our work  when discussing spinal decompositions in   Section \ref{section:MainSpinalDecomp}.   Under $P$,  we consider the measure-valued process $\eta := (\eta_t : t \geq 0)$  defined by 
    \begin{equation} \label{definition:eta}
        \eta_t(\dd r) := \beta \mathbbm{1}_{[0,H_{t}]}(r)\dd r+\mathop{\sum \limits_{0<s\leq t}}_{X_{s-}<I_{s,t}}(X_{s} - I_{s,t} )\:\delta_{H_{s}}(\dd r),\quad t\geq 0.
    \end{equation}
 The process $\eta$ is also rcll  with respect to the total variation distance of measures \cite[Corollary 3.1.6]{DLG02} and  takes values in  $\mathcal{M}_f(\mathbb{R}_+)$ \cite[Lemma 3.1.1]{DLG02}. Further, $P$-a.s. we have $H(\eta_t) = H(\rho_t)$ for every $t \geq 0$ and the set $\{ t \geq 0 : \eta_t = 0 \}$ coincides with \eqref{equation:zeros}. The pair $(\rho, \eta)$, is a strong Markov process \cite[Proposition 3.1.2]{DLG02}. For a complete account on $(\eta_t : t \geq 0 )$ we refer to  \cite[Section 3.1]{DLG02}, see as well \cite{LGLJ98}  for heuristics on $(\rho , \eta)$ in terms of queuing systems.  The argument we employed to define  $H$ under $N$ can be easily extended to  the pair
 $(\rho , \eta)$;  it is plain that under $N$, we still have $H=H(\rho)=H(\eta)$   and that properties (i)-(iii) still hold.   As a straightforward  consequence of our previous discussion and  excursion theory for  the reflected Lévy process $X-I$ we deduce that, under $P$, if we write $(a_i,b_i)_{i \in \mathbb{N}}$ for the connected component of $\{ t \geq 0 : X_t - I_t > 0 \}$, the random measure in $\mathbb{R}_+ \times \mathcal{M}_f(\mathbb{R}_+)^2$ defined by
\begin{equation} \label{definition:PoissonRMExcursionesRho}
    \sum_{i \in \mathbb{N}}\delta_{(-I_{a_i} , \rho_{(a_i + \cdot)\wedge{b_i}}, \eta_{(a_i + \cdot)\wedge{b_i}} )}
\end{equation}
is a Poisson point measure with intensity $\mathbbm{1}_{\mathbb{R}_+}(\ell)\dd \ell  \, N(\dd  \rho, \dd \eta    )$. Note that (iii) and our discussion on the process $\eta$ immediately yields that under $P$, the measure $(0,0) \in \mathcal{M}_f^2(\mathbb{R}_+)$ is regular and instantaneous for the Markov process $(\rho, \eta)$, and that $-I$ is a local time. It then follows that the corresponding excursion measure away from $(0,0)$ for the pair $(\rho , \eta)$ is precisely $N(\dd \rho , \dd \eta)$.  The process $\eta$  is often referred to as the dual of $\rho$. This terminology is justified by the following remarkable identity in distribution:
\begin{equation}\label{dualidad:etaRho}
     \big( (\rho_t, \eta_t) : t \in [0,\sigma_{H(\rho)}] \big) \overset{(d)}{=} \big( (\eta_{(\sigma_{H(\rho)} -t )-}, \rho_{(\sigma_{H(\rho)} -t )-} ) : t \in [0,\sigma_{H(\rho)}]  \big), \quad \text{under } N,
\end{equation}
with the convention that $(\eta_{0-}, \rho_{0-} ) = (\eta_{0}, \rho_{0} )$,  we refer to \cite[Corollary 3.1.6]{DLG02} for a proof. 
\par 
 Let us close this section  with a discussion which concerns the reconstruction of $H, \rho,\eta $ and $X$ under $P$ or $N$. Namely, we will prove that given any of these four processes,  the remaining three can be recovered in a measurable.  First, recall that  $H,\rho$ and $\eta$ are functionals of $X$, and that $H=H(\rho)=H(\eta)$. Thus,  it remains to check that one can recover $X$  from the height process $H$. To this end, let us recall that  by \cite[Lemma 1.3.2]{DLG02} and the monotonicity of $t\mapsto I_{t}$ we have:
\begin{equation}\label{temps:local:I}
    \lim \limits_{\epsilon \to 0}E \big[\sup\limits_{s\in[0,t]}\big| \epsilon^{-1} \int_{0}^{s}\dd r \, \mathbbm{1}_{\{0< H_r<\epsilon\}}+I_s\big|\big]=0,\quad \text{ for every }t\geq 0.
\end{equation}
The convergence in the previous display yields that  $-I$ can also be thought of as the local time at $0$ of $H$. In particular, by \eqref{temps:local:I} we can find a decreasing subsequence $(\varepsilon_k)_{k\geq 0}$ of positive numbers converging to $0$ along which $P$-a.s. we have
\begin{equation}\label{temps:local:I:p:s}
    \lim \limits_{k\to \infty}\sup\limits_{s\in[0,t]}\big| \epsilon^{-1}_k \int_{0}^{s}\dd r \, \mathbbm{1}_{\{0< H_r<\epsilon_k\}}+I_s\big|=0,\quad \text{ for every }t\geq 0.
\end{equation}
From now on, we fix the subsequence $(\varepsilon_k)_{k\geq 0}$. The following result is undoubtedly known, but since we have not been able to find a reference in the literature we give a quick proof.
\begin{lem}\label{coro:reconstruccionRhoX}
    Under $P$ and $N$ and  for each $t \geq 0$,  the following convergence holds a.e. 
\begin{equation} \label{equation:X-I}
    X_t - I_t =\lim \limits_{k\to \infty}\epsilon^{-1}_k\int_{t}^{\infty}  \mathrm{d}r \,  \mathbbm{1}_{\{0< m_H(t,r) <H_r< m_{H}(t,r) +\epsilon_k\}}.
\end{equation} 
\end{lem}
By right-continuity of $X$, we deduce from   the previous lemma and \eqref{temps:local:I:p:s} that under $P$ and $N$, we can recover the process   $X$ from $H$ in a measurable way. We shall refer to the process $X$ obtained from $H$ under $P$ (resp $N$) through \eqref{temps:local:I:p:s} and \eqref{equation:X-I} as the $\psi$-Lévy process (resp. $\psi$-Lévy excursion) associated with $H$.  
\begin{proof}
We only establish the lemma under $P$, the result under $N$ follows by the same arguments. In this direction we fix $t \geq  0$ and we start by expressing \eqref{equation:X-I} solely in terms of $\rho$. To this end,   recall from \eqref{eq:rho:1:mass} the identity $X_t - I_t = \langle \rho_t , 1 \rangle$ and  note  that as a direct consequence of \eqref{equation:zeros},  the variable  $\inf\{ s \geq t : H_s = 0 \}$ coincides with $\inf\{ s \geq t : \langle \rho_s , 1 \rangle = 0 \}$. Further,  by the Markov property, the law of $(\rho_{s+t}: s \geq 0 )$ under $P$ is $\textbf{P}_{\rho_t}$, where we recall that for $\mu \in \mathcal{M}_f(\mathbb{R}_+)$ the measure $\textbf{P}_\mu$ is the law under $P$  of $\rho^{\mu}$  as defined in \eqref{equation:LeyExploration}.  By our previous remarks,  to obtain \eqref{equation:X-I} it suffices to show that  for every fixed  $\mu\in \mathcal{M}_f(\mathbb{R}_+)$, it holds that: 
\begin{equation}\label{eq:Markov:reconstruction:lem}
    \langle\mu,1\rangle=\lim \limits_{k\to \infty}\epsilon^{-1}_k\int_{0}^{\inf\{ s \geq 0 : \rho^\mu_s = 0\}} \hspace{-8mm} \mathrm{d}r \,  \mathbbm{1}_{\{0<H(\rho^\mu_r)\leq \min_{[0,r]} H(\rho^\mu) +\epsilon_k\}}, \quad P\text{-a.s.}
\end{equation}
By construction, under $P$ we can write   $H(\rho_r^\mu) = H(\kappa_{-I_r}\mu ) + H(\rho_r)$ as well as  $\min_{[0,r]}H(\rho^\mu) = H(\kappa_{-I_r}\mu )$ for $r \geq 0$,  and remark that      $\inf\{ s \geq 0 : \rho^\mu_s = 0\}$ coincides with  $T_{\langle \mu , 1 \rangle} = \inf \{ s \geq 0 : -I_s = \langle \mu, 1 \rangle \}$ by \eqref{equation:LeyExploration}. This  gives that the convergence   \eqref{eq:Markov:reconstruction:lem}  can be written as  
$$  
\langle\mu,1\rangle=\lim \limits_{k\to \infty}\epsilon^{-1}_k\int_{0}^{T_{\langle\mu,1\rangle}}\mathrm{d}r \,  \mathbbm{1}_{\{0<H_r\leq\epsilon_k\}}, \quad P\text{-a.s.}
$$
However, the convergence in the last display  clearly holds  by  \eqref{temps:local:I:p:s}. This concludes the proof of the lemma. 
\end{proof}

In what follows, our analysis is carried over in terms of $\rho$, instead of $X$. In particular, we use the notation $H(\rho)$ instead of $H$. This is due to the fact that the exploration process is better suited to study Lévy trees. 
 
\subsubsection{The Lévy snake}\label{section:snake}

Fix a branching mechanism $\psi$ as in Section \ref{sec:trees:1} as well as the law of a continuous strong Markov  process $\Pi = (\Pi_y)_{y \in E}$ with values in some fixed  Polish space $(E, \dd_E)$. The goal of this section is to define the snake with  spatial motion   $\Pi$   driven by the height process of a $\psi$-Lévy process. This will give rise to the notion of a $\Pi$-Markov process indexed by a $\psi$-Lévy tree. In this direction,  recall that for $\mu \in \mathcal{M}_f(\mathbb{R}_+)$ the measure    $\mathbf{P}_\mu$  stands for the law in $\mathbb{D}(\mathbb{R}_+, \mathcal{M}_f(\mathbb{R}_+) )$ of the exploration process started from $\mu$. With a slight abuse of notation, it will be convenient to write $\rho = (\rho_t: t \geq 0)$ for the canonical process in $\mathbb{D}(\mathbb{R}_+, \mathcal{M}_f(\mathbb{R}_+) )$.  We shall argue that, for an appropriate subset of initial conditions $(\mu, \w) \in \mathcal{M}_f(\mathbb{R}_+) \times \mathcal{W}_E$ and pairs $(\psi, \Pi)$,  the probability measure $Q_\w^{H(\rho)}(\dd W)$ under $\mathbf{P}_\mu$ is well defined in $\mathbb{D}(\mathbb{R}_+, \mathcal{W}_E)$ and supported on $\mathcal{S}^\circ$.    First,  remark that it is imperative to consider initial conditions $(\mu, \mathrm{w})$ for which  $\zeta_{\mathrm{w}}= H(\mu)$, and ensuring the continuity of $H(\rho)$ under $\mathbf{P}_\mu$. By the discussion following \eqref{df:M:O:f}, it is clear that both conditions are  verified for any pair $(\mu , \w)$ in the following subset: 
\begin{equation*}
    \Theta:=\Big\{(\mu, \w) \in \mathcal{M}_f^0 \times \mathcal{W}_{E}:~ H(\mu)=\zeta_{\w}\Big\},
\end{equation*}
for  $\mathcal{M}^{0}_{f}$ defined as in \eqref{df:M:O:f}. Next, recall that in Section \ref{secsnake}, we gave, for a continuous non-negative function $h$ with $h(0) = \zeta_\w$, sufficient conditions [(i)-(ii) in Section \ref{secsnake}] on $(h,\Pi)$  ensuring that $Q_{\mathrm{w}}^{h}$ can be defined as a measure on $\mathbb{D}(\mathbb{R}_+, \mathcal{W}_E)$ with support on $\mathcal{S}^\circ$. We shall then enforce conditions on $\psi$ and $\Pi$ to ensure that 
conditions (i)-(ii) from Section \ref{secsnake} are verified $\mathbf{P}_{\mu}$-a.s. by $(H(\rho), \Pi)$. To this end,  we will always assume that $\psi$ and $\Pi$ verify\footnote{For instance, it can  be checked that condition \eqref{continuity_snake} is fulfilled if the Lévy tree has exponent $\psi(\lambda)= \lambda^\alpha$ for $\alpha \in (1,2]$ and $\Pi$ is a linear   Brownian motion.} the following hypothesis.
\begin{enumerate}
    \item[\mbox{}] \textbf{Hypothesis $(\textbf{H}_{0})$}  There exists  a constant $C_\Pi> 0$ and two positive numbers $p,q > 0$ such that,\\ for every $y \in E$ and $t\geq 0$, we have:
\begin{equation} \tag{$\mathbf{H_{0}}$} \label{continuity_snake}
    \hspace{-15mm}\Pi_{y}\big(  \sup_{0 \leq s \leq t } \dd_E(\xi_s , y)^p  \big) \leq C_\Pi\cdot t^{q}, \hspace{7mm} \text{ and } \hspace{7mm} q\cdot (1-\Upsilon^{-1})>1,
    \end{equation}
   where $\Upsilon:=\sup \big\{ r \geq 0 : \lim_{\lambda \rightarrow \infty} \lambda^{-r}\psi(\lambda) = \infty \big\}.$\footnote{ Remark that the convexity  of $\psi$ ensures that $\Upsilon \geq 1$. }
\end{enumerate}
Let us explain the implications of these conditions.   
Fix  $(\mu,\mathrm{w})\in \Theta$ and a pair $(p,q)$ satisfying assumption \eqref{continuity_snake}. Note that   in particular condition (i) from Section \ref{secsnake} is   satisfied. By  \cite[Theorem 1.4.4]{DLG02} (see as well the discussion following assumption H$_0$ in page 18 of \cite{2024structure}), for every  $\theta \in (0, 1-1/\Upsilon)$, the height process $H(\rho)$ is $\mathbf{P}_\mu$-a.s. locally $\theta$-Hölder continuous in   the excursion intervals of $H(\rho)$ over its running infimum. Thanks to the standing  assumption $q \cdot  (1- \Upsilon^{-1}) > 1$    we can further impose on $\theta$ that $ q \theta >1$, so that   condition (ii) from Section \ref{secsnake} is satisfied $\textbf{P}_{\mu}$--a.s. by $H(\rho)$ for every such $\theta$. We are then in position to apply Proposition 2.2 in \cite{2024structure} which entails that  $\textbf{P}_{\mu}$--a.s. the measure $Q_{{\w}}^{H(\rho)}$ is well defined in $\mathbb{D}(\mathbb{R}_+, \mathcal{W}_E)$ and is supported on $\mathcal{S}^\circ$. If, with a slight abuse of notation, we write  $(\rho , W)$ for the canonical process in $\mathbb{D}(\mathbb{R}_+, \mathcal{M}_f(\mathbb{R}_+)\times \mathcal{W}_E )$,  it now follows from our previous discussion that for every $(\mu,\text{w})\in \Theta$, we can define a probability measure in  $\mathbb{D}(\mathbb{R}_+, \mathcal{M}_f(\mathbb{R}_+)\times \mathcal{W}_E )$  by setting  
\begin{equation}\label{definition:Pmuw}
\mathbb{P}_{\mu,\text{w}}(\dd \rho,\: \dd W):=\textbf{P}_\mu(\dd \rho)\:Q^{H(\rho)}_{\text{w}}(\dd W). 
\end{equation} 
 The process $(\rho , W)$ under $\mathbb{P}_{\mu , \w}$ is called the $\psi$-Lévy snake with spatial motion $\Pi$ started from $(\mu , \w)$. The key reason to work with the pair $(\rho, W)$, instead of solely with $W$, is that the Lévy snake is a Markov process (while the process $W$ is only Markovian when $\pi \neq 0$). Namely, if we write $\mathcal{F} := (\mathcal{F}_t: t \in [0,\infty])$ for its canonical filtration,  by  \cite[Theorem 4.1.2]{DLG02}  the family $((\rho , W), (\mathbb{P}_{\mu , \w} : (\mu , \w) \in \Theta ))$ is a strong Markov process with respect to the filtration $(\mathcal{F}_{t+}:~t\in [0,\infty])$. We stress that by construction, under $\mathbb{P}_{\mu , \w}$ for $(\mu , \w) \in \Theta$ and outside of a negligible set,  the process  $W$  takes values in $\mathcal{S}^\circ$, and its lifetime process $\zeta$ is continuous and indistinguishable from  $H(\rho)$.  These facts will be used frequently in the sequel. 
\par  
Let us now  introduce the notion of \textit{snake path} which summarises the regularity properties of $(\rho,W)$ that we shall need.  We denote systematically  the elements of the   space  $\mathbb{D}(\mathbb{R}_+ , \mathcal{M}_f(\mathbb{R}_+) \times \mathcal{W}_E)$ by:  
\begin{equation*}
    (\uprho, \omega) = \big(  (\uprho_s, \omega_s) : \, s \in \mathbb{R}_+  \big).  
\end{equation*}
In particular, by definition we have $(\rho_s(\upvarrho), W_s(\omega)) = (\upvarrho_s, \omega_s )$ for $s \in \mathbb{R}_+$. 
\begin{def1}[Snake paths]\label{definition:snakePath}
Fix $(\mu , \w) \in \Theta$.  An element  $(\upvarrho , \omega) \in 
\mathbb{D}(\mathbb{R}, \mathcal{M}_f(\mathbb{R}_+) \times \mathcal{W}_{E})$  is called a snake path started from $(\mu , \w)$ if the following properties hold:
\begin{enumerate}
    \item[\emph{(S1)}]   $(\upvarrho_0 , \omega_0) = (\mu,  \w)$  and $\omega\in \mathcal{S}^\circ_\w$;
    \item[\emph{(S2)}] For every  $s \geq 0$, we have  $(\upvarrho_s, \omega_s) \in \Theta$; in  particular   $H(\upvarrho) = \zeta(\omega)$.
\end{enumerate}
The family of snake paths started from some fixed $(\mu , \w) \in \Theta$ is denoted by $\mathcal{S}_{\mu , {\w}}$, and when $(\mu , \w)$ is of the form $(0,y)$ for some $y \in E$, we simply write $\mathcal{S}_{y}$. 
We set 
\begin{equation*}
   \mathcal{S} := \bigcup_{(\mu , \w)\in \Theta}\mathcal{S}_{\mu , \w} 
\end{equation*}
for the collection of snake paths. 
\end{def1}
When working with a snake path  $(\upvarrho, \omega)$, the equivalent notations $\zeta_s(\omega)$ and  $H(\upvarrho_s)$ for $s \geq 0$ will be used indifferently, and note that in particular $\sigma_{H(\upvarrho)}$ and $\sigma(\omega)$ coincide. The continuity of  $H(\rho)$ and $\zeta$ under $\mathbb{P}_{\mu , \w}$ for $(\mu , \w) \in \Theta$ paired with the fact that $\mathbb{P}_{\mu , \w}$ a.s. the process $W$ belongs to $\mathcal{S}^\circ$,  immediately gives by construction of the measure $\mathbb{P}_{\mu , \w}$  that the process $((\rho , W), (\mathbb{P}_{\mu , \w} : (\mu , \w) \in \Theta ))$ takes a.s. values in $\mathcal{S}$. Said otherwise, for each $(\mu , \w) \in \Theta$ and  $\mathbb{P}_{\mu,\text{w}}\text{-a.s.}$ we have 
\[\zeta_{s}=H(\rho_{s}), \:\: \text{ for every } s \geq 0, \]
and  for any $t, t'  \geq 0$, 
\[W_{t}(r)=W_{t^{\prime}}(r), \:\: \text{ for all } 0 \leq r\leq m_{H(\rho)}(t,t^{\prime}). \]
  Recalling the discussion at the end of Section \ref{secsnake}, under  $\mathbb{P}_{\mu,\text{w}}$ and outside of a negligible set, we can  consider the tree-indexed process $(\widehat{\omega}_a:~a\in \mathcal{T}_{H(\uprho)})$ and we call it the Markov process $\Pi$ indexed by the $\psi$-Lévy tree $\mathcal{T}_{H(\rho)}$  started from $(\mu,\text{w})$.   It is straightforward to check  from the definition that $\mathcal{S}$ is a measurable subset of the space $\mathbb{D}(\mathbb{R}, \mathcal{M}_f(\mathbb{R}_+) \times \mathcal{W}_{E})$.
\medskip 
\\
 The study of the Lévy snake and the geometric properties of the corresponding tree-indexed process is often performed by working with  restrictions of $(\rho,W)$ to different sub-intervals of interest. We now introduce a notion  that is well suited for this purpose. 
 \\
 \\
\noindent \textbf{Subtrajectories.} Let us start by introducing some notation. First, for $\mu  \in  \mathcal{M}_f(\mathbb{R}_+) $ and $z\geq 0$,  we let $ \uptheta_z(\mu)$ be the element of $\mathcal{M}_f(\mathbb{R}_+)$  defined by the relation:
    \begin{equation} \label{operacion:translacion}
    \langle \uptheta_{z} (\mu), f \rangle   := \int  \mu (\dd r)  f(r - z ) \mathbbm{1}_{\{r > z\}},    
\end{equation}
    where $f: \mathbb{R}_+ \mapsto \mathbb{R}$ is any  measurable bounded function. Similarly, for 
$\w\in \mathcal{W}_{E}$ and $z\geq 0$ satisfying the additional condition $z\leq \zeta_\w$,  we write $\uptheta_z(\w)$ for the shifted path $\w(z+r), r\in [0,\zeta_\w-z]$. 
   Now, consider an arbitrary  $(\upvarrho , \omega) \in 
\mathbb{D}(\mathbb{R}, \mathcal{M}_f(\mathbb{R}_+) \times \mathcal{W}_E)$  and fix $0 \leq s_1 < s_2 \leq \infty $ such that $\zeta_{s_1}(\omega) \leq  \zeta_t(\omega)$  for every $t \in [s_1,s_2]$.  The subtrajectory of $(\upvarrho, \omega)$  associated with  $[s_1,s_2]$  is the element of $\mathbb{D}(\mathbb{R}, \mathcal{M}_f(\mathbb{R}_+) \times \mathcal{W}_E)$ defined  by
       \begin{equation}\label{eq:subtrajectory:a:b}
        \mathscr{T}_{s_1,s_2}(\upvarrho , \omega) :=\big( \uptheta_{\zeta_{s_1}(\omega)}(\upvarrho_{(s_1+t) \wedge s_2}), \uptheta_{\zeta_{s_1}(\omega)}(\omega_{(s_1+t) \wedge s_2} )\big), \quad \text{ for } t \geq 0. 
    \end{equation}  
Let us briefly comment on the use of this definition in our main case of interest.  Consider an arbitrary element $(\upvarrho , \omega)$ from  $\mathcal{S}$. Let $0 \leq s_1 < s_2 \leq \sigma(\omega)$ be as before and  assume additionally that  $\zeta_{s_1}(\omega) = \zeta_{s_2}(\omega)$. In this case, it follows by construction that $\mathscr{T}_{s_1,s_2}(\upvarrho , \omega)$  belongs  to $\mathcal{S}_{\widehat{\omega}_{s_1}}$.  One can think of  $\mathscr{T}_{s_1,s_2}(\upvarrho , \omega)$ as coding the labels of the restricted process $(\widehat{\omega}_v:~v\in  p_{H(\uprho)}([s_1, s_2]))$. 
\medskip 
\\
\par We conclude this section by  introducing the excursion measure of the  Lévy snake and by recalling  some of its main properties. Identity \eqref{equation:zeros} combined with the snake property yields  that under $\mathbb{P}_{0,y}$ for every $y \in {E}$, the point $(0,y) \in \mathcal{M}_f^0 \times  \mathcal{W}_{{E}}$  is  regular and instantaneous for $(\rho, {W})$ and that $-I$ is a local time for $(\rho , {W})$ at $(0, y )$. We write $\mathbb{N}_{y}$ for the corresponding excursion measure in $\mathbb{D}(\mathbb{R}, \mathcal{M}_f(\mathbb{R}_+) \times \mathcal{W}_E)$.  Let $(a_i,b_i)_{i \in \mathbb{N}}$ be the connected components of  $\{ t\geq 0 : \rho_t  \neq  0  \}$ and for every $i \in \mathbb{N}$, write $(\rho^i, {W}^i)$ for the subtrajectory  associated with the excursion interval $(a_i , b_i)$. By excursion theory, under $\mathbb{P}_{0,y}$  the measure 
\begin{equation} \label{PPPsobreinfimo}
     \sum_{i \in \mathbb{N}} \delta_{(-I_{a_i} , \rho^i, {W}^i )} 
\end{equation}
is a Poisson point measure with intensity $\mathbbm{1}_{\mathbb{R}_+}(\ell) \dd \ell  \, \mathbb{N}_{y}(\dd \rho , \dd {W})$. Since the excursion measure of $\rho$ under $\mathbb{P}_{0,y}$ is $N(\dd\rho)$, it readily follows from the form of the conditional law of ${W}$ given $\rho$ that the measure $\mathbb{N}_{y}$ writes: 
\begin{equation}\label{N:H(rho)}
\mathbb{N}_{y}(\dd \rho, \: \dd {W})=N(\dd \rho )\: Q^{H(\rho)}_{y}(\dd {W}). 
\end{equation}
Recalling that $\eta$ under $N(\dd \rho)$ is a functional of $\rho$,  it follows that  $(\rho , \eta)$ under $\mathbb{N}_{y}$ is distributed as $(\rho , \eta)$ under the excursion measure $N$ and, conditionally on the pair $(\rho , \eta)$, the process ${W}$ has the law of a snake driven by $H(\rho)$ with spatial motion $\Pi_y$. Under $\mathbb{N}_y$  we write $\sigma$ for the lifetime of $\zeta$, which coincides with  $\sigma_{H(\rho)}$.  Let us briefly comment on the Markovian character of $\mathbb{N}_y$ as well as on its properties under time-reversal. Starting with the latter, by the disintegration  of  $\mathbb{N}_y$ given above and \eqref{dualidad:etaRho},  under  $\mathbb{N}_{y}$   the following  identity  holds in distribution: 
\begin{equation}\label{dualidad:etaRhoW}
     \big( (\rho_t, \eta_t, {W}_t) : t \in [0,\sigma] \big) \overset{(d)}{=} \big( (\eta_{(\sigma -t )-}, \rho_{(\sigma -t )-} , {W}_{\sigma -t } ) :t \in [0,\sigma] \big). 
\end{equation}
  We refer to the identity in the last display as the duality property of the Lévy snake.  Further, the Lévy snake is still a strong Markov process under $\mathbb{N}_y$, and the strong Markov property under  $\mathbb{N}_{y}$ takes the following form. Let $T > 0$ be a  $\mathcal{F}_{t+}$-stopping time and consider an arbitrary non-negative $\mathcal{F}_{T+}$-measurable function $\Phi$. By \cite[Theorem 4.1.2]{DLG02}, for every nonegative measurable  function   $F$  on $\mathbb{D}(\mathbb{R}_+, \mathcal{M}_f(\mathbb{R}_+) \times \mathcal{W}_E)$,   we have\footnote{Let us mention that \cite[Theorem 4.1.2]{DLG02} is stated for snakes taking values on a  more general space than $\mathcal{W}_{{E}}$, namely the space of killed $E$-valued rcll  paths endowed with  a different metric than $d_{\mathcal{W}_E}$. It is however straightforward to check that the proof of \cite[Theorem 4.1.2]{DLG02} still holds in our framework.} 
\begin{equation*}
    \mathbb{N}_{y} \big( \mathbbm{1}_{\{ T < \infty \}} \Phi \cdot  F (  \rho_{T+s} , {W}_{T+s} : s \geq 0 )  \big)=  \mathbb{N}_{y} \big( \mathbbm{1}_{\{ T < \infty \}} \Phi \cdot \mathbb{E}_{\rho_T , {W}_T}^\dag [F] \big)    
\end{equation*} 
where we write $\mathbb{P}_{\mu , {\w}}^\dag$ for the law of $(\rho , W)$ stopped at time $\inf\{ s \geq 0 : \rho_s = 0 \}$ under $\mathbb{P}_{\mu , \w}$.

\section{Setup for the excursion theory}\label{section:framework}

In this section we fix the framework that we shall consider for the rest of this manuscript and under which we  develop the excursion theory.  We work under the same assumptions as in the companion paper \cite[Sections 4-5]{2024structure},  these are restated here for the reader's convenience.
\medskip \\
Fix  a Polish space $E$ and  the  family of distributions  $\Pi := (\Pi_y)_{y \in E}$ on the canonical space $\mathbb{C}(\mathbb{R}_+ , E)$ of an $E$-valued continuous strong Markov process.  As usual, we write $\xi = (\xi_t : t \geq 0)$  for the canonical process on $\mathbb{C}(\mathbb{R}_+ , E)$.   We assume the existence of a reference point $x$ in $E$  satisfying the following conditions: 
\begin{equation}\label{Asssumption_2}
\bm{x} \:\:\textbf{is regular, instantaneous and recurrent for}\:\:  {\xi}, \tag*{\text{$\mathbf{(H_{1})}$}}
\end{equation}
and
\begin{equation}\label{Asssumption_3}
{\int_{0}^{\infty}\dd t \:\mathbbm{1}_{\{\xi_{t}= {x}\}} =0}, \quad  \quad \Pi_x -\text{a.s. } \tag*{\text{$\mathbf{(H_{2})}$}}
\end{equation}
The first two conditions in hypothesis \ref{Asssumption_2} ensure the existence of a local time  at $x$ for the Markov process, and  we denote  it by $\mathcal{L} = (\mathcal{L}_t : t \geq 0)$. The recurrence 
 hypothesis is assumed for convenience, and we expect our results to hold without it up to minor modifications.   The local time $\mathcal{L}$ is  unique up to a multiplicative constant, that we fix arbitrarily, and we write $\mathcal{N}$  for the corresponding (infinite) excursion measure.   We set $\uptau(\xi) := \inf \{s > 0 : \xi_s = x \}$ for the return time to $x$ of $\xi$, and to simplify notation when there is no risk of confusion we simply denote it by $\uptau$.   If we write $(a_i,b_i)_{i \in \mathbb{N}}$ for the excursion intervals away from $x$ of $\xi$ and we let  $(\xi^{i,*})_{i \in \mathbb{N}}$ be the corresponding excursions with respective  duration  $\uptau _i := \uptau(\xi^{i,*})$,  the point measure 
  \begin{equation}\label{eq:mea:xi=*}
      \sum_{i \in \mathbb{N}}\delta_{(\mathcal{L}_{a_i}, \xi^{i,*})}
  \end{equation}
  under $\Pi_x$ is a Poisson point measure with intensity $\mathbbm{1}_{\mathbb{R}_+}(t)\dd t \otimes \mathcal{N}$. Note that  since $x$ is recurrent, the duration of each excursion is finite $\Pi_x$-a.s. Since the family $(a_i,b_i)_{i \in \mathbb{N}}$ are the connected components of $\mathbb{R}_+ \setminus \{  t \geq 0 : \, \xi_t = x\}$,    condition \ref{Asssumption_3} ensures that for any non-negative measurable function $f: E \mapsto \mathbb{R}_+$ we have $\int_0^\infty \dd s f(\xi_s) = \sum_{i \in \mathbb{N}} \int_{0}^{\uptau_i} \dd s \,  f(\xi_s^{i,*}), \, \Pi_{x}$--a.s.  This fact will be used frequently in our computations without further notice, but it is by no means the   reason behind working under  \ref{Asssumption_3}. The main implication of  \ref{Asssumption_3} is  postponed for now and discussed  in \eqref{prelim:Branchingx} below, as it concerns the geometry of the corresponding  tree indexed process.  It will be crucial for us to not only work with the Markov process $\Pi$ indexed by the underlying  Lévy tree, but  to keep track as well of the value of its local time along the branches of the tree. This can be easily solved by  working instead with the  pair $\overline{\xi}_s := ( \xi_s , \mathcal{L}_s )$,   $s \geq 0$, under $\Pi_y$ for $y \in E$. Namely, this is still a strong Markov process taking values in the Polish space $\overline{E}:=E\times \mathbb{R}_+$, that we  equip with the product metric $d_{\overline{E}}$.  We  write $\Pi_{y,r}$ for  its law on $\mathbb{C}(\mathbb{R}_+ , \overline{E})$ started from an arbitrary point $(y,r)\in \overline{E}$ and  for convenience we assume that $\overline{\xi}$ is the canonical process in $\mathbb{C}(\mathbb{R}_+ , \overline{E})$.   We still write $\overline{\xi} = (\xi , \mathcal{L})$ for its first and second component and to simplify notation, we set $\overline{\Pi} := ({\Pi}_{y,r})_{(y,r) \in \overline{E}}$. Under ${\Pi}_{y,r}$, we also use the notation $\uptau$ for $\inf\{ t > 0 : \xi_t = x \}$.  
  \par
  We next introduce another family of measures, closely related to the excursion measure, that will play a crucial role in our work.  For each $r \geq 0$ we let $\mathcal{N}_{x,r}$ be the law in $\mathbb{C}(\mathbb{R}_+, \overline{E})$ of the path obtained by concatenating  $(({\xi}_t, r): 0 \leq  t <  \uptau)$ under $\mathcal{N}$, with an independent process distributed according to $\Pi_{x,r}$. This somewhat informal description can be made precise by working in the product space $\mathbb{C}(\mathbb{R}_+ , \overline{E})^2$ under the product measure $\mathcal{N} \otimes \Pi_{x,r}$, and by considering a concatenation operation like the one described in \cite[Chapter II - 14]{Sharpe}. Since this procedure is standard  we skip the details. Heuristically, the measure $\mathcal{N}_{x,r}$ codes the law of the Markov process and its local time shifted at the beginning  of an excursion away from $x$, when the value of the local time at the start of the excursion is given by $r \geq 0$. It is not difficult to check that the process $\overline{\xi}$ under $\mathcal{N}_{x,r}$ is still Markovian in the following sense. For any $t > 0$  and every non-negative $\sigma(\overline{\xi}_s : 0 \leq s \leq t)$ - measurable function $\Phi$ we have  
\begin{equation}\label{equation:markovExcursionquevive}
    \mathcal{N}_{x,r} \big(  \Phi \cdot  G( \bar{\xi}_{t + s} : s \geq 0) \big)
    = 
     \mathcal{N}_{x,r} \big(  \Phi \cdot {\Pi}_{\bar{\xi}_{t}} [ G ] \big) ,
\end{equation}
where $G$ is an arbitrary non-negative measurable function on $\mathbb{C}(\mathbb{R}_+, \overline{E})$.   Let us mention that the collection $({\mathcal{N}}_{x,r}: r \geq 0)$ fall in the much more general framework of exit systems introduced by Maisonneuve in \cite{Maisonneuve}, and   the heuristic description of the measure $\mathcal{N}_{x,r}$ given above is justified by the so called exit formula \cite[Theorem 4.1]{Maisonneuve}.  See as well  \cite[Theorem 5.1]{Maisonneuve}  for a proof of the Markov property in a much more general setting. 
\par
Let us now return to the setting of Lévy snakes. Recall that we assume that $(\psi, \Pi)$ satisfies \eqref{continuity_snake}, and that  we write  $\mathbb{N}_{y}$, $y\in E$, and $\mathbb{P}_{\mu, \w}$, $(\mu,\w)\in \Theta$, for the family of measures associated with the $(\psi, \Pi)$-Lévy snake introduced in Section \ref{secsnake}.  In order to keep track of the local time along  the branches of the tree, we will often need to work with the  $(\psi, \overline{\Pi})$-Lévy snake, so that the spatial motion now consists in the Markov process paired with its local time $(\xi, \mathcal{L})$. This can be done as soon as additional regularity assumptions on the pair $(\psi, \overline{\Pi})$ are satisfied. To this end, we shall henceforth assume that  $(\psi, \overline{\Pi})$ satisfies the following slightly stronger version of hypothesis \eqref{continuity_snake}. 
\begin{enumerate}
    \item[\mbox{}] \textbf{Hypothesis $(\textbf{H}_{0}')$}.  There exists  a constant $C_{\overline{\Pi}}> 0$ and two positive numbers $p,q > 0$ such that, \\ for every $y \in E$ and $t\geq 0$, we have:
\begin{equation*} \label{continuity_snake_2}
    \hspace{-15mm}\Pi_{y,0}\Big(  \sup_{0 \leq s \leq t } d_{\overline{E}}\big((\xi_s, \mathcal{L}_s) , (y,0) \big)^p  \Big) \leq C_{\overline{\Pi}}\cdot t^{q} \hspace{7mm} \text{ and } \hspace{7mm} q\cdot (1-\Upsilon^{-1})>1.  \tag*{\text{$\mathbf{(H_{0}')}$}}
\end{equation*}
\end{enumerate}
Hence, the framework that we have discussed in Section \ref{section:snake} can be applied to the motion  $\overline{\Pi}$ and it allows us to work, when required,  with the  $\psi$-Lévy snake with spatial motion $\overline{\Pi}$.
\par 
Let us  briefly comment on some  notations that will be used when working with the $(\psi,\overline{\Pi})$-Lévy snake.  Every element of $\mathcal{W}_{\overline{E}}$ can be written in the form $\overline{\w} = (\w, \ell)$ for some $\w \in \mathcal{W}_E$ and $\ell \in \mathcal{W}_{\mathbb{R}_+}$ with identical lifetimes.   We  use respectively the notation   $\overline{\Theta}, \overline{\mathcal{S}}^\circ,  \overline{\mathcal{S}}_{\mu , \overline{\w}}, \overline{\mathcal{S}}$ for  the sets defined as $\Theta, \mathcal{S}^\circ,  \mathcal{S}_{\mu , \w}, \mathcal{S}$ but   associated with the Polish space $\overline{E}$. For every $(\mu , \overline{\w}) \in \overline{\Theta}$  the law of the $(\psi, \overline{\Pi})$-Lévy snake started from $(\mu , \overline{\w})$ is denoted by  $\mathbb{P}_{\mu , \overline{\w}}$,   and for every $(y,r) \in \overline{E}$ we let $\mathbb{N}_{y,r}$ be its excursion measure away from $(y,r)$; note that these are now measures on  $\mathbb{D}(\mathbb{R}_+, \mathcal{M}_f(\mathbb{R}_+)\times \mathcal{W}_{\overline{E}})$.  Further, we write $(\rho , \overline{W}) = (\rho , W, \Lambda)$  for the  canonical process in $\mathbb{D}(\mathbb{R}_+, \mathcal{M}_f(\mathbb{R}_+)\times \mathcal{W}_{\overline{E}})$,  where   $W_{s}:[0,\zeta_{{s}}]\mapsto E$ and $\Lambda_{s}:[0,\zeta_{{s}}]\mapsto \mathbb{R}_+$ for every $s \geq 0$,  and  we  write $\overline{\mathcal{F}} := (\overline{\mathcal{F}}_t: t \in [0,\infty])$ for its canonical filtration. 
Since the law of $(\rho , W)$ under $\mathbb{P}_{\mu , \overline{\w}}$ and $\mathbb{N}_{y,r}$  is respectively $\mathbb{P}_{\mu , {\w}}$ and $\mathbb{N}_{y}$, when the local time of the motion is not needed in our reasoning we systematically work under the latter measures. Note from the definition of $\mathbb{P}_{\mu ,\overline{\w}}$ that for every $(y,r) \in \overline{E}$ and under $\mathbb{P}_{0,y,r}$, for each fixed $s\geq 0$ and  conditionally on $\zeta_s$, the pair  $(W_s,\Lambda_s) = \big( (W_s(h) , \Lambda_s(h)\big): h \in [0,\zeta_s]\big)$ is distributed as $(\xi_h, \mathcal{L}_h)_{h\in [0,\zeta_s]}$ under $\Pi_{y,r}$. In particular,  the  Lebesgue-Stieltjes  measure induced by  $\Lambda_s$ is  supported on the closure of  $\{ h \in [0,\zeta_s) : W_s(h) = x \}$,  $\mathbb{P}_{0,y,r}$--a.e. Note however that  this property might fail if we work under $\mathbb{P}_{\mu , \overline{\w}}$ for an arbitrary $(\mu , \overline{\w}) \in \overline{\Theta}$. 
Therefore, it will be convenient for our purposes to impose further restrictions on the set of initial conditions $(\mu , \overline{\w}) \in \overline{\Theta}$ that we shall work with. In this direction, we let $\overline{\Theta}_x$ be the subset of $\overline{\Theta}$ conformed by pairs $(\mu , \overline{\w})$, with $\overline{\w} = (\w, \ell)$,  
 satisfying the following conditions: 
 \begin{enumerate} 
      \item[\rm{(i)}] The measure $\mu$ does not charge the set $\{ h \in [0,\zeta_\w] : \w(h) = x\}$, i.e.   
     \begin{equation*}
         \int_{[0, \zeta_{\w}]} \mu(\dd h) \, \mathbbm{1}_{\{ \w(h) = x  \}} = 0.
     \end{equation*}
     \item[\rm{(ii)}] $\ell$   is a non-decreasing continuous function and the support  of its Lebesgue-Stieltjes  measure is $$\overline{\big\{ h \in [0,\zeta_\w) : \w(h) = x \big\}}.$$ 
 \end{enumerate}
  The following lemma taken from \cite{2024structure} justifies our definition for $\overline{\Theta}_x$. 
 \begin{lem} \label{lemma:tequedasenThetax}\emph{\cite[Lemma 5]{2024structure} } For every  $(\mu, \overline{\w})\in \overline{\Theta}_x$ and $(y,r) \in \overline{E}$,  the process $(\rho,\overline{W})$ under  $\mathbb{P}_{\mu,\overline{\w} }$ and $\mathbb{N}_{y,r}$  takes values in  $\overline{\Theta}_x$ a.e.
 \end{lem}
This holds for instance for initial conditions  $(\mu, \overline{\w} )$ of the form $(0,y,r)$ for $(y,r) \in \overline{E}$, making the space $\overline{\Theta}_x$ the natural subset of initial condition to work with. 
With a slight abuse of notation, for every $(\mu , \overline{\w}) \in \overline{\Theta}_x$ and $(y,r) \in \overline{E}$, under  $\mathbb{P}_{\mu , \overline{\w}}$ and $\mathbb{N}_{y,r}$ we write 
 \begin{equation} \label{definition:treeindexedprocess:2}
    \big( (\widehat{W}_a, \widehat{\Lambda}_a) : a \in \mathcal{T}_{H(\rho)}  \big) 
\end{equation}
for the labels induced on $\mathcal{T}_{H(\rho)}$ by $(\widehat{W},\widehat{\Lambda})$ in the sense of Section \ref{secsnake}. Finally,  assumption \ref{Asssumption_3} entails that, for every  $(\mu,{\mathrm{w}})\in {\Theta}_x$ and $y \in {E}$,  under $\mathbb{P}_{\mu, {\w}}$ and $\mathbb{N}_{y}$, we have: 
\begin{equation} \label{prelim:Branchingx}
    \big\{ a \in \mathcal{T}_{H(\rho)} : \widehat{W}_a = x  \big\} \cap  \big\{ a \in \text{Mult}_3 (\mathcal{T}_{H(\rho)}) \cup \text{Mult}_\infty(\mathcal{T}_{H(\rho)}) : \, a  \neq 0 \big\} = \emptyset
\end{equation}
and we refer to  \cite[Proposition 4.4]{2024structure} for a proof. In words, since branching points in a Lévy tree can only be of multiplicity $3$ or infinite,  the process  $( \widehat{W}_a : a \in \mathcal{T}_{H(\rho)}  )$ can not take the value $x$ at a branching point in $\mathcal{T}_{H(\rho)} \setminus \{ 0 \}$ .

\section{\texorpdfstring{Excursions and trajectories stemming from $x$}{Excursions and trajectories stemming from x}}\label{section:excursionstrajectories}
In this section we introduce, the notion of excursion away from $x$,  the one of a trajectory stemming away from $x$, and related functionals crucial in their study. Our definitions are reminiscent from classical excursion theory of $\mathbb{R}_+$-indexed processes. First,  we shall need   a notion of trajectory $(\upvarrho , \omega) \in \mathbb{D}(\mathbb{R}_+, \mathcal{M}_f(\mathbb{R})\times \mathcal{W}_E )$ stopped at its first return to $x$. To this end, we define in Section \ref{section:truncationboundary} a truncation operation on the elements of $\mathbb{D}(\mathbb{R}_+, \mathcal{M}_f(\mathbb{R})\times \mathcal{W}_E )$ which informally,  prunes each $(\upvarrho_s, \omega_s)$, for $s \geq 0$,   upon its first return to $x$. For functions in $\mathbb{R}_+$ with values in $E$, the interface between the   function stopped at its first return to $x$  and the piece of path coming afterwards consists of a single point. In our setting this scenario is more complex, since the interface consists of a family of points that typically  are of fractal nature. Making use of the theory of exit local times, in Section \ref{section:truncationboundary} we introduce a functional well suited to study these interfaces. Section \ref{sub:sect:debut} is then devoted to the definition of excursions away from $x$, trajectories  stemming from $x$, and to addressing some elementary geometric properties closely related to these objects. 

\subsection{Truncation and exit local times}\label{section:truncationboundary}

We start defining the truncation operation  of an  arbitrary element $(\upvarrho , \omega)\in \mathbb{D}(\mathbb{R}_+ , \mathcal{M}_f(\mathbb{R}_+) \times \mathcal{W}_{E} )$, that we fix from now on.  To this end we first need to introduce some notations.  For $\w \in \mathcal{W}_{E}$, we write 
\begin{equation} \label{equation:taudef}
   \uptau(\w) := \inf \{ t > 0 : \w(t) = x \},
\end{equation}
for the first return time to $x$ of the path $\w$  and   we set 
\begin{equation}\label{definition:V:star}
    V_t (\upvarrho,\omega) := \int_0^t \dd s  \,  {1}_{\{ H(\upvarrho_s) \leq \uptau(\omega_s) \}},\quad t\geq 0.  
\end{equation} 
When $V_{\sigma(\omega)} (\upvarrho,\omega)  > 0$,  we denote the right-inverse of $(V_t (\upvarrho,\omega)  : t\geq 0)$ by $(\Gamma_t( \upvarrho,\omega)  : t\geq 0 )$, i.e. 
\begin{equation} \label{definition:gamma}
    \Gamma_t (\upvarrho,\omega)  = \inf \big\{ s \geq 0 : V_s (\upvarrho,\omega)  > t \big\}, \quad \text{for } t\in [0,  V_{\sigma(\omega)} (\upvarrho,\omega) ) ,
\end{equation}
with the convention that  $\Gamma_t (\upvarrho,\omega)  = \Gamma_{V_{\sigma} - } (\upvarrho,\omega) $,  for $t \geq V_{\sigma(\omega)} (\upvarrho,\omega) $.  For definiteness, when $V_{\sigma(\omega)} (\upvarrho,\omega)  = 0$ we simply set  $\Gamma_t(\upvarrho, \omega) = 0$, for every $t \geq 0$.  We shall write $\mathrm{tr}(\upvarrho, \omega)$ for the element of $\mathbb{D}(\mathbb{R}_+ , \mathcal{M}_f(\mathbb{R}_+) \times \mathcal{W}_{E} )$ defined by the relation: 
\begin{equation} \label{definition:troncature*}
    \mathrm{tr} (\upvarrho, \omega)_t := ( \upvarrho_{\Gamma_t(\upvarrho, \omega)}, \omega_{\Gamma_t(\upvarrho, \omega)} ), \quad \quad  \text{ for } t \geq 0.  
\end{equation} 
 We write $\text{tr}(\upvarrho)$ and $\text{tr}(\omega)$ respectively for the first and second coordinate of $\text{tr}(\upvarrho, \omega)$. We stress  that  $\text{tr}(\upvarrho ,\omega)$ is rcll since $\upvarrho, \omega$  and $\Gamma(\upvarrho,\omega)$ are rcll processes.    Observe that if $V_{\sigma(\omega)}(\upvarrho, \omega) = 0$ the truncation  $\mathrm{tr}(\upvarrho , \omega)$   only takes the value   $(\upvarrho_0, \omega_0)$. Hence, this operation is only of interest when applied to elements $(\upvarrho , \omega)$ for which $V_{\sigma(\omega)}(\upvarrho, \omega)$ is non-null. This is for instance  the case if  $(\upvarrho, \omega) \in \mathcal{S}$ and for  some $t \geq 0$ we have $\uptau(\omega_t) > 0$, and this is the scenario we are interested in.  Roughly speaking,  when applied to an element $(\upvarrho, \omega)$ of $\mathcal{S}$, since $H(\upvarrho)=\zeta(\omega)$  the truncation operation removes the trajectories $(\upvarrho_s, \omega_s)$  from $(\upvarrho,\omega)$ for which $\uptau(\omega_s) < \zeta_s(\omega)$  and glues the remaining endpoints. We also stress  that when $(\upvarrho, \omega)$ is an element of $\mathcal{S}$, the truncation $\mathrm{tr}(\upvarrho , \omega)$ is still in $\mathcal{S}$  since   $\mathrm{tr} (\omega)$ is  a snake trajectory  by   \cite[Proposition 10]{ALG15}, and  the others conditions in Definition \ref{definition:snakePath} are clearly satisfied.
\\
\\
The goal now is to define a function ``counting" the time spent  by $(\upvarrho , \omega)$  at those values of $s \geq 0$ at which $\uptau(\omega_s) = H(\upvarrho_s)$. To this end, we now introduce a notion of \textit{exit local time}  suited for this purpose;  we refer to \cite[Section 4.3]{DLG02} and \cite[Section 3.1]{2024structure} for background and the general theory. Recalling the definition of the sequence $(\varepsilon_k)_{k\geq 0}$ introduced in \eqref{temps:local:I:p:s}, we consider  the functional $L : \mathbb{D}(\mathbb{R}_+,\mathcal{M}_f(\mathbb{R}_+)\times \mathcal{W}_E) \rightarrow \mathbb{R}_+^{\mathbb{R}_+}$  defined for every $t\geq 0$ by the relation  
\begin{equation} \label{definition:exitMultiusos} 
  L_t(\upvarrho , \omega) :=    \liminf_{k\to \infty }   \frac{1}{\epsilon_k} \int_0^t \dd s~ \mathbbm{1}_{\{ \uptau(\omega_s) < H(\upvarrho_s) < \uptau(\omega_s) + \epsilon_k \}}.   
\end{equation}
When the previous convergence holds uniformly in compact intervals, with a limit instead of a liminf,   the process  $L(\upvarrho , \omega) = (L_s(\upvarrho, \omega) : s \geq 0)$ is  continuous and  non-decreasing. In particular,  we can consider the associated Stieltjes measure $\dd L(\upvarrho , \omega)$ which by \eqref{definition:exitMultiusos} is
 supported on $\{s\geq 0:~H(\upvarrho_s)=\uptau(\omega_s)\}$. The following lemma establishes that the previous discussion holds  for a.e. $(\upvarrho , \omega)$ under $\mathbb{P}_{0,y}$  and $\mathbb{N}_{y}$, for  $y \in {E}$, with $y \neq x$. 
 
\begin{lem}\label{lem:equa:convL}  For every $y \in {E}$ with $y \neq x$, under $\mathbb{P}_{0,y}$  and $\mathbb{N}_{y}$,  a.e.  for every $M \geq 0$ we have 
    \begin{equation}\label{equation:convL}
        \lim \limits_{k\to \infty}\sup_{t \leq M} \Big|  \frac{1}{\epsilon_k} \int_0^t \dd s ~\mathbbm{1}_{\{ \uptau(W_s) < H(\rho_s) < \uptau(W_s) + \epsilon_k \}} - L_t(\rho , W)   \Big|=0. 
        \end{equation}
\end{lem}
 \begin{proof}
 We fix an arbitrary $y \in E$ and  we start by establishing that the convergence \eqref{equation:convL} holds under $\mathbb{P}_{0,y}$. For every $t \geq 0$, set  
\[\sigma_{t}:=\inf\Big\{s\geq 0:\:\int_{0}^{s} \dd u~\mathbbm{1}_{\{H(\rho_u)> \uptau(W_u) \}} > t\Big\},\]
and we consider the process $(\rho'_{t})_{t\geq 0}$ taking values in $\mathcal{M}_{f}(\mathbb{R}_{+})$ defined,   for any bounded measurable function $f: \mathbb{R}_+ \to \mathbb{R}_+$, by the relation:
\begin{equation}\label{def:rho:D}
\langle  \rho'_{t}, f \rangle :=\int \rho_{\sigma_{t}}(\dd h)f\big(h-\uptau(W_{\sigma_{t}})\big) \mathbbm{1}_{\{ h>\uptau(W_{\sigma_t}) \}}.
\end{equation}
Then, by Proposition 4.3.1 in \cite{DLG02},  $\rho'$ and $\rho$ have the same distribution under $\mathbb{P}_{0,y}$. In particular,  $\langle \rho' , 1 \rangle$ has the same law as the reflected Lévy process $X-I$ under $P$ and, by \eqref{temps:local:I:p:s}, there exists a continuous process  $\ell'= (\ell'(t) : t \geq 0 )$ such that for every $M \geq 0$,  we  have 
     \begin{equation} \label{equation:convtloc}
      \lim_{k \rightarrow \infty}  \sup_{t \leq M} \Big|  \frac{1}{\epsilon_k} \int_0^t \dd s ~\mathbbm{1}_{\{  0   < H(\rho'_s) <  \epsilon_k \}} -  \ell'(t)  \Big|  = 0, \quad \mathbb{P}_{0,y} \text{ a.s.}
     \end{equation}
 Now, we define a continuous non-decreasing function $L' = (L'_t)_{t \geq 0}$ by setting
 \begin{equation*}
        L'_t = \ell'\Big( \int_0^t \dd s \, \mathbbm{1}_{\{ H(\rho_s) > \uptau(W_s)\}}   \Big), \quad t \geq 0 
    \end{equation*}
    and we next check that \eqref{equation:convL} holds with $L$ replaced by $L'$.  To simplify notation, we  write   $\widetilde{V}_t := \int_0^t \dd s \, \mathbbm{1}_{\{ H(\rho_s) > \uptau(W_s) \}}$ for $t \geq 0$ and  observe that 
     \begin{align*}
          \sup_{t \leq M} \Big|  \frac{1}{\epsilon_k} \int_0^t \dd s~ \mathbbm{1}_{\{ \uptau(W_s) < H(\rho_s) < \uptau(W_s) + \epsilon_k \}} - L'_t \Big| 
         &= 
         \sup_{t \leq M} \Big|  \frac{1}{\epsilon_k} \int_0^{\widetilde{V}_t} \dd s \,  \mathbbm{1}_{\{ \uptau(W_{\sigma_s}) < H(\rho_{\sigma_s}) < \uptau(W_{\sigma_s}) + \epsilon_k \}} -  \ell' ( \widetilde{V}_t ) \Big| \\
          &= 
          \sup_{t \leq \widetilde{V}_M} \Big|  \frac{1}{\epsilon_k} \int_0^t \dd s \,  \mathbbm{1}_{\{  0   < H(\rho'_s) <  \epsilon_k \}}   -  \ell'(t) \Big|,
     \end{align*} 
     where in the last equality we used that if $\uptau(W_{\sigma_s}) <  H(\rho_{\sigma_s})$, we have $H(\upvarrho_{{\sigma_s}}) = \uptau(W_{{\sigma_s}}) + H(\rho'_s)$. Now, the expression in  last display converges a.s. to $0$ by \eqref{equation:convtloc}, completing the proof. The statement under $\mathbb{N}_y( \, \cdot \,  | \,  \epsilon > 0)$ now easily follows by working under $\mathbb{P}_{0,y}$ and by considering the first excursion of $(\rho , W)$ away from $(0,y)$ with duration $\sigma > \epsilon$ and by applying the previous result, we skip the details. 
 \end{proof}
We conclude this section by   observing that for $(\upvarrho , \omega)\in \mathbb{D}(\mathbb{R}_+ , \mathcal{M}_f(\mathbb{R}_+) \times \mathcal{W}_{E} )$, the process   $L(\upvarrho , \omega)$ is not defined in the time-scale of $\mathrm{tr}(\upvarrho , \omega)$, but rather in the one of  $(\upvarrho , \omega)$. For this reason, we will sometimes work with the  time-changed version       $L_{\Gamma_t(\upvarrho, \omega)}(\upvarrho , \omega),$ $t \geq 0.$

\subsection{Debut points}\label{sub:sect:debut}

In this section we work with a fixed deterministic element $(\upvarrho , \omega) \in \mathcal{S}_x$. We work within $\mathcal{S}_x$ rather than in the general space  $\mathbb{D}(\mathbb{R}_+ , \mathcal{M}_f(\mathbb{R}_+) \times \mathcal{W}_{E} )$ to avoid pathological cases that can only occur on sets of null $\mathbb{N}_{x}$ and $\mathbb{P}_{0,x}$ measure. For definiteness, the notions defined below will be extended  to elements in $\mathbb{D}(\mathbb{R}_+, \mathcal{M}_f(\mathbb{R}_+) \times \mathcal{W}_E )$ which are not in $\mathcal{S}_x$ in a trivial way. Recall  that on $\mathcal{S}_x$, the process $\zeta(\omega)$ is continuous, null at $0$, and we have   $H(\upvarrho)=\zeta(\omega)$. Recall as well that we write $(\widehat{\omega}_a)_{ a \in \mathcal{T}_{H(\upvarrho)}}$ for the corresponding tree-indexed process.  
\par
Closely related versions of the  results that we now present have already been discussed  in \cite[Section 3]{ALG15} and we   rely on similar arguments. 

\begin{def1}\label{definition:debut} A point $u \in \mathcal{T}_{H(\upvarrho)}$ is called an excursion debut for $(\widehat{\omega}_a)_{ a\in \mathcal{T}_{H(\upvarrho)}}$ if the following properties hold: 
    \begin{enumerate}
        \item[\emph{(i)}] We have  $\widehat{\omega}_u = x$;
        \item[\emph{(ii)}] We can find $v \succeq u$ such that $\widehat{\omega}_a\neq x$ for every $a$ in  $\rrbracket u,v \llbracket$. 
    \end{enumerate}
    We denote the collection of excursion  debuts by $\mathcal{D}(\upvarrho, \omega)$. For every $u \in \mathcal{D}(\upvarrho, \omega)$,  we write  $\mathcal{C}_u(\upvarrho, \omega)$ for the subset of points $v \in \mathcal{T}_{H(\upvarrho)}$ fulfilling that $v \succeq u$ with  $\widehat{\omega}_a \neq x$ for every $a \in \, \rrbracket u,v \llbracket$ and we set
\begin{equation*}
    \mathcal{C}^\circ_u(\upvarrho, \omega) := \mathcal{C}_u(\upvarrho, \omega) \cap \{ a \in \mathcal{T}_{ H(\upvarrho) } : \widehat{\omega}_a \neq x \}. 
\end{equation*}
\end{def1}
 For definiteness, when $(\upvarrho , \omega)$ is not in $\mathcal{S}_x$ we set $\mathcal{D}(\upvarrho , \omega)$ as $\emptyset$ and  to avoid degenerate cases we henceforth assume that $\mathcal{D}(\upvarrho, \omega)$ is non-empty.  Let us briefly  discuss some basic properties of the components $\mathcal{C}_u(\upvarrho, \omega)$, for $u\in\mathcal{D}(\upvarrho, \omega)$. First, note that by definition, if  $u$ is a debut point, then it belongs to $\mathcal{C}_u(\upvarrho, \omega)$. Moreover, since for every point $a \in \mathcal{C}_u(\upvarrho, \omega)$ we have that $\llbracket u,a \rrbracket \subset \mathcal{C}_u(\upvarrho,\omega)$, the subset $\mathcal{C}_u(\upvarrho, \omega)$ is path-wise connected. Hence,  $\mathcal{C}_u(\upvarrho, \omega)$  is a tree on its own right that we root at $u$, and it is not difficult to check from our definitions that it is a closed subset of $\mathcal{T}_{H(\upvarrho)}$.  Using that the mapping $a \mapsto \widehat{\omega}_a$ is continuous, it is easy to check   that $\mathcal{D}(\upvarrho, \omega)$ is countable, and that for every $u \in \mathcal{D}(\upvarrho, \omega)$ the sets $\{ a \in \mathcal{T}_{H(\upvarrho)} : \widehat{\omega}_a \neq x \}$ and  $\mathcal{C}^\circ_u(\upvarrho, \omega)$  are open. Remark however that in general,  $\text{Int}(\mathcal{C}_u(\upvarrho, \omega))$ and  $\mathcal{C}^\circ_u(\upvarrho, \omega)$ may differ. Indeed,    the existence of an isolated point $b$ of the set $\mathcal{C}_u(\upvarrho, \omega) \cap \{ a \in \mathcal{T}_{H(\upvarrho)}   : \widehat{\omega}_a= x  \}$ satisfying $b \in \text{Mult}_1( \mathcal{T}_{H(\upvarrho)} )$ would yield  both that    $b \in \text{Int}(\mathcal{C}_{u}( \upvarrho, \omega))$ and $\widehat{\omega}_b= x$. Note that    $\mathcal{C}^\circ_u(\upvarrho, \omega)$ is not necessarily path connected (consider a component $\mathcal{C}_u(\upvarrho, \omega)$ for an element $u \in \mathcal{D}(\upvarrho, \omega)$ which is a branching point). However,  as the next lemma shows, this property does hold as soon as this scenario has been ruled out. 
\begin{lem} \label{lemma:connectedcomponents} 
Suppose that  $(\widehat{\omega}_a)_{ a\in \mathcal{T}_{H(\upvarrho)}}$ does not take the value $x$ at branching points. Then, the family $(\mathcal{C}^\circ_u(\upvarrho, \omega))_{u \in \mathcal{D}(\upvarrho, \omega)}$ are the connected components of the open set 
     $\{ a \in \mathcal{T}_{H(\upvarrho)} : \widehat{\omega}_a \neq x \}$. 
\end{lem}
\begin{proof}
Let us first prove that 
$$
\cup_{u \in \mathcal{D}(\upvarrho, \omega)} \mathcal{C}^\circ_u(\upvarrho, \omega) =\{ a \in \mathcal{T}_{H(\upvarrho)} : \widehat{\omega}_a \neq x \}.
$$
The inclusion $\cup_{u \in \mathcal{D}(\upvarrho, \omega)} \mathcal{C}^\circ_u(\upvarrho, \omega) \subset\{ a \in \mathcal{T}_{H(\upvarrho)} : \widehat{\omega}_a \neq x \}$ follows  by definition, and to obtain the reverse one let us fix an arbitrary  $v \in \mathcal{T}_{H(\upvarrho)}$ such that $\widehat{\omega}_v \neq x$.  Since the set $\llbracket 0 , v \rrbracket \cap \{ a \in \mathcal{T}_{H(\upvarrho)} : \widehat{\omega}_a \neq x \}$ is non-empty,  the continuity of $(\widehat{\omega}_a)_{ a\in \mathcal{T}_{H(\upvarrho)}}$, combined with \eqref{equation:snakeproperty2} and  the fact that $\widehat{\omega}_0  = x$,  ensure that there exists a unique $u \in \llbracket 0 , v\rrbracket$ with  $\widehat{\omega}_u = x$ satisfying that   $\widehat{\omega}_a \neq x$ for every  $a \in \rrbracket u,v \rrbracket$. The continuity of $(\widehat{\omega}_a)_{ a\in \mathcal{T}_{H(\upvarrho)}}$ in $\llbracket 0, v\rrbracket$ and  the fact that $\widehat{\omega}_0  = x$  then ensure that there exists a unique $u \in \llbracket 0 , v\rrbracket$ with  $\widehat{\omega}_u = x$ satisfying that   $\widehat{\omega}_a \neq x$ for every  $a \in \rrbracket u,v \rrbracket$. It now follows from our definitions that $u$ is an excursion debut and that $v \in \mathcal{C}^\circ_u(\upvarrho, \omega)$, which proves the reverse inclusion.   It remains to prove that the sets $(\mathcal{C}^\circ_u(\upvarrho, \omega)$,  $u\in \mathcal{D}(\upvarrho, \omega))$, are disjoint and that each $\mathcal{C}^\circ_u(\upvarrho, \omega)$ is   connected.  For the latter, let us argue  that for any $v_1,v_2\in \mathcal{C}^\circ_u(\upvarrho, \omega)$,  we have 
$$
\llbracket v_1,v_2\rrbracket =  \llbracket v_1 \curlywedge v_2 ,v_1\rrbracket\cup \llbracket v_1 \curlywedge v_2 ,v_2\rrbracket\subset  \mathcal{C}^\circ_u(\upvarrho, \omega), 
$$
where we recall that the notation $v_1 \curlywedge v_2$ stands for the common ancestor.  The fact that both $v_1$ and $v_2$ are descendents of $u$  guarantees that $v_1 \curlywedge v_2 \in  \llbracket u , v_1 \rrbracket \cap \llbracket u , v_2 \rrbracket$, and the equality $u = v_1 \curlywedge v_2$ is ruled out since $u$ can not be a branching point by assumption. This ensures that $v_1 \curlywedge v_2$ is a strict descendent of $u$ and since there is no point with label $x$ in $\rrbracket u , u_1 \curlywedge u_2 \rrbracket$, we get that $u_1 \curlywedge u_2$ belongs to $\mathcal{C}_u^\circ(\upvarrho, \omega)$. The contention in the previous display now follows by definition of  $\mathcal{C}^\circ_u(\upvarrho, \omega)$.   Finally, let us check that if $u, u'$ are distinct debut points, then $\mathcal{C}_u^\circ(\upvarrho, \omega)$ and $\mathcal{C}^\circ_{u'}(\upvarrho, \omega)$ are disjoint. Arguing by contradiction,  if we had some $v \in \mathcal{C}^\circ_u(\upvarrho, \omega) \cap \mathcal{C}^\circ_{u'}(\upvarrho, \omega)$, then it must hold that $u \preceq u' \preceq v$ or $u' \preceq u \preceq v$. In any case, we get respectively that $u' \in \, \rrbracket u,v \llbracket $ with $\widehat{\omega}_{u^\prime}  = x$ and $u \in \,  \rrbracket u',v \llbracket $ with $\widehat{\omega}_u  = x$, in contradiction with the fact that  $v$ belongs to  $\mathcal{C}^\circ_u(\upvarrho, \omega)\cap \mathcal{C}^\circ_{u'}(\upvarrho, \omega)$. 
\end{proof}

For every $u\in\mathcal{D}(\upvarrho, \omega)$, we refer to $(\widehat{\omega}_a)_{a \in \mathcal{C}_u(\upvarrho, \omega) }$ as the excursion away from $x$ of $(\widehat{\omega}_a)_{a \in \mathcal{T}_{H(\upvarrho)}}$ associated with the component  $\mathcal{C}_u(\upvarrho, \omega)$. Each excursion away from $x$  is a tree indexed process on its own right,  and we shall now argue that it can be coded by a snake path.  In order to be more precise, we   recall the notation   $\mathfrak{g}(u)$ and $\mathfrak{d}(u)$ from \eqref{definition:gauchedroite} for respectively  the first time and the last time at which the exploration $p_{H(\upvarrho)}$ visits the component $\mathcal{C}_u(\upvarrho, \omega)$. Then, we set:
\begin{equation}\label{eq:rho:w:u:D:stemming}
    (\upvarrho^u , \omega^u  ) 
    :=    
       \mathscr{T}_{\mathfrak{g}(u),\mathfrak{d}(u)}(\upvarrho,\omega)  
\end{equation}
for the  subtrajectory 
associated with the interval $[\mathfrak{g}(u) , \mathfrak{d}(u)]$, where we recall the notation $ \mathscr{T}_{\mathfrak{g}(u),\mathfrak{d}(u)}(\upvarrho,\omega)$ from \eqref{eq:subtrajectory:a:b}. Informally, the snake path    $(\upvarrho^u , \omega^u  )$ encodes the restriction of $(\widehat{\omega}_a)_{ a\in \mathcal{T}_{H(\upvarrho)}}$ to  $\{v\in \mathcal{T}_{H(\upvarrho)}:~u\preceq v\}$,  the subtree stemming from $u$. For this reason, we refer to $(\upvarrho^u , \omega^u  )$  as the subtrajectory stemming from $u$. The  snake path coding the  excursion $(\widehat{\omega}_a)_{a \in \mathcal{C}_u(\upvarrho, \omega)}$ is then obtained by  pruning the paths of $(\upvarrho^u , \omega^u  )$ at its first return time to $x$, namely 
\begin{equation}\label{eq:rho:w:u:D}
    (\upvarrho^{u,*} , \omega^{u,*}  ) 
    :=    
     \mathrm{tr}   (\upvarrho^u,\omega^u)  .
\end{equation}
The process $(\upvarrho^{u,*} , \omega^{u,*}  )$   is an element of $\mathcal{S}_x$ with (strictly) positive  duration $\sigma(\omega^{u,*})$ which  encodes  $(\widehat{\omega}_a)_{a \in \mathcal{C}_u(\upvarrho, \omega)}$, in the sense that by construction the tree coded by $\zeta( \omega^{u,*}) = H(\upvarrho^{u,*})$ is   isometric to  the subset $\mathcal{C}_u(\upvarrho, \omega)$ by an isometry  preserving the root and the labels. 
 \begin{def1}\label{definition:excursion} 
The family $( ( \upvarrho^{u,*}, \omega^{u,*}) : u \in \mathcal{D}(\upvarrho, \omega) )$ is referred to as the family of excursions away from $x$ of the snake path $(\upvarrho , {\omega})$.
\end{def1}
 Recalling the functional $L$ introduced in \eqref{definition:exitMultiusos},  we interpret $L_\infty(\upvarrho^u,\omega^u)$ as the boundary size of $\mathcal{C}_u$.    We conclude this section with a discussion under the measures $\mathbb{N}_{x,0}$ and $\mathbb{P}_{0,x,0}$. We start by introducing  some  notations that will be used throughout this work. Under $\mathbb{N}_{x,0}$  and $\mathbb{P}_{0,x,0}$, the process $(\rho,\overline{W}):=(\rho, W,\Lambda)$ belongs a.e. to  $\overline{\mathcal{S}}_x$.   Therefore, $(\rho,W)$ is in $\mathcal{S}_x$ a.e. and we can then consider, outside of a 
 negligible set,  the sets $\mathcal{D}$ and $\mathcal{C}_u, \mathcal{C}_u^\circ$ for $u\in \mathcal{D}$ taken at  $(\rho, W)$. As usual, we omitted the dependence on  $(\rho, W)$ to simplify notation.  We write 
$$
(\rho^u,W^u)_{u\in \mathcal{D}}\quad \text{ and }\quad (\rho^{u,*},W^{u,*})_{u\in \mathcal{D}} 
$$
for the corresponding family of   subtrajectories stemming from $u \in \mathcal{D}$, as well as the family of excursions. Note that for $u \in \mathcal{D}$ we have  $p_{H(\rho)}^{-1}(\{u \}) = \{ \mathfrak{g}(u), \mathfrak{d}(u) \}$ since $u$ is not a branching point by \eqref{prelim:Branchingx}; in particular,    the statement of Lemma \ref{lemma:connectedcomponents} holds under $\mathbb{N}_{x}$  and $\mathbb{P}_{0,x}$.   The next lemma shows in particular that  under $\mathbb{N}_{x,0}$ and $\mathbb{P}_{0,x,0}$,  the local time process $\Lambda$  can be used to index 
the family of excursions away from $x$. 

\begin{lem} \label{lemma:contanteExcursion}
Under $\mathbb{N}_{x,0}$  and $\mathbb{P}_{0,x,0}$,   for every $u \in \mathcal{D}$ the process $(\widehat{\Lambda}_a)_{a \in \mathcal{T}_{H(\rho)}}$ is constant on $\mathcal{C}_u$ and its value  is given by $\widehat{\Lambda}_u$. Moreover, if $u'$ is another arbitrary element of $\mathcal{D}$ with $u \neq u'$, we have  $\widehat{\Lambda}_u \neq \widehat{\Lambda}_{u^\prime}$.  
\end{lem}

\begin{proof} 
 We only prove the result under $\mathbb{N}_{x,0}$, the arguments under $\mathbb{P}_{0,x,0}$ are identical. To simplify notation we write  $H$ for the height process instead of $H(\rho)$. We start proving the first claim and in this direction,  recall from  Lemma~\ref{lemma:tequedasenThetax} that $\mathbb{N}_{x,0}$-a.e., the Lévy snake $(\rho , \overline{W})$ takes values in $\overline{\Theta}_x$. Therefore, we can consider a measurable subset $\Omega_0$ of full $\mathbb{N}_{x,0}$ measure at which the process  $(\rho_t , \overline{W}_t )_{t\geq 0}$  stays in  $\overline{\Theta}_x$,    and for the rest of the argument   we work on  $\Omega_0$.   Fix  $u \in \mathcal{D}$,  an arbitrary element  $a \in \mathcal{C}_u \setminus \{ u \}$,  and let us prove that $\widehat{\Lambda}_{u} = \widehat{\Lambda}_{a}$. Recall that by definition, we have   $\widehat{W}_v \neq x$ for every $v \in \rrbracket u , a \llbracket$. The condition $u \preceq a$ ensures that $ H_{\mathfrak{g}(u)} = \min_{[\mathfrak{g}(u), \mathfrak{g}(a)]}H \leq H_{\mathfrak{g}(a)}$, and  by   \eqref{equation:snakeproperty2} we get that 
$W_{\mathfrak{g}(a)}(h) \neq x$ for every $h \in (H_{\mathfrak{g}(u)}, H_{\mathfrak{g}(a)})$. Now, from the fact that   $(\rho_{\mathfrak{g}(a)}, \overline{W}_{\mathfrak{g}(a)}) \in \overline{\Theta}_x$   and the snake property we infer  $\widehat{\Lambda}_{\mathfrak{g}(a)} =  \Lambda_{\mathfrak{g}(a)}(H_{\mathfrak{g}(u)} ) = \widehat{\Lambda}_{\mathfrak{g}(u)}$. This entails that $\widehat{\Lambda}_u = \widehat{\Lambda}_a$,  and since $a$ is arbitrary, we conclude  that  $\widehat{\Lambda}_a$, $a\in \mathcal{C}_u$, is identically equal to $\widehat{\Lambda}_u$.
 \par 
 The proof of the second claim of the lemma  relies on the following elementary fact. Under $\Pi_{y,r}$ for $(y,r) \in E \times \mathbb{R}_+$  consider an independent copy $\xi^\prime$ of $\xi$ and write $\tau^\prime = \inf\{ t > 0 : \xi^\prime_t = x \}$. Then,  the excursion process of $(\xi_{\uptau + t})_{t \geq 0}$ and $(\xi^\prime_{\uptau^\prime + t})_{t \geq 0}$  are  independent Poisson point measures with intensity $\mathbbm{1}_{\{\ell>r\}}\mathrm{d}\ell~\mathcal{N}$ and therefore, the first entries in the atoms of these two measures are a.s. all pairwise distinct. Let us  infer from this the second statement of the lemma. To this end, fix two arbitrary rational numbers $0 < t<t^{\prime}$, write  $t \curlywedge t'$ for an arbitrary element of $p_H^{-1}( p_H(t)  \curlywedge p_H(t'))$ and denote the connected components  of $\{h \in[H_{t\curlywedge t^{\prime} },H_t]:~W_{t}(h)\neq x\}$ (resp. $\{h\in[H_{t\curlywedge t^{\prime} },H_{t^\prime}]:~W_{t^\prime}(h)\neq x\}$) that do not contain $H_{t\curlywedge t^{\prime}}$ by $(a_i,b_i)_{i\in I}$  (resp. $(a_i^\prime,b_i^\prime)_{i\in I^\prime}$). Note that both of these collections are  empty only when $p_H(t)$ and $p_H(t')$ belong to the same excursion component.  Conditionally on $H$ and $\overline{W}_t(H_{t\curlywedge t^{\prime} })$, the processes $(\overline{W}_t(H_{t\curlywedge t^{\prime} }+h))_{h\leq H_t-H_{t\curlywedge t^{\prime} }}$ and $(\overline{W}_{t^\prime}(H_{t\curlywedge t^{\prime} }+h))_{h\leq H_{t^\prime}-H_{t\curlywedge t^{\prime} }}$ are independent and  distributed as  $(\xi,\mathcal{L})$ under $\Pi_{\overline{W}_{t}(H_{t\curlywedge t^{\prime} })}$ stopped respectively at time $H_t-H_{t\curlywedge t^{\prime} }$ and $H_{t^\prime}-H_{t\curlywedge t^{\prime} }$. From the elementary observation given above we get that the quantities $(\Lambda_t(a_i))_{i\in I}$ and $(\Lambda_{t^\prime}(a_i^\prime))_{i\in I^\prime}$ are all distinct when $I,I' \neq \emptyset$. Since by Lemma \ref{lemma:connectedcomponents}   the sets $\mathcal{C}_u^\circ$ for $u \in \mathcal{D}$ are open, it is plain  by \eqref{equation:snakeproperty2} and the definition of debut points that this ensures   $\widehat{\Lambda}_u$, $\widehat{\Lambda}_{u'}$ a.s. differ  if $u, u'$ are distinct debut points. 
\end{proof}

For every $u \in \mathcal{D}$, the local time process $\widehat{\Lambda}$ is constant in the connected components $\mathcal{C}_u$,  but  this is no longer the case if we consider instead its restriction to the corresponding stemming subtree  
$\{v\in \mathcal{T}_{H(\rho)}:~u\preceq v\}$ containing $\mathcal{C}_u$. It will be in fact  crucial for our purposes to enrich the family $(\rho^u,W^u)_{u\in \mathcal{D}}$  by including the corresponding restriction of the local time $\Lambda$. 
\begin{def1} Under $\mathbb{N}_{x,0}$ and $\mathbb{P}_{0,x,0}$,  for every $u\in \mathcal{D}$, we shall write  $(\rho^u,\overline{W}^u)$ for the subtrajectory of $(\rho,\overline{W})$ associated with the interval $[\mathfrak{g}(u), \mathfrak{d}(u)]$. 
\end{def1}
\noindent Note that under $\mathbb{N}_{x,0}$ and $\mathbb{P}_{0,x,0}$, the lifetime of $W^u$ is  finite and given by $\mathfrak{d}(u) - \mathfrak{g}(u)$.  

\section{\texorpdfstring{The measures $\mathbf{N}_{x,r}$ and  $\mathbf{N}_x^*$}{}}\label{section:excursionmeasure}

Now that we have defined the notion of an excursion away from $x$ and the corresponding  subtrajectory stemming from $x$, the next step is to construct candidates for the corresponding (infinite) measures. This task requires some preliminary constructions and notations that we will now introduce. 
\par
We consider the family of measures  $\bar{R}_{a,b}( \overline{\w} , \dd \overline{\w}')$ for  $\overline{\w} \in \mathcal{W}_{\overline{E}}$,  $0 \leq a \leq \zeta_\w$ and $b \geq a$ introduced  in  Section \ref{secsnake}  associated with  $\overline{\Pi} :=  ({\Pi}_{y,r})_{(y,r) \in \overline{E}}$. In this  section  we work with a fixed $r\geq 0$ and  we recall the notation $\mathcal{N}_{x,r}$ for the measure introduced in Section \ref{section:framework}.   For every continuous function $h:\mathbb{R}_+\to \mathbb{R}_+$ and $t \geq 0$,   we write $\nu^h_t (\dd \overline{\w} )$ for the law of $( ({\xi}_s, \mathcal{L}_s) : 0 \leq s \leq h(t))$ under  $\mathcal{N}_{x,r}$. Note that $\nu_t^h$ is a measure on $\mathcal{W}_{\overline{E}}$;  the dependence on $r\geq 0$ was only omitted to simplify notation. The interest in the family $(\nu_t^h : t \geq  0)$ arises from the following result.
\begin{prop}
Fix a  continuous function $h:\mathbb{R}_+\to \mathbb{R}_+$ with $\sigma_h \in(0, \infty)$  and  $h(0) = h(\sigma_h) = 0$. There exists  a unique  measure  $\mathbf{Q}^h_{x,r}(\dd \overline{W})$ on $\mathcal{W}_{\overline{E}}^{\mathbb{R}_+}$ such that, for every  $n \geq 1$, $A_0, A_1,  \dots, A_n$ measurable subsets of $\mathcal{W}_{\overline{E}}$ and  $0 = t_0 < t_1 < t_2 \dots < t_n$, we have 
\begin{align}\label{equation:consistency}
      \mathbf{Q}^h_{x,r} \Big( &  \overline{W}_{t_0}\in A_0, \, \overline{W}_{t_1} \in A_1,  \dots, \overline{W}_{t_n}\in A_n\Big) \nonumber
   \\&= \mathbbm{1}_{A_0}(x,r)  \int_{A_1} \nu_{t_1}^h(\dd \overline{\mathrm{w}}_1) \int_{A_2\times \dots\times A_n} \bar{R}_{ m_{h}(t_1,t_2),h(t_2) }(\overline{\mathrm{w}}_1,\dd \overline{\mathrm{w}}_2) \dots \bar{R}_{ m_{h}(t_{n-1},t_n),h(t_n) }(\overline{\mathrm{w}}_{n-1},\dd \overline{\mathrm{w}}_n) .
\end{align}
\end{prop}

 \begin{proof}
 First,  remark that for every  $0 < s < t$ and every bounded measurable function $f:\mathcal{W}_{\overline{E}}\to \mathbb{R}_+$, we have
 \begin{equation*}
 \int_{\mathcal{W}_{\overline{E}}} \nu_s^h (\dd \overline{\w}) \int_{\mathcal{W}_{\overline{E}}} \bar{R}_{m_h(s,t) , h(t) }(\overline{\w} , \dd \overline{\w}') f(\overline{\w}') 
        = \mathcal{N}_{x,r}\big( f(\overline{\xi}_u : 0 \leq u \leq h(t)) \big) = \nu^h_t( f ), 
\end{equation*}
 where  in the first equality we applied the Markov property \eqref{equation:markovExcursionquevive} at time $m_h(s,t)$. In other terms, it holds that:
$$ 
\nu^h_t\big(\mathrm{d} \overline{\mathrm{w}}^\prime\big)=\int_{\mathcal{W}_{\overline{E}}}  \nu^h_s\big(\mathrm{d}\overline{\mathrm{w}}\big) \bar{R}_{m_h(s,t),h(t)}( \overline{\w} ,\mathrm{d}\overline{\w}^\prime).
$$
If we write $\mathfrak{Q}_{t_0, t_1, \dots t_n}$ for the measure in $\mathcal{W}_{\overline{E}}^{n+1}$ defined by the right-hand side of  \eqref{equation:consistency},  our previous observation entails that the family of measures
$\mathfrak{Q}_{t_0, t_1,\dots, t_n }$, for $n\geq 1$ and  $0=t_0<t_1<\dots<t_n$ satisfies    Kolmogorov's  consistency criterion. The statement of the proposition will now follow  by a straightforward argument which involves Kolmogorov’s extension theorem. Note that since the measures   $\mathfrak{Q}_{t_0, t_1, \dots t_n}$ are infinite, we must proceed with some care. In this direction, fix $s>0$ such that $h(s)>0$ as well as an arbitrary   $\delta>0$. For any collection of  times  $0=t_0< t_1< \dots <t_n$,  we let  $\mathfrak{Q}_{t_0, t_1,\dots, t_n }^\delta$ be the measure in $\mathcal{W}_{\overline{E}}^{n+1}$ defined at each measurable set $A_0 \times A_1 \times \dots \times A_n$ as: 
\begin{equation*}
    \mathfrak{Q}_{t_0, t_1,\dots, t_j, \dots ,  t_n } \big(A_0 \times A_1, \times \dots \times \big(A_j \cap \{\uptau(\w) > \delta \} \big)\times \dots A_n \big) 
\end{equation*}
if $s = t_j$ for some $1 \leq j \leq n$,   as 
\begin{equation*}
    \mathfrak{Q}_{t_0, t_1,\dots, t_j, s, t_{j+1}, \dots ,  t_n } \big(A_0 \times A_1, \times \dots A_j \times  \{\uptau(\w) > \delta \}\times A_{j+1} \times \dots \times A_n \big) 
\end{equation*}
if $t_j < s < t_{j+1}$ for some $0 \leq j \leq n-1$, and  as 
\begin{equation*}
    \mathfrak{Q}_{t_0, t_1,\dots, t_n, s} \big(A_0 \times A_1, \times \dots A_n \times  \{\uptau(\w) > \delta \} \big)   
\end{equation*}
if  $t_n < s$.  The family $\mathfrak{Q}_{t_0, t_1,\dots, t_n }^\delta$ for $n\geq 1$ and  $0=t_0<t_1<\dots<t_n$ 
 still  satisfies    Kolmogorov's  consistency criterion,  and each one of these measures  has finite total mass given by $\nu_s^h(\uptau(\w)>\delta)\leq \mathcal{N}_{x,r}(\uptau(\xi)>\delta)<\infty$. Therefore, Kolmogorov's  extension theorem ensures the existence of a unique measure $\mathbf{Q}^{h,\delta}_{x,r}$ in the product space $\mathcal{W}_{\overline{E}}^{\mathbb{R}_+}$  with total mass $\nu_s^h(\uptau(\w)>\delta)$ and having for finite-dimensional marginal distributions  the measures $\mathfrak{Q}_{t_0, t_1,\dots, t_n }^\delta$, for $n\geq 1$ and  $0=t_0<t_1<\dots<t_n$. By construction, for every $0 < \delta_1 < \delta_2$, the measure $\mathbf{Q}_{x,r}^{h,\delta_2}$ is the restriction of $\mathbf{Q}_{x,r}^{h,\delta_1}$ to $\{\uptau(W_s)>\delta_2\}$. Hence,   we can consider the measure $\mathbf{Q}^{h}_{x,r}$  in $\mathcal{W}^{\mathbb{R}_+}_{\overline{E}}$  defined at every measurable subset $\mathscr{C}$ by 
\begin{equation*}
    \mathbf{Q}^{h}_{x,r}(\mathscr{C}):= \lim_{\delta \rightarrow 0}\mathbf{Q}^{h,\delta}_{x,r}(\mathscr{C}). 
\end{equation*}
Using that $\nu_{s}^{h}(\uptau(\w)=0)=0$, we infer by monotone convergence that $\mathbf{Q}^{h}_{x,r}$ satisfies \eqref{equation:consistency}. Finally, the uniqueness comes from the fact that if $\mathbf{Q}^\prime$ also verifies  \eqref{equation:consistency},   by  uniqueness of  $\mathbf{Q}^{h,\delta}_{x,r}$,  for each $\delta>0$ we get that the restriction of $\mathbf{Q}^\prime$ to $\{\uptau(W_s)>\delta\}$ must be $\mathbf{Q}^{h,\delta}_{x,r}$.  Since  $\{\uptau(W_s)=0\}$ is still a $\mathbf{Q}^\prime$-null set, it now follows by monotone convergence\footnote{The uniqueness follows also easily from the monotone class theorem.}  that $\mathbf{Q}^{h}_{x,r} = \mathbf{Q}^\prime$.
\end{proof} 
By construction, for any fixed time $t \geq 0$,   the variable $\overline{W}_t$ under ${\textbf{Q}}^h_{x,r}$ is distributed as  $\overline{\xi}$ under $\mathcal{N}_{x,r}$ restricted to $[0,h(t)]$.   For this reason, the canonical process $\overline{W}$ under $\textbf{Q}_{x,r}^h$ can be interpreted as the snake driven by $h$ with spatial motion distributed according to  $\mathcal{N}_{x,r}$. We stress that, in contrast to the framework presented in Section \ref{secsnake}, the measure $\mathcal{N}_{x,r}$ is an infinite measure, and as a result, $\mathbf{Q}^h_{x,r}$ is as well an infinite measure.

\begin{lem}\label{prop:Q:mathcal:N} Fix a   continuous function $h:\mathbb{R}_+\to \mathbb{R}_+$ satisfying  $\sigma_h \in(0, \infty)$, with $h(0) = h(\sigma_h) = 0$ and set  $h_{\trev}:= (h((\sigma_h - t)\vee 0) : t\geq 0)$. The following properties hold: 
\begin{enumerate}
    \item[\emph{(i)}] The distribution of $(\overline{W}_{(\sigma_h - t)\vee 0} :  t \geq 0 )$ under $\mathbf{Q}^h_{x,r}$ coincides with the law of  
    $(\overline{W}_{t} : t\geq 0)$ under $\mathbf{Q}^{h_{\trev}}_{x,r}$.  
    \item[\emph{(ii)}] Let $q>0$ be as in \ref{continuity_snake_2}.  If   $h$ is $\theta$-Hölder-continuous  with $\theta q > 1$,   the canonical process $\overline{W}$ under $\mathbf{Q}_{x,r}^h$ possesses a continuous modification. 
\end{enumerate}     
\end{lem}

\begin{proof} 
    For $0 \leq s<t \leq \sigma_h$, we shall write $P^h_{s,t}$ and $P_{s, t }^{h_{\trev }}$ for the transition semigroup from time $s$ to time $t$ of the snake with spatial motion $\overline{\Pi}$ with driving function $h$ and $h_{\trev}$ respectively. 
    Note that its explicit expressions can be easily derived from \eqref{equation:fddsSnake}.  Starting with (i),  observe that the result will follow as soon as we establish that
    \begin{equation} \label{equation:predualidad}
        \nu_{\sigma_h  - t}^h f_2P_{\sigma_h - t , \sigma_h - s }^h f_1 
        =
        \nu_{s}^{h_{ \trev }} f_1 P_{s, t }^{h_{\trev }} f_2
    \end{equation}
    for every $0 < s  < t < \sigma_h$ and non-negative measurable functions $f_1$ and  $f_2$ on $\mathcal{W}_{\overline{E}}$. 
    Indeed, if  the previous identity  holds,  by inductively applying \eqref{equation:predualidad} and noting that $\nu_{t}^{h_{\trev}} = \nu_{\sigma_h - t}^{h}$,  we infer  that   for every $0 < t_1 < \dots < t_k< \sigma_h$ and  non-negative  measurable functions $f_1, \dots,  f_k$  on $\mathcal{W}_{\overline{E}}$, 
    \begin{align*}
        \nu^h_{\sigma_h  - t_{k}} f_k P_{\sigma_h - t_{k} , \sigma_h - t_{k-1} }^h f_{k-1} \dots  P_{\sigma_h - t_2 , \sigma_h - t_1 }^h f_1 
        =
        \nu^{h_{\sigma_h \shortminus \hspace{0,2mm}  \bigcdot} }_{t_{1}} f_1 P_{t_{1} ,   t_{2} }^{h_{\trev}} f_2 \dots P_{t_{k-1} ,  t_k}^{h_{ \trev}} f_k. 
    \end{align*}
     Since for $t\geq \sigma_h$ we plainly have  $W_{0}=W_{t}=(x,r)$, the  equality in the last display yields that  $(\overline{W}_{(\sigma_h - t)\vee 0} : t\geq 0)$ under $\textbf{Q}^h_{x,r}$  and   
    $(\overline{W}_{t} : t\geq 0)$ under $\textbf{Q}^{h_{\trev}}_{x,r}$  have the same finite-dimensional distributions, proving  (i). Let us then establish \eqref{equation:predualidad}. In this direction, note that  the left-hand side of \eqref{equation:predualidad} is given by 
    \begin{align*}
        \nu_{\sigma_h  - t}^h f_2P_{\sigma_h - t , \sigma_h - s }^h f_1  
        =
        \int_{\mathcal{W}_{\overline{E}}} \nu_{\sigma_h - t}^h (\dd \overline{\w} ) f_2( \overline{\w} ) \int_{\mathcal{W}_{\overline{E}}} \bar{R}_{m_{h}( \sigma_h - t, \sigma_h - s) , h(\sigma_h - s) }( \overline{\w} , \dd \overline{\w}') f_1( \overline{\w}'). 
    \end{align*}
    Said otherwise, the law of  the pair $(\overline{W}_{\sigma_h - t}, \overline{W}_{\sigma_h - s})$  under $\textbf{Q}^{h}_{x,r}$ is characterised by the following two properties: 
\begin{itemize}
    \item Up to time $m_h(\sigma_h - t , \sigma_h - s)$, the paths $\overline{W}_{\sigma_h - t}$ and $\overline{W}_{\sigma_h - s}$ coincide and are distributed as the canonical process $(\xi , \mathcal{L})$  under $\mathcal{N}_{x,r}$ restricted to the time-interval    $[0,m_h(\sigma_h - t , \sigma_h - s)]$. 
    \item Conditionally on $\big( \overline{W}_{\sigma_h-t}(u) : u \in[0,  m_h(\sigma_h - t , \sigma_h - s)] \big)$,  the restrictions     
    \begin{equation*}
        \overline{W}_{\sigma_h - t} \big(m_h(\sigma_h - t , \sigma_h - s) + u \big), \quad  u \in \big[0, h(\sigma_h - t) - m_h(\sigma_h - t , \sigma_h - s) \big],  
    \end{equation*}
    and 
    \begin{equation*}
          \overline{W}_{\sigma_h - s} \big(m_h(\sigma_h - t , \sigma_h - s) + u\big), \quad  u \in \big[0, h(\sigma_h - s) - m_h(\sigma_h - t , \sigma_h - s) \big],
    \end{equation*}
    are independent with respective distributions $(\xi, \mathcal{L})$ under  $\overline{\Pi}_y$ taken at $y = {\overline{W}_{\sigma_h -t}( m_h(\sigma_h - t , \sigma_h -s) )}$, and  stopped at time $h(\sigma_h  - t) - m_h(\sigma_h - t , \sigma_h - s)$ and $h( \sigma_h - s) - m_h(\sigma_h - t , \sigma_h - s)$ respectively. 
\end{itemize}
    The first identity in distribution in the last point is a consequence of the Markov property \eqref{equation:markovExcursionquevive} under $\mathcal{N}_{x,r}$ while the second follows by definition of the measures $(\bar{R}_{a,b}( \overline{\w} , \dd \overline{\w}') : \overline{\w} \in \mathcal{W}_{\overline{E}}, 0 \leq a \leq \zeta_{\overline{\w}} \text{ and } b \geq a )$. A similar inspection to the right-hand side of \eqref{equation:predualidad} gives that this is precisely the law of $(\overline{W}_{s} , \overline{W}_{t} )$ under $\textbf{Q}_{x,r}^{h_{\trev }}$, and  concludes the proof of (i). 
    \par 
    Let us now turn our attention to (ii). Recall that we are working under  \ref{continuity_snake_2} and  that  $(h, \overline{\Pi})$ satisfies the regularity  conditions (i) and (ii) of Section \ref{secsnake}. In particular,  the shifted function $h' := (h_{\sigma_h/3 \,  + t} : t \geq 0)$  is still $\theta$-Hölder continuous and  the pair $(h', \overline{\Pi})$ still satisfies conditions (i) and (ii) of Section \ref{secsnake}. For $\overline{\w}\in \mathcal{W}_{\overline{E}}$, let us write $Q_{\overline{\w}}^{h'}$ for the law of the snake driven by $h'$ with spatial motion $\overline{\Pi}$ and recall that  under our assumptions, the process $\overline{W}$ under $Q_{\overline{\w}}^{h'}$ possesses a continuous modification. Noting from \eqref{equation:consistency}  that conditionally on $\overline{W}_{\sigma_h/3}$, the distribution of $(\overline{W}_{\sigma_h/3 + t}: t \geq 0)$ is ${Q}^{h'}_{\overline{W}_{\sigma_h/3}}$, we infer that under    $\textbf{Q}_{x,r}^h$ the process $\overline{W}$ possesses a continuous modification on $( \sigma _h/3,\infty)$. To extend the modification to the non-negative real line we make use of (i). More precisely,  by (i) we have the equality in distribution 
    \begin{equation*}
        \Big((\overline{W}_s : \sigma_h/3 \leq s \leq  \sigma_h) \text{ under }  \textbf{Q}^{h_{\trev}}_{x,r} \Big)
        \overset{(d)}{=} 
        \Big((\overline{W}_{\sigma_h - s} : \sigma_h /3\leq s \leq  \sigma_h) \text{ under }  \textbf{Q}_{x,r}^{h}\Big),
    \end{equation*}
and  applying the previous argument to the process in the left hand side yields    that $\overline{W}$ under $\textbf{Q}_{x,r}^h$ possesses as well  a continuous modification in $[0, 2\sigma_h/3 )$, and therefore in  $\mathbb{R}_+$.
\end{proof}
Now, we randomise the driving  function $h$ by setting: 
\begin{equation*} 
    \mathbf{N}_{x,r}(\dd \rho ,  \dd \overline{W}) := N(\dd \rho) \textbf{Q}^{H(\rho)}_{x,r}(\dd \overline{W}).  
\end{equation*}
Since under our standing assumption \ref{continuity_snake_2}   the process 
  $H(\rho)$ under $N(\dd \rho)$ is a.e. $\theta$-Hölder continuous for some $\theta$ satisfying $\theta q > 1$, by Lemma \ref{prop:Q:mathcal:N}-(ii) the process $\overline{W}$ under $\mathbf{N}_{x,r}$  possesses a continuous modification. Therefore, arguing as in Section \ref{section:snake} we infer that the  measure $\mathbf{N}_{x,r}$ is well defined in the canonical space $\mathbb{D}(\mathbb{R}_+ , \mathcal{W}_f(\mathbb{R}_+) \times \mathcal{W}_{\overline{E}} )$ and supported on $\overline{\mathcal{S}}_{x,r}$. From now on, it will be implicitly assumed that $\mathbf{N}_{x,r}$ is a measure on $\mathbb{D}(\mathbb{R}_+ , \mathcal{W}_f(\mathbb{R}_+) \times \mathcal{W}_{\overline{E}} )$.   The canonical process $(\rho , \overline{W})$ under $\mathbf{N}_{x,r}$  can be interpreted as the $\psi$-Lévy snake with spatial motion distributed according to  $\mathcal{N}_{x,r}$, where we recall that this description is   informal since $\mathcal{N}_{x,r}$ is an infinite measure.  Since for any $r \geq 0$, the process $(\rho , W, \Lambda - r)$ under $\mathbf{N}_{x,r}$ is distributed $\textbf{N}_{x,0}$, it is often only necessary to prove statements holding under $\textbf{N}_{x,r}$ in the case $r = 0$, this will be used frequently in the sequel. Note  that $\mathbf{N}_{x,r}$ is a sigma-finite measure.   For example, it gives finite mass to events of the form  $\{    \uptau(W_\epsilon) \wedge H(\rho_\epsilon) > \epsilon \}$ for $\varepsilon>0$,  since by construction:
  \begin{equation*}
      N\big(  \mathbbm{1}_{\{ H(\rho_{\epsilon}) > \epsilon \}} \mathbf{Q}^{H(\rho)}_{x,r}(  \uptau(W_\epsilon) > \epsilon  ) \big)  = N(H(\rho_{\epsilon}) > \epsilon) \mathcal{N}_{x,r}(\uptau(\xi)> \epsilon) < \infty.
  \end{equation*}
  Let us now establish some important properties of $\mathbf{N}_{x,r}$. First, note that since under $N$ the process $\eta$ is a functional of $\rho$, we can consider the triplet $(\rho,\eta,\overline{W})$ under $\mathbf{N}_{x,r}$.   The definition of $\mathbf{N}_{x,r}$, paired with Lemma \ref{prop:Q:mathcal:N}-(i) and \eqref{dualidad:etaRho},  allows us   to recover the  duality property \eqref{dualidad:etaRhoW} of the Lévy snake under $\mathbf{N}_{x,r}$.
\begin{cor}\label{corollary:dualidadNblack}
    Under $\mathbf{N}_{x,r}$, the processes $((\rho_{s},\eta_s, \overline{W}_{s}) : 0 \leq s \leq \sigma  )$ and $((\eta_{ (\sigma - s)-}, \rho_{ (\sigma - s)-},\overline{W}_{\sigma - s}) : 0 \leq s \leq \sigma  )$  have the same distribution. 
\end{cor}  
It is not difficult to check from our definitions that  $(\rho,\overline{W})$ under $\mathbf{N}_{x,r}$ satisfies the following  simple Markov property. For every $t>0$ and conditionally on $\overline{\mathcal{F}}_t$, the shifted process $(\rho_{t+s},\overline{W}_{t+s})_{s\geq 0}$ is distributed according to $\mathbb{P}_{\rho_t,\overline{W}_t}^{\dag}$. In the following proposition, we strengthen this result by proving that $(\rho , \overline{W})$ satisfies the strong Markov property under $\mathbf{N}_{x,r}$.  
\begin{prop}\label{eq:strong:Markov:N} Let $T$ be a $(\overline{\mathcal{F}}_{t+})$-stopping time such that  $T > 0$, $\mathbf{N}_{x,r}$-a.e. For every non-negative $\overline{\mathcal{F}}_{T+}$-measurable function $\Phi$  and non-negative measurable function $F$ on $\mathbb{D}(\mathbb{R}_+, \mathcal{M}_{f}(\mathbb{R}_+)\times \mathcal{W}_{\overline{E}})$, we have 
\begin{equation} \label{equation:strongMarkov}
    \mathbf{N}_{x,r} \big( \mathbbm{1}_{\{T<\infty\}} \Phi \cdot  F \big( (\rho_{T + s }, \overline{W}_{T + s } : s \geq 0  ) \big)  \big) 
    = 
    \mathbf{N}_{x,r} \big( \mathbbm{1}_{\{T<\infty\}}  \Phi \cdot  \mathbb{E}^{\dag}_{\rho_T, \overline{W}_T} [ F ]   \big).  
\end{equation}
\end{prop}

\begin{proof}  
Our arguments  are taken from Section 4.1.3 of \cite{DLG02}. First, notice that
 it suffices to prove the result for an arbitrary bounded stopping  time $T$ taking values in $(0,\infty)$, that we fix from now on.  For  $t \geq 0$, we let $\lceil t \rceil$ be the smallest integer satisfying   $t < \lceil t \rceil$  and for every $n\geq 1$, we set $T^{(n)}:= \lceil Tn \rceil/n $.  We write $d_{\text{P}}$ for the Prokhorov metric on $\mathcal{M}_f(\mathbb{R}_+)$.  Next, consider $f: \mathcal{M}_f(\mathbb{R}_+) \times \mathcal{W}_{\overline{E}} \mapsto \mathbb{R}_+$  a bounded non-negative Lipschitz-continuous function with respect to the product metric $d_{\text{P}}\wedge 1 + d_{\mathcal{W}_{\overline{E}}} \wedge 1$ as well as  a bounded non-negative $\overline{\mathcal{F}}_{T+}$-measurable function  $\Phi$, non-null on a set with finite $\mathbf{N}_{x,r}$ measure. If we let $(\texttt{P}_t)_{t \geq 0}$   be the transition 
 semigroup of the  $ (\psi, \overline{\Pi} )$-Lévy snake, the statement of the proposition reduces to  show that  
\begin{equation*}
    \mathbf{N}_{x,r} \big(  \Phi \cdot  f(\rho_{T + t }, \overline{W}_{T + t }   )   \big) 
    = 
    \mathbf{N}_{x,r} \big(  \Phi \cdot  \texttt{P}_t f(\rho_T, \overline{W}_T)   \big).  
\end{equation*}
In this direction, for every $n \geq 1$,  by the simple Markov property  we have 
\begin{align}
    \mathbf{N}_x( \Phi \cdot  f(\rho_{T^{(n)}+t}, \overline{W}_{T^{(n)}+t} ) )
    &= \sum_{k=1}^\infty \mathbf{N}_{x,r}( \Phi \cdot\mathbbm{1}_{ \{ T^{(n)} = \frac{k}{n} \} }f(\rho_{T^{(n)}+t}, \overline{W}_{T^{(n)}+t} ) ) \nonumber \\
    &= \sum_{k=1}^\infty \mathbf{N}_{x,r}\big( \Phi \cdot  \mathbbm{1}_{\{ T^{(n)} =\frac{k}{n} \} } \texttt{P}_t f( \rho_{\frac{k}{n}} , \overline{W}_{\frac{k}{n}}) \big) \nonumber \\
    &=  \mathbf{N}_{x,r}\big( \Phi  \cdot  \texttt{P}_t f( \rho_{T^{(n)}} , \overline{W}_{T^{(n)}}) \big).   \label{equation:SMPeq1}
\end{align}
Since  $(\rho , \overline{W})$ under $\mathbf{N}_{x,r}$ is right-continuous, it holds that $\lim_{n \rightarrow \infty} f(\rho_{T^{(n)}+t}, \overline{W}_{T^{(n)}+t} ) = f(\rho_{T+t}, \overline{W}_{T+t} )$,  and to conclude it suffices to prove  $\limsup_{n \rightarrow \infty} | \texttt{P}_t f( \rho_{T}, \overline{W}_{T} ) -  \texttt{P}_t  f( \rho_{T^{(n)}}, \overline{W}_{T^{(n)}} )   |=0$, $\mathbf{N}_{x,r}$-a.e. To this end, for every $\epsilon > 0$ and $(\mu , \overline{\w} ) \in \overline{\Theta}$,  set  
 \begin{align}\label{equation:Vepsilon}
     \mathcal{V}_{\epsilon}  (\mu, \overline{\w} ) 
     := & \Big\{  (\mu', \overline{\w}') \in \overline{\Theta} : \,   \exists \,  \epsilon_0, \epsilon_1 \in [0,\epsilon) \text{ such that }  \kappa_{\epsilon_0} \mu = \kappa_{\epsilon_1}\mu' , \nonumber \\ 
     & \quad  \text{ and } \big(\overline{\w}(h) : 0 \leq h \leq H( \kappa_{\epsilon_0} \mu) \big) =  
     \big(\overline{\w}'(h) : 0 \leq h \leq H(\kappa_{\epsilon_1} \mu' ) \big) 
     \Big\},
 \end{align}
 where we recall that $\kappa_a$ for $a \geq 1$ is the pruning operation defined in \eqref{definition:pruning}.  We  introduce this set since for every $\varepsilon>0$, it is plain that $\mathbf{N}_{x,r}$-a.e., the pair $(\rho_{T+s} , \overline{W}_{T+s})$ belongs to $\mathcal{V}_\epsilon(\rho_T, \overline{W}_T)$ for $s$ small enough.  Therefore, $\mathbf{N}_{x,r}$-a.e., 
 \begin{equation*}
     \limsup_{n \rightarrow \infty} \big| \texttt{P}_t  f( \rho_{T}, \overline{W}_{T} ) -  \texttt{P}_t  f( \rho_{T^{(n)}}, \overline{W}_{T^{(n)}} )   \big| 
     \leq \sup_{(\mu', \overline{\w}') \in \mathcal{V}_{\epsilon}(\rho_T, \overline{W}_T )} 
     \big| \texttt{P}_t  f( \rho_{T}, \overline{W}_{T} ) -  \texttt{P}_t  f(\mu', \overline{\w}') \big|.
 \end{equation*}
To conclude, we rely on results from \cite{DLG02}. Namely, it was established in the proof of \cite[Lemma 4.1.3]{DLG02} by means of a coupling argument\footnote{Lemma 4.1.3 in \cite{DLG02} works on a  more general space than $\mathcal{W}_{\overline{E}}$, namely the space of killed rcll  paths. Despite this difference, the same coupling argument applies, leading to the conclusion in \eqref{equation:uniformconv}.}   that for every $(\mu , \overline{\w}) \in \overline{\Theta}$, we have 
\begin{equation} \label{equation:uniformconv}
    \lim_{\epsilon \rightarrow 0} \sup_{(\mu', \overline{\w}') \in \mathcal{V}_{\epsilon}(\mu, \overline{\w} )} 
     | \texttt{P}_t  f( \mu , \overline{\w} ) -  \texttt{P}_t f(\mu', \overline{\w}') | = 0.
\end{equation}
This completes the proof of the proposition.
\end{proof}
\noindent We conclude this section by   introducing our candidate measure for the law of the excursions away from $x$.  
\begin{def1}\label{definition:N*} We write   $\mathbf{N}_x^*$ for the law in $\mathbb{D}(\mathbb{R}_+, \mathcal{M}_f(\mathbb{R}_+)\times \mathcal{W}_E)$ of  $\mathrm{tr}(\rho , W) \text{ under }\mathbf{N}_{x,r}.$  
\end{def1}
Observe that this definition of $\mathbf{N}_x^*$  does not depend on  $r$. Therefore, the measure $\mathbf{N}_x^*$  is well-defined, and we refer to it as the excursion measure\footnote{The terminology might be slightly misleading, the measure $\mathbf{N}_x^*$ should not be confused with   $\mathbb{N}_{x,0}$.} away from $x$ of the $\Pi$-Markov process indexed by the  $\psi$-Lévy tree. Note that by construction, under  $\mathbf{N}_{x,r}$, the pair $(\rho,\overline{W})$ belongs to $\overline{\mathcal{S}}_x$, whereas under $\mathbf{N}_x^*$ we have  $(\rho,W) \in \mathcal{S}_x$.  In particular, in both cases we can consider the corresponding tree-indexed process.  We stress that,  while under $\mathbf{N}_{x,r}$ the process $\rho$ can be understood\footnote{ Under $\mathbf{N}_{x,r}$, the conditional distribution of $W$ given $\rho$ is an infinite measure, and consequently this statement must be handled with care} as the exploration process of a Lévy tree (and the corresponding tree $\mathcal{T}_{H(\rho)}$   as a Lévy tree),  this is no longer the case under $\mathbf{N}_x^*$. This makes the measure $\mathbf{N}_{x,r}$ much simpler to work with, as one can naturally extend classical results on Lévy snakes under $\mathbf{N}_{x,r}$ by making use of the  machinery developed in \cite{DLG02, 2024structure}.  In order to study $\mathbf{N}_x^*$, it is often simpler to start working under $\mathbf{N}_{x,r}$, since one can then extend the analysis to $\mathbf{N}_x^*$ through the identity $\mathbf{N}^*_{x} = \mathbf{N}_{x,r}( \text{tr}(\rho , W) \in \, \cdot \,  )$.

\section{Spinal decompositions}\label{section:MainSpinalDecomp}

On the forthcoming sections, the study of spatial properties of the Lévy snake and its excursions away from $x$ rely in a precise description of the law of $(\rho, \eta, \overline{W})$  at typical times taken under different measures. In order to state our results  we first need to introduce some notations. In this section we still work with an  arbitrary fixed $r\geq 0$. 
\par 
We start by introducing a pointed version of   $\mathbb{N}_{x,r}$. Namely, let $\mathbb{N}_{x,r}^\bullet$ be the  measure on  $\mathbb{D}(\mathbb{R}_+, \mathcal{M}_f^2(\mathbb{R}_+)\times \mathcal{W}_E)  \times \mathbb{R}_+$ defined by $\mathbb{N}_{x,r}^\bullet := \mathbb{N}_{x,r}(\dd \rho  , \dd \overline{W})  \dd s \, \mathbbm{1}_{\{ s \leq \sigma \}}$. Recall that under $\mathbb{N}_{x,r}$ the process $\eta$ is a functional of $\rho$, so that $\eta$ is well defined under 
 $\mathbb{N}^\bullet_{x,r}$. If  we write  $U : \mathbb{R}_+ \mapsto \mathbb{R}_+$ for the identity function $U(s) = s$ for $s \geq 0$, the law of  $(\rho , \eta,  \overline{W}, U)$ under $\mathbb{N}_{x,r}^\bullet$ is then  characterised by the relation 
    \begin{equation*}
        \mathbb{N}_{x,r}^\bullet \big( \Phi( \rho , \eta , \overline{W} , U ) \big) 
 = \mathbb{N}_{x,r}\Big( \int_0^{\sigma} \dd s \, \Phi( \rho  , \eta, \overline{W} , s )\Big). 
\end{equation*}
    Roughly speaking $U$ is a point taken, conditionally on $(\rho , \eta, \overline{W})$,  uniformly at random with respect to the Lebesgue measure on $[0,\sigma]$. Our goal now is to describe the law of $(\rho_U , \eta_U, \overline{W}_U)$ under  $\mathbb{N}_{x,r}^\bullet$.  To this end,  in an auxiliary probability space $(\Omega_0, \mathcal{F}_0, P^0)$,  we consider a $2$-dimensional subordinator $(U^{(1)}, U^{(2)} )$ with Laplace exponent given by:
\begin{equation} \label{identity:exponenteSubord}
    - \log E^0  \Big[ \exp \big( - \lambda_1 U^{(1)}_1 - \lambda_2 U^{(2)}_1 \big) \Big] := 
    \begin{cases}
    \frac{\psi(\lambda_1) - \psi(\lambda_2)}{\lambda_1 - \lambda_2} -\alpha\quad \quad \text{if } \lambda_1 \neq  \lambda_2,  \\
    \psi'(\lambda_1) - \alpha    \quad \quad \hspace{8mm} \text{if } \lambda_1 = \lambda_2.  
    \end{cases}
 \end{equation}
 In particular, we have the identity in law   $(U^{(1)},U^{(2)}) \overset{(d)}{=}(U^{(2)},U^{(1)})$ and note that both subordinators $U^{(1)}$ and $U^{(2)}$  have Laplace exponent $\psi(\lambda)/\lambda - \alpha$, for $\lambda \geq 0$. Still under $P^0$ and for $a \in (0,\infty]$, we write $(\mathcal{J}_a , \widecheck{\mathcal{J}}_a)$ for the measure in $\mathbb{R}_+^2$ defined by the relation
\begin{equation*}
   (\mathcal{J}_a , \widecheck{\mathcal{J}}_a)  := \big(\mathbbm{1}_{[0,a]}(t)~ \dd U^{(1)}_t , \mathbbm{1}_{[0,a]}(t) ~\dd U^{(2)}_t \big)
\end{equation*}
 with the usual convention $[0,\infty] := [0,\infty)$.  Our interest behind these measures comes from the fact that, by formula (18) in \cite{DLG02}, we have: 
 \begin{equation}\label{eq:many:to:one:N}
     N\Big( \int_0^{\sigma} \dd s \, \Phi \big( \rho_s, \eta_s   \big) \Big)   
    = \int_0^\infty \dd a \, \exp\big(- \alpha a\big)\cdot  
    E^0 \big(  \Phi( \mathcal{J}_a , \widecheck{\mathcal{J}}_a ) \big),
    \end{equation} 
     where $\Phi$ is a non-negative measurable function on $\mathbb{D}(\mathbb{R}_+,  \mathcal{M}^ 2_f(\mathbb{R}_+ ))$, and where we recall from Section \ref{section:MPonLT} that $N$ is the excursion measure of the  underlying reflected Lévy process. The previous formula can be easily enriched with the process $\overline{W}$ and takes the form:
\begin{equation}\label{eq:many:to:one:N:x}
     \mathbb{N}_{x,r}^{\bullet} \big(  \Phi\big( \rho_U,\eta_U, \overline{W}_U\big)  \big)   
    = \int_0^\infty \dd a \, \exp\big(- \alpha a\big)\cdot  
    E^0 \otimes \Pi_{x,r} \Big( \Phi\big( \mathcal{J}_a , \widecheck{\mathcal{J}}_a  , (\overline{\xi}_t : t \leq a)   \big) \Big),
\end{equation} 
we refer to \cite[Lemma 2.4]{2024structure} for a proof. \par 
    We will now apply the same treatment to the Lévy snake under pointed versions of the measures $\mathbf{N}_{x,r}$. In this direction, we introduce the pointed measure 
 \begin{equation}\label{N:bullet:black:def}
 \mathbf{N}^{\bullet}_{x,r}:= \mathbf{N}_{x,r}(\dd \rho ,   \dd \overline{W})  \dd s \, \mathbbm{1}_{\{ s \leq \sigma \}},
 \end{equation}
 and  recall the definition of the measure  
 $\mathcal{N}_{x,r}$ given in  Section \ref{section:framework}. 
\begin{lem} \label{proposition:only:spineN*}    For every non-negative measurable function $\Phi$  on $\mathcal{M}^2_f(\mathbb{R}_+) \times \mathcal{W}_{\overline{E}}$  we have: 
\begin{equation*} \label{equation:spineNblack}
 \mathbf{N}_{x,r}^{\bullet} \big(  \Phi\big( \rho_U,\eta_U, \overline{W}_U\big)  \big)  
    =  
     E^0 \otimes \mathcal{N}_{x,r} \Big( \int_0^\infty \dd a \, \exp (- \alpha a) \cdot  \Phi \big( \mathcal{J}_a , \widecheck{\mathcal{J}}_a  , (\overline{\xi}_t : t \leq a)  \big) \Big).  
\end{equation*}
\end{lem}
\begin{proof}

Fix a  non-negative measurable function  
 $\Phi$  on $\mathcal{M}^2_f(\mathbb{R}_+) \times  \mathcal{W}_{\overline{E}}$.
By definition of $\mathbf{N}_{x,r}$,  for every fixed $s\geq 0$, under $\mathbf{N}_{x,r}$ and conditionally on $(\rho, \eta)$,  the variable $\overline{W}_s$ is distributed as  $(\xi_t , \mathcal{L}_t)_{t \geq 0}$ under $\mathcal{N}_{x,r}$ restricted to $[0,H(\rho_s)]$. 
Hence, 
$$ 
\int_0^{\infty} \dd s \, \mathbf{N}_{x,r} \Big( \mathbbm{1}_{\{s \leq \sigma \}} \Phi\big(\rho_s , \eta_s,  \overline{W}_s \big) \Big)
=
 \int_{0}^{\infty} \dd s  \, N \Big(  \mathbbm{1}_{\{s \leq \sigma \}} \mathcal{N}_{x,r}  \Big[ \Phi\big(\rho_s , \eta_s,  \big( \overline{\xi}_t: t\leq H(\rho_s)\big) \big)\Big] \Big), 
$$
where now $\mathcal{N}_{x,r} \big[ \Phi\big(\rho_s, \eta_s , (\overline{\xi}_t: t \leq H(\rho_s))\big) \big]$ is a function of $(\rho_s, \eta_s)$.   An  application of Tonelli's theorem combined with \eqref{eq:many:to:one:N}    now entails the statement of the lemma, noting that for $a > 0$, we have  $H(J_a) = a$. This identity immediately follows from the fact the Lévy measure of $\psi(\lambda)/\lambda - \alpha$ is the measure in $\mathbb{R}_+$ given by  $\mathbbm{1}_{\mathbb{R}_+}(\ell)\dd \ell \, \pi([\ell ,\infty))$, and therefore it has infinite mass. We refer to  \cite[Section 1.1.2]{DLG02}  for details. 
\end{proof} 
    The formulations in  \eqref{eq:many:to:one:N:x}  and Lemma \ref{proposition:only:spineN*} are not strong enough for our purposes. Namely, it will be crucial for the sequel to also keep track of  the distribution of the pieces of trajectory before and after time $U$,   in both  \eqref{eq:many:to:one:N:x} and Lemma \ref{proposition:only:spineN*}. To this end,  for each  $(\upvarrho, \overline{\omega})\in \mathbb{D}(\mathbb{R}_+, \mathcal{M}_f(\mathbb{R}_+)\times \mathcal{W}_{\overline{E}})$,  we introduce for every $s\geq 0$, the notation
 \begin{equation*}
     \big(\upvarrho, \overline{\omega}\big)^{s,\leftarrow}:=(\upvarrho_{ ( 0 \vee  (s - t))-},\overline{\omega}_{ ( 0 \vee (s - t))-})_{t \geq 0}\quad\text{ and }\quad\big(\upvarrho, \overline{\omega}\big)^{s,\rightarrow}:=(\upvarrho_{s+t}, \overline{\omega}_{s+t})_{t\geq 0},
 \end{equation*}
 for the pieces of path coming before and after time $s$.  
  As usual, we use the convention $(\upvarrho_{0-},\overline{\omega}_{0-}):=(\upvarrho_{0},\overline{\omega}_{0})$. For $s\geq 0$, we shall refer to the pair
 \begin{equation}\label{equation:spinals}
     \big( \big(\upvarrho, \overline{\omega}\big)^{s,\leftarrow},\big(\upvarrho, \overline{\omega}\big)^{s,\rightarrow}\big) 
 \end{equation}
 as the spinal decomposition of $(\upvarrho, \overline{\omega})$ at time $s$.
 Observe that for every fixed $s\geq 0$, the process $(\upvarrho,\overline{\omega})$ can be easily    recovered from  \eqref{equation:spinals}  by time reversing $\big(\upvarrho, \overline{\omega}\big)^{s,\leftarrow}$
 and concatenating the resulting piece of path  with $\big(\upvarrho, \overline{\omega}\big)^{s,\rightarrow}$. 
 \par 
 Our goal now is to characterise  the law of $(\rho,\overline{W})^{U,\leftarrow}$ and $(\rho,\overline{W})^{U,\rightarrow}$ under the measures $\mathbb{N}^{\bullet}_{x,r}$ and $\mathbf{N}^{\bullet}_{x,r}$. In this direction, recall that $\mathbb{P}_{\mu,\overline{\mathrm{w}}}$ stands for the law of $(\rho, \overline{W})$ started from $(\mu,\overline{\mathrm{w}})$. Characterising the distribution of $(\rho,\overline{W})^{U,\rightarrow}$ in terms of the measures $(\mathbb{P}_{\mu , \overline{\mathrm{w}}}, (\mu , \overline{\w}) \in \overline{\Theta}_x)$  can be easily done by an application of the Markov property, but the one of  $(\rho,\overline{W})^{U,\leftarrow}$ is more delicate. Indeed, the latter  relies on the duality property \eqref{dualidad:etaRhoW},  and 
 in a description of the law of $(\rho , \eta , \overline{W})$ started from an arbitrary element $(\mu , \nu , \overline{\w})$ of $\mathcal{M}^2_f(\mathbb{R}_+)\times \mathcal{W}_{\overline{E}}$. Since the conditional distribution of $\overline{W}$   given $(\rho , \eta)$ only depends on  $H(\rho)$, we only need to  introduce the law of the pair $(\rho, \eta)$ started from $(\mu , \nu)$ - we recall that by   \cite[Proposition 3.1.2.]{DLG02}  the pair  $(\rho , \eta )$ is a Markov process. To do so, we first   extend the pruning and concatenation operations, introduced in Section 
 \ref{subsection:explorationprocess} for elements of $\mathcal{M}_{f}(\mathbb{R}_+)$, to arbitrary elements  $(\mu,\nu)$ in $\mathcal{M}^2_{f}(\mathbb{R}_+)$. Our presentation follows \cite[Section 3.1]{DLG02}.  First, for every $a\geq 0$, we write: 
 $$
 \kappa_a(\mu ,\nu ):=(\tilde{\mu} ,\tilde{\nu} ),
 $$ 
 for $\tilde{\mu} :=\kappa_a\mu $, and where the measure $\tilde{\nu}$ is the unique  element of $\mathcal{M}_{f}(\mathbb{R}_+)$ such that $(\mu +\nu )|_{[0, H(\tilde{\mu} )]}=\tilde{\mu} +\tilde{\nu}$. Notice that $\tilde{\nu}$ and $\nu|_{[0, H(\tilde{\mu})]}$ may only differ at the point $H(\tilde{\mu})$. Next, we define the concatenation of two elements $(\mu_1,\nu_1),(\mu_2,\nu_2)\in\mathcal{M}^2_{f}(\mathbb{R}_+)$, with $H(\mu_1) < \infty$, by setting:
 \begin{equation*}
     [(\mu_1,\nu_1),(\mu_2,\nu_2)] := \big([\mu_1,\mu_2],\theta\big) 
 \end{equation*}
  where $\theta$ is the element of $\mathcal{M}_f(\mathbb{R}_+)$ defined by $\langle \theta,f \rangle:=\int \nu_1(\mathrm{d} s) \mathbbm{1}_{[0,H(\mu_1)]}(s)f(s)+ \int \nu_2(\mathrm{d} s ) f(H(\mu_1)+ s)$, for any non-negative measurable function $f$ in $\mathbb{R_+}$. Note that in general  $\theta$ does not coincide with $[\mu_2, \nu_2]$. 
 Then, for every $(\mu,\nu)\in \mathcal{M}_{f}^2(\mathbb{R}_+)$, we let $\mathbf{P}_{\mu,\nu}$ be the distribution under $\mathbf{P}_0$ in $\mathbb{D}(\mathbb{R}_+, \mathcal{M}_f^2(\mathbb{R}_+))$ of the process 
 \begin{equation}\label{equation:leyetarho}
     \big[\kappa_{-I_t}(\mu,\nu ), (\rho_t,\eta_t)\big],  \quad \text{ for }{t\geq 0}. 
 \end{equation}
 The probability measure $\mathbf{P}_{\mu,\nu}$ is the law of the strong Markov process $(\rho , \eta)$ started from $(\mu, \nu)$, we refer to \cite[Section 3.1]{DLG02} for a detailed account. Since under $\mathbf{P}_0$ the processes $\eta$ and $(I_t)_{t\geq 0}$ are functionals of $\rho$, it  follows that  $\eta$ is also a functional of $\rho$ under $\mathbf{P}_{\mu,\nu}$. For $(\mu,\overline{\w})\in \overline{\Theta}$ and $\nu\in \mathcal{M}_{f}(\mathbb{R}_+)$, we define a probability measure in  $\mathbb{D}(\mathbb{R}_+, \mathcal{M}^2_f(\mathbb{R}_+)\times \mathcal{W}_{\overline{E}})$  by setting  
\begin{equation}\label{equation:pmunu}
\mathbb{P}_{\mu,\nu,\overline{\text{w}}}(\dd \rho,\dd \eta,\dd \overline{W}):=\textbf{P}_{\mu,\nu}(\dd \rho, \dd \eta)\:Q^{H(\rho)}_{\overline{\w}}(\dd \overline{W}). 
\end{equation} 
We write $\mathbb{P}^\dag_{\mu,\nu,\overline{\w}}$ for the law of $(\rho,\eta,\overline{W})$  under $\mathbb{P}_{\mu,\nu,\overline{\w}}$  stopped at  time $\inf\{t\geq 0:~\langle \rho_t,1\rangle=0\}$.

\begin{prop} \label{manytoone-lebesgue}  Under $\mathbb{N}_{x,r}^\bullet$ and  $\mathbf{N}_{x,r}^\bullet$, conditionally on $(\rho_U,\eta_U, \overline{W}_U)$, the
processes $(\rho,\overline{W})^{U,\leftarrow}$ and  $(\rho,\overline{W})^{U,\rightarrow}$ are independent and their 
 respective   conditional laws are as follows:
\begin{itemize}
\item the process $(\rho,\overline{W})^{U,\leftarrow}$ is distributed as  $(\eta, \overline{W})$ under $\mathbb{P}_{\eta_U, \rho_U,\overline{W}_U}^\dagger$; 
\item the process $(\rho,\overline{W})^{U,\rightarrow}$ is distributed as  $(\rho, \overline{W})$ under $\mathbb{P}_{\rho_U,\overline{W}_U}^\dagger$.
\end{itemize}
\end{prop}
Observe that \eqref{eq:many:to:one:N:x} and  Lemma \ref{proposition:only:spineN*} paired with Proposition \ref{manytoone-lebesgue} completely characterise the law of the spinal decomposition at time $U$ of $(\rho , \overline{W})$ under $\mathbb{N}_{x,r}^\bullet$ and  $\mathbf{N}_{x,r}^\bullet$. 
\begin{proof} 
The proof is a straightforward consequence of the Markov property and the duality  \eqref{dualidad:etaRhoW}.   We start proving the result under $\mathbb{N}_{x,r}^\bullet$. We will then briefly explain how to extend the proof to the measure  $\mathbf{N}_{x,r}^\bullet$. Let  $\Phi_1$ be a  non-negative measurable function on $\mathcal{M}^2_f(\mathbb{R}_+) \times \mathcal{W}_{\overline{E}}$ and  $\Phi_2, \Phi_3$ two non-negative measurable functions on $\mathbb{D}(\mathbb{R}_+, \mathcal{M}_f(\mathbb{R}_+)\times \mathcal{W}_{\overline{E}})$. The statement of the proposition is equivalent to proving that 
 \begin{align}\label{equation:spinal1}
 &\mathbb{N}_{x,r}^\bullet \Big( \Phi_1\big(\rho_U, \eta_U,  \overline{W}_U \big) \Phi_2((\rho,\overline{W})^{U,\leftarrow}) \Phi_3((\rho,\overline{W})^{U,\rightarrow}) \Big) \nonumber \\
  &= \mathbb{N}_{x,r}^\bullet \Big( \Phi_1\big(\rho_U, \eta_U,  \overline{W}_U \big) \mathbb{E}_{\eta_{U} , \rho_U, \overline{W}_{U}}^\dagger\big[\Phi_2 (\eta,\overline{W}) \big] \mathbb{E}_{\rho_{U} , \overline{W}_{U}}^\dagger\big[\Phi_3(\rho,\overline{W})\big]   \Big) .
  \end{align}
In this direction, notice that  the left hand side in the last display is given by 
\begin{align*}
 \int_0^\infty \mathrm{d } s~  \mathbb{N}_{x,r} \Big(  \mathbbm{1}_{\{s\leq \sigma\} }\, \Phi_1\big(\rho_s , \eta_s,  \overline{W}_s \big) \Phi_2((\rho,\overline{W})^{s,\leftarrow}) \Phi_3((\rho,\overline{W})^{s,\rightarrow}) \Big).
 \end{align*} 
The simple Markov property under $\mathbb{N}_{x,r}$ at time $s$ followed by the  change of variable $s\mapsfrom \sigma-s$  gives that the  previous display writes 
$$  
\int_0^{\infty} \dd s~  \mathbb{N}_{x,r} \Big(\mathbbm{1}_{\{ s\leq \sigma\} }\, \Phi_1\big(\rho_{\sigma - s } , \eta_{\sigma - s},  \overline{W}_{\sigma - s} \big) \Phi_2((\rho,\overline{W})^{\sigma-s,\leftarrow}) \mathbb{E}_{\rho_{\sigma-s} , \overline{W}_{\sigma-s}}^\dagger\big[\Phi_3(\rho,\overline{W})\big]  \Big).  
$$
By an application of the duality identity \eqref{dualidad:etaRhoW}   the last display equals 
\begin{align*}
\int_0^\infty \mathrm{d } s~  \mathbb{N}_{x,r} \Big(  \mathbbm{1}_{\{ s\leq \sigma\}}\, \Phi_1\big(\eta_s , \rho_s,  \overline{W}_s \big)  \Phi_2\big((\eta_{s+t},\overline{W}_{s+t}:~t\geq 0)\big) \mathbb{E}_{\eta_{s} , \overline{W}_{s}}^\dagger\big[\Phi_3(\rho,\overline{W})\big] \Big), 
 \end{align*} 
and we can now use again the simple Markov property    to get by analogous arguments that the previous expression is given by  
\begin{equation*}
    \mathbb{N}_{x,r} \Big(  \int_0^\sigma \mathrm{d} s~\Phi_1\big(\eta_s , \rho_s,  \overline{W}_s \big) 
 \mathbb{E}_{\rho_{s} , \eta_s, \overline{W}_{s}}^\dagger\big[\Phi_2\big(\eta,\overline{W}\big)\big] \mathbb{E}_{\eta_{s} , \overline{W}_{s}}^\dagger\big[\Phi_3(\rho,\overline{W})\big]  \Big).
\end{equation*}
When applying \eqref{dualidad:etaRhoW} we used the fact that, since $(\rho, \eta)$ under $\mathbb{N}_{x,r}$ is a rcll process,  the set of its jump-times has null Lebesgue measure $\mathbb{N}_{x,r}$-a.e.  Finally,  the change of variable $s \mapsfrom \sigma - s$ and the duality property applied as before yields \eqref{equation:spinal1}, which proves the desired result under $\mathbb{N}_{x,r}^\bullet$. The same arguments apply under $\mathbf{N}_{x,r}^\bullet$, making use of Corollary \ref{corollary:dualidadNblack} and the  Markov property \eqref{equation:strongMarkov} instead. We skip the details. 
\end{proof}

\section{The excursion theory}\label{section:excursiontheory}

In this section we present the main results of the excursion theory. It is divided in tree sections. We start by recalling in Section \ref{section:additivefuncionals} some properties of the local time at $x$ of $(\widehat{W}_t: t \geq 0)$ that will be used frequently in the upcoming sections. Section \ref{subsection:exitformula} is devoted to the proofs of the exit formulas , and finally, in Section \ref{section:Poisson} we prove that the excursion process of the Lévy snake is a Poisson point measure.

\subsection{\texorpdfstring{The local time at $x$}{}} \label{section:additivefuncionals}
In this section,  we recall some important properties of an additive functional $(A_t)_{t\geq 0}$ of the Lévy snake  called  the local time at $x$ of $(\widehat{W}_t:~t\geq 0)$. This process was recently introduced in \cite{2024structure} and plays a role analogous to that of the classical local time for $\mathbb{R}_+$-indexed processes.  Roughly speaking, at each time $t \geq 0$, the variable $A_t$  measures how much time  $\widehat{W}$ has spent at $x$ up to time $t$. Let us stress that, although closely related (see \eqref{equation:soporteAditiva} below), the processes $(A_t)_{t \geq 0}$ and $(\Lambda_t)_{t \geq 0}$ are drastically different. Recall that vaguely,  for each fixed $t \geq 0$, the path $\Lambda_t$  codes the local time at $x$ of  the Markov process $\Pi$ along the ancestral path of $p_{H(\rho)}(t)$ in the tree $\mathcal{T}_{H(\rho)}$. 
\par  
For each $z \geq  0$ and $\overline{\w}\in \mathcal{W}_{\overline{E}}$, we set:
\begin{equation*}
    \tau_z(\overline{\w}) := \inf \{  h \geq 0 :~\overline{\w}(h)=(x,z)\},
\end{equation*}
with the usual convention $\inf \emptyset=\infty$. Under $\mathbb{P}_{0, x,0}$ and $\mathbb{N}_{x,0}$, the continuous, non-decreasing process $A = (A_t)_{t \in \mathbb{R}_+}$  is  defined by the relation 
\begin{equation}
 \label{equation:aproximacionA}
  A_t = \lim_{\varepsilon \downarrow 0} \frac{1}{\varepsilon}  \int_0^t \dd u  \int_{\mathbb{R}_+} \dd z \,  \mathbbm{1}_{\{ \tau_z(\overline{W}_u ) <  H(\rho_u) < \tau_z(\overline{W}_u) + \varepsilon \}},\quad t\geq 0, 
\end{equation} 
 the  convergence  holding uniformly on compact intervals  in measure under $\mathbb{P}_{0,x,0}$ and $\mathbb{N}_{x,0}( \, \cdot \, \cap \{ \sigma > \delta \} )$  for every   $\delta >0$. As a consequence, we can find a decreasing sequence
$(\varepsilon_k^\prime)_{k\geq 0}$ of positive real numbers converging to $0$ along which the convergence \eqref{equation:aproximacionA} holds uniformly in compact intervals under   $\mathbb{P}_{0,x,0}$ and $\mathbb{N}_{x,0}$ outside of a 
 negligible set. It will be convenient to define $A_t(\upvarrho,\overline{\omega})$ for every $(\upvarrho , \overline{\omega})\in \mathbb{D}(\mathbb{R}_+ , \mathcal{M}_f(\mathbb{R}_+) \times \mathcal{W}_{\overline{E}} )$ at $t \geq 0$ by setting
\begin{equation}
 \label{equation:aproximacionA:def}A_t(\upvarrho,\overline{\omega})= \liminf_{k\to\infty}  \frac{1}{\varepsilon_k^\prime}\int_0^{t}\dd u  \int_{\mathbb{R}_+} \dd z \,  \mathbbm{1}_{\{ \tau_z(\overline{\omega}_u ) < H(\upvarrho_u) < \tau_z(\overline{\omega}_u) + \varepsilon_k \}}.\end{equation}
This definition is consistent with \eqref{equation:aproximacionA} under $\mathbb{P}_{0,x,0}$ and $\mathbb{N}_{x,0}$ up to an negligible set.  We shall now argue that one can make use of this process to index the excursions away from $x$, and that this indexing is compatible with the order induced by time $t \geq 0$.   To this end, we recall that Theorem 4.19 in \cite{2024structure} states that the support  of the Stieltjes measure $\dd A$,  denoted by $\text{supp } \dd A$, satisfies:
\begin{equation}\label{equation:soporteAditiva}
     \text{supp } \dd A = [0,\sigma] \setminus \mathcal{C}^*\subset  \{ t \in [0,\sigma] : \widehat{W}_t = x \}, \quad \mathbb{P}_{0,x,0 } \text{ and } \mathbb{N}_{x,0} \text{ a.e.}~,
\end{equation}
where we write $\mathcal{C}^*$ for the subset of  all times $t \geq 0$ at 
 which  $\widehat{\Lambda}$ is constant on  some open neighborhood  of~$t$. 
The following decomposition in excursions of the functional $A$ will be used frequently in our computations. Under $\mathbb{P}_{0,x,0}$, let $(\rho^i, \overline{W}^i)_{i \in \mathbb{N}}$ be the subtrajectories associated to the connected components of $\{ t \geq 0 :  H(\rho_t) >  0  \}$. Then,
\begin{equation} \label{equation:aditiva1}
    A_\sigma = \sum_{i \in \mathbb{N}} A_{\sigma_i}(\rho^i , \overline{W}^i), 
     \quad \mathbb{P}_{0,x,0}  \text{-a.s.}
\end{equation}
We refer to the end of Section 4.2 in \cite{2024structure} for other  related results. 
 Recalling that by Lemma \ref{lemma:contanteExcursion}, the variables $\widehat{\Lambda}_{\mathfrak{g}(u)}$, $u\in \mathcal{D}$ under $\mathbb{N}_{x,0}$ and $\mathbb{\mathbb{P}}_{0,x,0}$ are a.e.  distinct,  it follows 
from  \eqref{equation:soporteAditiva} and the continuity of the map $t\mapsto \widehat{\Lambda}_t$ that:
\begin{cor} \label{cor:A:distinct}
Under $\mathbb{P}_{0,x,0}$ and $\mathbb{N}_{x,0}$,  all the values $A_{\mathfrak{g}(u)}$, $u\in \mathcal{D}$, are distinct. 
\end{cor}

 It might be worth mentioning that,   in stark contrast with $(A_{\mathfrak{g}(u)}, u \in \mathcal{D})$,  the  ordering induced by $(\widehat{\Lambda}_{\mathfrak{g}(u)}:u \in \mathcal{D})$ in $\mathcal{D}$ is not compatible with the order induced by time since $(\widehat{\Lambda}_t)_{t \geq 0}$ is not monotone.   \par 

We conclude this short section establishing a  spinal decomposition in the same vein as in Section \ref{section:MainSpinalDecomp}  but at a point sampled according  to the measure $\dd A$. This description will be used when establishing 
the relationship between the  $(\psi, \overline{\Pi})$-Lévy snake under $\mathbb{N}_{x,0}, \mathbb{P}_{0,x,0}$, and the measures $(\mathbf{N}_{x,r}$, $r\geq 0)$ in Section \ref{subsection:exitformula}.  To this end, 
we introduce the pointed measure 
\begin{equation*}
    \mathbb{N}^{\bullet, A}_{x,r}:= \mathbb{N}_{x,r}(\dd \rho ,   \dd \overline{W})  \dd A_s  
\end{equation*}
 in $\mathbb{D}(\mathbb{R}_+, \mathcal{M}_f^2(\mathbb{R}_+)\times \mathcal{W}_{\overline{E}})  \times \mathbb{R}_+$. 
With a slight abuse of notation, we extend the definition of $\tau_z$, $z\geq 0$, under $\Pi_{x,0}$ by setting $\tau_z(\overline{\xi}):=\inf\{h\geq 0:~\overline{\xi}_h=(x,z)\}$, which coincides with the first time at which  the local time $\mathcal{L}$ takes the value $z$. When there is no risk of  confusion we will drop the dependence on $\overline{\xi}$, and to simplify notation we write $\bar{\xi}^{\tau_z}:=\big(\bar{\xi}_t:~  t\in [0,\tau_z]\big)$.

\begin{prop} \label{manytoone-lebesgue-Addi}
Fix $r \geq 0$. For every non-negative measurable function $\Phi$ on $\mathcal{M}^2_f(\mathbb{R}_+) \times \mathcal{W}_{\overline{E}}$, we have:
\begin{align}\label{eq:prop:many:A}
     & \mathbb{N}_{x,r}^{ \bullet, A } \big(  \Phi\big( \rho_U,\eta_U, \overline{W}_U\big)  \big)  
  =   E^0 \otimes \Pi_{x,r} \Big(   \int_r^\infty \dd a \, \exp\big(- \alpha \tau_{a}\big)\Phi( \mathcal{J}_{\tau_{a}}, \widecheck{\mathcal{J}}_{\tau_{a}} , \bar{\xi}^{\tau_{a}} )  \Big).
\end{align}
Furthermore, under $\mathbb{N}_{x,r}^{ \bullet, A }$, conditionally on $(\rho_U,\eta_U, \overline{W}_U)$, the
processes $(\rho,W)^{U,\leftarrow}$ and  $(\rho,W)^{U,\rightarrow}$ are independent and their  conditional laws are as follows:
\begin{itemize}
\item the process $(\rho,W)^{U,\leftarrow}$ is distributed as  $(\eta, \overline{W})$ under $\mathbb{P}_{\eta_U, \rho_U,\overline{W}_U}^\dagger$; 
\item the process $(\rho,W)^{U,\rightarrow}$ is distributed as  $(\rho, \overline{W})$ under $\mathbb{P}_{\rho_U,\overline{W}_U}^\dagger$.
\end{itemize}
\end{prop}
In particular, the only difference with Proposition \ref{manytoone-lebesgue} lies in  the distribution of $ (\rho_U,\eta_U, \overline{W}_U)$. 

\begin{proof}
The proof is similar to the one of Proposition \ref{manytoone-lebesgue} so we will be brief. First, by \cite[Lemma 4.11]{2024structure}, we already know that:
\begin{equation}\label{eq:1:prop:spinal:A}
     \mathbb{N}_{x,r} \Big( \int_0^{\sigma} \dd A_s \, \Phi \big( \rho_s, \eta_s , \overline{W}_s  \big) \Big)   
    =   
    E^0 \otimes \Pi_{x,r} \Big( \int_r^\infty \dd a \, \exp\big(- \alpha \tau_a\big) \Phi \big( \mathcal{J}_{\tau_a} , \widecheck{\mathcal{J}}_{\tau_a}  , \bar{\xi}^{\tau_{a}}   \big) \Big).
\end{equation}
Furthermore, if  $\Phi_2,\Phi_3$ are non-negative measurable functions on  $\mathbb{D}(\mathbb{R}_+, \mathcal{M}_f(\mathbb{R}_+)\times \mathcal{W}_{\overline{E}})$, by strong Markov property we get that:
\begin{align}\label{eq:2:prop:spinal:A}
   & \mathbb{N}_{x,r} \Big( \int_0^{\sigma} \dd A_s \, \Phi\big(\rho_s, \eta_s,  \overline{W}_s \big) \Phi_2((\rho,\overline{W})^{s,\leftarrow}) \Phi_3((\rho,\overline{W})^{s,\rightarrow}) \Big) \nonumber \\
 &=  \mathbb{N}_{x,r} \Big( \int_0^{\sigma} \dd A_s \, \Phi\big(\rho_s , \eta_s,  \overline{W}_s \big) \Phi_2((\rho,\overline{W})^{s,\leftarrow}) \mathbb{E}_{\rho_s , \overline{W}_s}^\dagger\big[\Phi_3(\rho,\overline{W})\big] \Big).
 \end{align} 
We can now use the duality identity \eqref{dualidad:etaRhoW} paired  with \eqref{equation:aproximacionA} to infer that
\begin{equation*}
     \Big( \big(\rho_t, \eta_t, \overline{W}_t, A_t\big) : t \in [0,\sigma] \Big) \overset{(d)}{=} \Big( \big(\eta_{(\sigma -t )-}, \rho_{(\sigma -t )-} , \overline{W}_{\sigma -t },      A_{\sigma} - A_{\sigma-t} \big) :t \in [0,\sigma] \Big), \quad \text{ under } \mathbb{N}_{x,r}.
\end{equation*}
Therefore,  as in the proof of Proposition  \ref{manytoone-lebesgue}, by performing the change of variable $s \mapsfrom \sigma-s$ and applying again the strong Markov property, we obtain that \eqref{eq:2:prop:spinal:A} equals:
$$  \mathbb{N}_{x,r} \Big( \int_0^{\sigma} \dd A_s \, \Phi\big(\rho_s , \eta_s,  \overline{W}_s \big)  \mathbb{E}_{\eta_s ,\rho_s, \overline{W}_s}^\dagger\big[\Phi_2(\eta,\overline{W})\big] \mathbb{E}_{\rho_s , \overline{W}_s}^\dagger\big[\Phi_3(\rho,\overline{W})\big] \Big).$$
The statement of the proposition  follows by combining the previous display with \eqref{eq:1:prop:spinal:A}.
\end{proof}

\subsection{The exit formula}\label{subsection:exitformula}

In this section we establish the connection between the family of measures $(\mathbf{N}_{x,r})_{r\geq 0}$ introduced in Section~\ref{section:excursionmeasure} and the $(\psi, \overline{\Pi})$-Lévy snake, through the  \textit{exit formula} stated in Theorem \ref{theorem:exit}. Let us first introduce some notation.  Under $\mathbb{N}_{x,0}$ and $\mathbb{P}_{0,x,0}$, for each $u  \in \mathcal{D}$ we write $(\mathrm{T}^u_t)_{t \geq 0}$ for the right-inverse of the functional 
\begin{equation}\label{equation:trimming}
    \int_0^t \dd s \, \mathbbm{1}_{  s \,  \in \,  (\mathfrak{g}(u),\mathfrak{d}(u) )  }, \quad t \geq 0. 
\end{equation}
It is plain that the time-changed process $\texttt{T}_u (\rho,\overline{W}) := (\rho_{\mathrm{T}_t^{u}}, W_{\mathrm{T}_t^{u}})_{t \geq 0}$ is an element of $\overline{\mathcal{S}}_x$ which roughly speaking, is obtained from gluing the two pieces of path $( (\rho_t, \overline{W}_t) : 0 \leq t < \mathfrak{g}(u))$ and $( (\rho_t, \overline{W}_t) :  \mathfrak{d}(u) \leq t)$  so that the subtrajectory $(\rho^u, \overline{W}^u)$ has been ``trimmed" from   $(\rho, \overline{W})$. Recall that since the image under $p_{H(\rho)}$ of $[\mathfrak{g}(u), \mathfrak{d}(u)]$ is the subtree stemming from $u$,  in terms of labelled trees, the subtrajectory $(\rho^u, \overline{W}^u)$ can be  thought of as  coding the labels on $\{v \in \mathcal{T}_{H(\rho)} : u \preceq v  \}$ and therefore   $\texttt{T}_u (\rho,\overline{W})$ as coding  the labels on its complement $\mathcal{T}_{H(\rho)} \setminus \{v \in \mathcal{T}_{H(\rho)} : u \preceq v  \}$.
\begin{theo}\label{theorem:exit}
For every  non-negative measurable functions $F,G$ on $\mathbb{D}(\mathbb{R}_+, \mathcal{M}_f(\mathbb{R}_+)\times \mathcal{W}_{\overline{E}})\times \mathbb{R}$ and 
 $\mathbb{D}(\mathbb{R}_+, \mathcal{M}_f(\mathbb{R}_+)\times \mathcal{W}_{\overline{E}})$, we have 
    \begin{equation}\label{eq:theorem:exit}
    \mathbb{N}_{x,0}\Big(  \sum_{u \in \mathcal{D}} F\big( \,  \emph{\texttt{T}}_u (\rho,\overline{W}), \mathfrak{g}(u) \big) \cdot G(\rho^u , \overline{W}^u)  \Big)  = 
    \mathbb{N}_{x,0}\Big( \int_0^\sigma \mathrm{d} A_s  ~F(\rho,\overline{W},s ) \cdot \mathbf{N}_{x, \widehat{\Lambda}_s}\big(G(\rho,\overline{W})\big)\Big).
\end{equation}
\end{theo}
This formula shares strong similarities with the so-called exit formula of Maisonneuve \cite[Theorem 4.1]{Maisonneuve} and, with a slight abuse of terminology, we shall refer to it in this way. Most of this section is devoted to the proof of Theorem \ref{theorem:exit}. We will then derive a version of this result holding under the probability measure $\mathbb{P}_{0,x,0}$ that will be needed in Section \ref{section:Poisson}.  It will be used at multiple instances in our arguments the fact that by   \ref{Asssumption_3} and \eqref{eq:many:to:one:N:x}, it holds that 
\begin{equation}\label{equation:Lebx}
    \text{Leb}\big( \{ s \geq 0 : \widehat{W}_s = x \} \big) = 0, \quad \mathbb{N}_{x,0} - \text{a.e.}
\end{equation}
\begin{proof}
The proof can be divided in two main steps. We first start by rewriting the left and right-hand side of \eqref{eq:theorem:exit} in terms of three  spinal decompositions. We will then check that the two sides coincide by  making use of  \eqref{eq:many:to:one:N:x}, and of Propositions \ref{manytoone-lebesgue} and  \ref{manytoone-lebesgue-Addi}. Let us  start with the former task. In this direction, remark that by a simple continuity argument, under $\mathbb{N}_{x,0}$ and for every $u\in \mathcal{D}$, we have $0 < \sigma(\mathrm{tr}(W^u)) < \infty$ a.e., and that 
 under $\mathbf{N}_{x,0}$, we have $0 < \sigma(\mathrm{tr}(W)) < \infty$ a.e.
Therefore, by replacing the functions $F(\upvarrho, \overline{\omega}), G(\upvarrho, \overline{\omega})$ by $F(\upvarrho, \overline{\omega} ) / \sigma( \mathrm{tr}(\omega) )$ and $G(\upvarrho, \overline{\omega} ) / \sigma( \mathrm{tr}(\omega) )$ respectively, we get that to obtain \eqref{eq:theorem:exit} it suffices to prove  
    \begin{equation}\label{proof:theorem:exit:1:formula}
    \mathbb{N}_{x,0}\Big(  \sum_{u \in \mathcal{D}} F\big( \, \texttt{T}_u (\rho,\overline{W}), \mathfrak{g}(u)\big) \sigma(\mathrm{tr}(W^u))G\big(\rho^u , \overline{W}^u\big)  \Big)  = 
    \mathbb{N}_{x,0}\Big( \int_0^\sigma \mathrm{d} A_s  ~F\big(\rho,\overline{W},s \big) \cdot \mathbf{N}_{x, \widehat{\Lambda}_s}\big(\sigma(\mathrm{tr}(W))G\big(\rho,\overline{W})\big)\Big). 
\end{equation}
Let us now rewrite  the right-hand side of   \eqref{proof:theorem:exit:1:formula}.  First, recalling  the definition of $\uptau$ given in \eqref{equation:taudef} and the form of the time change $\Gamma$ appearing in the definition of  $\text{tr}(\upvarrho , \omega)$ in \eqref{definition:troncature*}, the right-hand side of   \eqref{proof:theorem:exit:1:formula} can be written in the following form:
\begin{equation}\label{equation:exiteq0}
\mathbb{N}_{x,0}\Big( \int_0^\sigma \mathrm{d} A_s  ~F\big(\rho,\overline{W},s \big) \cdot  \mathbf{N}_{x, \widehat{\Lambda}_s}\Big(\int_{0}^{\sigma(W)} \dd t \,  \mathbbm{1}_{\{ \uptau(W_t)=\infty \} }G\big(\rho,\overline{W})\Big)\Big). 
\end{equation}
Next, under $\mathbb{N}_{x,0}$ and $\mathbf{N}_{x,r}$ for $r\geq 0$, we introduce the notation:
$$
\texttt{Sp}_s (\rho, \overline{W}) : = \big((\rho,\overline{W})^{s,\leftarrow}, (\rho,\overline{W})^{s,\rightarrow} \big)
$$
for every  $s \geq 0$. Observe that for any $s \in (0,\sigma(W))$, we can recover the  process $(\rho,\overline{W})$  by considering $(\rho,\overline{W})^{s,\leftarrow}$ run backwards in time,  shifted at the last time at which it takes the value $(0,x,0) \in \mathcal{M}_f(\mathbb{R}_+)\times \mathcal{W}_{\overline{E}}$, and by  concatenating the resulting process with  $(\rho,\overline{W})^{s,\rightarrow}$. This operation on the pair $\texttt{Sp}_s (\rho, \overline{W})$ does not depend on $s$, and we denote it by $\mathfrak{R}$.   With this notation in hand, we get that  \eqref{equation:exiteq0} equals:
\begin{equation}\label{equation:exitproofEq3}
    \mathbb{N}_{x,0}\Big( \int_0^\sigma \mathrm{d} A_s  ~F\big(\mathfrak{R} \circ \texttt{Sp}_s (\rho, \overline{W}) , s\big) \cdot  \mathbf{N}_{x, \widehat{\Lambda}_s}\Big(\int_{0}^{\sigma(W)} \dd t \,  \mathbbm{1}_{\{ \uptau(W_t)=\infty \}}G\big(\mathfrak{R}\circ \texttt{Sp}_t (\rho, \overline{W}) \big) \Big) \Big).
\end{equation}
Let us now apply a similar treatment to the left-hand  side  of \eqref{proof:theorem:exit:1:formula}. In this direction we notice that, under $\mathbb{N}_{x,0}$, for each $u \in \mathcal{D}$ the positive variable $\sigma(\mathrm{tr}(W^u))$ is precisely the Lebesgue measure of $\{s\in[0,\sigma(W)]:~ p_{H(\rho)}(s)\in \mathcal{C}_u\}$.  It now follows   that the left-hand side of \eqref{proof:theorem:exit:1:formula} writes 
\begin{align*}
    \mathbb{N}_{x,0}\Big( \sum_{u \in \mathcal{D}}  \int_0^{\sigma(W)} \dd s~\mathbbm{1}_{ \{ p_H(s)\in \,  \mathcal{C}_u \} } F\big( \, \texttt{T}_u (\rho,\overline{W}), \mathfrak{g}(u)\big) G(\rho^u, \overline{W}^u) \Big).
\end{align*}
By Lemma \ref{lemma:connectedcomponents} and \eqref{equation:Lebx},  for Lebesgue almost every $s\in[0,\sigma(W)]$,   there is a unique $u(s)\in \mathcal{D}$ such that $p_{H(\rho)}(s)$ belongs to $\mathcal{C}_{u(s)}$, and by convention we set $u(s)=0$ when the latter fails. Recalling \eqref{equation:Lebx}, we can then rewrite the previous display in the following form
\begin{equation}\label{equation:exiteq2}
    \mathbb{N}_{x,0}\Big(  \int_0^{\sigma(W)} \dd s~ F\big( \, \texttt{T}_{u(s)} (\rho,\overline{W}), \mathfrak{g}(u(s)) \big) \cdot  G(\rho^{u(s)} , \overline{W}^{u(s)})  \Big). 
\end{equation}
In the same vein as before, the idea now is to express  $ \mathrm{T}_{u(s)} (\rho,\overline{W})$, $\mathfrak{g}(u(s))$ and $(\rho^{u(s)} , \overline{W}^{u(s)})$ in terms of $\texttt{Sp}_s$ and  $\mathfrak{
R}$. To this end, setting  $T^{\rightarrow}_s:=\inf\{t \geq 0:~H(\rho_{s+t})=\varsigma_x(W_s)\}$ with $\varsigma_x(\w) = \sup \{ 0 \leq s \leq  \zeta_\w : \w(s) = x \}$,  we ``cut''  $\big(\rho, \overline{W}\big)^{\rightarrow,s}$ in the two following  processes:
\begin{align*}
\big(\rho, \overline{W}\big)^{\rightarrow,s,\prime}_t &:=\big(\rho_{s+ T_s^{\rightarrow}+t}, \overline{W}_{s+ T_s^{\rightarrow}+t}\big),\quad t\geq 0,  \\
\big(\rho, \overline{W}\big)^{\rightarrow,s,\prime\prime}_t &:=\uptheta_{\varsigma_x (W_s)}\big(\rho_{(s+t)\wedge (s+T_s^{\rightarrow})}, \overline{W}_{(s+t)\wedge (s+T_s^{\rightarrow})}\big),\quad t\geq 0,    
\end{align*}
where  $\uptheta$ is the operation defined  in \eqref{operacion:translacion}. In the last display and for the remainder of the proof, when $\uptheta$ is applied to a tuple, we apply it coordinate-wise. We now proceed similarly with  the process $\big(\rho, \overline{W}\big)^{\leftarrow,s}$. Namely, we set $T^{\leftarrow}_s:=\inf\{t \geq 0:~H(\rho_{0 \vee(s-t)})=\varsigma_x(W_s)\}$   and we consider the following pair
\begin{align*}
    \big(\rho, \overline{W}\big)^{\leftarrow,s,\prime}_t & :=\big(\rho_{ {\, (0 \vee } (s- T_s^{\leftarrow}-t))-}, \overline{W}_{{ (0 \vee }(s- T_s^{\leftarrow}-t))-}\big), \quad t \geq 0, \\
    \big(\rho, \overline{W}\big)^{\leftarrow ,s,\prime \prime}_t &:=\uptheta_{\varsigma_x (W_s)}\big(\rho_{((s-t)\vee (s-T_s^{\leftarrow}))-}, \overline{W}_{((s-t)\vee (s-T_s^{\leftarrow}))-}\big), \quad t \geq 0.
\end{align*}
To simplify notation, for every $s \geq 0$, we set 
$$
\texttt{Sp}^\prime_s (\rho, \overline{W}) : = \big( (\rho,\overline{W})^{s,\leftarrow,\prime}, (\rho,\overline{W})^{s,\rightarrow,\prime} \big)\quad \text{ and } \quad \texttt{Sp}^{\prime\prime}_s (\rho, \overline{W}) : = \big( (\rho,\overline{W})^{s,\leftarrow,\prime\prime}, (\rho,\overline{W})^{s,\rightarrow,\prime\prime} \big).$$
Let us also denote by  $\Sigma_s^\prime$  the first time at which  $\big(\rho, \overline{W}\big)^{\leftarrow,s,\prime}$ takes the value $(0,x,0)$. The key now is that by \eqref{equation:Lebx},  under $\mathbb{N}_{x,0}$, a.e.  for Lebesgue-almost every $s \in [0,\sigma(W)]$ we have   
$$
\mathfrak{R}\circ  \texttt{Sp}^\prime_s (\rho, \overline{W}) = \texttt{T}_{u(s)}(\rho, \overline{W}), \quad \quad  \mathfrak{R}\circ \texttt{Sp}^{\prime\prime}_s (\rho, \overline{W}) = (\rho^{u(s)}, \overline{W}^{u(s)}) 
$$
and $\Sigma_s^{\prime}= \mathfrak{g}(u(s))$. The formula in   \eqref{equation:exiteq2} can now be written in terms of these identities, and recalling that the right-hand side of \eqref{proof:theorem:exit:1:formula} is given by  \eqref{equation:exitproofEq3}, we deduce that the desired identity \eqref{proof:theorem:exit:1:formula} can be re-written as follows:
\begin{align}\label{equation:exitformulareducida}
&\mathbb{N}_{x,0}\Big(  \int_0^{\sigma(W)} \dd s~ F\big( \mathfrak{R}\circ\texttt{Sp}^\prime_s (\rho, \overline{W}), \Sigma_s^\prime\big) \cdot  G\big(\mathfrak{R}\circ \texttt{Sp}^{\prime\prime}_s (\rho, \overline{W}) \big)  \Big)  \nonumber \\
&=
\mathbb{N}_{x,0}\Big( \int_0^{\sigma(W)} \mathrm{d} A_s  ~F\big(\mathfrak{R} \circ \texttt{Sp}_s (\rho, \overline{W}) , s\big) \cdot  \mathbf{N}_{x, \widehat{\Lambda}_s}\Big(\int_{0}^{\sigma(W)} \dd t \,  \mathbbm{1}_{\{ \uptau(W_t)=\infty \}}G\big(\mathfrak{R}\circ \texttt{Sp}_t (\rho, \overline{W}) \big) \Big) \Big). 
\end{align}
Observe that both $\Sigma_s'$ and $s$ are just  the hitting time of $(0,x,0)$ of respectively  the first element in    $\texttt{Sp}'_s (\rho, \overline{W})$ and $\texttt{Sp}_s (\rho, \overline{W})$. Hence, to establish \eqref{equation:exitformulareducida} it suffices to prove that for every non-negative measurable functions $\Phi, \Psi$   on  triples,  we have 
\begin{align}\label{equation:exitformulareducida-2}
&\mathbb{N}_{x,0}\Big(  \int_0^{\sigma(W)} \dd s~ \Phi\big(\texttt{Sp}^\prime_s (\rho, \overline{W})\big) \cdot  \Psi\big( \texttt{Sp}^{\prime\prime}_s (\rho, \overline{W}) \big)  \Big)  \nonumber \\
&=
\mathbb{N}_{x,0}\Big( \int_0^{\sigma(W)} \mathrm{d} A_s  ~\Phi\big( \texttt{Sp}_s (\rho, \overline{W}) \big) \cdot  \mathbf{N}_{x, \widehat{\Lambda}_s}\Big(\int_{0}^{\sigma(W)} \dd t \,  \mathbbm{1}_{\{ \uptau(W_t)=\infty \}} \Psi\big( \texttt{Sp}_t (\rho, \overline{W}) \big) \Big) \Big) . 
\end{align}
We will now simplify this identity  by exploiting the simple nature of the conditional law of the spine at time $U$ under the pointed measures $\mathbb{N}_{x,0}^\bullet$, $\mathbb{N}_{x,0}^{\bullet,A}$ and  $\mathbf{N}_{x,r}^\bullet$. First, under $\mathbb{N}_{x,0}$ and for every $s\in [0,\sigma(W)]$, set 
\begin{align*}
    & \big(\rho^\prime_s,\eta^\prime_s, \overline{W}_s^\prime\big):=\Big(\mathbbm{1}_{\{ h\leq \varsigma_x(W_s)  \}} \rho_s(\mathrm{d} h), \mathbbm{1}_{ \{ h\leq \varsigma_x(W_s) \} } \eta_s(\mathrm{d} h), \big(\overline{W}_s(t):~t\in [0, \varsigma_x(W_s) ]\big)\Big), \\
    & \big(\rho^{\prime\prime}_s, \eta^{\prime\prime}_s, \overline{W}_s^{\prime\prime}\big):= \uptheta_{\varsigma_x (W_s)}\big(\rho_s, \eta_s,  \overline{W}_s\big).     
\end{align*}
where again,  $\uptheta$  is applied coordinatewise.  Recall from Proposition \ref{manytoone-lebesgue}  that, under $\mathbb{N}_{x,0}^\bullet$ and conditionally on $(\rho_U,\eta_U,\overline{W}_U)$,  the processes $(\rho,\overline{W})^{\leftarrow, U}$ and $(\rho,\overline{W})^{\rightarrow, U}$ are  distributed according to $\mathbb{P}_{\eta_U,\rho_U, \overline{W}_U}^\dagger$ and $\mathbb{P}_{\rho_U,\overline{W}_U}^\dagger$ respectively. It  readily follows  from the previous construction, the Markov property, and the definition of the Lévy snake  that, under $\mathbb{N}_{x,0}^\bullet$ and conditionally on $(\rho_U,\eta_U,\overline{W}_U)$, the processes  $\big(\rho, \overline{W}\big)^{\leftarrow,U,\prime}$,  $\big(\rho, \overline{W}\big)^{\rightarrow,U,\prime}$,  $\big(\rho, \overline{W}\big)^{\leftarrow ,U,\prime \prime}$ and $\big(\rho, \overline{W}\big)^{\rightarrow,U,\prime\prime}$  are independent and  with respective distributions  $\mathbb{P}_{\eta_U^\prime,\rho_U^\prime, \overline{W}_U^\prime}^\dagger$, $\mathbb{P}_{\rho_U^\prime,\overline{W}_U^\prime}^\dagger$, $\mathbb{P}_{\eta_U^{\prime\prime},\rho_U^{\prime\prime}, \overline{W}_U^{\prime\prime}}^\dagger$ and $\mathbb{P}_{\rho_U^{\prime\prime},\overline{W}_U^{\prime\prime}}^\dagger$.  This fact applied to the left-hand side of \eqref{equation:exitformulareducida-2}, paired with an application of Propositions \ref{manytoone-lebesgue} and \ref{manytoone-lebesgue-Addi} 
 to the right-hand side of \eqref{equation:exitformulareducida-2} yields that the proof of \eqref{equation:exitformulareducida-2}  can be reduced  to establishing that:
\begin{align}\label{eq:interm:exit:system}
& \mathbb{N}_{x,0}\Big(\int_{0}^{\sigma(W)}\dd s~\Phi\big(\rho_s^\prime,\eta_s^\prime, \overline{W}_s^\prime\big) \cdot  \Psi\big(\rho_s^{\prime\prime},\eta_s^{\prime\prime}, \overline{W}_s^{\prime\prime}\big) \Big) \nonumber \\
&= \mathbb{N}_{x,0}\Big(\int_0^{\sigma(W)} \mathrm{d} A_s ~\Phi\big(\rho_s,\eta_s,\overline{W}_s\big) \cdot \mathbf{N}_{x,\widehat{\Lambda}_s}\Big(\int_0^{\sigma(W)}\dd t~  \mathbbm{1}_{\{ \uptau(W_t)=\infty \}} \Psi(\rho_t,\eta_t,\overline{W}_t)\Big) \Big),  
\end{align}
where now $\Phi, \Psi$ are non-negative measurable functions on $\mathcal{M}^2_f(\mathbb{R}_+) \times \mathcal{W}_{\overline{E}}$. 
By \eqref{eq:many:to:one:N:x}, the left-hand side in the last display is given by 
\begin{equation}\label{equation:pruebaExit}
    \int_0^\infty \dd a \, \exp(- \alpha a) \, \Pi_{x,0}\otimes E^0\Big( \Phi \big( \mathcal{J}_{\varsigma_x(\xi^a)}, \widecheck{\mathcal{J}}_{\varsigma_x(\xi^a)}, (\overline{\xi}_t : 0 \leq t \leq  {\varsigma_x(\xi^a)} )  \big) \cdot \Psi \big(  \uptheta_{\varsigma_x(\xi^a)}(\mathcal{J}_{a}, \widecheck{\mathcal{J}}_{a}, \overline{\xi^a})  \big) \Big)  
\end{equation}
where to simplify notation, we write $\xi^a, \overline{\xi^a}$ for $(\xi_t : 0 \leq t \leq a)$ and $(\overline{\xi}_t : 0 \leq t \leq a)$ respectively. This expression can now be simplified by making use of excursion theory for  $\overline{\xi}$ under $\Pi_{x,0}$. Under $\Pi_{x,0}$ and   for $r \geq 0$, let $\tau_r^+$ be the right limit at $r \geq 0$ of the process $(\tau_t)_{t \geq 0}$ defined before Proposition \ref{manytoone-lebesgue-Addi}. In particular, $(\tau_t^+)_{t \geq 0}$ is a subordinator and it is the right-inverse of $\mathcal{L}$.  We write $(r_j)_{j \in \mathbb{N}}$ for the jump-times of  $(\tau_{t}^+)_{t \geq 0}$. If for each $j \in \mathbb{N}$ we let $\xi^j$ be the excursion away from $x$ of $\xi$ corresponding to the excursion interval $( \tau_{r_j}, \tau_{r_j}^+)$,  the measure $\sum_{j \in \mathbb{N}}\delta_{(r_j, \xi^j)}$ is a Poisson point measure with intensity $\mathbbm{1}_{\mathbb{R}_+}(t)\dd t \otimes \mathcal{N}$. Since $\mathcal{L}$ is constant and identically equal to $r_j$ on the excursion intervals  $(\tau_{r_j} , \tau^+_{r_j})$, we can express  \eqref{equation:pruebaExit} as follows: 
\begin{align*}
     & \Pi_{x,0}\otimes E^0\Big( \sum_{j \in \mathbb{N}} \exp\big(- \alpha \sum_{r_i < r_j } \uptau(\xi^{i}) \big)\Phi \big( \mathcal{J}_{\tau_{r_j}}, \widecheck{\mathcal{J}}_{\tau_{r_j}}, \overline{\xi^{ \, \tau_{r_j}}} \,  \big) \\  & \hspace{20mm} \cdot  \int_0^{\uptau(\xi^{j})} \dd a \, \exp \big(- \alpha a \big) \Psi \Big(  \uptheta_{\tau_{r_j}}\big(\mathcal{J}_{\tau_{r_j} +a}, \widecheck{\mathcal{J}}_{\tau_{r_j}+a}\big) , \big((\xi_t^{j}, r_j) : 0 \leq t \leq a )\big)  \Big).
\end{align*}
An application of the compensation formula entails that the previous display writes: 
\begin{align*}
      \int_0^\infty \dd r \, 
   E^0 \otimes \Pi_{x,0}  \Big(
    \exp \big(- \alpha \tau_r \big)  \Phi \big( \mathcal{J}_{\tau_{r}}, \widecheck{\mathcal{J}}_{\tau_r}, \overline{\xi^{\tau_r}}  \big)   \Big) E^0 \otimes \mathcal{N}_{x,r} \Big(  \int_0^{\uptau(\xi)} \dd a \, \exp (- \alpha a )  \Psi(  \mathcal{J}_a, \widecheck{\mathcal{J}}_a  , \overline{\xi^a}  ) \Big). 
\end{align*}
Making now use of Proposition \ref{manytoone-lebesgue-Addi} and Lemma \ref{proposition:only:spineN*}, we infer that the formula in the previous display coincides with the right-hand side of \eqref{eq:interm:exit:system}. This concludes the proof of the theorem.
\end{proof}

We conclude the section deriving a version of Theorem \ref{theorem:exit} under the probability measure $\mathbb{P}_{0,x,0}$. 

\begin{cor}\label{corollary:exitPx}
For every  non-negative measurable functions $F,G$ on $\mathbb{D}(\mathbb{R}_+, \mathcal{M}_f(\mathbb{R}_+)\times \mathcal{W}_{\overline{E}}) \times \mathbb{R}$ and $\mathbb{D}(\mathbb{R}_+, \mathcal{M}_f(\mathbb{R}_+)\times \mathcal{W}_{\overline{E}})$ we have 
\begin{equation}\label{eq:cor:exit:formula}
    \mathbb{E}_{0,x,0}\Big[   \sum_{u \in \mathcal{D}} F\big( \,  \emph{\texttt{T}}_u (\rho,\overline{W}), \mathfrak{g}(u) \big) G(\rho^u , \overline{W}^u) \Big]  = 
    \mathbb{E}_{0,x,0}\Big[ \int_0^\infty \mathrm{d} A_s  ~F(\rho,\overline{W}, s ) \mathbf{N}_{x, \widehat{\Lambda}_s}\big(G(\rho,\overline{W})\big)\Big].
\end{equation}
\end{cor}

\begin{proof}
This corollary is  a consequence of Theorem \ref{eq:theorem:exit} and standard Poissonian calculus. In this proof we argue under $\mathbb{P}_{0,x,0}$  and for simplicity we replace the function $F$ in the statement by a non-negative measurable function of the form $F(\upvarrho, \overline{\omega})f(t)$ for $(\upvarrho , \overline{\omega}, t) \in \mathbb{D}( \mathbb{R}_+,  \mathcal{M}_f(\mathbb{R}_+) \times \mathcal{W}_{\overline{E}} )\times \mathbb{R}_+$. Recall that $(0,x,0)$  is regular and  instantaneous  for $(\rho, \overline{W})$ and that the reflected running infimum $-I$ of the underlying Lévy process is a local time for $(\rho , \overline{W})$ at $(0, x,0 )$. Let $(a_i,b_i)_{i \in \mathbb{N}}$ be the connected components of the complement of $\{ t\geq 0 : (\rho_t,\overline{W}_t) = (0,x,0)   \}$ and for every $i \in \mathbb{N}$, write $(\rho^i, \overline{W}^i)$ for the subtrajectory associated with the interval $[a_i , b_i]$. Setting  $\ell_i:=-I_{a_i}$ for $i \in \mathbb{N}$, as discussed in  \eqref{PPPsobreinfimo}  the measure
\begin{equation*}
  \mathcal{N} =   \sum_{i \in \mathbb{N}} \delta_{ (\ell_i, \rho^i, \overline{W}^i )} 
\end{equation*}
is a Poisson point measure with intensity $\mathbbm{1}_{\mathbb{R}_+}(\ell) \dd \ell  \, \mathbb{N}_{x,0}(\dd \rho , \dd \overline{W})$, and it is classic that we can recover $(\rho,\overline{W})$ from  $\mathcal{N}$. Let us be more precise. Consider the process 
\begin{equation*}
    \Sigma_\ell:=\sum \limits_{i\in \mathbb{N}, \,  \ell_i \leq \ell}\sigma(W^i), \quad \text{ for }\ell \geq 0, 
\end{equation*}
 which is clearly nondecreasing  and right-continuous.  Note that if $\ell$ is a discontinuity time for $\Sigma$,  there exists a  unique $i \in \mathbb{N}$ such that $\ell_{i}=\ell$ and $\Sigma_{\ell_i} -\Sigma_{\ell_i-}=\sigma(\overline{W}^i)$. Since $\Sigma_\ell\to \infty $, as $\ell\to \infty$, for every fixed $t\geq 0$ there exists a unique $\ell \geq 0$ such that $\Sigma_{\ell-}\leq t\leq \Sigma_\ell$. Still for $t$ and  $\ell$ as before,  if there exists $i\in \mathbb{N}$ such that $\ell_i=\ell $ we can write: 
\begin{equation*}
    (\rho_t,\overline{W}_t)= \big(\rho^i_{t-\Sigma_{\ell_i-}}, \overline{W}^i_{t-\Sigma_{\ell_i-}}\big), \quad \quad A_t(\rho,\overline{W})= \sum_{\ell_j< \ell_i} A_{\sigma(\overline{W}^j)}(\rho^j,\overline{W}^j)+ A_{t-\Sigma_{\ell_i-}}(\rho^i,\overline{W}^i),
\end{equation*}
and otherwise, we have $(\rho_t,\overline{W}_t)=(0,x,0)$ as well as  $A_t= \sum_{\ell_j< -I_t} A_{\sigma(\overline{W}^i)}(\rho^j,\overline{W}^j)$. We stress that we used the fact that the measure $\dd A$ does not charge the set $\{ t \geq 0 : (\rho , \overline{W}) = (0,x,0) \}$ by \eqref{equation:aditiva1}   as well as its additive property.  
Let us now infer from these facts  identity \eqref{eq:cor:exit:formula}. First, observe that by \eqref{equation:aditiva1} we can rewrite the right-hand side of \eqref{eq:cor:exit:formula} in the following  form:
\begin{equation}\label{equation:displayExitPx}
\mathbb{E}_{0,x,0}\Big[\sum_{i\in \mathbb{N}} \int_0^{\sigma(W^i)} \mathrm{d} A^i_s  ~F(\rho,\overline{W}) f(  \Sigma_{\ell_i-}+s ) \cdot \mathbf{N}_{x, \widehat{\Lambda}^i_s}\big(G(\rho,\overline{W})\big)\Big],   
\end{equation}
where $A^i$ stands for the process $A(\rho^i, \overline{W}^i)$. Next, remark that for any $i \in \mathbb{N}$, we can recover $(\rho,\overline{W})$ by   concatenating consecutively the three following paths  
\begin{equation}\label{equation:tripletMecke}
    \Big( \big(\rho,\overline{W}\big)_{\cdot\wedge \Sigma_{\ell_i-}}, (\rho^i,\overline{W}^i),  (\rho,\overline{W})_{ \, \cdot +\Sigma_{\ell_i}} \Big) 
\end{equation}
where at each step we concatenate at the first time at which the initial process remains identically equal to $(0,x,0)$. This operations does not depend on $i \in \mathbb{N}$ and for simplicity we denote it  by $\mathfrak{G}$. With this notation in hand,  we get that \eqref{equation:displayExitPx}  equals
\begin{equation*}
\mathbb{E}_{0,x,0}\Big[\sum_{i\in \mathbb{N}} \int_0^{\sigma(W^i)} \mathrm{d} A_s^i  ~F \circ  \mathfrak{G}\big( (\rho,\overline{W}\big)_{\cdot\wedge \Sigma_{\ell_i-}}, (\rho^i,\overline{W}^i), (\rho,\overline{W})_{\cdot +\Sigma_{\ell_i}} \big) f(\Sigma_{\ell_i-}+s) \cdot  \mathbf{N}_{x, \widehat{\Lambda}^i_s}\big(G(\rho,\overline{W})\big)\Big].
\end{equation*}
For each  $i \in \mathbb{N}$, the first and third processes in \eqref{equation:tripletMecke}  can be expressed in terms of the point measure $ \mathcal{N} \setminus (\ell_i, \rho^i, \overline{W}^i)$ and the atom $(\ell_i, \rho^i, \overline{W}^i)$. More precisely, they can  be  recovered respectively from the pair of point measures $\mathbbm{1}_{(0,\ell_i)}(\ell)\mathcal{N}(\mathrm{d }\ell,\mathrm{d }\upvarrho, \mathrm{d }\omega)$, and  $\mathbbm{1}_{(\ell_i, \infty)}(\ell)\mathcal{N}(\mathrm{d }\ell,\mathrm{d }\upvarrho, \mathrm{d }\omega)$  shifted by $-\ell_i$ on its first  coordinate,  by making use of the construction recalled above. Thus,  Mecke's formula and Theorem \ref{theorem:exit} yield that the previous display is given by: \medskip \\ 
\begin{equation}\label{eq:fin:cor:exit}
     \mathbb{E}_{0,x,0}\Big[\sum_{i\in \mathbb{N}} \sum \limits_{u\in \mathcal{D}(\rho^i,\overline{W}^i)}~F \circ  \mathfrak{G}\big( (\rho,\overline{W})_{\cdot\wedge \Sigma_{\ell_i-}}, \texttt{T}_u(\rho^i,\overline{W}^i), (\rho,\overline{W})_{\cdot +\Sigma_{\ell_i}} \big) f\big(\Sigma_{\ell_i-}+\mathfrak{g}_i(u)\big) \cdot  G(\rho^{i,u},\overline{W}^{i,u})\Big],
  \end{equation}
 where for every $u\in \mathcal{D}(\rho^i,\overline{W}^i)$,   $\texttt{T}_u(\rho^i,\overline{W}^i)$ and $(\rho^{i,u},\overline{W}^{i,u})$ stand respectively for  the subtrajectory stemming from $u$ and the excursion of $(\rho^i,\overline{W}^i)$ corresponding with such debut, while  
 $\mathfrak{g}_i(u)$ is the first time at which $(p_{H(\rho^i)}(t): t \geq 0)$ visits the debut $u$. Finally, notice that for every element $u$ of 
$\cup_{i\geq 1} \mathcal{D}(\rho^i,\overline{W}^i)$ we can  canonically associate a unique debut $u'$ of $(\rho,\overline{W})$ in such a way that
$$
\mathfrak{g}(u')=\Sigma_{\ell_i-}+\mathfrak{g}_i(u),
 \quad \text{ and }
 \quad (\rho^{u'},\overline{W}^{u'})=(\rho^{i,u},\overline{W}^{i,u}).
$$
The correspondence $u \leftrightarrow u'$ is clearly one-to-one, and for any pair $u,u'$  in correspondence through this bijection we have 
\begin{equation*}
    \texttt{T}_{u'}(\rho,\overline{W})
=
 \mathfrak{G}\big( (\rho,\overline{W})_{\cdot\wedge \Sigma_{\ell_i-}}, \texttt{T}_u(\rho^i,\overline{W}^i), (\rho,\overline{W})_{\cdot +\Sigma_{\ell_i}} \big).  
\end{equation*}
We infer from these identities that \eqref{eq:fin:cor:exit} and the left-hand side of \eqref{eq:cor:exit:formula} coincide, 
which gives the desired result.
\end{proof}
 
\subsection{The excursion process}\label{section:Poisson}

For every element $(\upvarrho,\overline{\omega}):=(\upvarrho,\omega, \lambdaC)$ in $\mathbb{D}(\mathbb{R}_+, \mathcal{M}_f(\mathbb{R}_+) \times \mathcal{W}_{\overline{E}})$, recall the notation $\mathcal{D}(\upvarrho , \omega)$ and $(\upvarrho^{u,*}, \omega^{u,*})$ for $u \in \mathcal{D}(\upvarrho , \omega)$ for respectively the set of debuts of $(\upvarrho , \omega)$ and the family of excursion away from $x$  in the sense of Section \ref{sub:sect:debut}. For every such element of $\mathbb{D}(\mathbb{R}_+, \mathcal{M}_f(\mathbb{R}_+) \times \mathcal{W}_{\overline{E}})$ we define a point measure in $\mathbb{R}_+ \times \mathbb{D}(\mathbb{R}_+, \mathcal{M}_f(\mathbb{R}_+)\times \mathcal{W}_E)$ by setting 
\begin{equation*}
    \mathcal{E}(\upvarrho, \overline{\omega}):= \sum_{u \in \mathcal{D}(\upvarrho,\omega)} \delta_{(A_{\mathfrak{g}(u)}(\upvarrho,\overline{\omega}), (\upvarrho^{u,*},\omega^{u,*}))}. 
\end{equation*}
 Recall that by convention, if $(\upvarrho, \omega)$ is not in $\mathcal{S}_x$, the set $\mathcal{D}(\upvarrho, \omega)$ is  empty and therefore  $\mathcal{E}(\upvarrho , \overline{\omega})$ is the null measure. In analogy to the nomenclature used in classical excursion theory, the measure in the last display will be referred to as the excursion  process of $(\upvarrho , \overline{\omega})$. When working under $\mathbb{P}_{0,x,0}$ and when there is no risk of confusion, we simply write $\mathcal{E}$ for $\mathcal{E}(\rho, \overline{W})$.  The  section is devoted to proving the following theorem. 
\begin{theo}\label{theorem:excursionPPP}
 Under $\mathbb{P}_{0 , x ,0 }$, the measure $\mathcal{E}$  is a Poisson measure on $\mathbb{R}_+ \times  \mathbb{D}(\mathbb{R}_+, \mathcal{M}_f(\mathbb{R}_+) \times \mathcal{W}_E)$ with intensity $\mathbbm{1}_{\mathbb{R}_+}(t)\dd t \otimes \mathbf{N}_x^*(\dd \upvarrho, \dd \omega)$.  
\end{theo}
\begin{proof}
Since $\mathbf{N}_x^*$ is a sigma-finite measure and $\mathbb{D}(\mathbb{R}_+, \mathcal{M}_f(\mathbb{R}_+) \times \mathcal{W}_E)$ is a Polish space  that we  equip with its  Borel sigma-field $\mathcal{B}(\mathbb{D})$, by standard characterisation results on Poisson point measures it suffices to show that:
\begin{itemize}
\item[(i)] For every $\mathcal{U}\in \mathcal{B}(\mathbb{D})$ such that $\mathbf{N}^*_x(\mathcal{U})<\infty$, we have 
\begin{equation*}
    \mathbb{E}_{0,x,0}\big[ \mathcal{E}\big( [0,1] \times \mathcal{U}   \big)\big]= \mathbf{N}^*_x(\mathcal{U}). 
\end{equation*}
\item[(ii)]  $\mathbb{P}_{0,x,0}$-a.s.,   for every $\mathcal{U}\in \mathcal{B}(\mathbb{D})$ and $t \geq 0$, it holds that $\mathcal{E}\big( \{t\} \times  \mathcal{U} \big)\in \{0,1\}$  and  $\mathcal{E}\big( \{0\} \times   \mathcal{U}  \big)=0$.
\item[(iii)] For every $n\geq 1$ and  $\mathcal{U}_1,\dots , \mathcal{U}_n\in \mathcal{B}(\mathbb{D})$ satisfying 
$\mathbf{N}_x^*( \mathcal{U}_j  )<\infty$  for every $1\leq j\leq n$, the process
\begin{equation*}
    \big(\mathcal{E}\big( [0,t] \times \mathcal{U}_1   \big),\dots, \mathcal{E}\big( [0,t] \times \mathcal{U}_n   \big) \big), \quad  \text{ for } t\geq 0, 
\end{equation*}
has stationary and independent increments.
\end{itemize}
Let us start by deducing point (i) from Corollary \ref{corollary:exitPx}.  First, note that for each $u \in \mathcal{D}$, the  process $\texttt{T}_u(\rho, \overline{W})$ restricted to the interval $[0,\mathfrak{g}(u)]$ is just $\big( (\rho_t , \overline{W}_t), 0 \leq t \leq \mathfrak{g}(u)\big)$. Hence, using the approximation \eqref{equation:aproximacionA}  we infer that $A_{\mathfrak{g(u)}}$ is a measurable function of the pair $\texttt{T}_u(\rho , \overline{W})$, $\mathfrak{g}(u)$. It now readily follows from this fact and Corollary \ref{eq:cor:exit:formula}  that 
\begin{equation*}
    \mathbb{E}_{0,x,0}\Big[   \sum_{u \in \mathcal{D}} f( A_{\mathfrak{g}(u)}) \cdot G\circ \mathrm{tr}( \rho^u , {W}^u )  \Big]  = 
    \mathbb{E}_{0,x,0}\Big[ \int_0^\infty \mathrm{d}s  ~ f(s) \cdot  \mathbf{N}_{x, \widehat{\Lambda}_{A^{-1}_s}}\big(G \circ \mathrm{tr}( \rho , {W} )\big)\Big],
\end{equation*}
for every non-negative measurable functions $f,G$ on $\mathbb{R}$ and $\mathbb{D}(\mathbb{R}_+, \mathcal{M}_f(\mathbb{R}_+) \times \mathcal{W}_E)$ respectively.  Since the distribution of $\mathrm{tr}(\rho,W)$ under $\mathbf{N}_{x,r}$ for any $r \geq 0$ is precisely $\mathbf{N}^*_x$,  point (i) follows.    
\par 
The first claim of point (ii) is an immediate consequence of Corollary \ref{cor:A:distinct}, which states that under  $\mathbb{P}_{0,x,0}$,   the  variables $A_{\mathfrak{g}(u)}$, for $u\in \mathcal{D}$,  are distinct. Further, the measure $\mathcal{E}$ can not possess an atom having $0$ as first coordinate since under $\mathbb{P}_{0,x,0}$, the point $0$ belongs to the support of $\dd A$ \cite[Lemma 4.15]{2024structure}, and it is easy to check that $p_{H(\rho)}( 0 )$ is  not a debut.  This proves the second claim in point (ii). 
\par 
It remains to establish point (iii) and to do so, we  rephrase the problem in a slightly different form.  Under $\mathbb{P}_{0,x,0}$ and for an arbitrary $r \geq 0$ that we fix for the rest of the proof,     we consider the following pair of measures
\begin{equation*}
 \mathbbm{1}_{(0,r]}\mathcal{E}
    :=  \sum_{u \in \mathcal{D} , A_{\mathfrak{g}(u)} \leq r } \delta_{(A_{\mathfrak{g}(u)} , (\rho^{u,*} , W^{u,*}))} \quad  \quad \text{and }\quad   \quad \mathcal{E}^{\prime } 
    :=
     \sum_{u \in \mathcal{D} , A_{\mathfrak{g}(u)} >  r } \delta_{(A_{\mathfrak{g}(u)} -  r, (\rho^{u,*} , W^{u,*}))}.  
\end{equation*}
Noting that for any $\mathcal{U} \in \mathcal{B}(\mathbb{D})$, with $\mathbf{N}_x^*( \mathcal{U}) < \infty$, we can write 
\begin{equation*}
    \mathcal{E}\big( (0,r+h] \times \mathcal{U}   \big)  - \mathcal{E}\big( (0,r] \times \mathcal{U}   \big)
    =  \mathcal{E}\big( (r,r+h] \times \mathcal{U}   \big)
    =  \mathcal{E}' \big( [0,h]\times \mathcal{U}\big),  \quad \text{ for }h \geq 0,  
\end{equation*}
it is clear that to obtain point (iii) it suffices to establish that:
\begin{itemize}
    \item[(iii')] The pair of measures $\mathbbm{1}_{(0,r]}\mathcal{E}$ and $\mathcal{E}^{\prime}$ 
are independent, and    $\mathcal{E}^{\prime}$ is distributed as $\mathcal{E}$ under $\mathbb{P}_{0,x,0}$. 
\end{itemize}
In short, the idea is to encode both measures $\mathbbm{1}_{(0,r]}\mathcal{E}$ and $ \mathcal{E}'$ in terms of pieces of path of $(\rho , \overline{W})$ that are independent and of more tractable nature. 
 \par 
Starting with the measure $\mathbbm{1}_{(0,r]}\mathcal{E}$,  we consider  $(\mathscr{I}_s:~s \geq A_{r}^{-1})$, the running infimum of 
\begin{equation}\label{equation:shiftedLP}
    \langle \rho_{s} , 1 \rangle- \langle \rho_{A_r^{-1}} , 1 \rangle, \quad \text{ for } s \geq A_{r}^{-1},  
\end{equation}
and we set $T := \inf\{ t \geq A^{-1}_{r} : \langle \rho_s ,1 \rangle = 0\}$. Denote the excursion intervals of \eqref{equation:shiftedLP} over its running infimum  occurring before time $T$  by $(a_i, b_i)_{i \in \mathcal{I}}$ and write $(\rho^i, \overline{W}^i)_{i \in \mathcal{I}}$ for the the corresponding subtrajectories. Now, we claim that  $\mathbbm{1}_{(0,r]}\mathcal{E}$ can be recovered from the following pair  
\begin{equation}\label{eq:Poisson:proof:1}
\big( (\rho_t,\overline{W}_t), \, 0 \leq t\leq A_r^{-1}\big) \quad  \quad \text{ and }\quad  \quad  \sum_{i \in \mathcal{I}}\delta_{(-\mathscr{I}_{a_i},\mathrm{tr}(\rho^{i}, \overline{W}^i)) }, 
\end{equation}
where the truncation operator \eqref{definition:troncature*}  is defined on elements of $\mathbb{D}(\mathbb{R}_+, \mathcal{M}_f(\mathbb{R}_+) \times \mathcal{W}_{\overline{E}})$ in an obvious way.  This essentially follows from our definitions but let us sketch a short argument. Fix an arbitrary atom  $(A_{\mathfrak{g}(u)} , \rho^{u,*} , W^{u,*} )$ of $\mathbbm{1}_{(0,r]}\mathcal{E}$ and note that the condition  $A_{\mathfrak{g}(u)} \leq r$ ensures that  $\mathfrak{g}(u) < A^{-1}_{r}$,  the inequality being strict since $A^{-1}_r$  can not coincide with $\mathfrak{g}(u)$ a.s. (indeed, noting that $\rho_{A^{-1}_r}(\{H(\rho_{A^{-1}_r})\}) = 0$ by Lemma \ref{lemma:tequedasenThetax},  the fact that the point $0$ is regular and instantaneous for   the underlying Lévy process and the strong Markov property  of the exploration process gives  that a.s. the point $A^{-1}_r$ is not a point of right-increase for $H(\rho)$). If it further holds that $\mathfrak{d}(u) \leq A^{-1}_{r}$, it is plain that the triple $(A_{\mathfrak{g}(u)} , \rho^{u,*} , W^{u,*} )$ only depends on  $((\rho_t , \overline{W}_t) : 0 \leq t \leq A^{-1}_{r})$. On the other hand, if $\mathfrak{g}(u) < A^{-1}_{r} < \mathfrak{d}(u)$,  the excursion $(\rho^{u,*}, W^{u,*})$   can be recovered from the piece of path  $( (\rho_t , W_t) : 0 \leq  t \leq A^{-1}_{r} )$ and the atoms of  the point measure in \eqref{eq:Poisson:proof:1} satisfying $\Lambda^{i}_0 = \widehat{\Lambda}_{\mathfrak{g}(u)}$,  by making use of the snake property and the definition of subtrajectory. Here, we use the fact that the value of $\widehat{\Lambda}$ on each excursion component is a.s. distinct by Lemma \ref{lemma:contanteExcursion}. We leave the details to the reader.  
\par  
The previous discussion yields that proving the independence between $\mathbbm{1}_{(0,r]}\mathcal{E}$ and $ \mathcal{E}'$ can be reduced to showing that $\mathcal{E}^{\prime}$ is independent from  the pair  \eqref{eq:Poisson:proof:1}. 
To this end, we next show that $\mathcal{E}'$ is the excursion process of a Lévy snake distributed  $\mathbb{P}_{0,x,0}$ which is independent from the pair \eqref{eq:Poisson:proof:1}. In this direction  and with the same notations as before, we set 
\begin{equation} \label{equation:exitlocalx}
    \mathfrak{L}_t:=\sum\limits_{i\in \mathcal{I}} {L}_{t\wedge b_i-t\wedge a_i}(\rho^{i}, {W}^i),\quad t\geq 0.   
\end{equation}
The process $(\mathfrak{L}_t)_{t \geq 0}$ can be used to index the family of subtrajectories  escaping $\overline{E}\setminus ( \{x\}\times \mathbb{R}_+)$ of the collection $(\rho^{i}, W^{i},\Lambda^{i}-\Lambda^i_0)$, for $i\in \mathcal{I}$. Specifically, for each  $i \in \mathcal{I}$,  denote  the connected components of the open set $\big\{t\geq 0:\:\uptau(W_{t}^i)< \zeta_t(W^i) \big\} $ by  $\big((s_{k},t_{k}) : k\in \mathcal{K}_i \big)$, 
where  $\mathcal{K}_i $ is an indexing set that might be empty, and  for $k \in \mathcal{K}_i$   let $(\rho^{i,k},W^{i,k},\Lambda^{i,k})$ be the subtrajectory of $(\rho^i,W^i,\Lambda^{i})$ associated with the interval $[s_{k},t_{k}]$ in the sense of Section \ref{section:snake}, translated by $-(0,0, \Lambda^i_0)$. In particular, every such  subtrajectory is an element of $\overline{\mathcal{S}}_x$ with $\Lambda^{i,k}_0 = 0$. Note that in the time scale of $((\rho_{t},\overline{W}_{t}):t\geq 0)$, the  triple   $(\rho^{i,k}, W^{i,k},\Lambda^{i,k}+\Lambda_0^i)$ is the subtrajectory associated with  the interval $[a_{i,k},b_{i,k}]$,
where $a_{i,k}:=a_i+s_{k}$ and $b_{i,k}:=b_i + t_{k}$. The strong Markov property at time $A^{-1}_r$ followed by the   special Markov property of Lévy snakes established in \cite[Theorem 3.7]{2024structure} entails that, conditionally on $\mathfrak{L}_\infty$, the point measure 
\begin{equation}
    \label{definition:medidaordenadatilde2_2}
\sum\limits_{i\in \mathcal{I}, k\in \mathcal{K}_i} \delta_{(\mathfrak{L}_{s_{k}},\rho^{i,k},\overline{W}^{i,k})}
\end{equation}
is a Poisson point measure with intensity $\mathbbm{1}_{[0,\mathfrak{L}_\infty]}(\ell)\dd \ell ~\mathbb{N}_{x,0}(\dd\rho,\dd\overline{W})$ independent of  the pair  \eqref{eq:Poisson:proof:1}. Observe that  $\mathfrak{L}_{\infty}<\infty$ a.s.  since otherwise, $T = \inf \{t \geq A^{-1}_{r} : \langle \rho_t , 1 \rangle = 0  \}$ would be infinite with positive probability.   Further, by  the strong Markov property, the process $((\rho_{T+t}, W_{T+t}) : t \geq 0)$ is distributed $\mathbb{P}_{0,x,0}$ and is independent from $\overline{\mathcal{F}}_T$; and   in particular, it is independent of \eqref{eq:Poisson:proof:1} and from the measure in the previous display. If we write  $(c_j,d_j)_{j \in \mathcal{J}}$  for the excursion intervals away from $0$ of $( \rho_{t},  \,  t \geq T)$, we let $(\rho^j,\overline{W}^j)_{j \in \mathcal{J}}$ be  the associated subtrajectories and $(\ell_j)_{j\in \mathcal{J}}$  the corresponding value of the local time of $( \rho_{T+t},  \,  t \geq 0)$ at $0$  at the start of each  excursion,   it follows from our previous discussion   that    the measure 
\begin{equation}
    \label{definition:medidaordenadatilde2_2_2}
\mathcal{M}' = \sum\limits_{i\in \mathcal{I}, k\in \mathcal{K}_i} \delta_{(\mathfrak{L}_{s_{k}},\rho^{i,k},\overline{W}^{i,k})}+ \sum_{j \in \mathcal{J}} \delta_{( \mathfrak{L}_\infty +\ell_j , \rho^j , \overline{W}^j )} 
\end{equation}
is a Poisson point measure,  independent of \eqref{eq:Poisson:proof:1}, and with intensity $\mathbbm{1}_{\mathbb{R}_+}(\ell) \dd \ell \otimes \mathbb{N}_{x,0}$. Since $\mathbb{N}_{x,0}$ is precisely the excursion measure away from $(0,x,0)$ of $(\rho , \overline{W})$, by the same classic arguments detailed in the proof of Corollary \ref{corollary:exitPx},  we can concatenate the atoms of  $\mathcal{M}'$ according to the order induced by $(\mathfrak{L}_t)_{t \geq 0}$    to obtain a process $(\rho^\prime,\overline{W}^\prime)$ independent from \eqref{eq:Poisson:proof:1} and  distributed $\mathbb{P}_{0,x,0}$. 
\par 
Let us now  argue that $\mathcal{E}'$ is precisely the excursion point  process of $(\rho', \overline{W}')$. This would show both, that $\mathcal{E}'$  is  distributed as $\mathcal{E}$ under $\mathbb{P}_{0,x,0}$, and by our previous discussion that it is independent from the pair \eqref{eq:Poisson:proof:1}. Note that this will  conclude the proof of point (iii').  Let  $\mathcal{D}^{ \prime}  := \mathcal{D}(\rho', W')$ be the collection of debut times of $(\rho',\overline{W}')$,   and write  $(\rho'^{u'}, W'^{u'})_{u' \in \mathcal{D}^{\prime}}$ for the corresponding family of excursions. We set 
\begin{equation*}
    \mathcal{E}(\rho', \overline{W}')  = 
        \sum_{u' \in \mathcal{D}^{ \prime}} \delta_{(A_{g(u')}(\rho', W'), \rho'^{u'} , W'^{u'})}
    \end{equation*}
for the associated excursion process. Since any debut such that $\mathfrak{g}(u) > A^{-1}_r$ is necessarily an excursion away from $x$ for either, a subtrajectory $(\rho^{i,k}, W^{i,k})$ for some $i \in \mathcal{I}$, $k \in \mathcal{K}_i$, or for some $(\rho^j, W^j)$ for some $j \in \mathcal{J}$,   from our construction of the process $(\rho', \overline{W}')$ it follows that there exists  a bijection between $\{ u \in \mathcal{D} : \mathfrak{g}(u) > A^{-1}_{r} \}$ and $\mathcal{D}^{\prime}$. More precisely,   for every debut point $u$ in the set  $\{u \in \mathcal{D}  : \mathfrak{g}(u) > A^{-1}_{r} \}$, we can find $u' \in \mathcal{D}^{\prime}$ with $W^{u} = W'^{u'}$. Moreover, as we shall now briefly sketch, for such a pair  $u, u'$ it holds that  
\begin{equation}\label{equation:Aditivaigual}
    A_{\mathfrak{g}(u)}(\rho , \overline{W}) - r = A_{\mathfrak{g}(u')}(\rho', \overline{W}'),  
\end{equation}
proving that $\mathcal{E}'$ is precisely $\mathcal{E}(\rho', \overline{W}')$. 
We stress however that the processes  $(A_{A^{-1}_{r} + t }-r: t \geq 0)$ and $A(\rho',\overline{W}')$ do differ. Let us briefly explain why the identity in the last display holds. On one hand, by  equations (4.29)  and    equation  (4.30) in \cite{2024structure} we can write 
\begin{equation}\label{equation:decompA}
    A_t(\rho , \overline{W} ) -r =  \sum\limits_{i\in \mathcal{I}, k \in \mathcal{K}_i} A_{t\wedge b_{i,k}-t\wedge a_{i,k}}(\rho^{i,k}, {W}^{i,k}) + \sum\limits_{j\in \mathcal{J}} A_{t\wedge d_j-t\wedge c_j}(\rho^{i}, {W}^i) ,\quad \text{ for } t\geq A^{-1}_r.
\end{equation}
The process $A(\rho', \overline{W}')$ can be decomposed as well in a similar form. More precisely, since by constriction  the collection $(\rho^{i,k},\overline{W}^{i,k})_{i \in \mathcal{I}, k \in \mathcal{K}_i}$, $(\rho^j, \overline{W}^{j})_{j \in \mathcal{J}}$ are precisely the excursions away from $(0,x,0)$ of $(\rho', \overline{W}')$, by making use of  \eqref{equation:aditiva1}    we can as well write an analogous decomposition for $A(\rho', \overline{W}')$ in terms of this family of subtrajectories. A direct comparison between  the latter and \eqref{equation:decompA} readily 
 yields identity \eqref{equation:Aditivaigual}.  
\end{proof}

\section{\texorpdfstring{The exit local time under $\mathbf{N}_{x,r}$ and $\mathbf{N}_{x}^*$}{} }\label{section:L}

We work with an arbitrary fixed $r\geq 0$. In this section we aim to develop a theory analogous to that of exit local times under the measures   $\mathbf{N}_{x,r}$ and $\mathbf{N}_{x}^*$. We start in Section \ref{sect:unif:conv:special:N:black} by studying the process $L$ introduced in Section \ref{section:truncationboundary} under $\mathbf{N}_{x,r}$,  and we establish a special Markov property,  similar to the one proved in Theorem 3.7 of \cite{2024structure}. In Section \ref{sub:inside:approx}, we  extend this theory to the excursion measure $\mathbf{N}_x^*$. Finally, in Section \ref{sec:Spinal:L} we derive a spinal decomposition under $\mathbf{N}_{x,r}$ with respect to a point sampled uniformly according to the measure $\dd L$. As an application, we identify the distribution of $L_\sigma$ under $\mathbf{N}_{x,r}$. 
\par
Let us start by introducing some notation and  by recalling standard results on  Lévy snakes  that  will be used frequently in our reasoning. In this direction,  remark that for every  $\mu\in \mathcal{M}_f(\mathbb{R}_+)$,  by property (iii)  of Section \ref{subsection:explorationprocess} and by the definition of $\mathbf{P}_{\mu}$ given in  \eqref{equation:LevyPmu},  we have:

\begin{equation}\label{Lebesgue:no:quiere:verlo}
    \text{Leb}\big( \{s \geq 0 : H(\rho_s)=\inf_{[0,s]} H(\rho)\}\big) = 0,\quad \mathbf{P}_{\mu}-\text{a.s.}
\end{equation}
 Next, fix an arbitrary element $(\mu , \overline{\w})$ of $\overline{\Theta}_x$. Under $\mathbb{P}_{\mu , \overline{\w}}^\dag$,   write  $(a_i,b_i)_{i\in \mathcal{I}}$ for the connected components of $\{s \geq 0 : H(\rho_s)-\inf_{[0,s]} H(\rho) > 0\}$ and for every $i\in \mathcal{I}$,  let $(\rho^i,\overline{W}^i):=(\rho^i,W^i,\Lambda^i)$ be the  subtrajectory associated with the interval $[a_i, b_i]$. It was shown in  \cite[Lemma 4.2.4]{DLG02} that the measure 
\begin{equation}\label{equation:PoissonH}
    \sum \limits_{i\in \mathcal{I}} \delta_{(H(\rho_{a_i}), \rho^i,\overline{W}^i)} 
\end{equation}
is a Poisson point measure with intensity $\mu(\dd h) \mathbb{N}_{\overline{\w}(h)}(\dd \rho , \dd \overline{W})$. By elementary properties of Poisson point measures,   the condition $(\mu , \overline{\w}) \in \overline{\Theta}_x$ ensures that $\mathbb{P}_{\mu , \overline{\w}}^\dag$-a.s. we have  $W_0^{i} \neq x$, for every $i \in \mathcal{I}$. 
 Finally, for latter use, we note that equation  (4.11) in \cite{DLG02} and   Lemma \ref{lem:equa:convL} entail that,  under $\mathbb{N}_{y}$ for $y\neq x$, the process  $L$ is the exit local time of $(\rho,W)$ from the domain $E\setminus\{x\}$. In particular, the first moment formula  \cite[Proposition 4.3.2]{DLG02} gives: 
\begin{equation}\label{equation:mediaexit}
    \mathbb{N}_{y}(L_\sigma) = \Pi_y\big(\exp(-\alpha \uptau) \big) , \quad \text{ for all } y \in E \setminus \{x \}.
\end{equation}

\subsection{The special Markov property}\label{sect:unif:conv:special:N:black}
Recall 
 the notation  $(\varepsilon_k)_{k\geq 0}$ for the sequence used in the definition \eqref{definition:exitMultiusos} of the functional $L$. The goal  of the section is to prove the following result.
\begin{prop} \label{proposition:L*} Under $\mathbf{N}_{x,r}$, the following convergence  holds     \begin{equation} \label{equation:aproximacionL*Pmu}
        \lim_{k\to\infty} \sup_{t\geq 0}\Big|L_{t}(\rho,W)- \frac{1}{\epsilon_k} \int_0^t \dd s \,  \mathbbm{1}_{\{ \uptau(W_s) <H(\rho_s) < \uptau(W_s) + \epsilon_k \}} \Big|=0, \quad \text{a.e.} 
    \end{equation}
    Furthermore, if under $\mathbf{N}_{x,r}$  we write $(s_i,t_i)_{i\in \mathcal{I}}$ for the connected components of $\{H(\rho_s)>\uptau(W_s):~s\geq 0\}$ and we let $(\rho^i,\overline{W}^i)_{i \in \mathcal{I}}$ be the associated subtrajectories, conditionally on $L_\sigma(\rho,W)$,  the point measure 
         \begin{equation}\label{equation:SpecialNblack}
             \sum_{i \in \mathcal{I}}\delta_{(L_{s_i}(\rho,W), \rho^i, \overline{W}^i )}
         \end{equation}
        is a Poisson point measure with intensity $\mathbbm{1}_{[0,L_\sigma(\rho,W)]}(\ell) \dd \ell \,  \mathbb{N}_{x,0}(\dd \rho , \dd \overline{W})$, independent of $\mathrm{tr}(\rho,W)$. 
\end{prop}
The first statement of the proposition is an extension of Lemma \ref{lem:equa:convL} under $\mathbf{N}_{x,r}$. We refer to the second statement as the special Markov property under $\mathbf{N}_{x,r}$, by analogy with the special Markov property of Lévy snakes \cite[Theorem 3.7]{2024structure}. In Section \ref{sub:inside:approx}, we will strengthen this result by showing that $L_\sigma(\rho,W)$ is a measurable function of $\mathrm{tr}(\rho,W)$; its proof is postponed to Section \ref{sub:inside:approx} since it requires additional  estimates. 
\par Note that Theorem \ref{theorem:exit} combined with Lemma \ref{lemma:tequedasenThetax} ensures that under $\mathbf{N}_{x,r}$, the process $(\rho,\overline{W})$ takes values in $\overline{\Theta}_x$. Proposition \ref{proposition:L*}  will easily follow by combining this remark with the following technical lemma.
\begin{lem}\label{lem:L:under:P:mu:w}
Let $(\mu,\overline{\w}):=(\mu,\w,\lambdaC)\in \overline{\Theta}_x$. Under $\mathbb{P}^\dag_{\mu , \overline{\w}}$, the convergence \eqref{equation:aproximacionL*Pmu} holds. Moreover, if under $\mathbb{P}^\dag_{\mu , \overline{\w}}$ we consider $(a_i,b_i)_{i\in \mathcal{I}}$ the connected components of $\{s \geq 0 : H(\rho_s)- \inf_{[0,s]}H(\rho)>0\}$ and we write $(\rho^i, W^i,\Lambda^i)_{i\in \mathcal{I}}$ for the associated subtrajectories, then a.s. we have: 
\begin{equation}\label{equation:Lpmu}
    L_t(\rho,W)=\sum \limits_{i\in \mathcal{I}} \mathbbm{1}_{ H(\rho_{a_i})<\uptau(\w)} L_{t\wedge b_i-t\wedge a_i}(\rho^i,W^i),\quad t\geq 0.
\end{equation}
\end{lem}

\begin{proof}
Fix an initial condition $(\mu , \overline{\w}):=(\mu,\w,\lambdaC)$ in $\overline{\Theta}_x$ and let  $T = \inf\{ t \geq 0 : H(\rho_t) < \uptau(\w) \}$, with the usual convention $\inf \emptyset=\infty$. The statement of the lemma will immediately follow from establishing that, under $\mathbb{P}^{\dag}_{\mu, \overline{\w}}$, the two following convergences hold  a.s.
\begin{align*}
    &\mathrm{(i)} \,  \lim_{k\to \infty }\epsilon_k^{-1} \int_0^T \dd s \,  \mathbbm{1}_{\{ \uptau(W_s)  < H(\rho_s)  < \uptau(W_s) + \epsilon_k  \}}   = 0;  \\
    &\mathrm{(ii)}   \lim_{k\to \infty } \sup_{t \geq 0}\Big| \epsilon_k^{-1}  \int_{T}^{t \wedge T} \dd s \,  \mathbbm{1}_{\{ \uptau(W_s)  < H(\rho_s)  < \uptau(W_s) + \epsilon_k  \}} -  \sum \limits_{i\in \mathcal{I}} \mathbbm{1}_{ H(\rho_{a_i})< \uptau(\w)} L_{t\wedge b_i-t\wedge a_i}(\rho^i,W^i)\Big| = 0.
\end{align*}
We start with some simplifications, and in this direction the following fact will be used  repeatedly in our arguments.  If $\sum_{j\in \mathcal{J}}\delta_{\rho^j}$ is a Poisson point measure with intensity $C\cdot N(\dd \rho)$ for some constant $C \geq 0$, then  
\begin{equation}\label{eq:C:poisson:H}
\lim \limits_{k\to \infty} \varepsilon_k^{-1} \sum_{j\in \mathcal{J}} \int_{0}^\infty \dd t~ \mathbbm{1}_{\{ 0<H(\rho^j_t)<\varepsilon_k \}}=C,\quad \text{a.s.}
\end{equation}
This convergence is an simple consequence of \eqref{temps:local:I:p:s} and \eqref{PPPsobreinfimo}.  
In what follows,  we argue under  $\mathbb{P}^\dag_{\mu , \overline{\w}}$ 
and we start with some simplifications. First  remark that, outside of a negligible set, by \eqref{Lebesgue:no:quiere:verlo} we can write 
\begin{equation*}
    \int_{0}^{t } \dd s \,  \mathbbm{1}_{\{ \uptau(W_s) < H(\rho_s) < \uptau(W_s) + \epsilon_k \}}
    =
    \sum_{i\in \mathcal{I}} \int_{a_i\wedge t}^{b_i\wedge t} \dd s \, \mathbbm{1}_{\{ \uptau(W_s) < H(\rho_s) < \uptau(W_s) + \epsilon_k \}}, \quad \text{ for } t \geq 0, 
\end{equation*}
and notice that $\mathcal{I}^{\prime}:= \{i\in \mathcal{I}:~H(\rho_{a_i}) < \uptau(\w)\}$   coincides with $\{i\in \mathcal{I}:~a_i > T\}$. For any $i \in \mathcal{I}$ and $s \in [a_i,b_i]$, it holds that $H(\rho_s)=H(\rho_{a_i})+ H(\rho^i_{s-a_i})$ and since the process  $(\rho,\overline{W})$ is a.s. in  $\overline{\mathcal{S}}$, we have  as well  
$$
\uptau(W_s)=\uptau(\w), \text{ if } i \in\mathcal{I}\setminus \mathcal{I}^\prime, \quad \text{ and } \quad \uptau(W_s)=H(\rho_{a_i})+\uptau(W_{s-a_i}^i), \text{ if } i \in\mathcal{I}^\prime. 
$$
With these remarks in hand, it is plain that in order to derive the points (i) and (ii) it suffices to prove respectively that a.s. we have  
\begin{align*}
    & \text{({i}')} \,   \lim_{k\to \infty}\varepsilon_k^{-1} \sum_{i\in \mathcal{I}\setminus \mathcal{I}^\prime} \mathbbm{1}_{H(\rho_{a_i})\leq \uptau(\w)+\varepsilon_k} \int_{0}^{\infty} \dd s \,  \mathbbm{1}_{\{ 0 < H(\rho_s^i) <  \epsilon_k \}}=0;   \\
    & \text{({ii}')}  \lim_{k\to \infty} \sum_{i\in \mathcal{I}^\prime}   \sup_{t\geq 0} \Big|L_t(\rho^i,W^i)-\varepsilon_k^{-1}\int_{0}^{t} \dd s \,  \mathbbm{1}_{\{ \uptau(W_s^i) < H(\rho^i_s) < \uptau(W_s^i) + \epsilon_k \}}\Big|=0.  
\end{align*}
 We stress that  $\mathcal{I}\setminus \mathcal{I}'$ or $\mathcal{I}'$ might be empty depending on the choice of the initial condition $(\mu , \overline{\w})$. Let us start by proving (i'). Fix $\delta>0$ and note that, since  \eqref{equation:PoissonH} is a Poisson point measure with intensity $\mu(\dd h ) \mathbb{N}_{ \overline{\w}(h)}(\dd \rho, \dd W)$, an application of  \eqref{eq:C:poisson:H} gives
 \begin{align*}
 &\limsup_{k\to \infty}\varepsilon_k^{-1} \sum_{i\in \mathcal{I}\setminus \mathcal{I}^\prime} \mathbbm{1}_{H(\rho_{a_i})\leq \uptau(\w)+ \varepsilon_k} \int_{0}^{\infty} \dd s \, \mathbbm{1}_{\{0< H(\rho_s^i) <  \epsilon_k \}}\\
 &\leq \limsup_{k\to \infty}\varepsilon_k^{-1} \sum_{i\in \mathcal{I}\setminus \mathcal{I}^\prime} \mathbbm{1}_{H(\rho_{a_i})\leq \uptau(\w)+\delta} \int_{0}^{\infty} \dd s \,  \mathbbm{1}_{\{ 0 < H(\rho_s^i) <  \epsilon_k \}}= \mu(  [  \uptau(\w),\uptau(\w)+\delta]),   \quad \text{a.s.}
 \end{align*}
 with the convention that $[\uptau(\w), \uptau(\w) + \delta]$ is empty if  $\uptau(\w) = \infty$. The condition $(\mu,\overline{\w})\in \overline{\Theta}_x$ ensures  $\mu(\{ \uptau(\w) \})=0$, and  (i') follows by taking the limit as $\delta \to 0$. We now  turn our attention to (ii'). Observe that by Lemma \ref{lem:equa:convL}, we have the a.s. convergence:  
\begin{equation}\label{eq:L:P:mu:w:x:c:d:g:1}
\lim_{k\to \infty} \sup_{t\geq 0} \Big|L_t(\rho^i,W^i)-\varepsilon_k^{-1}\int_{0}^{t} \dd s \,  \mathbbm{1}_{\{ \uptau(W_s^i) < H(\rho^i_s) < \uptau(W_s^i) + \epsilon_k \}}\Big|=0, \quad \text{ for } i\in \mathcal{I}^\prime,
\end{equation}
and to deduce  (ii') we need a domination argument. To this end, remark that:
\begin{equation}\label{eq:L:P:mu:w:x:c:d:g:3}\sup_{t\geq 0} \Big|L_t(\rho^i,W^i)-\varepsilon_k^{-1}\int_{0}^{t} \dd s \,  \mathbbm{1}_{\{ \uptau(W_s^i) < H(\rho^i_s) < \uptau(W_s^i) + \epsilon_k \}}\Big|\leq L_\infty(\rho^i,W^i) + \epsilon_k^{-1} \int_{0}^{\infty} \dd s \,  \mathbbm{1}_{\{ \uptau(W_s^i) < H(\rho^i_s) < \uptau(W_s^i) + \epsilon_k \}}.
\end{equation}
Recalling that the intensity measure of the Poisson point measure  \eqref{equation:PoissonH} is $\mu(\dd h)\mathbb{N}_{\overline{\w}(h)}(\dd \rho, \,  \dd \overline{W})$, using  \eqref{equation:mediaexit} we obtain the bound   $\mathbb{E}_{\mu,\w}^\dagger\big[\sum_{i \in \mathcal{I}^\prime} L_\infty(\rho^i,W^i)\big]\leq  \langle \mu , 1 \rangle <\infty$. Hence, the variable $\sum_{i\in \mathcal{I}^\prime} L_\infty(\rho^i,W^i)$ is finite a.s.  and  we claim that 
\begin{equation}\label{eq:L:P:mu:w:x:c:d:g:2}
\lim_{k\to \infty}\sum_{i\in\mathcal{I}^\prime} \varepsilon_k^{-1}\int_{0}^{\infty} \dd s  \, \mathbbm{1}_{\{ \uptau(W_s^i) < H(\rho^i_s) < \uptau(W_s^i) + \epsilon_k \}}= \sum_{i\in \mathcal{I}^\prime} L_\infty(\rho^i,W^i), \quad \text{ a.s.}
\end{equation}
Note that if this convergence holds,  the points  \eqref{eq:L:P:mu:w:x:c:d:g:1}, \eqref{eq:L:P:mu:w:x:c:d:g:3}, \eqref{eq:L:P:mu:w:x:c:d:g:2}  and an application of the general dominated convergence theorem \cite[Theorem 19, Chapter 4]{royden} immediately give (ii'). Let us then prove  \eqref{eq:L:P:mu:w:x:c:d:g:2}. For every   $i \in \mathcal{I}^\prime$,  denote  the connected components of the open set $\big\{t\geq 0:\:\uptau(W_{t}^i)< H(\rho^i_t) \big\}$  by  $\big((s_{j},t_{j}) : j\in \mathcal{K}_i \big)$, 
where  $\mathcal{K}_i $ is an indexing set that might be empty. For $j \in \mathcal{K}_i$  let $(\rho^{i,j},W^{i,j},\Lambda^{i,j})$ be the subtrajectory of $(\rho^i,W^i,\Lambda^{i})$ associated with the interval $[s_{j},t_{j}]$ in the sense of Section \ref{section:snake}, translated by $-(0,0, \Lambda^i_0)$. If we let $\mathfrak{L}_\infty^\prime$ be the right-hand side of  \eqref{eq:L:P:mu:w:x:c:d:g:2},  Theorem 3.7 in \cite{2024structure} entails that, conditionally on $\mathfrak{L}_\infty^\prime$,  the point measure 
\begin{equation}
    \label{definition:medidaordenadatilde2_2_l}
\sum\limits_{i\in \mathcal{I}^\prime, j\in \mathcal{K}_i} \delta_{(\rho^{i,j},\overline{W}^{i,j})}
\end{equation}
is a Poisson point measure with intensity $\mathfrak{L}_\infty^\prime \cdot \mathbb{N}_{x,0}(\dd\rho,\dd\overline{W})$. By   similar arguments  as before,   we can write  
\begin{equation*}
    \sum \limits_{i\in \mathcal{I}^\prime}\varepsilon_k^{-1}\int_{0}^{\infty} \dd s  \, \mathbbm{1}_{\{ \uptau(W_s^i) < H(\rho^i_s) < \uptau(W_s^i) + \epsilon_k \}}= \sum\limits_{i\in \mathcal{I}^\prime, j\in \mathcal{K}_i} \varepsilon_k^{-1}\int_{0}^{\infty} \dd s \,  \mathbbm{1}_{\{0< H(\rho^{i,j}_s) <  \epsilon_k \}}
\end{equation*}
and since $\sum \limits_{i\in \mathcal{I}^\prime}L_{\infty}(\rho^i,W^i)<\infty$ a.s.  it follows from \eqref{eq:C:poisson:H} that
\begin{equation*}
    \lim \limits_{k\to \infty} \sum\limits_{i\in \mathcal{I}^\prime, j\in \mathcal{K}_i} \varepsilon_k^{-1}\int_{0}^{\infty} \dd s \,  \mathbbm{1}_{\{0< H(\rho^{i,j}_s) <  \epsilon_k \}}=\sum \limits_{i\in \mathcal{I}^\prime}L_{\infty}(\rho^i,W^i), \quad \text{a.s.}
\end{equation*}
This completes the proof of the lemma.
 \end{proof}

We can now proceed with the proof of Proposition \ref{proposition:L*}.

\begin{proof}[Proof of Proposition \ref{proposition:L*}]
Recall that under $\mathbf{N}_{x,r}$, the process $(\rho,\overline{W})$ takes values in $\overline{\Theta}_x$. Therefore,  by the Markov property under $\mathbf{N}_{x,r}$ proved in Proposition \ref{eq:strong:Markov:N} and  Lemma \ref{lem:L:under:P:mu:w},  it follows that under $\mathbf{N}_{x,r}$ and for every fixed $s> 0$, the process
$$
\frac{1}{\epsilon_k} \int_s^t \dd u \,  \mathbbm{1}_{\{ \uptau(W_u) < H(\rho_u) < \uptau(W_u) + \epsilon_k \} },\quad t\geq s,  
$$
  converges uniformly as $k\to \infty$ towards a continuous function.
  This fact paired with the duality property established in Corollary \ref{corollary:dualidadNblack} yields that under $\mathbf{N}_{x,r}$,  for every fixed  $s> 0$, the process
\begin{align*}
    \frac{1}{\epsilon_k} \int_{(\sigma-s)\vee 0}^{(\sigma-t)\vee 0} \dd u \,  \mathbbm{1}_{\{ \uptau(W_u) < H(\rho_u) < \uptau(W_u) + \epsilon_k \}},\quad t\geq s,  
\end{align*}
 converges also uniformly as $k\to \infty$  towards a continuous function. These two facts combined prove the a.e. convergence \eqref{equation:aproximacionL*Pmu}  under $\mathbf{N}_{x,r}$. Finally,  the special Markov property under  $\mathbf{N}_{x,r}$ follows by the Markov property and   \cite[Theorem 3.8]{2024structure}  by similar arguments to the ones employed in the proof of Theorem \ref{theorem:excursionPPP}. For this reason we only sketch the main steps. An application of the  Markov property of Proposition \ref{eq:strong:Markov:N} under $\mathbf{N}_{x,r}$ at time $t > 0$ followed  with the special Markov property \cite[Theorem 3.8]{2024structure} yield that, on the event $  \{\sigma > t\}$ and conditionally on $(L_t(\rho,W),L_\sigma(\rho,W))$, 
the measure:
\begin{equation}
    \label{definition:medidaordenadatilde2_4}
\sum\limits_{i\in \mathcal{I}} \mathbbm{1}_{\{ L_{s_i}(\rho,W) > L_t(\rho,W) \}} \delta_{(L_{s_i}(\rho,W),\rho^{i},\overline{W}^{i})}
\end{equation}
is a Poisson point measure with intensity $\mathbbm{1}_{[L_{t}(\rho,W),L_\sigma(\rho,W)]}(\ell)\dd \ell \,  \mathbb{N}_{x,r}(\dd\rho,\dd\overline{W})$ independent of $\mathrm{tr}(\rho , W)$.  The result now readily follows by taking the limit as $t \to 0$, noting that the first statement of the proposition ensures that $\lim_{t\to 0}L_t(\rho,W)=L_0(\rho,W) = 0$ a.e. 
\end{proof}

\subsection{An auxiliary approximation of the exit local time}\label{sub:inside:approx}

The purpose  of this section is to extend the notion of exit local time under the excursion measure $\mathbf{N}_x^*$. Note  that since under $\mathbf{N}^*_x$ the set $\{s \geq 0 : \uptau(W_s) < H(\rho_s) \}$   is empty, the functional  $(L_t)_{t \geq 0}$  introduced in \eqref{definition:exitMultiusos} under $\mathbf{N}^*_x$ is  identically equal to $0$, and therefore  not suited for our purposes. 
At this point, we recall from the discussion at the end of  Section \ref{section:truncationboundary} that  for any given  $(\upvarrho , \omega)\in \mathbb{D}(\mathbb{R}_+ , \mathcal{M}_f(\mathbb{R}_+) \times \mathcal{W}_{E} )$ we interpret the process $L_{\Gamma_t(\upvarrho, \omega)}(\upvarrho , \omega),$ $t \geq 0$ as $L(\upvarrho , \omega)$  in the time-scale of $\mathrm{tr}(\upvarrho , \omega)$.  Since $\mathbf{N}_{x}^*$ is the law of $\mathrm{tr}(\rho , W)$ under $\mathbf{N}_{x,0}$, 
 extending the notion of exit local time under $\textbf{N}^*_x$  boils down to finding under $\mathbf{N}_{x,r}$  a continuous non-decreasing functional  of $\mathrm{tr}(\rho , W)$  which coincides with   $L_{\Gamma}(\rho , W)$.  To this end we rely on an  alternative approximation of $L$ under $\mathbf{N}_{x,0}$ given in Proposition \ref{prop:inside:L}. Specifically, if for  every $\w\in \mathcal{W}_E$ and $\delta,\upsilon\geq 0$, we set:
\begin{equation*}
    \uptau_{\delta,\upsilon}(\w):=\inf\{t\in[\delta,\zeta_\w]:~ d_{E}(\w(t),x)\leq \upsilon \}
\end{equation*}
with the usual convention $\inf \varnothing=\infty$, then we have: 

\begin{prop}\label{prop:inside:L}
 There exists
  sequences $(\delta_k)_{k \geq 0}$,  $(\upsilon_k)_{k \geq 0}$ and   $(\epsilon_k^{\prime\prime})_{k \geq 0}$  of positive numbers  converging towards $0$  as $k \to \infty$ such that,  under $\mathbf{N}_{x,r}$, the following convergence  holds a.e.
\begin{equation}\label{equation:aproximationInside}
    \lim \limits_{k\to \infty} \sup_{t\geq 0}  \Big| L_t(\rho,W)-\frac{1}{\epsilon_{k}^{\prime \prime}}\int_{0}^{t} \dd s \, \mathbbm{1}_{\{\uptau_{\delta_k, \upsilon_k}(W_s)<H(\rho_s)<\uptau_{\delta_k, \upsilon_k}(W_s)+\varepsilon_{k}^{\prime\prime}<\uptau(W_s)\}}\Big|=0. \quad 
\end{equation}
\end{prop}
Before proving this result let us discuss  its main implications. First, note that the approximation  \eqref{equation:aproximationInside} does not depend on $\Lambda$ or on $r\geq 0$. Hence, the same result holds under $\mathbf{N}_{x,r^\prime}$ for every $r^\prime\geq 0$ for the same  sequences $(\delta_k)_{k \geq 0}$,  $(\upsilon_k)_{k \geq 0}$ and   $(\epsilon_k^{\prime\prime})_{k \geq 0}$. These are fixed from now on. With this remark in mind, for $(\upvarrho, \omega) \in \mathbb{D}( \mathbb{R}_+, \mathcal{M}_f(\mathbb{R}_+) \times \mathcal{W}_E)$ we define a functional $L^*(\upvarrho, \omega) = (L^*_t(\upvarrho, \omega): t \geq 0)$,  by setting
\begin{equation*}
    L^*_t(\upvarrho,\omega):= \liminf_{k\to\infty}\frac{1}{\epsilon_{k}^{\prime \prime}}\int_{0}^{t} \dd s \,  \mathbbm{1}_{\{ \uptau_{\delta_k, \upsilon_k}(\omega_s)<H(\upvarrho_s)<\uptau_{\delta_k, \upsilon_k}(\omega_s)+\varepsilon_{k}^{\prime\prime}<\uptau(\omega_s) \}}, \quad  t \geq 0. 
\end{equation*}
As a first consequence of Proposition \ref{prop:inside:L} we obtain the following result.

\begin{cor}\label{corollary:L*sirvepatodo}
    We can find a measurable subset of full   $\mathbf{N}_{x,r}$ measure at which we have
     \begin{equation}\label{definition:L*}
    L^*_t(\rho, W)= L_t(\rho , W), \quad t \geq 0,  
\end{equation}
and
    \begin{equation}\label{definition:L*:?}
    L^*_t(\emph{tr}(\rho, W))= L_{\Gamma_t(\rho,W)}(\rho , W), \quad t \geq 0.   
\end{equation}
In particular,    in a  measurable subset of full $\mathbf{N}^*_x$  measure, the process $(L^*_t: t \geq 0)$ is  continuous, non-decreasing and null at $0$. 
\end{cor}
\begin{proof}
The identity \eqref{definition:L*} is an  
 immediate consequence of the approximation  \eqref{equation:aproximationInside}, while   \eqref{definition:L*:?} plainly follows    from  \eqref{equation:aproximationInside}  
 and a change of variable. The last claim is a consequence of   \eqref{definition:L*:?}  and the fact that $\mathbf{N}_x^*$ is  the distribution of $\mathrm{tr}(\rho , W)$ under $\mathbf{N}_{x,r}$.
\end{proof}
In other words, Corollary \ref{corollary:L*sirvepatodo}  entails that for  $\mathbf{N}_{x,r}$-a.e. $(\upvarrho , \omega)$,  the functional $L^*(\upvarrho, \omega)$ agrees with $L(\upvarrho, \omega)$ on $(\upvarrho , \omega)$, and when taken at  $\mathrm{tr}(\upvarrho , \omega)$ coincides  with $L_{\Gamma}(\upvarrho, \omega)$.  These relationships translate as follows when working with the excursions away from $x$. 
\begin{cor}\label{corollary:identificacionL*}
Under $\mathbb{N}_{x,0}$ and $\mathbb{P}_{0,x,0}$, for every $u \in \mathcal{D}$, we have 
\begin{equation} \label{equation:L*=lu}
    ( L^*_t(\rho^{u,*} ,  W^{u,*}) : t \geq 0 ) = ( L_{\Gamma_t}( \rho^u,W^u ) : t \geq 0). 
\end{equation}
\end{cor}
\begin{proof}
We start proving the statement under $\mathbb{N}_{x,0}$, and  in this direction we consider the set $$\Omega_0:=\big\{(\upvarrho,\omega)\in  \mathbb{D}(\mathbb{R}_+, \mathcal{M}_f(\mathbb{R}_+)\times \mathcal{W}_E): ~L^{*}(\mathrm{tr}(\upvarrho,\omega))\neq L_{\Gamma( \upvarrho , \omega)}(\upvarrho,\omega)\big\}.$$
Recalling that by definition, for every $u\in \mathcal{D}$, the truncated path $\mathrm{tr}(\rho^u,W^u)$ is precisely $(\rho^{u,*},W^{u,*})$, we can write:
$$
\mathbb{N}_{x,0}\Big(\sum \limits_{u\in \mathcal{D}} f\big(A_{\mathfrak{g}(u)}\big)\mathbbm{1}_{L^*(\rho^{u,*},W^{u,*})\neq L_{\Gamma}(\rho^u,W^u)}\Big)=\mathbb{N}_{x,0}\Big(\sum \limits_{u\in \mathcal{D}} f\big(A_{\mathfrak{g}(u)}\big)\mathbbm{1}_{(\rho^u,W^u)\in\Omega_0}\Big), 
$$
where  $f$ is  an arbitrary non-negative measurable function on $\mathbb{R}$. Now an application of Theorem~\ref{theorem:exit} paired with \eqref{definition:L*:?} yields that the right hand side in the last display is null. This proves that   \eqref{equation:L*=lu} holds a.e. for every $u\in \mathcal{D}$  under $\mathbb{N}_{x,0}$. The same argument holds under $\mathbb{P}_{0,x,0}$ applying  Corollary~\ref{corollary:exitPx} instead of Theorem~\ref{theorem:exit}.  
\end{proof}
Finally, Corollary \ref{corollary:L*sirvepatodo}  enables  us to strengthen the special Markov property 
 stated in Proposition \ref{proposition:L*}  under $\mathbf{N}_{x,r}$.  More precisely, since $L_\sigma = L_\infty^*(\mathrm{tr}(\rho , W))$  under $\mathbf{N}_{x,r}$,    Proposition \ref{proposition:L*} and   Corollary \ref{corollary:L*sirvepatodo} yield the following. 
\begin{cor}\label{corollary:special2}
    Under $\mathbf{N}_{x,r}$ and conditionally on $\emph{tr}(\rho , W)$,  the point measure \eqref{equation:SpecialNblack} is a Poisson point measure with intensity $\mathbbm{1}_{[0,L_\sigma(\rho,W)]}(\ell) \dd \ell \,  \mathbb{N}_{x,r}(\dd \rho , \overline{W})$.  
\end{cor}

The rest of the section is devoted to the proof of Proposition \ref{prop:inside:L}. In this direction, we start with a technical lemma.

\begin{lem}\label{lem:eq:inside:L:N}
For every fixed $y \neq x$ and $\delta,\upsilon\geq 0$, $\varepsilon > 0$ set:
\begin{equation*}
    J_{\delta,\upsilon,\varepsilon}(y):=\mathbb{N}_{y}\Big(\sup_{t\geq 0}\Big|L_t(\rho,W)- \epsilon^{-1}\int_{0}^{t} \dd s \,  \mathbbm{1}_{\{\uptau_{\delta, \upsilon}(W_s)<H(\rho_s)<\uptau_{\delta, \upsilon}(W_s)+\varepsilon<\uptau(W_s) \}}\Big|\Big).
\end{equation*}
The following properties hold:
\begin{itemize}
\item[$\mathrm{(i)}$] $J_{\delta, \upsilon,\varepsilon}(y)\leq 2$;
\item[$\mathrm{(ii)}$] $ \limsup_{\upsilon \to 0} \limsup_{\varepsilon \to 0} J_{0,\upsilon,\varepsilon}(y)=0$. 
\end{itemize}
\end{lem}

\begin{proof} 
Both points are direct consequences of classic estimates on functionals of  Lévy snakes. For simplicity, we omit the dependence on $(\rho,W)$ in our functionals when there is no risk of ambiguity.  Fix $y\in E\setminus \{x\}$ as well as  $\delta,\upsilon\geq 0$, $\varepsilon>0$, and  remark that:
$$
J_{\delta,\upsilon,\varepsilon}(y)\leq \mathbb{N}_{y}(L_\sigma)+\epsilon^{-1}\mathbb{N}_{y}\big(\int_{0}^{\sigma} \dd s~ \mathbbm{1}_{\{\uptau_{\delta, \upsilon}(W_s)<H(\rho_s)<\uptau_{\delta, \upsilon}(W_s)+\varepsilon<\uptau(W_s) \}}\big).
$$
By \eqref{equation:mediaexit} we already know that $\mathbb{N}_{y}(L_\sigma)\leq 1$, and  to obtain an upper  bound for the second term in the right-hand side of the last display we recall from   \cite[Lemma 2.4]{2024structure}  that for every non-negative  measurable function $\Phi$ on $\mathcal{M}^2_f(\mathbb{R}_+) \times \mathcal{W}_{{E}}$,  we have
\begin{equation}\label{eq:many:to:one:N:y}
     \mathbb{N}_{y} \big(\int_0^\sigma \dd s~  \Phi\big( \rho_s,\eta_s, {W}_s\big)  \big)   
    = \int_0^\infty \dd a \, \exp\big(- \alpha a\big)\cdot  
    E^0 \otimes \Pi_{y} \Big( \Phi\big( \mathcal{J}_a , \widecheck{\mathcal{J}}_a  , ({\xi}_t : t \leq a)   \big) \Big). 
\end{equation} 
In particular,  we have
\begin{equation*}
    \varepsilon^{-1} \mathbb{N}_{y}\big(\int_{0}^{\sigma} \dd s~ \mathbbm{1}_{\{ \uptau_{\delta, \upsilon}(W_s)<H(\rho_s)<\uptau_{\delta, \upsilon}(W_s)+ \epsilon <\uptau(W_s)\}}\big)
 = 
   \varepsilon^{-1} \Pi_y\big( \int_{0}^{  \infty} \dd a \,  \exp(- \alpha a) \mathbbm{1}_{ \{\uptau_{\delta, \upsilon}(\xi) < a < \uptau_{\delta, \upsilon}(\xi) + \epsilon<\uptau(\xi^a) \} } \big) \leq 1,
\end{equation*}
letting $\uptau_{\delta,\upsilon}(\xi):=\inf\{t\geq \delta:~\dd_{E}(\xi,x)\leq \upsilon \}$ under $\Pi_y$.
This completes the proof of (i) and we now turn our attention to  (ii). Still for  fixed  $y\in E\setminus\{x\}$,   note that $J_{0, \upsilon,\varepsilon}(y)$ is bounded above by:
\begin{align*}
&\mathbb{N}_{y}\Big(\sup_{t\geq 0}\Big|L_t- \epsilon^{-1}\int_{0}^{t} \dd s \, \mathbbm{1}_{\{\uptau_{0, \upsilon}(W_s)<H(\rho_s)<\uptau_{0, \upsilon}(W_s)+\varepsilon \}}\Big|\Big) +\epsilon^{-1}\mathbb{N}_{y}\Big(\int_{0}^{\sigma} \dd s \, \mathbbm{1}_{H(\rho_s),  \uptau(W_s) \,  \in \,  (\uptau_{0, \upsilon}(W_s), \uptau_{0, \upsilon}(W_s)+\varepsilon ]}\Big).
\end{align*}
On one hand, the first moment formula \eqref{eq:many:to:one:N:y} gives:
\begin{align*}
\epsilon^{-1}\mathbb{N}_{y}\Big(\int_{0}^{\sigma} \dd s \,  \mathbbm{1}_{H(\rho_s),  \uptau(W_s) \,  \in \,  (\uptau_{0, \upsilon}(W_s), \uptau_{0, \upsilon}(W_s)+\varepsilon ]} \Big)&=\varepsilon^{-1} \Pi_y\big(\int_{0}^{\infty}\dd a \, \exp(-\alpha a) \mathbbm{1}_{ a,  \uptau(\xi^a) \,  \in \,  (\uptau_{0, \upsilon}(\xi) ,  \uptau_{0, \upsilon}(\xi) +\varepsilon]} \big) \\
&\leq \Pi_y\big(\uptau(\xi)\leq\uptau_{0, \upsilon}(\xi) +\varepsilon \big),
\end{align*}
and since $\xi$ has continuous sample paths,  this last  term  converges to $0$ as $\varepsilon\to 0$. On the other hand,  (4.11) in \cite{DLG02} yields  that for every $0 \leq \upsilon<\dd_{E}(x,y)$, there exists a continuous non-decreasing process $L^{\upsilon} = (L^\upsilon_t)_{t \geq 0}$ known as the exit local time from the open set $D_\upsilon:=\{z\in E:~\dd_{E}(x,z)>\upsilon\}$, such that
\begin{equation}\label{def:eq:approx:L:upsilon:}
\lim_{\varepsilon \to 0} \mathbb{N}_{y}\Big(\sup_{t\geq 0}\Big|L_t^{\upsilon}- \epsilon^{-1}\int_{0}^{t} \dd s \,  \mathbbm{1}_{\{\uptau_{0, \upsilon}(W_s)<H(\rho_s)<\uptau_{0, \upsilon}(W_s)+\varepsilon\}}\Big|\Big)=0.
\end{equation}
Hence, to derive (ii) it suffices to establish that
\begin{equation}\label{equation:Lv}
    \lim_{\upsilon \to 0} \mathbb{N}_{y}\big(\sup_{t\geq 0}|L_t^{\upsilon}-L_t|\big)=0.  
\end{equation}
It is enough to show that for every sequence $(\upsilon_{k})_{k\geq 0}$ of positive  real numbers smaller than $\mathrm{d}_E(x,y)$ and strictly  decreasing  towards $0$, we can find a subsequence $(\upsilon_{k}^\prime)_{k\geq 0}$ along which \eqref{equation:Lv} converges to $0$. Let us fix any such sequence $(\upsilon_{k})_{k\geq 0}$ and note that then,  $(D_{\upsilon_k})_{k\geq 0}$ is  a non-decreasing sequence of open sets  containing $y$ with   $\cup_{k\geq 0} D_{\upsilon_k}=E\setminus\{x\}$. Since the closure of $D_{\upsilon_k}$ is included in $E\setminus\{x\}$, by  \cite[Lemma 3.6]{2024structure} we can find a subsequence  $(\upsilon_{k}^\prime)_{k\geq 0}$ along which  $\sup_{t \geq 0} |L^{\upsilon_{k}^\prime}-L| \to 0$ as $k \to \infty$,   $\mathbb{N}_y$-a.e. and to conclude that this convergence holds in $L_1(\mathbb{N}_y)$ we need a domination argument. To this end, remark that the term $\sup_{t \geq 0}  \big|  L^{\upsilon_k^\prime}  - L  \big|$ is $\mathbb{N}_y$-a.e. bounded above by  $L^{\upsilon^\prime_k}_\sigma+ L_\sigma$,  which converges  towards $2 L_\sigma$ as $k\to \infty$.   Moreover, by continuity of $\xi$ under $\Pi_y$ and \eqref{equation:mediaexit}, we have  
\begin{align*}
\lim_{k \to \infty}\mathbb{N}_{y}(L_\sigma^{\upsilon_k^\prime})=\lim_{k \to \infty}\Pi_y\big(\exp(-\alpha \uptau_{0,\upsilon_k^\prime}(\xi))\big)=\Pi_y\big(\exp(-\alpha \uptau(\xi))\big) = \mathbb{N}_y(L_\sigma).
\end{align*}
Now, an application of the general dominated convergence theorem yields that $\mathbb{N}_{y}(\sup_{t\geq 0}|L_t^{\upsilon_k^\prime}-L_t|) \to 0$ as $k \to \infty$. 
\end{proof} 

We are now in position to prove  Proposition \ref{prop:inside:L}.
\begin{proof}[Proof of Proposition \ref{prop:inside:L}]
 Recall the definition of the pointed measure $\mathbf{N}_{x,r}^\bullet$ given in \eqref{N:bullet:black:def} and observe that,  by an absolute continuity argument,  it suffices to establish  \eqref{equation:aproximationInside} under $\mathbf{N}_{x,r}^\bullet$. In this direction, for $\delta ,\upsilon \geq 0$ and $\epsilon > 0$ we set: 
\begin{align*}
& \texttt{I}_1(\delta, \upsilon,\varepsilon):=\sup \limits_{t\in [0, U]} \Big|L_U(\rho,W)-L_t(\rho,W)-\epsilon^{-1}\int_{t}^{U} \dd s \,  \mathbbm{1}_{\{\uptau_{\delta, \upsilon}(W_s)<H(\rho_s)<\uptau_{\delta, \upsilon}(W_s)+\varepsilon< \uptau(W_s)\}}\Big|,  \\
& \texttt{I}_2(\delta, \upsilon,\varepsilon):=\sup \limits_{t\in [U,\sigma]} \Big|L_t(\rho,W)-L_U(\rho,W)-\epsilon^{-1}\int_{U}^{t} \dd s \, \mathbbm{1}_{\{\uptau_{\delta, \upsilon}(W_s)<H(\rho_s)<\uptau_{\delta, \upsilon}(W_s)+\varepsilon< \uptau(W_s)\}} \Big|.  
\end{align*}
Notice that under $\mathbf{N}_{x,r}^\bullet$, the variable
$$
\sup \limits_{t\geq 0} \Big|L_t(\rho,W)-\epsilon^{-1}\int_{0}^{t} \dd s \,  \mathbbm{1}_{\{\uptau_{\delta, \upsilon}(W_s)<H(\rho_s)<\uptau_{\delta, \upsilon}(W_s)+\varepsilon< \uptau(W_s) \}} \Big|
$$
is dominated by $2 \texttt{I}_1(\delta, \upsilon,\varepsilon) +  \texttt{I}_2(\delta, \upsilon,\varepsilon)$.   Moreover, by   Corollary \ref{corollary:dualidadNblack} we have under $\mathbf{N}_{x,r}^\bullet$ the following identity in distribution: 
\begin{equation*}
    \big(\rho_{U+t},\eta_{U+t}, \overline{W}_{U+t} :~t \geq 0 \big)\overset{(d)}{=} \big(\eta_{((U-t)\vee 0) _-  },\rho_{((U-t)\vee 0) _-  }, \overline{W}_{((U-t)\vee 0) _-  } :~ t \geq 0\big), 
\end{equation*}
with the usual convention that the left limit at time $0$ is just $0$. Thus, since under $\mathbf{N}_{x,r}^\bullet$  we have that  $H(\rho)=H(\eta)$ a.e., it is plain that $\texttt{I}_1(\delta, \upsilon,\varepsilon)$ and $\texttt{I}_2(\delta, \upsilon,\varepsilon)$ have the same distribution under $\mathbf{N}_{x,r}^\bullet$. Given our previous observations,  the  convergence \eqref{equation:aproximationInside} under $\mathbf{N}_{x,0}^{\bullet}$  can  be reduced to establishing that we can find $\big(\delta_k, \upsilon_k,\varepsilon_k^{\prime\prime}\big)_{k\geq 0}$ as in the statement of the proposition such that:
\begin{equation}\label{eq:L:dentro:proof:1}
\lim \limits_{k\to \infty}  \texttt{I}_2\big(\delta_k, \upsilon_k,\varepsilon_k^{\prime\prime}\big)=0,\quad \mathbf{N}_{x,r}^\bullet\text{--a.e.}
\end{equation}
We are going to deduce \eqref{eq:L:dentro:proof:1} thanks to the Markov property and Lemma \ref{lem:eq:inside:L:N}.  For each $K > 1$, let us write $\Omega_K$ for the event $\{ \uptau(W_U)>K^{-1}, \langle \rho_U,1\rangle< K\}$. In order to verify   \eqref{eq:L:dentro:proof:1}, we claim that it suffices to prove that for every $K>1$, we have
\begin{equation}\label{eq:L:dentro:proof:2}  \limsup \limits_{\delta\to 0} \limsup \limits_{\upsilon\to 0} \limsup \limits_{\varepsilon\to 0}\mathbf{N}_{x,0}^{\bullet}\big(\mathbbm{1}_{\Omega_K} \texttt{I}_2(\delta, \upsilon, \varepsilon)\big)=0. 
\end{equation}
Actually, \eqref{eq:L:dentro:proof:2}  implies that  on the event $\Omega_K$, the 
 convergence \eqref{eq:L:dentro:proof:1} holds along a subsequence that depends on $K$. Since this holds for arbitrarily large $K$, by a diagonal argument we can  find a deterministic subsequence   $(\delta_k,\upsilon_k,\varepsilon_{k}^{\prime\prime})$ converging towards $(0,0, 0)$ as $k \to \infty$ and  satisfying \eqref{eq:L:dentro:proof:1}  on the event $\{ \uptau(W_U)>0, \langle \rho_U,1\rangle< \infty\}$. It readily follows from   Lemma \ref{proposition:only:spineN*}   that  this event  is of full $\mathbf{N}_{x,r}^\bullet$ measure,  and thus  \eqref{eq:L:dentro:proof:1} holds. 
\par 
Let us then prove  \eqref{eq:L:dentro:proof:2}. Consider  the connected components $(a_i,b_i)_{i\in \mathcal{I}}$ of the set $\{t \geq 0 : \, H( \rho_{U+t})- \inf_{s\in [U,U+t]}H(\rho)>0\}$  and as usual write $(\rho^i, \overline{W}^i)_{i \in \mathcal{I}}$ for the associated subtrajectories. The Markov property paired with the discussion following \eqref{equation:PoissonH} yields that the point measure
\begin{equation*}
    \sum_{i \in  \mathcal{I} } \delta_{(H(\rho_{a_i}),  \rho^i ,   \overline{W}^i )}
\end{equation*}
 is a Poisson point measure with intensity $ \rho_U(\dd h) \, \mathbb{N}_{\overline{W}_U( h ) } ( \dd \rho ,  \dd \overline{W} )$. 
By Lemma \ref{lem:L:under:P:mu:w} combined with \eqref{Lebesgue:no:quiere:verlo} and the fact that  $\mathbf{N}_{x,r}^{\bullet}$ is supported on $\overline{\mathcal{S}}$, we infer that  under $\mathbf{N}_{x,r}^{\bullet}$ we  must have
\begin{equation*}
    \texttt{I}_2(\delta, \upsilon, \varepsilon)\leq \sum \limits_{i\in \mathcal{I}} \mathbbm{1}_{H(\rho_{a_i})\leq \uptau(W_U)}\sup_{t\geq 0} \Big|L_t(\rho^i, \overline{W}^i) -\epsilon^{-1}\int_{0}^{t} \dd s \, \mathbbm{1}_{\{\tau_{\delta_i, \upsilon}(W_s^i)<H(\rho_s^i)<\tau_{\delta_i, \upsilon}(W_s^i)+\varepsilon<\uptau(W_s^i)\}} \Big|, 
\end{equation*}
where we write $\delta_i:= \delta -H(\rho_{a_i})$, if $H(\rho_{a_i})\leq \delta$, and $\delta_i:=0$ otherwise. 
Thus, with the notation of Lemma~\ref{lem:eq:inside:L:N}, we get that the term $\mathbf{N}_{x,0}^{\bullet}\big(\mathbbm{1}_{\Omega_K} \texttt{I}_2(\delta, \upsilon, \varepsilon)\big)$ is bounded above by:
\begin{align}\label{eq:N:J:delta:n:epsilon}
\mathbf{N}_{x,0}^{\bullet}\Big(\mathbbm{1}_{\Omega_K}\cdot \Big(  \int_{0}^{\delta} \rho_U(\dd t) J_{\delta - t,\upsilon,\varepsilon}(W_U(t))+\int_{\delta}^{\uptau(W_U)} \rho_U(\dd t) J_{0,\upsilon,\varepsilon}(W_U(t))\Big)\Big).
\end{align}
Then, by Lemma \ref{lem:eq:inside:L:N}-(i) we deduce that  \eqref{eq:N:J:delta:n:epsilon} is dominated by:
\begin{equation*}
2 \, \mathbf{N}_{x,0}^{\bullet}\big(\mathbbm{1}_{\Omega_K} \cdot \rho_U([0,\delta])\big) 
+  
\mathbf{N}_{x,0}^{\bullet}\Big(\mathbbm{1}_{\Omega_K}\cdot  \int_{\delta}^{\uptau(\overline{W}_U)} \rho_U(\dd t)J_{0, \upsilon,\varepsilon}(W_U(t))\Big). 
\end{equation*}
Note that $\Omega_K$ has finite $\mathbf{N}^\bullet_{x,0}$ measure since by Lemma \ref{proposition:only:spineN*} we have $\mathbf{N}_{x,r}^\bullet (\uptau(W_U)>K^{-1} )<\infty$.  Now, the convergence \eqref{eq:L:dentro:proof:2} follows applying  the reverse  Fatou lemma and Lemma \ref{lem:eq:inside:L:N}, noting that Lemma \ref{proposition:only:spineN*} entails that    $\rho_U(\{0 \})=0$, $\mathbf{N}_{x,0}^{\bullet}$-a.e. This completes the proof of the proposition. 
\end{proof}

\subsection{The spinal decomposition with respect to the exit local time}\label{sec:Spinal:L}
In this section, we  characterise the law of the spinal decomposition  \eqref{equation:spinals} under a pointed version of $\mathbf{N}_{x,r}$ at a  point sampled 
 according to the measure $\dd L$.  In this direction, for each $r \geq 0$  we consider the  following pointed measure 
\begin{equation*}
    \mathbf{N}^{\bullet, L}_{x,r}:= \mathbf{N}_{x,r}(\dd \rho ,   \dd \overline{W})  \dd L_s
\end{equation*}
on $\mathbb{D}(\mathbb{R}_+, \mathcal{M}^2_f(\mathbb{R}_+)\times \mathcal{W}_{\overline{E}})  \times \mathbb{R}_+$.  Recall the notation $\mathbb{P}^\dag_{\mu,\nu,\overline{\w}}$ from Section \ref{section:MainSpinalDecomp} and that we write $U: \mathbb{R}_+ \to \mathbb{R}_+$ for the identity function.   Recall as well that $\uptau$ under $\mathcal{N}_{x,r}$ stands for the first return to $x$ of $\xi$. 

\begin{prop}\label{proposition:ManyToOneL*}
Fix $r\geq 0$.  For every non-negative measurable function $\Phi$ on $\mathcal{M}^2_f(\mathbb{R}_+) \times \mathcal{W}_{\overline{E}}$, we have:
\begin{align}\label{eq:prop:many:L}
     &  \mathbf{N}^{\bullet, L}_{x,r} \big(  \Phi\big( \rho_U,\eta_U, \overline{W}_U\big)  \big)  
  =   E^0 \otimes \mathcal{N}_{x,r} \Big(  \exp\big(- \alpha\cdot  \uptau \big)\Phi\big( \mathcal{J}_{\uptau}, \widecheck{\mathcal{J}}_{\uptau}, (\bar{\xi}_t:~t\in [0,\uptau] ) \big) \Big). 
\end{align}
Under $\mathbf{N}_{x,r}^{\bullet, L}$, conditionally on $(\rho_U,\eta_U, \overline{W}_U)$ the
processes $(\rho,\overline{W})^{U,\leftarrow}$ and  $(\rho,\overline{W})^{U,\rightarrow}$ are independent and their  conditional distributions are as follows:
\begin{itemize}
\item the process $(\rho,\overline{W})^{U,\leftarrow}$ is distributed as  $(\eta, \overline{W})$ under $\mathbb{P}_{\eta_U, \rho_U,\overline{W}_U}^\dagger$; 
\item the process $(\rho, \overline{W})^{U,\rightarrow}$ is distributed as  $(\rho, \overline{W})$ under $\mathbb{P}_{\rho_U,\overline{W}_U}^\dagger$.
\end{itemize}
\end{prop}
\noindent Note that the only difference with Proposition \ref{manytoone-lebesgue} comes from  the law of $(\rho_U,\eta_U,\overline{W}_U)$.
\begin{proof}
The same arguments applied in the proof of  Proposition \ref{manytoone-lebesgue-Addi} allows to infer from \eqref{eq:prop:many:L} the second statement of the proposition - using the strong Markov property  [Proposition \ref{eq:strong:Markov:N}] and the duality established in  Corollary~\ref{corollary:dualidadNblack}. Since this type of arguments have already been detailed in the proof of Proposition \ref{manytoone-lebesgue-Addi},  we only prove \eqref{eq:prop:many:L}.  Without loss of generality, we may assume that $\Phi$ is continuous and bounded above by~$1$.  
Noting that $\uptau(W_U)>0$, $\mathbf{N}_{x,r}^{\bullet, L}$-a.e., and that  $\uptau>0$, $\mathcal{N}_{x,r}$-a.e., by the monotone convergence theorem it suffices to establish that  we have 
\begin{equation}\label{equation:manyd}
    \mathbf{N}_{x,r}^{\bullet, L} \big(  \mathbbm{1}_{\{ \uptau(W_U)>\delta\}}\Phi\big( \rho_U,\eta_U, \overline{W}_U\big)  \big)= E^0 \otimes \mathcal{N}_{x,r} \Big( \mathbbm{1}_{\{ \uptau>\delta \}} \exp (- \alpha \cdot \uptau )\Phi\big( \mathcal{J}_{\uptau}, \widecheck{\mathcal{J}}_{\uptau}, (\bar{\xi}_t:~t\in [0,\uptau] ) \big) \Big),
\end{equation}
for every $\delta > 0$.   To this end,  note first that by Proposition \ref{proposition:only:spineN*}, for every $k \geq 0$ we have 
\begin{align}\label{equation:manyLd2}
     &  \frac{1}{\epsilon_k}  \mathbf{N}_{x,r}\Big(\int_0^\sigma \dd s \, \mathbbm{1}_{\{ \delta<\uptau(W_s) < H(\rho_s) < \uptau(W_s) + \epsilon_k \}} \Phi(\rho_s ,\eta_s, \overline{W}_s ) \Big)  \nonumber \\
     &=
     \frac{1}{\epsilon_k} E^0 \otimes \mathcal{N}_{x,r}  \Big(\int_0^{\infty} \dd a \, e^{- \alpha a}
       \mathbbm{1}_{\{\delta< \uptau< a < \uptau + \epsilon_k \}}
     \Phi(\mathcal{J}_a, \widecheck{\mathcal{J}}_a, \big(\overline{\xi}_t:~t\in [0,a]) \big)  \Big).
\end{align}
Since $\mathcal{N}_{x,r}(\uptau>\delta)<\infty$, an application of the dominated convergence theorem  yields that  the term in the right-hand side of the  previous display converges towards the right-hand side of \eqref{equation:manyd}. 
 Now, it remains to show that   
\begin{equation}\label{eq:desired:L:spinal}
\mathbf{N}_{x,r}\Big( \int_0^\sigma \dd L_s  ~ \mathbbm{1}_{\{\uptau(W_s)>\delta \}} \Phi\big( \rho_s,\eta_s, \overline{W}_s\big)  \Big)= \lim_{k  \rightarrow \infty}  \frac{1}{\epsilon_k} \mathbf{N}_{x,r}\Big(\int_0^\sigma \dd s \, \mathbbm{1}_{\{ \delta<\uptau(W_s) < H(\rho_s) < \uptau(W_s) + \epsilon_k \}} \Phi\big( \rho_s,\eta_s, \overline{W}_s\big)  \Big).  
\end{equation}
Since  by Proposition \ref{proposition:L*} we have:
\begin{align}\label{equation:aproxL}
\int_0^\sigma \dd L_s  ~ \mathbbm{1}_{\{\uptau(W_s)>\delta \}}\Phi\big( \rho_s,\eta_s, \overline{W}_s\big)=\lim_{k  \rightarrow \infty} \frac{1}{\epsilon_k }\int_0^\sigma \dd s \, \mathbbm{1}_{\{ \delta< \uptau(W_s) < H(\rho_s) < \uptau(W_s) + \epsilon_k \}}\Phi(\rho_s ,\eta_s, \overline{W}_s ) 
   , \quad  \mathbf{N}_{x,r}\text{-a.e.} ,
\end{align}  
by the general dominated convergence theorem  it suffices to prove  \eqref{eq:desired:L:spinal} for $\Phi$ identically equal to $1$. Note that  Fatou's Lemma and  \eqref{equation:manyLd2}  already give that $\mathbf{N}_{x,r}\big( \int_0^\sigma \dd L_s  ~ \mathbbm{1}_{\{\uptau(W_s)>\delta \}}   \big)$ is bounded above by $\mathcal{N}_{x,r}(\uptau >\delta)<\infty$. Finally,     by the special Markov property of  Corollary~\ref{corollary:special2}, we have:
\begin{align*}
 \frac{1}{\epsilon_k} \mathbf{N}_{x,r}\Big(\int_0^\sigma \dd s \, \mathbbm{1}_{\{ \delta<\uptau(W_s) < H(\rho_s) < \uptau(W_s) + \epsilon_k \}}   \Big) 
&=  \epsilon_k^{-1} \cdot \mathbf{N}_{x,r}\Big(\int_0^\sigma \dd L_s\,\mathbbm{1}_{\{ \uptau(W_s)>\delta \}}\Big) \cdot N\Big( \int_0^\sigma\dd s \mathbbm{1}_{\{ 0< H(\rho_s) < \epsilon_k \}} \Big)  \\
&=  \mathbf{N}_{x,r}\Big(\int_0^\sigma \dd L_s\,\mathbbm{1}_{\{ \uptau(W_s)>\delta \}}\Big) \cdot  \frac{1-\exp(-\alpha\epsilon_k)}{\alpha\epsilon_k},
\end{align*} 
where in the second equality we used \eqref{eq:many:to:one:N} and the  convention $(1-\exp(-\alpha\epsilon_k))/(\alpha\epsilon_k)=1$ if $\alpha=0$.   The convergence \eqref{eq:desired:L:spinal} with $\Phi \equiv 1$ now follows taking the limit as $k\to \infty$ in the last display. 
\end{proof}
We will now make use of Proposition \ref{proposition:ManyToOneL*} to identify the distribution of $L_\sigma$ under $\mathbf{N}_{x,r}$. For $\lambda \geq 0$ and $y \neq x$  we set     $u_\lambda(y) := \mathbb{N}_{y,0}(1-\exp (- \lambda {L}_\sigma))$.  By \cite[Proposition 4.7]{2024structure}, the function 
\begin{equation} \label{definition:PhiTilde}
    \widetilde{\psi}(\lambda) := \mathcal{N}\Big( \int_0^\uptau\dd h ~ \psi\big(u_\lambda(\xi_h)\big)\Big), \quad \text{ for } \, \lambda \geq 0, 
\end{equation}
  is the characteristic exponent of a Lévy tree - in the sense that it is the Laplace exponent of a Lévy process satisfying conditions  (A1) -- (A4).  In particular, it is of  the form 
   \begin{equation}\label{equation:psitilde}
       \widetilde{\psi}(\lambda)=\widetilde{\alpha}\lambda+\widetilde{\beta} \lambda^2+ \int_{(0,\infty)}\widetilde{\pi}(\dd \ell)\big(\exp(-\lambda \ell)-1+\lambda \ell\big), \quad \lambda \geq 0,  
   \end{equation}
   for some constants $\widetilde{\alpha},\widetilde{\beta}\in \mathbb{R}_+$ and a Lévy measure $\widetilde{\pi}$ on $(0,\infty)$  satisfying $\int \widetilde{\pi}(\dd \ell) ~(\ell\wedge \ell^2)<\infty$. Moreover \cite[Corollary 4.9]{2024structure} gives that 
\begin{equation}\label{alpha:beta:tilde}
\widetilde{\beta}=0 \quad \quad \text{ and }\quad \quad \widetilde{\alpha} = \mathcal{N}(1-\exp(-\alpha \uptau)).
\end{equation}
Note that $\widetilde{\psi}$ does not have  Gaussian component.  
\begin{prop} \label{corollary:exponente}
The image measure  of $\mathbf{N}_{x,r}$ by  $L_\sigma$, restricted to $(0,\infty)$, is $\widetilde{\pi}(\dd \ell) $.
\end{prop}
In other words, the distribution of $L_\sigma$ under $\mathbf{N}_{x,r}(\cdot \cap \{L_\sigma>0\})$ is precisely $\widetilde{\pi}(\dd \ell)$. We stress that the proposition provides no information about the measure of the event  $\{L_\sigma =0\}$.    This set  might have positive  $\mathbf{N}_{x,r}$-measure in general.
\begin{proof} Since the Laplace exponent $\widetilde{\psi}$ has no Brownian component, this is equivalent to proving that
    \begin{equation}\label{equation:leyL}
        \mathbf{N}_{x,r}\big( \exp (- \lambda L_\sigma) - 1 + \lambda L_\sigma  \big) = \widetilde{\psi}(\lambda)-\widetilde{\alpha}\lambda, 
    \end{equation}
    for every $\lambda\geq 0$.  In this direction,  note that 
    \begin{align} \label{equation:Phitilde1}
       \mathbf{N}_{x,r}\big( \exp (- \lambda L_\sigma) - 1 + \lambda L_\sigma \big) 
        &= 
         \lambda \cdot \mathbf{N}_{x,r}\left(\int_0^\sigma \dd L_s \,  \Big(1 - \exp ( - \lambda \int_s^\sigma \dd L_u  ) \Big)  \right) \nonumber \\
        &=
        \lambda \cdot  \mathbf{N}_{x,r}\left(\int_0^\sigma \dd L_s \, \mathbb{E}_{\rho_s, \overline{W}_s}^{\dagger} \Big(1 - \exp ( - \lambda \int_0^\sigma \dd L_u  ) \Big)  \right), 
    \end{align}
    where in the last equality we used  the strong Markov property under $\mathbf{N}_{x,r}$. In order to compute  the last expression, recall that the point measure \eqref{equation:PoissonH} under $\mathbb{P}^\dag_{\mu , \overline{\w}}$ is a Poisson point measure with intensity $\mu(\dd h) \mathbb{N}_{\overline{\w}(h)}( \dd \rho , \dd \overline{W} )$.    Lemma  \ref{proposition:L*} and the formula for the Laplace transform of integrals with respect to Poisson point measures  yields that for any $(\mu , \overline{\w}) \in \overline{\Theta}_x$, 
    \begin{equation*}
        \mathbb{E}_{\mu, \overline{\w}}^{\dagger} \big[ \exp ( - \lambda  L_\sigma  )  \big] 
        { = } 
        \exp \Big( - \int_0^{\uptau(\w)} \mu(\dd h )\,  \mathbb{N}_{\w(h)} \big( 1- \exp (- \lambda L_\sigma) \big)    \Big) 
        { = }
        \exp \Big( - \int_0^{\uptau(\w)} \mu(\dd h )\,  u_\lambda\big(\w(h) \big)    \Big). 
    \end{equation*}
    Getting back to \eqref{equation:Phitilde1},  since $(\rho,\overline{W})$ under $\mathbf{N}_{x,0}$ takes values in $\overline{\Theta}_x$, we infer from the previous display and Proposition \ref{proposition:ManyToOneL*} that: 
 \begin{align*}
             \mathbf{N}_{x,r}\big( \exp (- \lambda L_\sigma) - 1 + \lambda L_\sigma  \big) 
        &=
        \lambda \cdot        \mathbf{N}_{x,r}\Big( \int_0^\sigma \dd L_s \, \Big( 1-  \exp \Big( - \int_0^{\uptau(W_s)} \rho_s(\dd h )\,  u_\lambda \big(W_s(h) \big)\Big) \Big)     \Big)\\
        &= \lambda \cdot   E^0\otimes \mathcal{N}_{x,r}  \Big( \exp (- \alpha \uptau ) \Big( 1-  \exp \big( - \int_0^\infty J_\uptau (\dd h )\,  u_\lambda (\xi_h)   \big) \Big)   \Big) \\
        &= \lambda \cdot   E^0\otimes \mathcal{N}_{x,r}  \Big(   \exp(- \alpha \uptau)-  \exp \big( - \int_0^\uptau \dd h \,  \psi (u_\lambda (\xi_h ) \big) / u_\lambda (\xi_h ) \big)     \Big), 
    \end{align*}
    where in the last equality we used that $J_\infty$ under $P^0$ is the jump measure of a subordinator with Laplace exponent $\psi(\lambda)/\lambda - \alpha$. Finally, noting that the last term in the previous display writes:
 \begin{equation*}
     \lambda \cdot  E^0\otimes \mathcal{N} \Big(   1-  \exp \Big( - \int_0^\uptau \dd h \, \psi (u_\lambda (\xi_h ) \big) / u_\lambda (\xi_h ) \Big)\Big)-\lambda \cdot  E^0\otimes \mathcal{N} \big( 1-  \exp(- \alpha \uptau) \big), 
 \end{equation*}
the desired  equality  \eqref{equation:leyL} now follows immediately by applying identity (4.21) in \cite{2024structure}. 
\end{proof}

\section{The excursion process and the  tree coded by the local time}
\label{section:joinlaw}

The goal of this section is to prove that the genealogy of the excursions  as well as their respective boundary sizes  can be   encoded in a random real tree. In the process, our study will shed  light on the law of the excursion process  $\mathcal{E}$ under $\mathbb{N}_{x,0}$,  thereby complementing the statement of Theorem \ref{theorem:excursionPPP} which solely concerned the law of $\mathcal{E}$ under $\mathbb{P}_{0,x,0}$. In this direction, recall 
the definition of $A$ given in \eqref{equation:aproximacionA:def} and for every element $(\upvarrho , \overline{\omega}) = (\upvarrho , \omega, \lambdaC )$ of $\mathbb{D}(\mathbb{R}_+, \mathcal{M}_f(\mathbb{R}_+)\times \mathcal{W}_{\overline{E}})$   write $\widetilde{H}(\upvarrho , \overline{\omega})$ for the process defined by the relation 
\begin{equation}\label{equation:Htildefuncional}
    \widetilde{H}_t(  \upvarrho , \overline{\omega} ) := 
    \widehat{\text{\lambdaC}}_{A^{-1}_t(\upvarrho , \overline{\omega})} , \quad t \geq 0. 
\end{equation}
By convention,  $\widetilde{H}(\upvarrho , \overline{\omega} )$ is set as the path identically  equal to $\widehat{\text{\lambdaC}}_0$ if the convergence in \eqref{equation:aproximacionA:def}
does not hold\footnote{When working under  measures associated to Lévy snakes, this will not occur outside of a negligible set, but we give a general definition in order to have a comprehensible framework.} uniformly  in compacts intervals (with a limit instead of a liminf).  As usual,  when there is no risk of  confusion we will drop the dependence on $(\upvarrho , \overline{\omega})$. The process $ \widetilde{H}$ was studied under $\mathbb{P}_{0,x,0}$ and $\mathbb{N}_{x,0}$ in \cite{2024structure} and for latter use we recall some of its main properties.
\par 
First, recall from \eqref{equation:soporteAditiva} that
\begin{equation}\label{equation:soporteAditiva:2}
     \text{supp } \dd A = [0,\sigma] \setminus \mathcal{C}^* , \quad \text{ under }\mathbb{P}_{0,x,0 } \text{ and } \mathbb{N}_{x,0},
\end{equation}
where $\mathcal{C}^*$ stands for the constancy intervals of $\widehat{\Lambda}$. As a consequence,  the process $\widetilde{H}$ under   $\mathbb{P}_{0,x,0}$ and $\mathbb{N}_{x,0}$ has  continuous sample paths, and   Theorem 5.1 in \cite{2024structure} (see as well the discussion right-after) states that  $\widetilde{H}$  is distributed as  the height process of  a $\widetilde{\psi}$-Lévy process and of a $\widetilde{\psi}$-Lévy excursion  respectively. Hence, under $\mathbb{P}_{0,x,0}$ and $\mathbb{N}_{x,0}$, the real tree $\mathcal{T}_{\widetilde{H}}$ is well-defined, and it is distributed  as a  forest of  $\widetilde{\psi}$-Lévy trees and  as a  $\widetilde{\psi}$-Lévy tree respectively, in the sense of Section \ref{subsection:height}.  Heuristically, since the constancy intervals of $A$ coincide precisely with those of   $\widehat{\Lambda}$, the   tree $\mathcal{T}_{\widetilde{H}}$ is obtained from  $\mathcal{T}_H$ by identifying each component $\mathcal{C}_u$ for $u \in \mathcal{D}$ in a single point, say $[u]$. The height (or distance to the root) of $[u]$ in $\mathcal{T}_{\widetilde{H}}$ is then given by $\widehat{\Lambda}_u$, the value of the local time in the corresponding excursion component. The tree $\mathcal{T}_{\widetilde{H}}$ codes the genealogy of the excursion components, in the sense that if $u \succeq u'$ in $\mathcal{T}_{H}$, we still have $[u] \succeq [u']$ in $\mathcal{T}_{\widetilde{H}}$. This heuristic discussion can be made precise  by making use of the notion of  ``subordination of trees by continuous non-decreasing functions'', introduced and studied in \cite{Subor}, when combined with results from  \cite[Section 5]{2024structure}. We refer to \cite{Subor} for related work and applications to Brownian geometry, as well as to \cite{BertoinLeGallLeJean} for connected results in the context of superprocesses. More precisely, 
we let $\widetilde{d}$ be the pseudo-distance in $\mathcal{T}_H$ defined  by the relation 
\begin{equation*}
    \widetilde{d}(a,b) := \widehat{\Lambda}_a + \widehat{\Lambda}_b - 2 \cdot \widehat{\Lambda}_{a \curlywedge b}, \quad a,b \in \mathcal{T}_H. 
\end{equation*}
For any $a,b \in \mathcal{T}_H$, we shall write  $a \approx b$ if and only if $\widetilde{d}(a,b) = 0$ or equivalently, if  $\widehat{\Lambda}_a = \widehat{\Lambda}_b =  \widehat{\Lambda}_{a \curlywedge b}$. Note that in particular,  for any $a,b \in \mathcal{C}_u$ we have that $a \approx b$ by Lemma \ref{lemma:contanteExcursion}. It readily follows that $\approx$ is an equivalence relation on $\mathcal{T}_H$,  and by  \cite[Theorem 5.1]{2024structure} and the discussion  therein   the metric space $ (\mathcal{T}_H / \approx, \tilde{d})$ pointed at the equivalence class of the root of  $\mathcal{T}_H$ is  a pointed  $\mathbb{R}$-tree  isometric to $(\mathcal{T}_{\widetilde{H}}, d_{\widetilde{H}},p_{\widetilde{H}}(0))$.\footnote{In the sense that there exists an isometry between both metric spaces preserving the roots.} 
\par   
 Before stating   our main result and its  implications we first need to introduce some notation. For every  $(\upvarrho, \overline{\omega}) = (\upvarrho, \omega , \lambdaC )$ in $\mathbb{D}(\mathbb{R}_+, \mathcal{M}_f(\mathbb{R}_+)\times \mathcal{W}_{\overline{E}})$, we partition the collection of debuts points $\mathcal{D}(\upvarrho , {\omega})$ in the two following sets 
\begin{equation*}
    \mathcal{D}_+(\upvarrho , {\omega}):=\{u\in \mathcal{D}(\upvarrho, \omega):~L_\sigma^*(\upvarrho^{u,*},\omega^{u,*})>0\},  \quad  \quad  \text{and} \quad  \quad \mathcal{D}_0(\upvarrho , {\omega}):=\{u\in \mathcal{D}(\upvarrho, \omega) : ~L_\sigma^*(\upvarrho^{u,*}, \omega^{u,*})=0\},  
\end{equation*}
  and  we introduce the point  measure:
  \begin{equation}\label{equation:jumpLevy}
       \widetilde{\mathcal{E}}(\upvarrho , \overline{\omega}) :=  \sum_{u \in \mathcal{D}_+(\upvarrho, \omega)} \delta_{A_{\mathfrak{g}(u)}(\upvarrho, \overline{\omega}) , \,  L_\sigma^*(\upvarrho^{u,*}, \omega^{u,*}) }.
    \end{equation}
 Recall that by  convention  $\mathcal{D}(\upvarrho, \omega)=\emptyset$ when $(\upvarrho, {\omega})$ does not belong to $\mathcal{S}_x$.  Theorem~\ref{theorem:excursionPPP}, Proposition~\ref{corollary:exponente}   and classic properties of Poisson measures yield that under $\mathbb{P}_{0,x,0}$, the measure $\widetilde{\mathcal{E}}$ is a Poisson point measure with intensity $\mathbbm{1}_{\mathbb{R}_+}(t) \dd t \widetilde{\pi}(\dd \ell)$, where we recall from Section~\ref{sec:Spinal:L} that $\widetilde{\pi}$ is the Lévy measure of $\widetilde{\psi}$. Recalling the form of $\widetilde{\psi}$ from \eqref{equation:psitilde} and the fact that  $\int \widetilde{\pi}(\dd \ell)~ (\ell \wedge  \ell^2)<\infty$,   under $\mathbb{P}_{0,x,0}$ we can  construct a $\widetilde{\psi}$-Lévy process by setting  
\begin{equation}\label{E:LevyIto}
\widetilde{X}_t \coloneqq \widetilde{\alpha} t + \int_{[0,t]\times \mathbb{R}_+} \widetilde{\mathcal{E}}^{(c)}( \dd  s,  \dd \ell)~ \ell,\quad t\geq 0,
\end{equation}
for $\widetilde{\alpha}$ as in  \eqref{alpha:beta:tilde} and where  $\widetilde{\mathcal{E}}^{(c)}$ stands for 
the compensated Poisson measure $\widetilde{\mathcal{E}}^{(c)}( \dd  s,  \dd \ell):=  \mathcal{E}(\dd  s,  \dd \ell)- \dd  s \widetilde{\pi}(\dd \ell).$  By considering the first excursion of $\rho$ with  duration greater than $\epsilon$ and the discussion at the end of the proof in \cite[Proposition 4.1]{2024structure}, classic arguments yield that the process $\widetilde{X}$ still makes sense under $\mathbb{N}_{x,0}$.
\begin{theo}\label{theo:X:H:tilde}
Under $\mathbb{N}_{x,0}$,
the process $\widetilde{X}$ is distributed as a $\widetilde{\psi}$-Lévy excursion. Moreover, under $\mathbb{P}_{0,x,0}$ and $\mathbb{N}_{x,0}$, the height process of $\widetilde{X}$ is precisely $\widetilde{H}$.
\end{theo} 
In particular,  Theorem \ref{theo:X:H:tilde} implies that under  $\mathbb{P}_{0,x,0}$ and $\mathbb{N}_{x,0}$, the height process $\widetilde{H}$ is a functional of $\mathcal{E}$. At the end of the section we will address the ``inverse''  problem, which consists in characterising the distribution of $\mathcal{E}$ conditionally on $\widetilde{H}$. 
\par 
Let us briefly discuss the main ingredients involved in 
 its proof. We start by  proving  Theorem \ref{theo:X:H:tilde} under $\mathbb{P}_{0,x,0}$, we then derive the results under $\mathbb{N}_{x,0}$ by standard excursion theory arguments. If under $\mathbb{P}_{0,x,0}$  we write $\widetilde{X}'$ for the $\widetilde{\psi}$-Lévy process  associated with $\widetilde{H}$ through Lemma \ref{coro:reconstruccionRhoX} and \eqref{temps:local:I:p:s},    proving the statement of Theorem \ref{theo:X:H:tilde} under $\mathbb{P}_{0,x,0}$ is equivalent to establishing that $\widetilde{X}'$ and $\widetilde{X}$ are indistinguishables under $\mathbb{P}_{0,x,0}$. Since  $\widetilde{\psi}$ has no  Brownian component,  $\widetilde{X}'$ is a purely discontinuous Lévy process and therefore it suffices to  show that  the respective jump processes $\Delta \widetilde{X}'_t$, $\Delta \widetilde{X}_t$ for $t \geq 0$  coincide. The key now is that since $\widetilde{H}$ is the height process of  $\widetilde{X}'$,  we know from \cite{FractalAspectsofLevyTrees}  that the jumps of $\widetilde{X}'$ are coded in the so-called fractal masses at the branching points of $\mathcal{T}_{\widetilde{H}}$, see \cite[Theorem 4.7]{FractalAspectsofLevyTrees}    and the discussion right after, while on the other hand,  the jumps of $\widetilde{X}$ are by construction given by the variables $L_\sigma(\rho^u, W^u)$, for $u \in \mathcal{D}_+$. We will show that these two   collections coincide by making use of  the exit formula of Corollary \ref{corollary:exitPx}.
\par 
These arguments rely in preliminary results under $\mathbf{N}_{x,r}$ for $r \geq 0$ that we will now address.  We start  with a brief study of the process $(A_t)_{t \geq 0}$ under the measure $\mathbf{N}_{x,r}$, for $r \geq 0$,  and in this direction the following identity will  be often used  in combination with \eqref{equation:soporteAditiva:2}:
\begin{equation}\label{eq:donde:se:anula:Lambda}
\{t\geq 0:~\widehat{\Lambda}_t=0\}= \{t\geq 0:~H(\rho_t)=0\}=\{t\geq 0:~\rho_t=0\}, \quad \text{ under }
\mathbb{P}_{0,x,0} \text{ and } \mathbb{N}_{x,0}. 
\end{equation}
The first equality is a consequence of   Lemma \ref{lemma:tequedasenThetax}  while the second  was already discussed in  \eqref{equation:zeros}.  Recall the notation $(\varepsilon'_k)_{k\geq 0}$ for the sequence used in the definition of $A$ given in \eqref{equation:aproximacionA:def}.  

\begin{lem}\label{lemma:subordinadoNblack}
Fix $r\geq 0$. Under $\mathbf{N}_{x,r}$ and outside of a negligible set, the  convergence 
\begin{equation}\label{equation:aproxAnblack}
    A_t(\rho,\overline{W})= \lim_{k\to \infty} \frac{1}{\varepsilon'_k}  \int_0^t \dd s \int_{r}^{\infty} \dd z \,  \mathbbm{1}_{\{ \tau_z(\overline{W}_s ) < H(\rho_s) < \tau_z(\overline{W}_s) + \varepsilon_k^\prime \}}
\end{equation}
 holds  uniformly in compacts intervals. In particular, the process $(A_t)_{t \geq 0}$ is continuous and non-decreasing. Furthermore, the support of the measure $\dd A$ coincides with the  complement of the constancy intervals of the process $\widehat{\Lambda}$.
\end{lem} 

\begin{proof} 
Recall that $(\rho,\overline{W})$ is a snake path under $\mathbb{N}_{x,0}$. It follows  from  \eqref{equation:aproximacionA} that,  under $\mathbb{N}_{x,0}$ and outside of a null set,  for any subtrajectory $(\rho' , \overline{W}')$ of $(\rho , \overline{W})$ corresponding to some interval $[a,b]$ for $0 \leq a < b < \infty$,  the process   $A^\prime:=A(\rho' , \overline{W}')$ is well defined, where the $\liminf$ in the definition of  $A(\rho', W')$  can be replaced by a limit, and with the convergence holding uniformly on compact intervals.  Moreover, it is plain that  $A_t(\rho', \overline{W}') = A_{(a+t) \wedge b}(\rho,\overline{W}) - A_{a}(\rho,\overline{W})$, for $t \geq 0$. By the latter identity and \eqref{equation:soporteAditiva:2}, we deduce as well that the support of the  measure $\dd A(\rho',\overline{W}')$ is precisely the complement of the constancy intervals of $\widehat{\Lambda}'$.  This analysis applies in particular to every $(\rho^u,\overline{W}^u)$ for $u \in \mathcal{D}$, and an application of Theorem \ref{theorem:exit}  then shows that the statement of the lemma  holds at least for some  $r\geq 0$. To extend it for every $r\geq 0$,  it suffices to remark that, since the distribution of $(\rho,W,\Lambda-r)$ under $\mathbb{N}_{x,r}$ is precisely $\mathbb{N}_{x,0}$,   if any of the previously mentioned properties  failed for some $r'$ it would fail for every $r \geq 0$.
\end{proof}

As an immediate consequence of Lemma~\ref{lemma:subordinadoNblack},  the process  $\widetilde{H}$ is  continuous  under $\mathbf{N}_{x,r}$ for $r \geq 0$.  In particular,  we can consider the associated real tree $\mathcal{T}_{\widetilde{H}}$ which, as usual, is rooted at $p_{\widetilde{H}}(0)$. In the next lemma, we establish the relation between the fractal masses at the branching points of $\mathcal{T}_{\widetilde{H}}$ and the excursion-boundary sizes $L_\sigma(\rho^u, W^u)$, for $u \in \mathcal{D}$. In this direction, let  $\widetilde{v}:\mathbb{R}_+\to \mathbb{R}_+$ be the unique function defined by the relation
\begin{equation}\label{int:tilde:v}
\int_{\widetilde{v}(a)}^\infty \frac{\dd \lambda}{\widetilde{\psi}(\lambda)} = a,\quad  \text{for }a > 0.
\end{equation}
In particular $\widetilde{v}$ is continuous and  we also have $ \widetilde{v}(a) \to  \infty$ as $a \to 0$ since $\widetilde{\psi}$ verifies (A4).
\begin{lem}\label{lemma:convergenceL}
Fix $r\geq 0$. Under $\mathbf{N}_{x,r}$ and for every $\varepsilon>0$, let  $ {\textbf{\emph{m}}} ( \widetilde{H}  , \epsilon )$ be the number of connected components of $\mathcal{T}_{\widetilde{H}}\setminus\{ 0 \}$ having  at least a point  at distance  $\varepsilon$ from the root $p_{\widetilde{H}}(0)$.  Then the following convergence holds 
\begin{equation}\label{eq:lim:L:Z}
L_\sigma(\rho,W)=\lim \limits_{\varepsilon \rightarrow 0} \frac{ {\textbf{\emph{m}}} (\widetilde{H}  , \epsilon )}{\widetilde{v}(\varepsilon)}, \quad \mathbf{N}_{x,r}\text{-a.e.} 
\end{equation}
\end{lem}
\begin{proof}
Since $(\rho,W,\Lambda -r)$ under $\mathbf{N}_{x,r}$ is distributed as $(\rho,W,\Lambda)$ under $\mathbf{N}_{x,0}$ it suffices to prove the result for $r = 0$. We argue under $\mathbf{N}_{x,0}$ and as usual, to simplify notation we  omit  the dependence on $(\rho,W)$ and $(\rho,\overline{W})$  in our functionals. Let us  write $(s_i,t_i)_{i\in \mathcal{I}}$ for the connected components of $\{H(\rho_s)>\uptau(W_s):~s\geq 0\}$ and let $(\rho^i,\overline{W}^i)_{i \in \mathcal{I}}$ be the associated subtrajectories. By the special Markov property of Proposition \ref{proposition:L*}, conditionally  $L_\sigma$, the point measure $ \sum_{i \in \mathcal{I}}\delta_{(L_{s_i}, \rho^i, \overline{W}^i )}$ is a Poisson point measure with intensity $L_\sigma\cdot \mathbb{N}_{x,0}$. 
In particular, if $L_\sigma=0$  the set $\mathcal{T}_{\widetilde{H}} \setminus \{ p_{\widetilde{H}}(0) \}$ is empty (since in that case $\widehat{\Lambda}$  is constant) and there is nothing to prove. In what follows we argue on the event $\{L_\sigma>0\}$. Using again the special Markov property of Proposition \ref{proposition:L*} followed by   \eqref{eq:donde:se:anula:Lambda} under $\mathbb{N}_{x,0}$, we get that the collection $\Lambda^i, i \in \mathcal{I}$, are precisely the subtrajectories  associated with  the connected components of $\{ t \geq 0 : \widehat{\Lambda}_t > 0 \}$. Now  the characterisation of the support of $\dd A$ obtained in Lemma \ref{lemma:subordinadoNblack} gives:
\begin{equation*}
  {\textbf{{m}}} ( \widetilde{H}  , \epsilon )=     \#\big\{i\in \mathcal{I}:~\sup_{t \geq 0} \widehat{\Lambda}_t^i \geq \varepsilon\big\}, 
\end{equation*}
for every $\epsilon > 0$. Since by \eqref{equation:soporteAditiva:2} the processes $\Lambda$,  $\widetilde{H}$ under $\mathbb{N}_{x,r}$ differ by a time-change, the quantities   $\mathbb{N}_{x,r}(\sup \widehat{\Lambda}  > \epsilon)$ and $\mathbb{N}_{x,0}(\sup \widetilde{H} > \epsilon)$ coincide, and we infer from our previous discussion that conditionally on $L_{\sigma}$, the random variable  ${\textbf{{m}}} ({\widetilde{H}}  , \epsilon )$ is Poisson  with intensity $L_\sigma\cdot \mathbb{N}_{x,0}(\sup \widetilde{H} > \epsilon)$. Since $\widetilde{H}$ under $\mathbb{N}_{x,0}$ is the height process of $\widetilde{\psi}$-Lévy  excursion,  \cite[Corollary 1.4.2]{DLG02} gives that the continuous, unbounded function $(\widetilde{v}(a): a > 0)$  is precisely $\mathbb{N}_{x,0}(\sup \widetilde{H} > a)$ for $a > 0$.  In particular,   standard properties of Poisson measures yield that conditionally on $L_\sigma$, the process $P_t :=  \textbf{m} ( \widetilde{H},  e^{-t} )$ for $t \geq 0$ is  a counting process  with continuous predicable compensator given by $\nu_t := L_\sigma \cdot \widetilde{v}(e^{-t})$ for  $t \geq 0$. It is   classical  that  the time-changed process $(P_{\nu^{-1}_t}: t \geq 0)$ is a standard Poisson process, see e.g. \cite[Corollary 25.26]{kallenberg},  and we derive the convergence \eqref{eq:lim:L:Z} by the strong law of large numbers for Poisson processes. 
\end{proof}

We are now in position to prove Theorem \ref{theo:X:H:tilde}. 

\begin{proof}[Proof of Theorem \ref{theo:X:H:tilde}]
Recall that under $\mathbb{P}_{0,x,0}$ and $\mathbb{N}_{x,0}$, we write  $\widetilde{X}^\prime$ for the Lévy process associated with $\widetilde{H}$ through Lemma \ref{coro:reconstruccionRhoX} and \eqref{temps:local:I:p:s}. As discussed above, we start by proving the result under $\mathbb{P}_{0,x,0}$, and to do so it suffices to establish that the  jump processes $\Delta \widetilde{X}'_t$, $\Delta \widetilde{X}_t$ for $t \geq 0$  coincide. We will then derive the statement under $\mathbb{N}_{x,0}$ by standard arguments. 
\par  
Since we already know that $\widetilde{X}$ and  $\widetilde{X}^\prime$ have the same distribution, it  suffices to check that for every jump-time $t \geq 0$ of $\widetilde{X}$, there exists $0 \leq t' \leq t$ such that $\Delta \widetilde{X}_{t} \leq \Delta \widetilde{X}_{t^\prime}^\prime$.  Recall from the construction \eqref{E:LevyIto} in terms of the point measure $\widetilde{\mathcal{E}}$ that the jump-times of  $\widetilde{X}$ are precisely the family $A_{\mathfrak{g}(u)}$, for $u\in \mathcal{D}_+$, and that the respective jumps are given by    $L_\sigma(\rho^u, \overline{W}^u)$, for $u \in \mathcal{D}_+$. We now relate these to the jumps of $\widetilde{X}^\prime$ by exploiting  classical results concerning the fractal masses at branching points on Lévy trees. To this end, first note that by straightforward arguments,  a.s. for every $u \in \mathcal{D}_+$ we have:
\begin{equation*}
    \widetilde{H}_{(A_{\mathfrak{g}(u)}+t)\wedge A_{\mathfrak{d}(u)}}(\rho,\overline{W})= \widetilde{H}_{t}(\rho^u,\overline{W}^u), \quad \text{ for } t\geq 0. 
\end{equation*}
To simplify notation we denote the process in the last display by $\widetilde{H}^u$.  For every $u\in \mathcal{D}_+$,  let $\mathcal{T}_{\widetilde{H}^u}$ be the tree coded by $\widetilde{H}^u$ and write $0_u$ for its root. We  let $\textbf{m}(\widetilde{H}^u , \epsilon )$ be  the number of connected components of $\mathcal{T}_{\widetilde{H}^u} \setminus \{ 0_u \}$ having   at least a point  at distance  $\varepsilon > 0$ from the root $0_u$. By the exit formula 
 \eqref{eq:cor:exit:formula} and Lemma \ref{lemma:convergenceL},  it follows that $\mathbb{P}_{0,x,0}$ a.s. for all $u \in \mathcal{D}_+$ we have  
\begin{equation*}
    \lim_{\varepsilon \rightarrow 0} \frac{\textbf{m}(\widetilde{H}^u , \epsilon )}{\widetilde{v}(\varepsilon)} =  L_{\sigma}(\rho^u,\overline{W}^u).   
\end{equation*}
Still under $\mathbb{P}_{0,x,0}$, for each point $[u] \in \mathcal{T}_{\widetilde{H}}$ with $u \in \mathcal{D}_+$ we let $\widetilde{M}([u],\epsilon)$ be the number of connected components of $\mathcal{T}_{\widetilde{H}} \setminus [u]$. Since $(\rho^u, \overline{W}^u)$ is the subtrajectory of $(\rho , \overline{W})$  associated with $[\mathfrak{g}(u), \mathfrak{d}(u)]$, it is plain that we have   $\widetilde{M}([u], \epsilon) \geq  \textbf{m}(\widetilde{H}^u , \epsilon )$ for every $u \in \mathcal{D}_+$ a.s. Hence, we deduce from the last display that 
\begin{equation*}
    \limsup_{\epsilon \rightarrow 0} \frac{\widetilde{M}([u],\epsilon)}{\widetilde{v}(\varepsilon)} \geq    L_{\sigma}(\rho^u,\overline{W}^u),   
\end{equation*}
for every $u\in \mathcal{D}_+$, $\mathbb{P}_{0,x,0}$ a.s. Since  $\widetilde{v}(\varepsilon) \to \infty$  as $\varepsilon \to 0$, we get that  a.s.  for every $u\in \mathcal{D}_+$, we have   $\widetilde{M}([u],\epsilon) \rightarrow \infty$  as $\epsilon \rightarrow 0$ at least along a subsequence, which gives that the point $[u] \in \mathcal{T}_{\widetilde{H}}$ is a point of infinite multiplicity for $\mathcal{T}_{\widetilde{H}}$.   An application of Theorem 4.7 in \cite{FractalAspectsofLevyTrees} (see as well  the discussion right-after) ensures   that the left-hand side in the previous display  converges a.s. towards $\Delta \widetilde{X}^\prime_{t_u}$, where $t_u = \inf\{ t \geq 0 : p_{\widetilde{H}}(t) = [u]  \}$. Finally,  since $p_{\widetilde{H}}(\{A_{g(u)}\}) = [u]$, it must hold that $t_u \leq A_{g(u)}$. This proves that under under $\mathbb{P}_{0,x,0}$, the processes $\widetilde{X}$, $\widetilde{X}^\prime$ are indistinguishables and therefore $\widetilde{H}$ is the height process of $\widetilde{X}$.  
\par 
Let us now explain how to derive from this the analogous  result under $\mathbb{N}_{x,0}$. Still under $\mathbb{P}_{0,x,0}$, consider      the connected components $(a_i,b_i)_{i \in \mathcal{I}}$ of   $\{ t\geq 0 : \rho_t \neq 0  \}$ and for each $i \in \mathcal{I}$ write $(\rho^i,\overline{W}^i)$ for the subtrajectory associated with $[a_i,b_i]$. Hence, $(\rho^i,\overline{W}^i)_{i\in \mathcal{I}}$ are the excursions of the Lévy snake $(\rho,\overline{W})$ away from $(0,x,0)$. Next, for every $i\in \mathcal{I}$, let us write $\widetilde{\mathcal{E}}_i:=\widetilde{\mathcal{E}}(\rho^i,\overline{W}^i)$, $\widetilde{H}^i:=\widetilde{H}(\rho^i,\overline{W}^i)$ and $\widetilde{X}^i$ for the process defined as in  \eqref{E:LevyIto} replacing $\widetilde{\mathcal{E}}$ by $\widetilde{\mathcal{E}}_i$. Now, we claim that the point measure:
\begin{equation}
\sum \limits_{i\in \mathcal{I}} \delta_{-\widetilde{X}_{A_{a_i}}, \widetilde{X}^i, \widetilde{H}^i}
\end{equation}
is a Poisson point measure with intensity $\mathbbm{1}_{\mathbb{R}_+}(\ell) \mathrm{d}\ell \widetilde{N}(\mathrm{d} X,  \mathrm{d} H)$, where $\widetilde{N}$ is the excursion measure above the minimum of a $\widetilde{\psi}$-Lévy process (with associated local time minus the running infimum, see Section \ref{section:MPonLT}) and  $\widetilde{N}(\mathrm{d} X, \mathrm{d} H)$ stands for the corresponding distribution of the Lévy and height excursion process. Before proving the claim let us explain why the desired result follows from it. Fix $\varepsilon>0$, and consider $(\rho^j,\overline{W}^j)$
the first excursion among $(\rho^i,\overline{W}^i)_{i\in \mathcal{I}}$ satisfying $\sup \widetilde{H}^i>\varepsilon$. Then, the distribution of $(\rho^j,\overline{W}^j)$  is precisely $\mathbb{N}_{x,0}(\cdot~ |\sup\widetilde{H}>\varepsilon)$. Therefore, the law of $(\widetilde{X}^j, \widetilde{H}^j)$ is $\mathbb{N}_{x,r}( \dd \widetilde{X}, \dd \widetilde{H} | \sup \widetilde{H} > \epsilon )$ and  by the claim, this  distribution  coincides with  $\widetilde{N}(\mathrm{d} X,  \mathrm{d} H ~|~\sup\widetilde{H}>\varepsilon)$. Since  this holds for every $\varepsilon>0$ and by \eqref{eq:donde:se:anula:Lambda} we  have $\mathbb{N}_{x,0}(\sup  \widetilde{H}=0)=0$, we derive that under $\mathbb{N}_{x,0}$,  the process $\widetilde{X}$ is a $\widetilde{\psi}$-Lévy excursion and 
 $\widetilde{H}$ is the associated  height process. It remains to establish the claim. In this direction, remark that, by  the support characterization \eqref{equation:soporteAditiva:2} and \eqref{eq:donde:se:anula:Lambda}, the points $a_i,b_i$, $i\in \mathcal{I}$, do not belong to $\{\mathfrak{g}(u):~u\in \mathcal{D}(\rho,W)\}$ and they are  points of both left and right increase for  $A(\rho,\overline{W})$. Further,   since the convergence  \eqref{equation:aproximacionA:def}
 holds uniformly  in compacts intervals (with a limit instead of a liminf), we must have  \begin{equation}\label{A:additive}
 A_{(a_i+t)\wedge b_i}(\rho,\overline{W})= A_{a_i}(\rho,\overline{W})+A_{t}(\rho^i,\overline{W}^i),
\end{equation} 
 for every $i\in \mathcal{I}$ and $t\geq 0$. It is now straightforward to deduce from our previous observations that 
 \begin{equation}\label{X:i:H:i:restriction}
\widetilde{X}^{i}_t= \widetilde{X}_{(A_{a_i}+t)\wedge A_{b_i}}- \widetilde{X}_{A_{a_i}}, \quad \text{ and }\quad \widetilde{H}^{i}_t= \widetilde{H}_{(A_{a_i}+t)\wedge A_{b_i}},
\end{equation}
for every $i\in \mathcal{I}$ and  $t\geq 0$.
Since $\widetilde{X}$ and $\widetilde{X}^\prime$ are indistinguishables and $\widetilde{H}$ is the height process of $\widetilde{X}^\prime$, it follows from excursion theory for Lévy processes that to obtain the claim it suffices to show that $(A_{a_i}, A_{b_i})$, $i\in \mathcal{I}$, are the excursion intervals of $\widetilde{X}^\prime$ above its running  infimum. By \eqref{equation:zeros} this is equivalent to establishing that the intervals  
$(A_{a_i}, A_{b_i})$, $i\in \mathcal{I}$, are the connected components of $\{t\geq 0:~\widetilde{H}_t>0\}$. However, this  follows from definition \eqref{equation:Htildefuncional}, the support characterization \eqref{equation:soporteAditiva:2} and \eqref{eq:donde:se:anula:Lambda}. This completes the proof of the theorem.
\end{proof}
We conclude this section by identifying the conditional distribution of $\mathcal{E}$ knowing $\widetilde{H}$ under $\mathbb{P}_{0,x,0}$ and $\mathbb{N}_{x,0}$.  Note that in particular,  this characterises   the distribution of  $\mathcal{E}$ under $\mathbb{N}_{x,0}$ since we already know  by  Theorem \ref{theo:X:H:tilde} that $\widetilde{H}$ under $\mathbb{N}_{x,0}$ is the height process of a $\widetilde{\psi}$-Lévy excursion. First, by Theorem \ref{theo:X:H:tilde} paired with Lemma \ref{coro:reconstruccionRhoX} and \eqref{temps:local:I:p:s}, we know that $\widetilde{X}$, its jump measure $\widetilde{\mathcal{E}}$, and  $\widetilde{H}$ are functionals of each-other.  Hence, it suffices to characterise the law of $\mathcal{E}$ conditionally on $\widetilde{\mathcal{E}}$ under $\mathbb{P}_{0,x,0}$ and $\mathbb{N}_{x,0}$. Since $\widetilde{\mathcal{E}}$ is the image of $\mathcal{E}$ under the map  $(a, \upvarrho, \omega )\mapsto (a, L_\sigma^* (\upvarrho,\omega))$, the description can now be obtained by  standard  methods. 
\par 
In this direction, by Proposition~\ref{corollary:exponente} and  classic disintegration theorems, there exists  a $\widetilde{\pi}$-a.e. uniquely determined family of probability measures  $(\mathbf{N}_{x}^{*, \ell})_{\ell>0}$ on   $\mathbb{D}(\mathbb{R}_+, \mathcal{M}_f(\mathbb{R}_+) \times \mathcal{W}_E)$ such that the function $\ell\mapsto \mathbf{N}_{x}^{*, \ell}$ is  Borel measurable, for every fixed $\ell>0$ the measure $\mathbf{N}_{x}^{*, \ell}$ is supported on $\{L^*_\sigma= \ell\}$, and we have
\begin{equation*}
    \mathbf{N}_{x}^*=\mathbf{N}_{x}^*( \, \cdot \cap \{L_\sigma^*=0\})+\int_{(0,\infty)} \widetilde{\pi}(\mathrm{d} \ell) \, \mathbf{N}_{x}^{*,\ell}. 
\end{equation*}
Finally, let us write $\mathcal{D}_{\widetilde{X}}$ for the set of  jump-times of $\widetilde{X}$; recall these are is in one-to-one correspondence with the elements of $\mathcal{D}_+$ by the map $u\mapsto A_{\mathfrak{g}(u)}$. If $u\in \mathcal{D}_+$ with $s=A_{\mathfrak{g}(u)}$, we write $(\rho^{(s)},W^{(s)}):= (\rho^{u,*}, W^{u,*})$. 
\begin{cor}\label{corolary:condlaw} Under $\mathbb{P}_{0,x,0}$ and conditionally on $\tilde{\mathcal{E}}$, the point measures 
$$\mathcal{E}_+(\mathrm{d} s, \mathrm{d}\upvarrho, \mathrm{d} \omega):=\mathbbm{1}_{L_\sigma^* (\upvarrho,\omega)>0}\mathcal{E}(\mathrm{d} s, \mathrm{d}\upvarrho, \mathrm{d} \omega) \quad \text{ and } \quad  \mathcal{E}_0(\mathrm{d} s, \mathrm{d}\upvarrho, \mathrm{d} \omega):=\mathbbm{1}_{L_\sigma^* (\upvarrho,\omega)=0}\mathcal{E}(\mathrm{d} s, \mathrm{d}\upvarrho, \mathrm{d} \omega)$$ are independent, and their conditional distributions are as follows:
\begin{itemize}
    \item[\emph{(i)}]The conditional distribution of $\mathcal{E}_+$ given $\tilde{\mathcal{E}}$ is characterised by the relation: 
     \begin{align}\label{equation:condlaw}
         \mathbb{E}_{0,x,0}\Big[  g( \, \widetilde{\mathcal{E}} \, ) \exp \big(-  \langle \mathcal{E}_+,f \rangle  \big)      
        \Big]  
        =  \mathbb{E}_{0,x,0}\Big[  g(\, \widetilde{\mathcal{E}} \, ) \prod_{z \in \mathcal{D}_{\widetilde{X}}} \mathbf{N}_x^{*, \Delta \widetilde{X}_z }\Big( \exp \big( -f(z , \rho, W ) \big) \Big)  \Big],  
    \end{align}
    which holds for every non-negative measurable function $f$ on $\mathbb{R}_+\times \mathbb{D}(\mathbb{R}_+, \mathcal{M}_f(\mathbb{R}_+)\times \mathcal{W}_{E})$.
    \item[\emph{(ii)}]  The point measure  $\mathcal{E}_0$ conditionally on $\widetilde{\mathcal{E}}$ is a Poisson point measure  with intensity measure given by 
    \begin{equation*}
        \mathbbm{1}_{[ 0,A_\sigma )}\dd t \otimes  \mathbf{N}_{x}^*( \, \cdot \cap \{L_\sigma^*=0\}). 
    \end{equation*}
\end{itemize}
Furthermore, the previous statements hold if we replace $\mathbb{P}_{0,x,0}$ by the excursion measure $\mathbb{N}_{x,0}$.
\end{cor}

In other words, point (i) gives that conditionally on $\widetilde{\mathcal{E}}$, the excursions $(\rho^{(z)}, W^{(z)})_{z \in \mathcal{D}_{\widetilde{X}}}$ are independent, and their respective distribution is given by $\mathbf{N}_{x}^{*,\Delta \widetilde{X}_z}$.   Under $\mathbb{P}_{0,x,0}$, the variable $A_\sigma$ is a.s. infinite, and point (ii) then yields that under $\mathbb{P}_{0,x,0}$ the point measure $\mathcal{E}_0$ is independent from $\widetilde{\mathcal{E}}$. Note that under $\mathbb{P}_{0,x,0}$ and $\mathbb{N}_{x,0}$, the variable  $A_\sigma$ is $\widetilde{\mathcal{E}}$ measurable since it is the lifetime of $\widetilde{H}$.  
 
\begin{proof} The statements under $\mathbb{P}_{0,x,0}$ follow immediately  form Theorem \ref{theo:X:H:tilde} and classic properties of Poisson point measures. The result under the measure $\mathbb{N}_{x,0}$ follows by standard arguments as the ones employed in the proof of Theorem \ref{theo:X:H:tilde}. We leave the details to the reader.  
\end{proof}

\section{Consistency with  Abraham-Le Gall's excursion theory}\label{sec:ref:consis:}

We conclude this work verifying the consistency of our excursion theory with the one developed by Abraham and Le Gall  \cite{ALG15} for Brownian motion indexed by the Brownian tree. In this section, we assume that $\psi(\lambda) := \lambda^2/2$ for $\lambda \in \mathbb{R}_+$, that the Polish space  $(E,\dd _E) $ is $\mathbb{R}$ equipped with its euclidean metric, and we set $x := 0$. For $(y,r) \in \mathbb{R}\times \mathbb{R}_+$ we let $\Pi_{y,r}$ be the law of a one-dimensional Brownian motion and its local time at $0$  started from $(y,r)$.  The local time is unique up to a multiplicative constant, that we fix here according to the following  approximation:
\begin{equation*}
    \mathcal{L}_t = \lim_{\epsilon \to 0} \frac{1}{2\epsilon} \int_0^t \dd s \,  1_{\{ |\xi_s| \leq \epsilon \}}, \quad \Pi_{y,r}\text{-a.s.} 
\end{equation*}
Straightforward computations give that assumptions \ref{continuity_snake_2}, \ref{Asssumption_2} and \ref{Asssumption_3} are satisfied, which allows us to consider the measures $\mathbb{N}_{y,r}$ for $(y,r) \in \mathbb{R}\times \mathbb{R}_+$. We write $\mathbf{N}_0^*$ for the corresponding excursion measure at $0$. In this setting,   under $\mathbb{N}_{y,r}, \mathbf{N}_0^*$ and for every  $s \geq 0$, the measure $\rho_s$ is just the Lebesgue measure restricted to $[0, H(\rho_s)]$, and  $H(\rho)$ under $\mathbb{N}_{y,r}$ is distributed as a non-negative Brownian excursion. Since the processes  $H(\rho)$ and $\zeta$ are indistinguishable, in this case  $\rho_t$ is a functional of $\overline{W}_t$. Hence, the exploration process is of no use,  and for this reason  in \cite{ALG15} the results are stated solely in terms of $\overline{W}$ (which is now a Markov process). However, in order to keep the same notation and framework as in the rest of the manuscript, we will still work with  the pair $(\rho,\overline{W})$. 
\par 
The tree-indexed process $(\widehat{W}_{a})_{a \in \mathcal{T}_{H(\rho)} }$ under $\mathbb{N}_{0}$ is the so-called  Brownian motion indexed by the Brownian tree, and the work \cite{ALG15} was devoted to the study of its  excursions away from $0$. In \cite{ALG15}, the authors established that there exists a unique measure $\mathbb{N}^*_0$ on  $\mathbb{D}(\mathbb{R}_+, \mathcal{M}_f(\mathbb{R}_+)\times \mathcal{W}_{\mathbb{R}})$  verifying
\begin{equation*}
    \mathbb{N}_0^*\big( \Phi(\rho, W)\big)=\lim \limits_{\varepsilon \rightarrow 0} \varepsilon^{-1} \mathbb{N}_{\varepsilon,0}\big( \Phi \circ \mathrm{tr}(\rho,W) \big),
\end{equation*}
for every non-negative bounded continuous function $\Phi$ vanishing in 
$$\big\{(\upvarrho,\omega)\in \mathbb{D}(\mathbb{R}_+, \mathcal{M}_f(\mathbb{R}_+)\times \mathcal{W}_{\mathbb{R}}):~ \sup_{s\geq 0} |\widehat{\omega}_s|<\delta\big\},
$$ 
for some $\delta>0$. Then, the following measure was considered:
\begin{equation*}
    \mathbb{M}_0:=\frac{1}{2}\big(\mathbb{N}_0^* +\check{\mathbb{N}}_0^*\big), 
\end{equation*}
where $\check{\mathbb{N}}_0^*$ is the image of $\mathbb{N}_0^*$ by the mapping $(\upvarrho,\omega)\mapsto (\upvarrho,-\omega)$. The measure $\mathbb{M}_0$ was named the excursion measure away from 
$0$ of Brownian motion indexed by the Brownian tree.  Their analysis often relies on the following master formula. For every non-negative measurable functions $g: \mathbb{R} \to \mathbb{R}_+$ and $\Phi : \mathbb{D}(\mathbb{R}_+, \mathcal{M}_f(\mathbb{R}_+)\times \mathcal{W}_{\mathbb{R}}) \to \mathbb{R}_+$, we have  
\begin{equation} \label{equation:averagingN+Localtime}
        \mathbb{N}_{0,0} \Big( \sum_{u \in \mathcal{D}} g(\widehat{\Lambda}_u) \Phi(\rho^{u,*}, W^{u,*}) \Big) = \int_0^\infty \dd \ell \,   g(\ell) 
 \mathbb{M}_{0}( \Phi ).  
\end{equation}
As we shall see,  \eqref{equation:averagingN+Localtime} is a special case of the exit formula proved in Theorem \ref{theorem:exit} and this fact will allow us to relate   the measure $\mathbb{M}_{0}$ with $\mathbf{N}_0^*$. 
\begin{prop}\label{lemma:consistency1}
    The measures $\mathbf{N}_0^*$ and $\mathbb{M}_0$ are identical.  
\end{prop}
\begin{proof}
Let $g$ and $\Phi$ be as in \eqref{equation:averagingN+Localtime}, and consider in the exit formula \eqref{eq:theorem:exit} the functional $\Phi \circ \mathrm{tr}(\upvarrho,\omega)$ for $(\upvarrho , \omega) \in \mathbb{D}(\mathbb{R}_+, \mathcal{M}_f(\mathbb{R}_+) \times \mathcal{W}_{\mathbb{R}})$.   Theorem \ref{theorem:exit} entails  that the left hand-side of \eqref{equation:averagingN+Localtime} equals:
\begin{equation*}
\mathbb{N}_{0,0}\Big( \int_0^\sigma \mathrm{d}A_s  ~ g(\widehat{\Lambda}_s) \cdot  \mathbf{N}_{0,{\widehat{\Lambda}_s}}\big(\Phi \circ \mathrm{tr}( \rho , W)  \big)\Big).     
\end{equation*}
Recalling the definition of $\widetilde{H}$ given in \eqref{equation:Htildefuncional}, we infer by a time-change   that the previous displays is given by: 
$$
\mathbb{N}_{0,0}\Big( \int_0^{A_\sigma} \mathrm{d}s  ~ g(\widetilde{H}_s) \cdot  \mathbf{N}_{0,{\widetilde{H}_s}}\big(\Phi \circ \mathrm{tr}( \rho , W) \big)\Big).
$$
Under $\mathbb{N}_{0,0}$, the process $\widetilde{H}$ is the height process of a $\widetilde{\psi}$-Lévy excursion 
 \cite[Theorem 5.1]{2024structure} and by \eqref{alpha:beta:tilde} we must have $\widetilde{\alpha}=0$. Since $A_\sigma$ is the lifetime of $\widetilde{H}$, an application of \eqref{eq:many:to:one:N}  then gives that the last display writes: 
\begin{equation*}
    \int_0^\infty \dd \ell \,   g(\ell) \cdot
 \mathbf{N}_{0,\ell}\big( \Phi\circ \mathrm{tr}( \rho , W)  \big). 
\end{equation*}
Getting back to  \eqref{equation:averagingN+Localtime}, the desired result follows by recalling  the fact that, for any $\ell \geq 0$, the measure $\mathbf{N}_{0}^*$ is the push-forward measure of  $\mathbf{N}_{0,\ell}$ by the map $(\upvarrho,\omega) \mapsto \mathrm{tr}(\upvarrho,\omega)$.
\end{proof}
Making use of the re-rooting formula \cite[Theorem 28]{ALG15} and  the scaling property of $\mathbb{N}_0^*$ (which are not available for general spatial motions), a notion of exit local time was introduced in \cite{ALG15}. A more precise version, which we now recall, was later proved in the work \cite{Disks}. Specifically, \cite[Corollary 37]{Disks} states that, if for every $(\upvarrho, \omega)\in\mathbb{D}(\mathbb{R}_+, \mathcal{M}_f(\mathbb{R}_+) \times \mathcal{W}_{\mathbb{R}})$, we set: 
 \begin{equation*}
        \widetilde{L}^*_t(\upvarrho,\omega)= \liminf_{\varepsilon \to  0} \varepsilon^{-2} \int_{0}^{t} \mathrm{d} s~ \mathbbm{1}_{ \{  |\widehat{\omega}_s| \, \leq \,  \varepsilon \}},
 \end{equation*}
 then, under the excursion measure $\mathbb{M}_0$, this convergence holds   a.e. uniformly in compact intervals, where the $\liminf$ can be replaced by a limit. In particular, the process $\widetilde{L}^{*}(\rho,W) = (\widetilde{L}^{*}_s(\rho,W), s \geq 0)$ under $\mathbb{M}_0$ is a continuous, non-decreasing functional of $(\rho,W)$  with Stieltjes measure $\dd \widetilde{L}^*$ supported in the set $\{ t \geq 0 : \widehat{W}_t = 0 \}$. As usual, we omit its dependence on $(\rho,W)$ when there is no risk of confusion. 

\begin{prop}\label{prop:equi:L}
Under $\mathbf{N}_0^*$, the processes $\widetilde{L}^{*}$ and $L^*$ are indistinguishables.
\end{prop}

\begin{proof}
For every  $(\upvarrho, \omega)\in \mathbb{D}(\mathbb{R}_+, \mathcal{M}_f(\mathbb{R}_+) \times \mathcal{W}_{\mathbb{R}})$ and $t \geq 0$, we set 
\begin{equation}\label{equation:convergenceLocaltime*}
        \widetilde{L}_t(\upvarrho,\omega) := \liminf_{\varepsilon \to  0} \varepsilon^{-2} \int_{0}^{t} \mathrm{d} s~ \mathbbm{1}_{ \{  |\widehat{\omega}_s| \, \leq \,  \varepsilon, \, H(\upvarrho_s)\leq\uptau(\omega_s)   \}}. 
\end{equation}
Since $\mathbf{N}_{0}^*$ is the push-forward measure of  $\mathbf{N}_{0,0}$ by the map $(\upvarrho,\omega) \mapsto \mathrm{tr}(\upvarrho,\omega)$, we infer from Lemma \ref{lemma:consistency1} and \cite[Corollary 37]{Disks}  that under $\mathbf{N}_{0,0}$ outside of a negligible set, the time-changed process
    \begin{equation}\label{equation:consistecy}
        \epsilon^{-2}\int_0^t\dd s \,  \mathbbm{1}_{\{ |\widehat{W}_{\Gamma_s} |\leq \epsilon \}} , \quad  \text{ for } t \geq 0, 
    \end{equation}
converges uniformly in compact intervals as $\epsilon \rightarrow 0$. This in turn implies that  the convergence in \eqref{equation:convergenceLocaltime*} holds under $\mathbf{N}_{0,0}$  uniformly in compact intervals (with a limit instead of a liminf), and that the limit of \eqref{equation:consistecy} as $\epsilon \rightarrow 0$ is precisely  
$(\widetilde{L}_{\Gamma_t}(\rho,W):t \geq 0)$. Thus, to complete the proof, it suffices to show that $L(\rho,W)$ and $\widetilde{L}(\rho,W)$ are indistinguishable under $\mathbf{N}_{0,0}$. Since  $L(\rho , W)$ and $\widetilde{L}(\rho , W)$ are continuous  with    $L_0(\rho,W)=\widetilde{L}_0(\rho,W)=0$, it is enough to prove  that $L_{\sigma}(\rho,W)-L_s(\rho,W)=\widetilde{L}_{\sigma}(\rho,W)-\widetilde{L}_s(\rho,W)$ for Lebesgue almost every $s\in (0, \sigma)$ or in other words, that under the pointed measure $\mathbf{N}_{0,0}^\bullet$  we have 
\begin{equation}\label{equation:const2}
    L_{\sigma}(\rho,W)-L_U(\rho,W)=\widetilde{L}_{\sigma}(\rho,W)-\widetilde{L}_U(\rho,W).  
\end{equation}
Observe that $\widetilde{L}(\rho , W)$ and $L(\rho , W)$ are constant on each connected component of the set  $\{s\geq 0:~\uptau(W_s)<H(\rho_s)\}$.  Moreover, Lemma \ref{proposition:only:spineN*} gives $\mathbf{N}^\bullet_{0,0}(\uptau(W_U)=H(\rho_U))=0$.   Hence, it suffices to establish the equality on the event  $\{H(\rho_U)<\uptau(W_U)\}$, and for the rest of the proof we work under the measure
\begin{equation*}
    \mathbf{N}^{\bullet}_{0,0}(\, \cdot \,  \cap \{H(\rho_U)<\uptau(W_U)\}).
\end{equation*}
Now, for every $\delta>0$,  set $T_\delta:=\inf\{s\geq U:~H(\rho_s)\leq \delta \}$. Remark that since $H(\rho_s)>0$ for $s\in(0,\sigma)$, the continuity of $\widetilde{L}(\rho , W)$ and $L(\rho , W)$ implies that  $\widetilde{L}_{T_\delta}(\rho,W) \to \widetilde{L}_{\sigma}(\rho,W)$ and ${L}_{T_\delta}(\rho,W) \to {L}_{\sigma}(\rho,W)$  as $\delta\to 0$.  Therefore, it suffices to show that \eqref{equation:const2} holds with $\sigma$ replaced by $T_\delta$, for any $\delta > 0$.  To this end, fix an arbitrary $\delta>0$  and let  $(a_i,b_i)_{i\in \mathcal{I}}$ be  the connected components  of the set $\{t \geq 0 : \, H(\rho_{U+t})- \inf_{s\in [U,U+t]}H(\rho_s)>0\}$. As usual we write $(\rho^i,\overline{W}^i)_{i \in \mathcal{I}}=(\rho^i, W^i, \Lambda^i)_{i \in \mathcal{I}}$ for the corresponding  subtrajectories.   By \eqref{Lebesgue:no:quiere:verlo},  \eqref{equation:convergenceLocaltime*}, Lemma \ref{lem:L:under:P:mu:w}  and  the Markov property given in Proposition \ref{eq:strong:Markov:N},   the proof of  \eqref{equation:const2} with $\sigma$ replaced by $T_\delta$ boils down  to establish that
\begin{equation}\label{eq:final:L:prime}
 \lim_{\varepsilon \to 0 }\sum \limits_{i\in \mathcal{I}} \mathbbm{1}_{ \delta<H(\rho_{a_i})< \uptau(W_U)} \Big|  L_{\infty}(\rho^i,W^i)-\epsilon^{-1}  \int_{a_i}^{b_i} \dd s \,  \mathbbm{1}_{\{|\widehat{W}_s^i|\leq \varepsilon,  
H(\rho^i_s) \leq \uptau(W_s^i)   \}} \Big| = 0, \quad \text{a.e.}
\end{equation}
 In this direction recall that by  the Markov property of Proposition \ref{eq:strong:Markov:N}  and the discussion following \eqref{equation:PoissonH},  the point measure
\begin{equation*}
    \sum_{i \in  \mathcal{I} } \delta_{(H(\rho_{a_i}),  \rho^i ,   \overline{W}^i )}
\end{equation*}
 is a Poisson point measure with intensity $ \rho_U(\dd h) \, \mathbb{N}_{\overline{W}_U( h ) } ( \dd \rho ,  \dd \overline{W} )$. In particular, it follows from \cite[Proposition 34]{Disks} that:
 \begin{equation*}
\lim_{\varepsilon\to 0} \Big|L_\infty(\rho^i,W^i)-\varepsilon^{-1}\int_{0}^{\infty} \dd s \,  \mathbbm{1}_{\{|\widehat{W}_s^i|\leq \varepsilon,  
H(\rho^i_s) \leq \uptau(W_s^i) \}} \Big|=0, \quad \text{ a.e.}
\end{equation*}
 for every $i\in \mathcal{I}$ such that $\delta<H(\rho_{a_i})< \uptau(W_U)$. The desired result  now follows since, by continuity, there exists   $\varepsilon_0>0$  such that there are only finitely many indices $i\in \mathcal{I}$, with $\delta<H(\rho_{a_i})< \uptau(W_U)$, verifying $\inf |\widehat{W}^i|<\varepsilon_0$. 
\end{proof}

\noindent \textbf{Notation index} 
\begin{itemize}
    \itemsep0em 
    \item $\mathbb{D}(\mathbb{R}_+, M)$, for an arbitrary Polish space  $M$,  stands for  the space of $M$-valued rcll paths indexed by $\mathbb{R}_+$, endowed with the Skorokhod topology
     \item $\sigma_h = \inf \big\{t\geq 0:~h(s)=0 \, \text{ for every }s \geq t\big\}$   lifetime of $h \in \mathbb{D}(\mathbb{R}_+, \mathbb{R})$ (Section \ref{sec:trees:1})
      \item $m_h(s,t)$  infimum of $h$ in the interval $[s\wedge t, s\vee t]$ (Section \ref{sec:trees:1})
      \item $d_h$  metric on $\mathcal{T}_h$ (Section \ref{sec:trees:1}) 
      \item $\mathcal{T}_h$ tree coded by the continuous  non-negative function $h:\mathbb{R}_+ \to \mathbb{R}_+$ (Section \ref{sec:trees:1})
    \item $p_h$  canonical projection from $\mathbb{R}_+$ onto $\mathcal{T}_h$ (Section \ref{sec:trees:1})
     \item $\text{Mult}_{i}(\mathcal{T}_h)$  points of multiplicity $i \in \mathbb{N}$ in $\mathcal{T}_h$ (Section \ref{sec:trees:1})

    \item $E$ a Polish space with metric $\dd_E$ (Section \ref{secsnake})
    \item $\xi$  canonical process in $\mathbb{C}(\mathbb{R}_+, E)$, the space of continuous functions indexed by $\mathbb{R}_+$ taking values in $E$ (Section \ref{secsnake})
    \item $\Pi_y$  law of an $E$-valued continuous Markov process started from $y\in E$  (Section \ref{secsnake})
    \item $\mathcal{W}_E$  space of finite $E$-valued paths (Section \ref{secsnake})
    \item $\zeta_\w$ lifetime of $\w \in \mathcal{W}_E$ (Section \ref{secsnake})
    \item $\widehat{\w} := \w(\zeta_\w)$  for  $\w \in \mathcal{W}_E$ (Section \ref{secsnake})  
    \item $\zeta(\omega) = (\zeta_{\omega_s}: s \geq 0)$  lifetime process of a continuous $\mathcal{W}_E$-valued path $\omega$ (Section \ref{secsnake})
    \item $X$  canonical process in $\mathbb{D}(\mathbb{R}_+, \mathbb{R})$ (Section \ref{subsection:height})
    \item $\psi$  Laplace exponent \eqref{equation:psi} of a spectrally positive Lévy process  (Section \ref{subsection:height})
    \item $H$  height process (Section \ref{subsection:height})
    \item $N$ excursion measure at $0$ of the reflected Lévy process $X-I$ (Section \ref{subsection:height})
    
     \item $\mathcal{M}_f(\mathbb{R}_+)$ set of finite measures on $\mathbb{R}_+$  (Section \ref{subsection:explorationprocess})
    \item $H(\mu) := \sup \text{supp } \mu$ for  $\mu \in \mathcal{M}_f(\mathbb{R}_+)$ (Section \ref{subsection:explorationprocess})
    \item $\kappa_a\mu$  pruning operation for $\mu \in \mathcal{M}_f(\mathbb{R}_+)$ and $a \geq 0$ (Section \ref{subsection:explorationprocess})
    \item $[\mu , \nu]$ concatenation of $\mu , \nu \in \mathcal{M}_f(\mathbb{R}_+)$ with $H(\mu) < \infty$ (Section \ref{subsection:explorationprocess})
    \item $\langle \mu , f \rangle$  integral of a measurable function $f: \mathbb{R}_+ \to \mathbb{R}$ with respect to $\mu$  (Section \ref{subsection:explorationprocess})
    \item $\rho^\mu$ exploration process started from $\mu \in \mathcal{M}_f(\mathbb{R}_+)$ (Section \ref{subsection:explorationprocess})
    \item $\textbf{P}_\mu$  law of $\rho^\mu$ for $\mu \in \mathcal{M}_f(\mathbb{R}_+)$ (Section \ref{subsection:explorationprocess})
    \item $\eta$  dual of the exploration process (Section \ref{subsection:explorationprocess})
     \item $\mathcal{M}^{0}_{f}:=\big\{\mu\in \mathcal{M}_{f}(\mathbb{R}_{+}):\:H(\mu)<\infty \ \text{ and } \text{supp } \mu = [0,H(\mu)]\big\}\cup\{0\}$  (Section \ref{subsection:explorationprocess})

    \item $\Theta:=\big\{(\mu, \w) \in \mathcal{M}_f^0 \times \mathcal{W}_{E}:~H(\mu)=\zeta_{\text{w}}\big\}$  subset of  initial conditions for the Lévy snake (Section \ref{section:snake})
    \item $(\rho , W)$  canonical process in $\mathbb{D}(\mathbb{R}_+, \mathcal{M}_f(\mathbb{R}_+) \times \mathcal{W}_E )$ (Section \ref{secsnake})
    \item $\mathbb{P}_{\mu , \w}$  law in $\mathbb{D}(\mathbb{R}_+, \mathcal{M}_f(\mathbb{R}_+)\times \mathcal{W}_E )$ of the Lévy snake started from $(\mu , \w) \in \Theta$ (Section \ref{section:snake})
    \item $\mathbb{N}_y$ excursion measure away from $(0,y) \in \mathcal{M}_f(\mathbb{R}_+) \times \mathcal{W}_E$ of the Lévy snake (Section \ref{section:snake})
    \item $\mathcal{S}_{\mu, \w}$ subset of $\mathbb{D}(\mathbb{R}_+, \mathcal{M}_f(\mathbb{R}_+) \times \mathcal{W}_E )$ of snake paths started from ${(\mu , \w) \, \in \, \Theta }$ (Section \ref{section:snake})
    \item $\mathcal{S} := \bigcup_{(\mu , \w) \, \in \, \Theta } \mathcal{S}_{\mu, \w}$  set of snake paths (Section \ref{section:snake})
   
    \item $\mathcal{N}$ excursion measure of $\xi$ away from $x$ (Section \ref{section:framework})
    \item $\overline{E} = E \times \mathbb{R}_+$  (Section \ref{section:framework})
    \item  $\overline{\w} = (\w, \ell)$  elements of $\mathcal{W}_{\overline{E}}$ (Section \ref{section:framework})
    \item $\overline{\Theta}:= \big\{(\mu, \overline{\w}) \in \mathcal{M}_f^0 \times \mathcal{W}_{\overline{E}}:~H(\mu)=\zeta_{\text{w}}\big\}$ (Section \ref{section:framework})
    \item $\overline{\Theta}_x$  subset of $\overline{\Theta}$ satisfying conditions (i) and (ii) from Section \ref{section:framework}  (Section \ref{section:framework})
    \item  $\Pi_{y,r}$  law of an $E$-valued continuous Markov process and its local time at $x$  started from $(y,r)\in \overline{E}$  (Section \ref{section:framework})
    \item  $(\rho , \overline{W})$  canonical process in $\mathbb{D}(\mathbb{R}_+, \mathcal{M}_f(\mathbb{R}_+) \times \mathcal{W}_{\overline{E}} )$ (Section \ref{section:framework})
    
    
    \item $\uptau(\text{w}):=\inf\big\{t > 0 : ~ \text{w}(t) = x\big\}$  return time to $x$ of  $\w \in \mathcal{W}_E$ (Section \ref{section:truncationboundary})
    \item $V_t(\upvarrho , \omega) := \int_0^t \dd s \,  \mathbbm{1}_{\{  \zeta_{\omega_s} \leq \uptau(\omega_s)    \}}$ for $(\upvarrho , \omega) \in \mathbb{D}(\mathbb{R}_+, \mathcal{M}_f(\mathbb{R}_+) \times \mathcal{W}_E )$  (Section \ref{section:truncationboundary}) 
    \item $\Gamma_s(\upvarrho, \omega):=\inf\big\{t\geq 0 :  V_t(\upvarrho , \omega) > s\big\}$, $s \geq 0$,  right-inverse of $V(\upvarrho , \omega)$ (Section \ref{section:truncationboundary})
    \item $\mathrm{tr}\big(\upvarrho ,\omega\big) :=(\upvarrho_{\Gamma_s(\upvarrho, \omega)},\omega_{\Gamma_s(\upvarrho , \omega)})_{s\in \mathbb{R}_+}$  truncation of $(\upvarrho , \omega) \in \mathbb{D}(\mathbb{R}_+, \mathcal{M}_f(\mathbb{R}_+) \times \mathcal{W}_E )$  (Section \ref{section:truncationboundary})
    \item $L(\upvarrho , \omega) = (L_t(\upvarrho , \omega) : t \geq 0)$ exit local time   (Section \ref{section:truncationboundary})
    \item $\mathcal{D}(\upvarrho, \omega)$ collection of debut times for $(\upvarrho, \omega) \in \mathcal{S}_x$  (Section \ref{sub:sect:debut})
    \item $(\upvarrho^u,\overline{\omega}^u)$ for $u \in \mathcal{D}(\upvarrho, \omega)$ subtrajectory stemming from $u$ for $(\upvarrho, \overline{\omega}) \in \overline{\mathcal{S}}_x$ (Section \ref{sub:sect:debut})
    \item $(\upvarrho^{u,*},{\omega}^{u,*})$ for $u \in \mathcal{D}(\upvarrho, \omega)$ excursion away from $x$ associated with $u$ of $(\upvarrho, \omega) \in \mathcal{S}_x$ (Section \ref{sub:sect:debut})
    \item $\mathbf{N}_{x,r}$ for $r \geq 0$ (Section \ref{section:excursionmeasure})
    \item $\mathbf{N}_{x}^*$ the excursion measure away from $x$ of $(\widehat{W}_a)_{a \in \mathcal{T}_H}$ (Section \ref{section:excursionmeasure})
    \item $(U^{(1)}, U^{(2)})$ a two-dimensional subordinator with exponent \eqref{identity:exponenteSubord} (Section \ref{section:MainSpinalDecomp})
    \item $(J_a , \widecheck{J}_a)  := \big(\mathbbm{1}_{[0,a]}(t)~ \dd U^{(1)}_t , \mathbbm{1}_{[0,a]}(t) ~\dd U^{(2)}_t \big)$  Lebesgue-Stieltjes measure of $(U^{(1)}, U^{(2)})$ restricted to $[0,a]$, for $a \geq 0$  (Section \ref{section:MainSpinalDecomp})
    \item $ \big(\upvarrho, \overline{\omega}\big)^{s,\leftarrow}$ and  $\big(\upvarrho, \overline{\omega}\big)^{s,\rightarrow}$ pieces of path before and after $s \geq 0$ of $(\upvarrho , \overline{\omega}) \in \mathbb{D}(\mathbb{R}_+, \mathcal{M}_f(\mathbb{R}_+)\times \mathcal{W}_{\overline{E}})$ \\(Section \ref{section:MainSpinalDecomp})
    \item $A(\upvarrho, \overline{\omega}) = (A_t(\upvarrho, \overline{\omega}))_{t \geq 0}$ the local time at $x$ of $(\widehat{\omega}_a)_{a \in \mathcal{T}_H}$  for $(\upvarrho, \overline{\omega}) \in \mathbb{D}(\mathbb{R}_+, \mathcal{M}_f(\mathbb{R}_+) \times \mathcal{W}_E )$  (Section~\ref{section:additivefuncionals})    
    \item $\texttt{T}_u(\rho, \overline{W})$ the Lévy snake ``trimmed'' from the subtrajectory  stemming from $u \in \mathcal{D}$ (Section \ref{subsection:exitformula})
    \item $\mathcal{E}(\upvarrho, \overline{\omega})$  the excursion process of $(\upvarrho, \overline{\omega}) \in \mathbb{D}(\mathbb{R}_+, \mathcal{M}_f(\mathbb{R}_+)\times \mathcal{W}_{\overline{E}})$ (Section \ref{section:Poisson})
    \item $L^* = (L^*_t: t \geq 0)$  (Section \ref{section:L})
    \item $\widetilde{H}(\upvarrho, \overline{\omega})$  for $(\upvarrho, \overline{\omega}) \in  \mathbb{D}(\mathbb{R}_+, \mathcal{M}_f(\mathbb{R}_+) \times \mathcal{W}_{\overline{E}})$  (Section \ref{section:joinlaw})
    \item $\widetilde{\psi}$  Laplace exponent of the tree coded by the local time (Section \ref{section:joinlaw})

\end{itemize}

\bibliographystyle{siam}

\end{document}